\documentclass[11pt,reqno]{amsart}

\usepackage{amssymb,mathrsfs,mathtools,dsfont}
\usepackage{upgreek}
\usepackage[foot]{amsaddr}
\usepackage[usenames,dvipsnames]{xcolor}
\usepackage{float,subfig}
\usepackage{graphicx,epsfig}
\usepackage{makecell}
\usepackage[hidelinks]{hyperref}
\usepackage{cite}
\hypersetup{
    colorlinks,
    linkcolor={red!50!black},
    citecolor={blue!50!black},
    urlcolor={blue!80!black}
}

\usepackage[top=1in, bottom=1.25in, left=1.2in, right=1.2in,footskip=.5in]{geometry}
\usepackage{tikz}

\newtheorem{theorem}{Theorem}[section] 
\newtheorem{corollary}[theorem]{Corollary}
\newtheorem{proposition}[theorem]{Proposition} 
\newtheorem{lemma}[theorem]{Lemma}
\newtheorem{linpro}{Linear Program}
\newtheorem*{theoremLambda1-Betti}{Theorem 1.1}
\newtheorem*{theoremLambda1}{Theorem 1.2}
\newtheorem*{theoremAs}{Theorem 1.3}
\newtheorem*{theoremSys}{Theorem 1.4}
\theoremstyle{definition}
\newtheorem{definition}[theorem]{Definition}

\newtheorem{problem}[theorem]{Open Problem}
\theoremstyle{remark}
\newtheorem{remark}[theorem]{Remark}
\numberwithin{equation}{section}
\renewcommand{\leq}{\leqslant}
\renewcommand{\geq}{\geqslant}
\renewcommand{\hat}{\widehat}

\newcommand{\hh}{{\mathbb{H}^3}}
\newcommand{\lap}{{\Delta_{M}}}
\newcommand{\xb}{{\bar{x}}}
\newcommand{\hb}{\widetilde{h}}
\newcommand{\be}{\begin{equation}}
\newcommand{\ee}{\end{equation}}
\newcommand{\ba}{\begin{equation}\begin{aligned}}
\newcommand{\ea}{\end{aligned}\end{equation}}
\newcommand{\nn}{\nonumber}
\newcommand*{\pFq}[5]{{}_{#1}F_{#2}\left(\genfrac..{0pt}{}{#3}{#4};#5\right)}
\makeatletter
\newcommand{\thickhline}{%
    \noalign {\ifnum 0=`}\fi \hrule height 1pt
    \futurelet \reserved@a \@xhline
}
\makeatother

\begin{document}

\title[Spectral Bounds on Hyperbolic 3-Manifolds]{Spectral Bounds on Hyperbolic 3-Manifolds: Associativity and the Trace Formula}

\author{James Bonifacio\textsuperscript{1}}
\email{bonifacio@phy.olemiss.edu}

\author{Dalimil Maz\'{a}\v{c}\textsuperscript{2}}
\email{dmazac@ias.edu}

\author{Sridip Pal\textsuperscript{2,3}}
\email{sridip@caltech.edu}

\address[1]{Department of Physics and Astronomy, University of Mississippi, University, MS 38677, USA}

\address[2]{Institute for Advanced Study, Princeton, NJ 08540, USA}

\address[3]{Walter Burke Institute for Theoretical Physics, California Institute of Technology, Pasadena, CA,
USA}

\begin{abstract}
We constrain the low-energy spectra of Laplace operators on closed hyperbolic manifolds and orbifolds in three dimensions, including the standard Laplace--Beltrami operator on functions and the Laplacian on powers of the cotangent bundle. Our approach employs linear programming techniques to derive rigorous bounds by leveraging two types of spectral identities. The first type, inspired by the conformal bootstrap, arises from the consistency of the spectral decomposition of the product of Laplace eigensections, and involves the Laplacian spectra as well as integrals of triple products of eigensections. We formulate these conditions in the language of representation theory of $\mathrm{PSL}_2(\mathbb{C})$ and use them to prove upper bounds on the first and second Laplacian eigenvalues. The second type of spectral identities follows from the Selberg trace formula. We use them to find upper bounds on the spectral gap of the Laplace--Beltrami operator on hyperbolic 3-orbifolds, as well as on the systole length of hyperbolic 3-manifolds, as a function of the volume.  Further, we prove that the spectral gap $\lambda_1$ of the Laplace--Beltrami operator on all closed hyperbolic 3-manifolds satisfies $\lambda_1 < 47.32$. Along the way, we use the trace formula to estimate the low-energy spectra of a large set of example  orbifolds and compare them with our general bounds, finding that the bounds are nearly sharp in several cases.
\end{abstract}

\maketitle

\newpage
\setcounter{tocdepth}{1}
\tableofcontents

\section{Introduction}\label{sec:introduction}

\subsection{Hyperbolic 3-manifolds}
In this paper, we study spectra of Laplace operators on closed hyperbolic 3-manifolds and 3-orbifolds. A hyperbolic 3-manifold is a Riemannian 3-manifold with constant sectional curvature $-1$. By Thurston's geometrization conjecture~\cite{Thurston82}, proved by Perelman~\cite{Perelman2002}, every closed 3-manifold can be canonically decomposed into pieces admitting one of eight possible geometries. Of these eight geometries, the hyperbolic manifolds correspond, in a sense, to the most abundant one and are the most difficult to classify. This means that hyperbolic manifolds are crucial for understanding geometry and topology in three dimensions.

Every complete, connected, simply-connected hyperbolic 3-manifold is isometric to the hyperbolic space $\hh$. This manifold can be realized as the following locus:
$$
\hh = \{x\in\mathbb{R}^{1,3}\colon\,x_0^2-x_1^2-x_2^2-x_3^2 = 1,\,x_0>0\},
$$
with Riemannian metric induced from the standard Lorentzian metric on $\mathbb{R}^{1,3}$. The group of orientation-preserving isometries of $\hh$ is $G = \mathrm{SO}_{0}(1,3)\cong\mathrm{PSL}_2(\mathbb{C})$, the connected component of $\mathrm{SO}(1,3)$ containing the identity.
The stabilizer in $G$ of a point in $\hh$ is a copy of the maximal compact subgroup $K = \mathrm{SO}(3)$, so we have $\hh = G/K$. It follows that every closed, connected, orientable hyperbolic 3-manifold can be realized as a quotient $M = \Gamma\backslash\hh$, where $\Gamma$ is a torsion-free discrete subgroup of $G$. More generally, we will be interested in hyperbolic 3-orbifolds, which amounts to allowing $\Gamma$ to contain elements of finite order. Thus, in the rest of this paper, we consider $M = \Gamma\backslash\hh$, where $\Gamma$ is a co-compact discrete subgroup of $G$. We will refer to any such quotient as a hyperbolic 3-orbifold, dropping the adjectives closed, connected and orientable for brevity.

\subsection{Laplace eigenvalues}

Our main object of interest will be the spectrum of the Laplacian $\lap  := -\mathrm{div}\circ \nabla$ acting on smooth functions on $M$, as well as on smooth sections of powers of the cotangent bundle of $M$.

The Laplace equation on functions reads
\be\label{eq:laplaceScalar}
\lap \varphi^{(0)}_k(x) = \lambda^{(0)}_k\,\varphi^{(0)}_k(x)\,,
\ee
where $\varphi^{(0)}_k(x)$ are smooth functions on $\hh$ satisfying $\varphi^{(0)}_k(\gamma x) = \varphi^{(0)}_k(x)$ for all $\gamma\in\Gamma$ and all $x\in\hh$. 
The eigenvalues $\lambda^{(0)}_k$ form a non-decreasing discrete subset of $\mathbb{R}_{\geq 0}$,
$$
0 = \lambda^{(0)}_0<\lambda^{(0)}_1\leq\lambda^{(0)}_2\leq\ldots\rightarrow\infty\,.
$$
It is useful to write the eigenvalues as $\lambda^{(0)}_k = (t^{(0)}_k)^2+1$, where $t^{(0)}_k \in \mathbb{R} \cup i(0,1]$. The superscript $(0)$ reminds us that this is the spectrum on functions.

More generally, it is natural to consider the spectral problem for the Laplacian on powers of the cotangent bundle $(T^{*}M)^{\otimes J}$, where $J\in \mathbb{Z}_{\geq 0}$. The case of functions corresponds to $J=0$. A smooth section of $(T^{*}M)^{\otimes J}$ means a smooth section of $(T^{*}\hh)^{\otimes J}$ invariant under the pull-back by any element of $\Gamma$. A section $\varphi^{(J)}$ of $(T^{*}M)^{\otimes J}$ is called \emph{symmetric} if $\varphi^{(J)}(X_1,\ldots,X_J)$ is a symmetric function of the vector fields $X_1,\,\ldots X_J$, and \emph{traceless} if it vanishes upon contraction with the inverse metric in any pair of entries. Finally, we call $\varphi^{(J)}$ \emph{divergence-free} if ${\mathrm{div}(\varphi^{(J)}) = 0}$.

The action of the Laplacian on smooth sections of $(T^{*}M)^{\otimes J}$ preserves the properties of symmetry, tracelessness and vanishing divergence. Thus, we will consider the family of spectral problems

\be\label{eq:spinJequation}
\lap \varphi^{(J)}_{k}(x) = \lambda^{(J)}_{k}\,\varphi^{(J)}_{k}(x)\,,
\ee
where the eigensections $\varphi_{k}^{(J)}$ are smooth, symmetric, traceless and divergence-free sections of $(T^{*}M)^{\otimes J}$. We will see in Section~\ref{sec:rep-theory} how these conditions follow naturally from the representation theory of $G$. The eigenvalues $\lambda^{(J)}_k$ for fixed $J>0$ form a non-decreasing discrete subset of $\mathbb{R}$
$$
\lambda_1^{(J)}\leqslant \lambda_2^{(J)}\leqslant\lambda_3^{(J)}\leqslant\ldots\rightarrow\infty\,.
$$
For $J>0$, we will also assume $\varphi^{(J)}_{k}$ are eigensections of the curl operator
\be\label{eq:curl}
\mathrm{curl}\, \varphi^{(J)}_{k}(x) = t_k^{(J)} \varphi^{(J)}_{k}(x)\,.
\ee
Here $\mathrm{curl}$ is a certain linear differential operator generalizing the standard curl acting on vector fields to any $J>0$. It is described precisely in Definition~\ref{def:curl}. Since curl is Hermitian with respect to the standard inner product on sections of $(T^{*}M)^{\otimes J}$, we have $t_k^{(J)}\in\mathbb{R}$. Furthermore, $ \lap\varphi^{(J)} = (\mathrm{curl}\circ\mathrm{curl} + J + 1)\varphi^{(J)}$ for any symmetric, traceless, divergence-free $\varphi^{(J)}$. It follows that any solution of~\eqref{eq:curl} is also a solution of~\eqref{eq:spinJequation}. The curl eigenvalues are related to the Laplace eigenvalues by ${\lambda^{(J)}_k = (t_k^{(J)})^2+J+1}$ and hence $\lambda_1^{(J)}\geqslant J+1$ for all $J>0$. We refer to the sign of $t_k^{(J)}$ as the chirality of the corresponding eigensection. Following physics terminology, we will often refer to $J$ as the spin and $\lambda_k^{(J)}$ or $t_k^{(J)}$ as the spin-$J$ eigenvalues, including the case $J=0$.

Eigensections with $J=1$ and $t_k^{(1)}=0$ are precisely the harmonic 1-forms. The number of linearly independent harmonic 1-forms is the first Betti number of $M$. More generally, the degeneracy of eigenvalues with $t_k^{(J)}=0$ for $J> 0$ can be computed as the dimension of a certain cohomology space~\cite{Cohomology}. 

\subsection{Bounds on the spectrum}
Our main results are new bounds on the Laplace spectra of hyperbolic 3-orbifolds. Previous works studying the Laplace spectra of hyperbolic 3-manifolds and 3-orbifolds include~\cite{Cheng1975, Schoen1982, Sarnak1983, Callahan-thesis, Burger-Sarnak1991, Aurich-Marklof, Inoue:1998nz, Cornish1999, Inoue2001, White2013, Lipnowski2016, Lin2017, Lin-Lipnowski2018, Hamenstadt2019, Baik2020, Lin-Lipnowski2020, Lin-Lipnowski2021, rudd2023}. We will employ two distinct approaches.

The first approach uses the associativity of the product of eigenfunctions to derive spectral identities satisfied by the eigenvalues and triple products of eigenfunctions, as in~\cite{Bonifacio:2019ioc, Bonifacio:2020xoc, Bonifacio:2021msa, Kravchuk:2021akc, Bonifacio:2021aqf}. We will systematize this approach in the language of unitary representation theory of ${G=\mathrm{PSL}_2(\mathbb{C})}$, following the case of hyperbolic 2-orbifolds and $\mathrm{PSL}_2(\mathbb{R})$ studied in~\cite{Kravchuk:2021akc}.

From these spectral identities, we derive computer-assisted bounds on the spectrum using linear and semidefinite optimization. The resulting bounds are called bootstrap bounds, in analogy with the conformal bootstrap program from conformal field theory, which inspires this approach~\cite{Belavin:1984vu, Rattazzi:2008pe, Rychkov:2009ij}, see~\cite{Poland:2018epd, Hartman:2022zik} for recent reviews. In that context, $G$ plays the role of the group of conformal isomorphisms of $S^2$ and the space of $L^2$ functions on $\Gamma\backslash G$ is replaced by the Hilbert space of physical states of a conformal field theory.

As a special case of the bootstrap bounds, we will prove the following result:
 \begin{theorem} \label{thm:lambda1-betti}
Let $M$ be a closed oriented hyperbolic 3-orbifold with positive first Betti number. The first positive eigenvalue of the Laplace--Beltrami operator on $M$ satisfies
$$
\lambda^{(0)}_1 < \frac{252549}{8000} \approx 31.569.
$$
Similarly, the first eigenvalue $\lambda_1^{(2)}$ of the $J=2$ spectral problem \eqref{eq:spinJequation} satisfies
$$
\lambda_1^{(2)}<\frac{46685}{8000}\approx 5.8356\,.
$$
\end{theorem}
These upper bounds are quite close to the values for the orbifold $M=\mathrm{m203}(0,3)(2,0)$, of approximate volume $0.4866$, for which $\lambda^{(0)}_1\approx 30.84$ and $\lambda^{(2)}_1\approx 5.8242$. The first Betti number of $\mathrm{m203}(0,3)(2,0)$ equals one.\footnote{The notation $\mathrm{m203}(0,3)(2,0)$ means that this orbifold is obtained from the cusped manifold m203 by filling its two cusps using Dehn fillings with coefficients $(0,3)$ and $(2,0)$. Here $\mathrm{m203}$ is the name assigned to the manifold in the Hodgson--Weeks census (it is the complement of the link $6_2^2$ in $S^3$).}

More generally, the bootstrap bounds we derive take the following form:
\begin{enumerate}
\item an upper bound on $|t_1^{(J)}|$ as a function of $|t_1^{(J')}|$ with $J \neq J'$, \label{list:bounds-1}
\item an upper bound on $|t_2^{(J)}|$ as a function of $|t_1^{(J)}|$. \label{list:bounds-2}
\end{enumerate}
The first and second part of Theorem~\ref{thm:lambda1-betti} will follow as special cases of the bounds of type~(\ref{list:bounds-1}) for $J=0,2$ and $J'=1$ by setting $|t_1^{(1)}| = 0$, which is equivalent to the positivity of the first Betti number.

\begin{figure}[t!]
\begin{center}
\epsfig{file=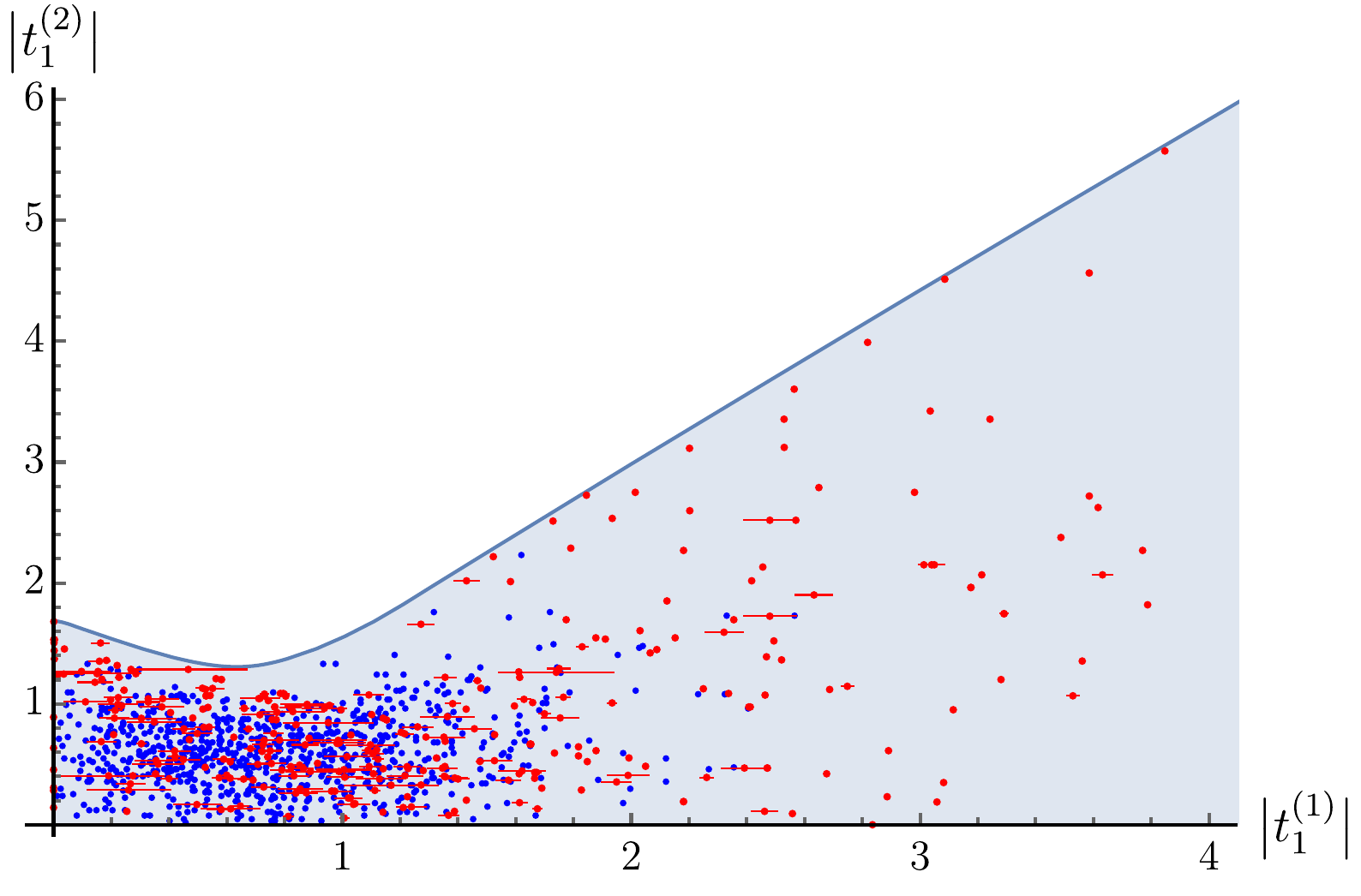,scale=.4}
\end{center}
\caption{The bootstrap upper bound on the first spin-2 eigenvalue $|t_1^{(2)}|$ for varying value of the first spin-1 eigenvalue $|t_1^{(1)}|$ for closed orientable hyperbolic 3-orbifolds. The region above the blue curve is excluded by the bootstrap method. The dots represent estimates of the values $|t_1^{(1)}|,\,|t_1^{(2)}|$ for various known 3-orbifolds, derived from their length spectra using the Selberg trace formula.}
\label{fig:bound_1_2}
\end{figure}

\begin{table}
\centering
  \begin{tabular}{  | c| c |c | c| c|c|}
  \thickhline
   Name & Approximate volume& $ \big| t^{(1)}_1 \big| $ & $\big| t^{(2)}_1 \big|$ & bootstrap bound  \\ \hline
 m$203(0,3)(2,0)$ & 0.4866 & 0 & 1.6805256(5) & 1.68386 \\ 
  m$032(3,0)$ & 0.6542 & 1.84475(2) & 2.72630(2) & 2.75673 \\ 
   m$015(3,0)$ & 0.3142 & 3.08538(2) & 4.5139(5) & 4.53661 \\ 
   orb$(0.1324, 1_1^2.20)$ &0.1342 & 3.847360(3) & 5.5744(2) & 5.60879 \\ \hline
\thickhline
\end{tabular}
\caption{Example orbifolds that are close to saturating the bound of Figure~\ref{fig:bound_1_2}. The parentheses denote the uncertainty in the last digit of the eigenvalue such that the corresponding interval contains the eigenvalue. The last column gives the bootstrap upper bound on the spin-2 eigenvalue for each orbifold, fixing $\big| t^{(1)}_1 \big|$ as the maximum of its possible values, computed using 101 spectral identities. orb$(0.1324, 1_1^2.20)$ is the small volume orbifold from Table 4.1 of Heard's thesis \cite{Heard-thesis} with volume $V \approx 0.1324$ and singular set described by the trivalent graph denoted $1_1^2.20$ in Table~4.2 of \cite{Heard-thesis}.
}
\label{tab:bound_1_2-examples}
\end{table} 

As a further example of bounds of type~(\ref{list:bounds-1}), we show in Figure~\ref{fig:bound_1_2} the bootstrap upper bound on the first spin-2 eigenvalue $|t_1^{(2)}|$ in terms of the first spin-1 eigenvalue $|t_1^{(1)}|$. The plot also shows many example orbifolds, displayed as dots, some of which are quite close to the bound.\footnote{The vertical position of the dots in Figure~\ref{fig:bound_1_2} is a lower bound on $|t_1^{(2)}|$. The red dots have horizontal error bars showing the value and uncertainty of $|t_1^{(1)}|$ (some of which are too small to be visible). The horizontal position of the blue dots is a lower bound on the value of $|t_1^{(1)}|$.} In Table~\ref{tab:bound_1_2-examples}, we list the eigenvalues of some of these examples and the corresponding bounds.
 Additional bounds of type (\ref{list:bounds-1}), with $(J, J')=(0, 2), (1, 0), (2, 0)$ and $(2,4)$, can be found in Figures~\ref{fig:bound_0_2},~\ref{fig:bound_1_0},~\ref{fig:bound_2_0},~and~\ref{fig:bound_2_4}. Bounds of type (\ref{list:bounds-2}) with $J=0$ and $J=2$ can be found in Figures~\ref{fig:bound_0_0}~and~\ref{fig:bound_2_2}.

Our second approach makes use of the Selberg trace formula. This formula relates a sum over the length spectrum of an orbifold, weighted by a test function $H$, to a sum over the eigenvalue spectrum, weighted by the Fourier transform of $H$; see Theorems~\ref{thm:trace-even} and~\ref{thm:trace-odd} in Section~\ref{subsec:STF} for precise statements. By optimizing over certain spaces of test functions, we can derive bounds on the spectra of specific orbifolds, as well as general bounds.

To use the trace formula to compute the spectra of specific orbifolds, we follow and extend the procedure of Booker--Str\"ombergsson and Lin--Lipnowski from~\cite{Booker-Strombergsson2007, Lin-Lipnowski2018, Lin-Lipnowski2020, Lin-Lipnowski2021}. We have used this approach to compute numerical bounds on the low-lying eigenvalues for many small-volume orbifolds. In particular, we used it to produce the dots of Figure~\ref{fig:bound_1_2} and the third and fourth column of Table~\ref{tab:bound_1_2-examples}. This procedure takes as input the complex length spectrum of an orbifold. The complex length spectrum can be computed numerically using the program \texttt{SnapPy}~\cite{SnapPy}, which employs the algorithm from~\cite{Hodgson-Weeks1994}. A summary of orbifolds with large $|t_1^{(J)}|$ for $J=0,\dots, 4$ can be found in Table~\ref{tab:large-eigenvalues}. The wealth of spectral data we obtain may also prove useful for other applications. 

To use the trace formula to derive general bounds, we apply linear programming to it. This follows an approach employed recently by Fortier Bourque and Petri to bound kissing numbers on hyperbolic manifolds~\cite{bourque2019kissing}, as well as the spectral gap (and a wealth of other invariants) of hyperbolic surfaces~\cite{bourque2023linear}. This idea goes back to the work of Cohn and Elkies~\cite{MR1973059}, who applied linear programming to the Poisson summation formula to prove bounds on density of sphere packings in $\mathbb{R}^n$.

We first find a numerical upper bound on $\lambda^{(0)}_1$ for manifolds or orbifolds of a given volume, shown in Figure~\ref{fig:selberg-upper-bound}. 
We then prove the following rigorous upper bound on the spectral gap of all hyperbolic 3-manifolds, using the fact that the smallest hyperbolic 3-manifold is the Weeks manifold~\cite{Gabai2009}, which has volume $V \approx 0.9427$:
\begin{theorem} \label{thm:gap-bound}
Let $\lambda^{(0)}_1$ be the first nonzero eigenvalue of the Laplace--Beltrami operator on an oriented closed hyperbolic 3-manifold. This satisfies the following upper bound:
$$
\lambda^{(0)}_1 < \frac{193785}{4096} \approx 47.3108.
$$
\end{theorem}
\noindent We similarly show that $\lambda^{(0)}_1 < 400$ for closed hyperbolic 3-orbifolds. In addition, we prove the following upper bound on the spectral gap for orbifolds of large volume:

\begin{theorem}\label{thm:lambda1As}
For any $\varepsilon>0,$  there exists $V_*>0$ such that for all oriented closed hyperbolic 3-orbifolds with volume $V>V_{*}$, we have
\begin{equation*}\label{eq:1}
\lambda^{(0)}_1< 1+ \frac{54\left(1+\varepsilon\right)}{\left[ \log V\right]^2}\,.
\end{equation*}
\end{theorem}
\noindent Finally, we also find upper bounds on the systole length (the length of the shortest closed geodesic) of closed hyperbolic 3-manifolds of a given volume, following~\cite{bourque2023linear}. Previous works studying bounds on the systole length of hyperbolic 3-manifolds include~\cite{Gromov83, adams_reid_2000, White2002, Katz-Schaps-Vishne2007, Gendulphe2011, lakeland2014, Murillo-thesis}. Let $\mathrm{sys}(M)$ denote the systole length of $M$. A plot of our numerical bound on $\mathrm{sys}(M)$ for small-volume manifolds can be found in Figure~\ref{fig:selberg-systole-bound}, while the bound for large volume is given by the following: 
\begin{theorem}\label{thm:systole}
There exists $V_*>0$ such that for all closed oriented hyperbolic 3-manifolds $M$ of volume $V>V_*$, we have
\begin{equation*}\label{eq:2}
\mathrm{sys}(M)< \log(V)- 0.8482\,.
\end{equation*}
\end{theorem}
\begin{figure}[H]
\begin{center}
\epsfig{file=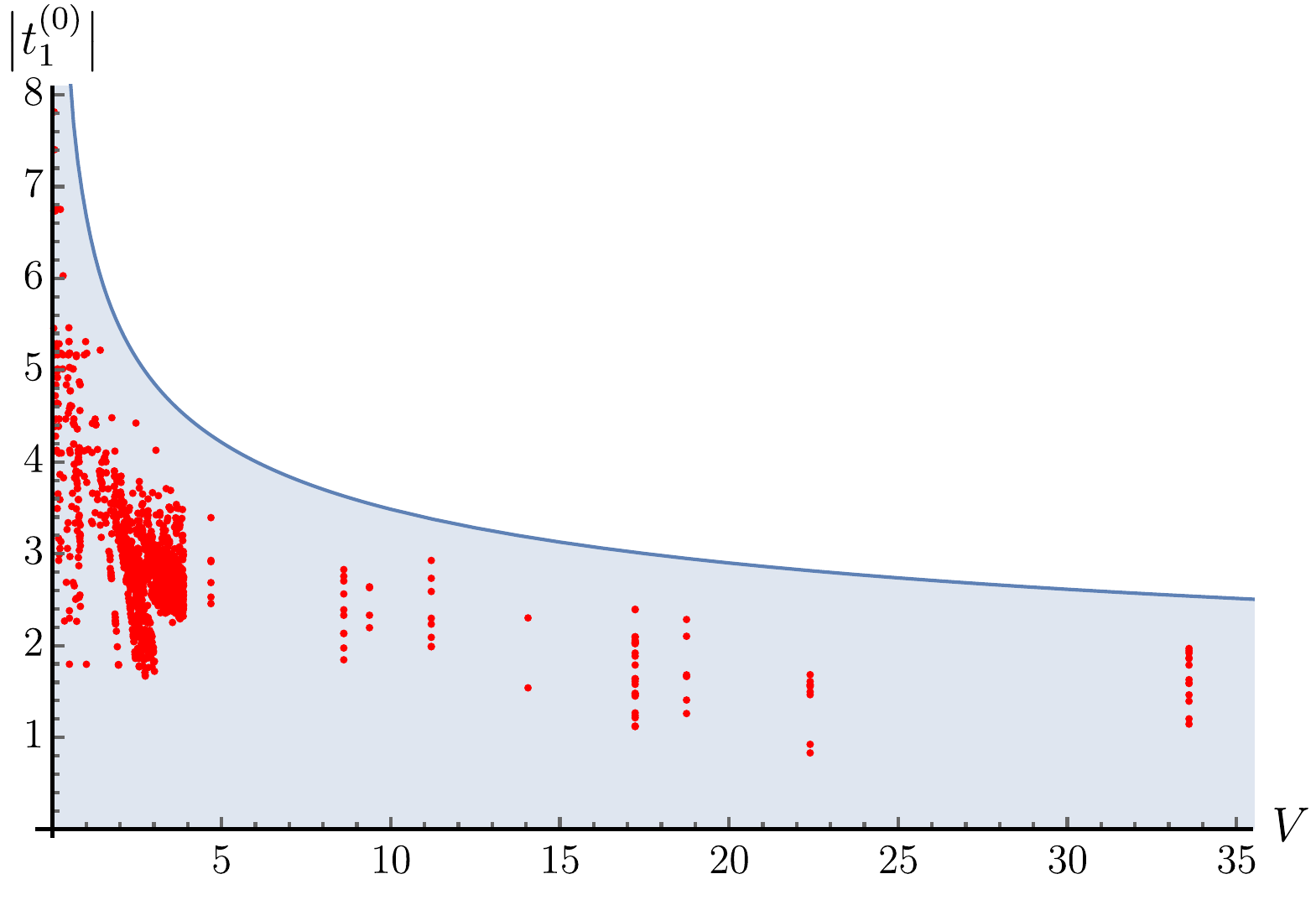,scale=.4}
\end{center}
\caption{An upper bound on the spectral gap of the Laplace--Beltrami operator from the Selberg trace formula. Here an upper bound $\big|t^{(0)}_1 \big|< c$ means that $\lambda^{(0)}_1 < c^2+1$. The red dots are lower bounds on the spectral gaps of various example orbifolds.}
\label{fig:selberg-upper-bound}
\end{figure}

\subsection{Summary of the bootstrap method}
Since the bootstrap method may be unfamiliar, we will now briefly summarize it. Fix a cocompact lattice $\Gamma$ in $G=\mathrm{PSL}_2(\mathbb{C})$. The Hilbert space $L^2(\Gamma\backslash G)$ is a unitary representation of $G$, the latter acting by right translation. $L^2(\Gamma\backslash G)$ decomposes as a discrete direct sum of unitary irreducible representations of $G$:
\be\label{eq:spectrumIntroOld}
L^2(\Gamma\backslash G) = \mathbb{C} \oplus \bigoplus\limits_{J=0}^{\infty}\bigoplus\limits_{k=1}^{\infty} R_{1+i t^{(J)}_k,J}\,.
\ee
Here $t_k^{(J)}$ with $J>0$ are the eigenvalues of the curl operator~\eqref{eq:curl} and $(t_k^{(0)})^2+1$ are the eigenvalues of the Laplace-Beltrami operator. $R_{1+ i t,J}$ are unitary irreducible representations of $G$, reviewed in Section~\ref{ssec:irrepsG}. In this way, the Laplacian spectrum of $\Gamma\backslash\hh$ is encoded representation-theoretically.

\begin{remark}\label{rmk:spectrumNotation}
Since we will often use the spectral decomposition~\eqref{eq:spectrumIntroOld}, it will be convenient to simplify our notation for the spectrum. Let $((t_k,J_k))_{k\in\mathbb{N}_{>0}}$ be a sequence of ordered pairs $(t_k,J_k)$ such that for all $J\geq0$, the subsequence with $J_k=J$ coincides with the spin-$J$ spectrum $((t^{(J)}_{\ell},J))_{\ell\in\mathbb{N}_{>0}}$. In other words, the spectral decomposition~\eqref{eq:spectrumIntroOld} can be equivalently written as
\be\label{eq:spectrumIntro}
L^2(\Gamma\backslash G) = \mathbb{C} \oplus \bigoplus\limits_{k=1}^{\infty} R_{1+ i t_{k},J_k}\,.
\ee
\end{remark}

The main ingredient for deriving bootstrap bounds is an infinite collection of spectral identities which follow from considering integrals of quadruple products
\be\label{eq:quadrupleIntro}
\int\limits\limits_{\Gamma\backslash G}F_1(g) F_2(g) F_3(g) F_4(g)dg\,.
\ee
Here $F_{1,2,3,4}$ are smooth functions on $\Gamma\backslash G$, each belonging to an irreducible summand on the right-hand side of~\eqref{eq:spectrumIntro}. In practice, we take $F_{1,2,3,4}$ to be $K$-finite, i.e.,~they are automorphic forms. A key fact is that such four-point correlations (as well as higher-point correlations) are uniquely determined by the spectrum $((t_k,J_k))_{k\in\mathbb{N}_{>0}}$ and integrals of triple products
$$
\hat{c}_{ijk} =\frac{1}{\mathrm{vol}(\Gamma\backslash \mathbb{H}^3)} \int\limits\limits_{\Gamma\backslash \hh}\varphi_i(x) \varphi_j(x) \varphi_k(x)d\mu(x)\,.
$$
See Section~\ref{ssec:triple} for the definition and discussion of the triple products for general spins $J_i,\,J_j,\,J_k$. Suppose $F_1 \in R_{1+i t_k,J_k}\subset L^2(\Gamma\backslash G)$ and $F_2\in R_{1+i t_\ell,J_\ell}\subset L^2(\Gamma\backslash G)$. The spectral decomposition~\eqref{eq:spectrumIntro} of the product $F_1(g)F_2(g)$ then takes the form
\be\label{eq:productIntro}
F_1(g)F_2(g) = \sum\limits_{m=0}^{\infty} \hat{c}_{k\ell m} \widetilde{F}_m(g)\,,
\ee
where $\widetilde{F}_m \in R_{1+i t_m,J_m}\subset L^2(\Gamma\backslash G)$ and the $m=0$ term represents the contribution of the trivial representation, where we set $\hat{c}_{k\ell 0} := \delta_{k\ell} $. Crucially, the dependence of $\widetilde{F}_m$ on $F_1$ and $F_2$ as these vary over vectors in $R_{1+i t_k,J_k}$ and $R_{1+i t_\ell,J_\ell}$ is completely determined by representation theory. This follows from the uniqueness of invariant trilinear forms for $G$ \cite{Oksak1973,Loke2001}.

Suppose $F_{1,2,3,4}$ in~\eqref{eq:quadrupleIntro} all belong to the same irreducible summand of $L^2(\Gamma\backslash G)$, $F_{1,2,3,4}\in R_{1+i t_\ell,J_\ell}\subset L^2(\Gamma\backslash G)$. Using the product expansion~\eqref{eq:productIntro} inside~\eqref{eq:quadrupleIntro} leads to the expansion
\be\label{eq:cbIntro}
\frac{1}{\mathrm{vol}(\Gamma\backslash \mathbb{H}^3)}  \int\limits\limits_{\Gamma\backslash G}F_1(g) F_2(g) F_3(g) F_4(g)dg =
\sum\limits_{m=0}^{\infty}(\hat{c}_{\ell\ell m})^2 \hat{\Psi}^{\ell\ell\ell\ell}_m(F_1,F_2,F_3,F_4)\,.
\ee
Thanks to the uniqueness of trilinear functionals, $\hat{\Psi}^{\ell\ell\ell\ell}_m(F_1,F_2,F_3,F_4)$ is explicitly computable in terms of the spectral data $(t_\ell,J_\ell)$, $(t_m,J_m)$, given a realization of $R_{1+i t_\ell,J_\ell}$ used to parametrize $F_{1,2,3,4}$, see Section~\ref{ssec:fourPoints} and in particular Proposition~\ref{prop:productExpG} for details.

Finally, the spectral identities follow from associativity of the product expansion~\eqref{eq:productIntro}. This can be implemented by imposing symmetry of~\eqref{eq:cbIntro} under the permutation $F_2\leftrightarrow F_4$, which leads to the identity
$$
\sum\limits_{m=0}^{\infty}(\hat{c}_{\ell\ell m})^2\left[ \hat{\Psi}^{\ell\ell\ell\ell}_m(F_1,F_2,F_3,F_4)
-\hat{\Psi}^{\ell\ell\ell\ell}_m(F_1,F_4,F_3,F_2) \right]= 0\,.
$$
As $F_{1,2,3,4}$ each ranges over a basis for the space of $K$-finite vectors in $R_{1+i t_\ell,J_\ell}$, this equation becomes an infinite set of spectral identities constraining the spectrum and $\hat{c}_{\ell\ell m}$. The full set of spectral identities is stated explicitly in Theorem~\ref{thm:crossingKFinite}. As an example, here are several spectral identities with $J_{\ell}=1$:
$$
\begin{aligned}
&\sum^{\infty}_{m=0} (\hat{c}_{\ell\ell m})^2\left[  \frac{(5 t_m^2+21)}{8(t_m^2+4)} \delta_{J_m 0} - \delta_{J_m 2}\right]=0, \\
&\sum^{\infty}_{m=0} (\hat{c}_{\ell\ell m})^2 \left[ t_{\ell} \delta_{J_m 0} - t_{m} \delta_{J_m 2} \right] =0, \\
&\sum^{\infty}_{m=0}(\hat{c}_{\ell\ell m})^2 \left[ \frac{(16t_{\ell}^2(t_m^2+4)-7 t_m^4-619t_m^2-2484 )}{8(t_m^2+4)}\delta_{J_m 0} - (t_{m}^2-120) \delta_{J_m 2} \right] =0, \\
&\sum^{\infty}_{m=0} (\hat{c}_{\ell\ell m})^2 \left[ t_{\ell} (4 t_{\ell}^2-3 t_m^2 -116)\delta_{J_m 0} - t_{m}(t_m^2-120) \delta_{J_m 2} \right] =0,
\end{aligned}
$$
where $\delta_{J J'}$ is the Kronecker delta.

To turn such identities into spectral bounds, we use the fact that the triple products $\hat{c}_{\ell\ell m}$ are real and hence $(\hat{c}_{\ell\ell m})^2\geqslant 0$. By judiciously choosing $F_{1,2,3,4}$, and taking linear combinations of such identities, we can arrange for the contribution of any element of the spectrum $(t_m,J_m)$ in some chosen set $U'$ to be non-negative. The vanishing of the sum over $m$ then guarantees that there must exist a point in the spectrum $(t_m,J_m)$ outside $U'$, yielding the spectral bound, which is the content of Proposition~\ref{prop:spectralBounds}.

The task of deciding whether or not there is a linear combination of spectral identities which is non-negative in $U'$ is an infinite-dimensional linear program which can be turned into a sequence of finite-dimensional semi-definite programs, thus allowing for an efficient implementation.

\subsection{Outline}
The rest of this paper is organized as follows: in Sections~\ref{sec:rep-theory},~\ref{sec:correlators}, and~\ref{sec:OPE-and-associativity}, we set up the machinery for deriving bounds from associativity---Section~\ref{sec:rep-theory} discusses the connection between the Laplace spectra of $\Gamma \backslash \hh$ and the unitary irreducible representations of $G$ occurring in the Hilbert space $L^2(\Gamma \backslash G)$, Section~\ref{sec:correlators} introduces correlations and triple products on $\Gamma \backslash G$, and Section~\ref{sec:OPE-and-associativity} explains how to derive the spectral identities. In Section~\ref{sec:STF}, we present the relevant trace formulas and discuss a procedure for deriving bounds for specific orbifolds. In Section~\ref{sec:bounds}, we present the bootstrap bounds. We use the trace formula to derive general bounds on the spectral gap and systole length in Section~\ref{sec:selberg-general-bounds}. Finally, we offer concluding remarks in Section~\ref{sec:conclusion}. Several technical discussions are presented in the appendices.

\subsection*{Acknowledgments}
We would like to thank Peter Sarnak and Akshay Venkatesh for useful discussions and Petr Kravchuk for useful discussions and collaboration on related ideas. DM acknowledges funding provided by Edward
and Kiyomi Baird as well as the grant DE-SC0009988 from the U.S. Department of Energy. SP acknowledges the support by the U.S. Department of Energy, Office of Science, Office of High Energy Physics, under Award Number DE-SC0011632 and by the Walter Burke Institute for Theoretical Physics.  

\section{Representation theory and the spectrum}
\label{sec:rep-theory}
The spectral problem for the Laplacian on hyperbolic 3-orbifolds can be naturally formulated using the representation theory of $G=\mathrm{SO}_0(1,3)$. In this section, we review this connection and set notation, starting with representations of the maximal compact subgroup $K = \mathrm{SO}(3)$.
\subsection{Irreducible representations of $\mathrm{SO}(3)$}
Let $V_1 = \mathbb{C}^3$ be the fundamental representation of $K=\mathrm{SO}(3)$ and let $V_n$ be the space of traceless elements of $\mathrm{Sym}^n(V_1)$. In other words, $V_n$ consists of tensors $w$ with components $w^{a_1\ldots a_n}$, where $a_1,\ldots,a_n\in\{1,2,3\}$ such that $w^{a_1\ldots a_n} = w^{a_{\sigma(1)}\ldots a_{\sigma(n)}}$ for any $\sigma\in S_n$, and $\delta_{a_1 a_2}w^{a_1\ldots a_n} = 0$. Here $\delta_{ab}$ is the Kronecker delta and we are using the Einstein summation convention, where a pair of repeated indices, one upper and one lower, is summed over.

$V_n$ is irreducible of dimension $2n+1$. The action of $k\in K$ on $w$ takes the form
$$
(k w)^{a_1\ldots a_n} = k^{a_1}_{\phantom{a_1}b_1}\ldots k^{a_n}_{\phantom{a_1}b_n} w^{b_1\ldots b_n}\,.
$$

We will use the following notation for the standard $K$-invariant maps $V_n\otimes V_n\rightarrow \mathbb{C}$ and $V_1\otimes V_1\otimes V_1 \rightarrow \mathbb{C}$:
\be\label{eq:dotAndSquare}
w_1\cdot w_2 = \delta_{a_1b_1}\ldots\delta_{a_n b_n}w_1^{a_1\ldots a_n} w_2^{b_1\ldots b_n}\,,\qquad
[v_1,v_2,v_3] = \epsilon_{abc} v_1^{a}v_2^{b}v_3^{c}\,,
\ee
$\epsilon_{abc}$ is the Levi--Civita tensor. We will normalize the $K$-invariant norm on $V_n$ as follows:
$$
\| w \|^2_n =\frac{w\cdot \overline{w}}{2n+1}\,.
$$
We say $v\in V_1$ is null if  $v\cdot v = 0$. For any null $v\in V_1$, let $w = v^n$ be the element of $V_n$ with components $w^{a_1\ldots a_n} = v^{a_1}\ldots v^{a_n}$.\footnote{We are aware of the notational clash between the components $v^a$ of $v\in V_1$ and the $n^{\mathrm{th}}$ tensor power $v^n\in V_n$. To avoid confusion, the component indices are denoted by letters from the start of the alphabet $a,\,b,\,c,\ldots$, and tensor powers by the letter $n$.}

Let us discuss the $K$-invariant trilinear forms $V_{n_1}\otimes V_{n_2} \otimes V_{n_3}\rightarrow \mathbb{C}$. Recall the decomposition of the tensor product into irreducibles
$$
V_{n_1}\otimes V_{n_2} = \bigoplus\limits_{n_3 = |n_1-n_2|}^{n_1+n_2} V_{n_3}\,.
$$
Thus, when the triplet $(n_1,n_2,n_3)$ satisfies the triangle inequalities, the space of $K$-linear forms on $V_{n_1}\otimes V_{n_2} \otimes V_{n_3}$ is one-dimensional, otherwise it is zero-dimensional. We will normalize these forms as follows:
\begin{definition}\label{def:Trilinear}
Suppose $n_1,n_2,n_3$ satisfy the triangle inequalities and let
$$
T_{n_1,n_2,n_3}:V_{n_1}\otimes V_{n_2} \otimes V_{n_3}\rightarrow \mathbb{C}
$$
be the unique $K$-invariant trilinear form normalized so that for any null $v_1,v_2,v_3\in V_1$, we have
$$
T_{n_1,n_2,n_3}(v^{n_1}_1,v^{n_2}_2,v^{n_3}_3) = 
(v_1\cdot v_2)^{\frac{n_1+n_2-n_3}{2}}(v_1\cdot v_3)^{\frac{n_1+n_3-n_2}{2}}(v_2\cdot v_3)^{\frac{n_2+n_3-n_1}{2}}
$$
for $n_1+n_2+n_3$ even and
$$
T_{n_1,n_2,n_3}(v^{n_1}_1,v^{n_2}_2,v^{n_3}_3) = \frac{[v_1,v_2,v_3]}{\sqrt{2}}
(v_1\cdot v_2)^{\frac{n_1+n_2-n_3-1}{2}}(v_1\cdot v_3)^{\frac{n_1+n_3-n_2-1}{2}}(v_2\cdot v_3)^{\frac{n_2+n_3-n_1-1}{2}}
$$
for $n_1+n_2+n_3$ odd.
\end{definition}
Consider the dual map $T^*_{n_1,n_2,n_3}: \mathbb{C}\rightarrow V^{*}_{n_1}\otimes V^{*}_{n_2} \otimes V^{*}_{n_3}$. The bilinear form $(w_1,w_2)\mapsto w_1\cdot w_2$ gives rise to an isomorphism $b_{n}:V_{n} \rightarrow V_{n}^{*}$. Then we have
\begin{lemma}\label{lem:qDefinition}
$$
T_{n_1,n_2,n_3}\circ (b_{n_1}\otimes b_{n_2}\otimes b_{n_3})^{-1} \circ T^*_{n_1,n_2,n_3} = q(n_1,n_2,n_3)\,,
$$
where\
\be\label{eq:qDefinition}
q(n_1,n_2,n_3) = \tfrac{\Gamma(n_1+n_2-n_3+1)\, \Gamma(n_1+n_3-n_2+1) \,\Gamma(n_2+n_3-n_1+1) \,\Gamma(n_1+n_2+n_3+2)}{\Gamma(2 n_1+1) \Gamma(2 n_2+1) \Gamma(2 n_3+1)}\,.
\ee
\end{lemma}
The proof of the lemma is deferred to Appendix~\ref{app:KHaar}.

We will also make use of the dual of $T_{n_1,n_2,n_3}$ with respect to the last entry:
\begin{definition}\label{def:Tcheck}
Suppose $n_1,\,n_2,\,n_3$ satisfy the triangle inequalities. We define $T^{\vee}_{n_1,n_2,n_3}: V_{n_1}\otimes V_{n_2}\rightarrow V_{n_3}$ to be the unique $K$-linear map such that for all $w_1\in V_{n_1}$, $w_2\in V_{n_2}$, $w_3\in V_{n_3}$, we have
$$
T^{\vee}_{n_1,n_2,n_3}(w_1,w_2)\cdot w_3 = T_{n_1,n_2,n_3}(w_1,w_2,w_3)\,.
$$
\end{definition}

\subsection{Invariant quadrilinear forms for $\mathrm{SO}(3)$}\label{ssec:invariant4K}
Next, let us discuss $K$-invariant quadrilinear forms $V_{n_1}\otimes V_{n_2}\otimes V_{n_3}\otimes V_{n_4}\rightarrow\mathbb{C}$. From the decompositions of $V_{n_1}\otimes V_{n_2}$ and $V_{n_3}\otimes V_{n_4}$ into irreducibles, we see that a basis for the space of all invariant forms consists of
\ba\label{eq:HDef}
&H^{n_1,n_2,n_3,n_4}_{n_5}:V_{n_1}\otimes V_{n_2}\otimes V_{n_3}\otimes V_{n_4}\rightarrow \mathbb{C}\\
&H^{n_1,n_2,n_3,n_4}_{n_5}(w_1,w_2,w_3,w_4)\coloneqq T^{\vee}_{n_1,n_2,n_5}(w_1,w_2)\cdot T^{\vee}_{n_3,n_4,n_5}(w_3,w_4)\,,
\ea
with $\mathrm{max}(|n_1-n_2|,|n_3-n_4|)\leq n_5\leq \mathrm{min}(n_1+n_2,n_3+n_4)$. The main goal of this subsection will be to compute $H^{n_1,n_2,n_3,n_4}_{n_5}(w_1,w_2,w_3,w_4)$ explicitly.

Note that any linear form $\alpha\in V_n^{*}$ can be fully recovered from its values on $w\in V_n$ of the form $w^{a_1\ldots a_n} = v^{a_1}\ldots v^{a_n}$, where $v\in V_1$ ranges over all null vectors.\footnote{To see this, note that $\alpha(w) = w^{a_1\ldots a_n}\alpha_{a_1\ldots a_n}$ for some symmetric and traceless $\alpha_{a_1\ldots a_n}$. Any symmetric tensor $\alpha_{a_1\ldots a_n}$ can be recovered from the values of the polynomial $v^{a_1}\ldots v^{a_n}\alpha_{a_1\ldots a_n}$ as $v$ ranges over all of $V_1$. If $v$ only ranges over null vectors, we can recover $\alpha_{a_1\ldots a_n}$ up to terms proportional to $\delta_{a_i a_j}$, but these terms are uniquely fixed by demanding $\alpha_{a_1\ldots a_n}$ to be traceless. For an explicit reconstruction of $\alpha_{a_1\ldots a_n}$ from $\alpha(v^n)$ with $v$ null using a differential operator, see Section 3.1 of~\cite{Costa:2011dw}.} Hence, without loss of generality, we will focus on computing $H^{n_1,n_2,n_3,n_4}_{n_5}(v^{n_1}_1,v^{n_2}_2,v^{n_3}_3,v^{n_4}_4)$ where $v_{1},\,v_2,\,v_3,\,v_4\in V_1$ are four null vectors.

Note that $H^{n_1,n_2,n_3,n_4}_{n_5}(v^{n_1}_1,v^{n_2}_2,v^{n_3}_3,v^{n_4}_4)$ is a homogeneous polynomial in the components of $v_i$ of degree $n_i$ for each $i\in\{1,2,3,4\}$. Since it is also $K$-invariant, it must in fact be a polynomial in the six variables $v_i\cdot v_j$ with $i< j$ and the four variables $[v_i,v_j,v_k]$ with $i< j< k$. To proceed, we will first find all the polynomial relations between these variables. Firstly, note that the product $[v_i,v_j,v_k][v_{i'},v_{j'},v_{k'}]$ can always be expressed as a polynomial in the six $v_{i}\cdot v_{j}$. Secondly, the four vectors $v_1,\,\ldots,v_4\in\mathbb{C}^3$ must be linearly dependent, which can be exploited as follows. Define $r$ to be the following invariant ratio:
\be\label{eq:rDefinition1}
r(v_1,v_2,v_3,v_4) \coloneqq \frac{(v_3\cdot v_4)[v_1,v_2,v_3]}{(v_1\cdot v_3)[v_2,v_3,v_4]}\,.
\ee
$r(v_1,v_2,v_3,v_4)$ has homogeneity degree 0 in all $v_i$. As a result of the linear dependence of $v_1,\,\ldots,v_4$, any homogeneity degree 0 invariant function of $v_1,\,\ldots,v_4$ is in fact a function of $r$ only. Indeed, by expressing $v_4$ as a linear combination of $v_1$, $v_2$, $v_3$ and using $v_i\cdot v_i=0$, it is not hard to show that
\be\label{eq:rDefinition2}
\frac{(v_1\cdot v_2)(v_3\cdot v_4)}{(v_1\cdot v_3)(v_2\cdot v_4)} = r^2 \,,\qquad
\frac{(v_1\cdot v_4)(v_2\cdot v_3)}{(v_1\cdot v_3)(v_2\cdot v_4)} = (1-r)^2,
\ee
and
\be\label{eq:rDefinition3}
\frac{(v_2\cdot v_3)[v_1,v_3,v_4]}{(v_1\cdot v_3)[v_2,v_3,v_4]} = 1-r\,,\qquad
\frac{(v_3\cdot v_4)[v_1,v_2,v_4]}{(v_1\cdot v_4)[v_2,v_3,v_4]} = \frac{r}{1-r}\,.
\ee

With this preparation, we can state the main result of this subsection:
\begin{lemma}\label{lem:quadrilinear}
Let $v_1,\,v_2,\,v_3,\,v_4\in V_1$ be null vectors. Then, if $n_1+n_2+n_3+n_4$ is even, we have
\ba\label{eq:Heven}
H^{n_1,n_2,n_3,n_4}_{n_5}(v^{n_1}_1,v^{n_2}_2,v^{n_3}_3,v^{n_4}_4) =
 &(v_1\cdot v_2)^{\frac{n_1+n_2-n_3-n_4}{2}} (v_1\cdot v_4)^{\frac{n_1-n_2-n_3+n_4}{2}}\\
 &\times(v_1\cdot v_3)^{n_3}(v_2\cdot v_4)^{\frac{-n_1+n_2+n_3+n_4}{2}}
\mathcal{H}^{n_1,n_2,n_3,n_4}_{n_5}(r)\,.
\ea
On the other hand, if $n_1+n_2+n_3+n_4$ is odd, we have
\ba\label{eq:Hodd}
H^{n_1,n_2,n_3,n_4}_{n_5}(v^{n_1}_1,v^{n_2}_2,v^{n_3}_3,v^{n_4}_4) =
&\frac{[v_1,v_2,v_4]}{\sqrt{2}}
 (v_1\cdot v_2)^{\frac{n_1+n_2-n_3-n_4-1}{2}} (v_1\cdot v_4)^{\frac{n_1-n_2-n_3+n_4-1}{2}}
 \\
 &\times(v_1\cdot v_3)^{n_3}(v_2\cdot v_4)^{\frac{-n_1+n_2+n_3+n_4-1}{2}}
\mathcal{H}^{n_1,n_2,n_3,n_4}_{n_5}(r)\,.
\ea
Here $r=r(v_1,v_2,v_3,v_4)$ is the homogeneity degree zero ratio defined by any of the equations~\eqref{eq:rDefinition1},~\eqref{eq:rDefinition2}, or~\eqref{eq:rDefinition3}, and $\mathcal{H}^{n_1,n_2,n_3,n_4}_{n_5}(r)$ is the following polynomial
\be\label{eq:HcalDef}
\mathcal{H}^{n_1,n_2,n_3,n_4}_{n_5}(r) \coloneqq r^{n_3+n_4-n_5}{}_2F_1(n_1-n_2-n_5,-n_3+n_4-n_5;-2n_5;r)\,.
\ee
\end{lemma}
\begin{proof}
The prefactors of $\mathcal{H}^{n_1,n_2,n_3,n_4}_{n_5}$ on the right-hand sides of equations~\eqref{eq:Heven} and~\eqref{eq:Hodd} are $K$-invariant and have the correct degree of homogeneity in all four $v_i$. It follows that $\mathcal{H}^{n_1,n_2,n_3,n_4}_{n_5}$ must be $K$-invariant of homogeneity degree zero. By the preceding discussion, $\mathcal{H}^{n_1,n_2,n_3,n_4}_{n_5}$ depends on $v_i$ only through the ratio $r$.

It remains to find $\mathcal{H}^{n_1,n_2,n_3,n_4}_{n_5}(r)$. Let $A_a$ with $a=1,2,3$ be the standard basis for the Lie algebra $\mathfrak{so}(3)$, so that $[A_a,A_b] = \sum_{c=1}^{3}\epsilon_{abc}A_c$. Consider any $K$-linear map $F:V_{n_1}\otimes V_{n_2}\rightarrow W$, where $W$ is any representation of $\mathrm{SO}(3)$. Then $A_a$ acts on $F(v_1^{n_1},v_2^{n_2})$ as the following linear differential operator:
$$
A_a F(v_1^{n_1},v_2^{n_2}) = -\sum\limits_{b,c=1}^{3}\epsilon_{abc}\left(v^b_1\frac{\partial}{\partial v^c_1}+v^b_2\frac{\partial}{\partial v^c_2}\right)F(v_1^{n_1},v_2^{n_2})\,.
$$
Let us define the quadratic Casimir of $\mathrm{SO}(3)$ as $c_2 = \sum_{a=1}^{3}A_a A_a$, so that
$$
c_2 F(v_1^{n_1},v_2^{n_2})= \sum\limits_{a,b,c,d,e=1}^{3}\epsilon_{abc}\epsilon_{ade}\left(v^b_1\frac{\partial}{\partial v^c_1}+v^b_2\frac{\partial}{\partial v^c_2}\right)
\left(v^d_1\frac{\partial}{\partial v^e_1}+v^d_2\frac{\partial}{\partial v^e_2}\right)F(v_1^{n_1},v_2^{n_2})\,.
$$
Note that $c_2$ acts on $V_{n}$ as identity times the constant $-n(n+1)$. Hence, if $W = V_{n_5}$ is irreducible, $F(v_1^{n_1},v_2^{n_2})$ satisfies the second-order differential equation
$$
c_2 F(v_1^{n_1},v_2^{n_2})= -n_5(n_5+1)F(v_1^{n_1},v_2^{n_2})\,.
$$
In particular, since $H^{n_1,n_2,n_3,n_4}_{n_5}$ is defined by~\eqref{eq:HDef}, it satisfies the same equation
$$
c_2 H^{n_1,n_2,n_3,n_4}_{n_5}(v^{n_1}_1,v^{n_2}_2,v^{n_3}_3,v^{n_4}_4) = 
 -n_5(n_5+1)
H^{n_1,n_2,n_3,n_4}_{n_5}(v^{n_1}_1,v^{n_2}_2,v^{n_3}_3,v^{n_4}_4)\,.
$$
Plugging in the ansatz~\eqref{eq:Heven},~\eqref{eq:Hodd} leads to an ordinary second-order differential equation for $\mathcal{H}^{n_1,n_2,n_3,n_4}_{n_5}(r)$. The only solutions of this equation which are rational in $r$ are constant multiples of~\eqref{eq:HcalDef}.

To fix the normalization, we consider the limit $v_1\rightarrow v_2$, so that $v_1\cdot v_2\rightarrow 0$ and hence $r\rightarrow 0$. Assume for simplicity that $n_1+n_2+n_5$ is even. Then
$$
T_{n_1,n_2,n_5}^{\vee}(v_1^{n_1},v_2^{n_2})\sim (v_1\cdot v_2)^{\frac{n_1+n_2-n_5}{2}}v_1^{n_5} \quad\text{as}\quad v_1\rightarrow v_2.
$$
When $n_3+n_4+n_5$ is also even, it follows that
$$
H^{n_1,n_2,n_3,n_4}_{n_5}(v^{n_1}_1,v^{n_2}_2,v^{n_3}_3,v^{n_4}_4) \sim 
(v_1\cdot v_2)^{\frac{n_1+n_2-n_5}{2}}
(v_1\cdot v_3)^{\frac{n_3-n_4+n_5}{2}}(v_1\cdot v_4)^{\frac{-n_3+n_4+n_5}{2}}
(v_3\cdot v_4)^{\frac{n_3+n_4-n_5}{2}}
$$
as $v_1\rightarrow v_2$. On the other hand, in the same limit
$$
\mathcal{H}^{n_1,n_2,n_3,n_4}_{n_5}(r) \sim r^{n_3+n_4-n_5} \sim
\left[\frac{(v_1\cdot v_2)(v_3\cdot v_4)}{(v_1\cdot v_3)(v_1\cdot v_4)}\right]^{\frac{n_3+n_4-n_5}{2}}\,.
$$
Which fixes the overall constant in~\eqref{eq:HcalDef}. The cases when $n_1+n_2+n_5$ or $n_3+n_4+n_5$ are odd lead to the same answer.
\end{proof}
\begin{remark}\label{rmk:Hdegree}
The polynomial $\mathcal{H}^{n_1,n_2,n_3,n_4}_{n_5}(r)$ has degree $\min(2n_3,n_2+n_3+n_4-n_1)$. Indeed, when $a,b,c$ are non-negative integers such that $c\geq\min(a,b)$, ${}_2F_1(-a,-b;-c,r)$ is a polynomial in $r$ of degree $\min(a,b)$.
\end{remark}

\subsection{Unitary irreducible representations of $\mathrm{SO}_0(1,3)$}\label{ssec:irrepsG}
The unitary irreducible representations of $G=\mathrm{SO}_0(1,3)$ have been classified in~\cite{Bargmann, Gelfand-Naimark}. They can all be realized as spaces of functions on the Riemann sphere $\mathbb{C}\cup\{\infty\}$. Recall that $G\cong\mathrm{PSL}_2(\mathbb{C})$, so $G$ acts on the Riemann sphere by fractional linear transformations
$$
\pm\begin{pmatrix}a & b\\c & d\end{pmatrix} \in \mathrm{PSL}_2(\mathbb{C})\,:\qquad
x\mapsto \frac{a x+b}{c x+d}\,.
$$
This action leads to the following construction of representations:

\begin{definition}\label{def:principal}
The (unitary) \emph{principal series} representation of $\mathrm{PSL}_2(\mathbb{C})$ will be denoted $R_{\Delta,J}$, where $\Delta = 1 + i t$ with $t\in\mathbb{R}$ and $J\in\mathbb{Z}$.  The underlying Hilbert space of $R_{\Delta,J}$ is $L^2(\mathbb{C})$, with the inner product normalized as
$$
(f_1,f_2)_{\Delta,J} = \pi\int\limits_{\mathbb{C}} f_1(x)\overline{f_2(x)} dx\,,
$$
where $dx$ is the standard Lebesgue measure on $\mathbb{C}$. The action of $g\in \mathrm{PSL}_2(\mathbb{C})$ on $f\in R_{\Delta,J}$ takes the form
\be\label{eq:principalAction}
(g\cdot f)(x) = |-c\,x +a|^{2(\Delta-2)}\left(\frac{-\bar{c}\,\xb+\bar{a}}{-c\,x + a}\right)^{J}
f\left(\tfrac{dx-b}{-cx+a}\right)\,,\quad\text{where}\quad
g=\pm\begin{pmatrix}a & b\\c & d\end{pmatrix}\,.
\ee
\end{definition}

\begin{definition}\label{def:complementary}
The \emph{complementary series} representation of $\mathrm{PSL}_2(\mathbb{C})$ will be denoted $R_{\Delta,0}$, where $\Delta\in(0,1)$. Its underlying space is the Hilbert space completion of $L^2(\mathbb{C})$ with respect to the inner product
$$
(f_1,f_2)_{\Delta,0} = (1-\Delta)\int\limits_{\mathbb{C}^2} \, |x-y|^{-2\Delta}f_1(x)\overline{f_2(y)} \, dx \, dy\,.
$$
The action of $g\in \mathrm{PSL}_2(\mathbb{C})$ on $f\in R_{\Delta,0}$ is given by~\eqref{eq:principalAction} with $J=0$.
\end{definition}

It is not hard to check that the principal and complementary series are unitary, i.e., that the action~\eqref{eq:principalAction} preserves the inner products given above. References~\cite{Bargmann,Gelfand-Naimark} proved the following general result:
\begin{theorem}
The complete list of unitary irreducible representations of $\mathrm{PSL}_2(\mathbb{C})$ consists of
\begin{enumerate}
\item the trivial representation $\mathbb{C}$,
\item the principal series $R_{\Delta,J}$, where $\Delta = 1 + i t$ with $t\in\mathbb{R}$ and $J\in\mathbb{Z}$,
\item the complementary series $R_{\Delta,0}$, where $\Delta\in(0,1)$.
\end{enumerate}
Furthermore, the only isomorphisms between these representations are between pairs of principal series $R_{\Delta,J}\cong R_{2-\Delta,-J}$.
\end{theorem}
For $(\Delta,J)\neq(1,0)$, the intertwining operator $\mathcal{S}_{\Delta,J}:R_{\Delta,J}\rightarrow R_{2-\Delta,-J}$ takes the form\footnote{The integral is convergent for $\mathrm{Re}(\Delta)<1$ and the analytic continuation from there to $\mathrm{Re}(\Delta) = 1$ is a well-defined distribution.}
\be\label{eq:shadow}
\mathcal{S}_{\Delta,J}(f) =\frac{1-\Delta + |J|}{\pi} \int\limits_{\mathbb{C}}(x-y)^{-\Delta+J}(\overline{x}-\overline{y})^{-\Delta-J}f(y) \, dy\,.
\ee
The normalization is chosen so that $\mathcal{S}_{\Delta,J}$ preserves the norm and ${\mathcal{S}_{2-\Delta,-J}\circ \mathcal{S}_{\Delta,J} = \mathrm{id}}$. We will also define $\mathcal{S}_{1,0} = \mathrm{id}$, which agrees with the analytic continuation of~\eqref{eq:shadow} from  $\mathrm{Re}(\Delta)<1$ to $\Delta = 1$.

The principal and complementary series behave identically for most of our purposes, which justifies the unified notation $R_{\Delta,J}$, distinguished only by the range of $(\Delta,J)$. Furthermore, thanks to the isomorphism $R_{\Delta,J}\cong R_{2-\Delta,-J}$, we can assume that $J\geq 0$, which we will do from now on.

Ultimately, our spectral bounds will follow from the fact that squares of real numbers are positive. To that end, let us exhibit the real structure on $R_{\Delta,J}$. Recall that a representation $R$ of a group $G$ is real if and only if there exists a $G$-invariant antilinear involution ${\rho}: R\rightarrow R$. Then we have:

\begin{proposition}\label{prop:realStructure}
All unitary irreducible representations $R_{\Delta,J}$ of $G$ are real. For $\Delta\in(0,1]$ and $J=0$, ${\rho}$ coincides with complex conjugation, ${\rho}(f)(x) = \overline{f(x)}$. For the remaining cases, ${\rho}(f) = \overline{\mathcal{S}_{\Delta,J}(f)}$, where $\mathcal{S}_{\Delta,J}$ is the intertwining operator given in~\eqref{eq:shadow}.
\end{proposition}

To proceed, it will be convenient to pick the basis for the Lie algebra $\mathfrak{g}=\mathfrak{sl}_2(\mathbb{C})$ consisting of elements $\{A_1,A_2,A_3,B_1,B_2,B_3\}$, where $A_a = -i\sigma_a/2$, $B_a = \sigma_a/2$ and $\sigma_a$ are the Pauli matrices
$$
\sigma_1 =  \left(
    \begin{array}{cc}
      0 & 1\\
      1 & 0
    \end{array}
    \right)\,,\quad
\sigma_2 =  \left(
    \begin{array}{cc}
      0 & -i\\
      i & 0
    \end{array}
    \right)\,,\quad
\sigma_3 =  \left(
    \begin{array}{cc}
      1 & 0\\
      0 & -1
    \end{array}
    \right).
$$
The commutation relations are
$$
[A_a,A_b] = \sum\limits_{c=1}^{3}\epsilon_{abc}A_c\,,\qquad
[A_a,B_b] = \sum\limits_{c=1}^{3}\epsilon_{abc}B_c\,,\qquad
[B_a,B_b] = -\sum\limits_{c=1}^{3}\epsilon_{abc}A_c\,,
$$
where $\epsilon_{abc}$ is the completely antisymmetric tensor with $\epsilon_{123}=1$. Note that $\{A_1,A_2,A_3\}$ is a basis for the Lie algebra $\mathfrak{su}_2$ of the maximal compact subgroup. The center of the universal enveloping algebra $U(\mathfrak{g})$ is generated by two quadratic Casimir elements,
$$
C_2 = \sum\limits_{a=1}^{3}(A_aA_a - B_aB_a)\,,\qquad C'_2 = \sum\limits_{a=1}^{3}A_aB_a = \sum\limits_{a=1}^{3}B_aA_a\,.
$$
Elements of $U(\mathfrak{g})$ act on the space of smooth functions in $R_{\Delta,J}$ by differential operators. The action of the Lie algebra elements follows directly from~\eqref{eq:principalAction}:
\be\label{eq:difGens}
\begin{aligned}
A_1+iA_2 &= i[\partial+\xb^2 \bar{\partial}+(2-\Delta-J)\xb]\\
A_1-iA_2 &= i[-x^2\partial-\bar{\partial}-(2-\Delta+J)x]\\
A_3 &= i[x\partial- \xb\bar{\partial}+J]
\end{aligned}\qquad
\begin{aligned}
B_1+iB_2 &= -\partial+\xb^2 \bar{\partial}+(2-\Delta-J)\xb\\
B_1-iB_2 &= x^2\partial-\bar{\partial}+(2-\Delta+J)x\\
B_3 &= -x\partial- \xb\bar{\partial}+\Delta-2\,.
\end{aligned}
\ee
Here $\partial = \partial/\partial x$ and $\bar{\partial} = \partial/\partial\xb$. The following lemma can be derived by combining the definition of the Casimir elements with~\eqref{eq:difGens}.
\begin{lemma}\label{lem:casimirs}
The Casimir elements $C_2$, $C'_2$ act on the space of smooth vectors in $R_{\Delta,J}$ by multiplication by the following constants:
$$
C_2(\Delta,J) = \Delta(2-\Delta) - J^2\,,\qquad
C'_2(\Delta,J) = i (\Delta-1)J\,.
$$
\end{lemma}

In practice, we will often find it useful to work directly with $K$-finite vectors in $R_{\Delta,J}$. Let us first state the $K$-content of $R_{\Delta,J}$, as found in \cite{Bargmann,Gelfand-Naimark}:
\begin{proposition}
Considered as a representation of $\mathrm{SO}(3)\subset G$, the representation $R_{\Delta,J}$ decomposes as a direct sum
$$
R_{\Delta,J} = \bigoplus\limits_{n=J}^{\infty} V_n\,.
$$
\end{proposition}

To parametrize $K$-finite vectors in $R_{\Delta,J}$, we will introduce injective $K$-linear maps $\kappa^{(n)}_{\Delta,J}:V_n\rightarrow R_{\Delta,J}$, which exist for $n\geq J$. We will require all of them to preserve the real structure and also require $\kappa^{(J)}_{\Delta,J}$ to preserve the norm, which fixes it up to a minus sign. To fix this ambiguity, we simply prescribe the sign of $\kappa^{(J)}_{\Delta,J}(z_{+}^J)$, where $z_{\pm} = (1,\pm i,0)/\sqrt{2}$ are null vectors. To define $\kappa^{(n)}_{\Delta,J}$ for $n>J$, we use the fact that by acting with the generators $B_a$ of the noncompact part of the Lie algebra on elements of $V_{n}$ can produce elements of $V_{n+1}$.
\begin{definition}\label{def:kappaMaps}
Let $\kappa^{(J)}_{\Delta,J}: V_{J}\rightarrow R_{\Delta,J}$ be the unique $K$-linear, norm-preserving, real map which satisfies
$$
\kappa^{(J)}_{\Delta,J}(z_{+}^{J}) = \frac{i^{J}}{\pi}\overline{x}^{2J}(|x|^2+1)^{\Delta-J-2}\,.
$$
For $n\geq J$, let
$$
\kappa_{\Delta,J}^{(n)}(w) = \sum\limits_{a_1,\ldots,a_{m}=1}^{3}B_{a_1}\ldots B_{a_{m}}\cdot\kappa_{\Delta,J}^{(J)}(w^{a_1\ldots a_{m} \bullet})\,,
$$
where $m=n-J$ and $w^{a_1\ldots a_{m} \bullet}$ stands for the element of $V_{J}$ obtained by restricting the first $m$ indices of $w$ to the values $a_1,\ldots,\,a_{m}$.
\end{definition}

\begin{lemma}\label{lem:fPlusMinus}
The images under $\kappa_{\Delta,J}^{(n)}$ of the highest- and lowest-weight vectors $z_{\pm}^{n}\in V_n$ take the form
$$
\begin{aligned}
\kappa_{\Delta,J}^{(n)}(z_{+}^{n}) &= \frac{i^{J}2^{\frac{n-J}{2}}(2+J-\Delta)_{n-J}}{\pi}\,\overline{x}^{n+J}(|x|^2+1)^{\Delta-n-2} \, ,\\
\kappa_{\Delta,J}^{(n)}(z_{-}^{n}) &= \frac{(-i)^{J}2^{\frac{n-J}{2}}(2+J-\Delta)_{n-J}}{\pi}\,x^{n-J}(|x|^2+1)^{\Delta-n-2}\,.
\end{aligned}
$$
Here $(a)_{b}=\Gamma(a+b)/\Gamma(a)$ is the Pochhammer symbol.
\end{lemma}
\begin{proof}
Indeed, note that by definition
$$
\begin{aligned}
\kappa^{(J)}_{\Delta,J}(z_{-}^{J}) &= \tfrac{1}{(2J)!}(A_1-i A_2)^{2J}\kappa^{(J)}_{\Delta,J}(z_{+}^{J}) \, ,\\
\kappa^{(n)}_{\Delta,J}(z_{+}^{n}) &= 2^{-\frac{n-J}{2}}(B_1+i B_2)^{n-J}\kappa^{(J)}_{\Delta,J}(z_{+}^{J}) \, ,\\
\kappa^{(n)}_{\Delta,J}(z_{-}^{n}) &= 2^{-\frac{n-J}{2}}(B_1-i B_2)^{n-J}\kappa^{(J)}_{\Delta,J}(z_{-}^{J})\,.
\end{aligned}
$$
The lemma now follows by acting on $\kappa^{(J)}_{\Delta,J}(z_{+}^{J})$  with an appropriate sequence of the differential operators given in~\eqref{eq:difGens}.
\end{proof}

One can see from the definition that $\kappa_{\Delta,J}^{(n)}$ preserves the real structure. However, it does not preserve the norm for $n>J$, as shown in the following lemma:
\begin{lemma}\label{lem:normN}
Let $w \in V_n$. The norm in $R_{\Delta, J}$ of $\kappa_{\Delta,J}^{(n)}(w)$ is given by
$$
\|\kappa_{\Delta,J}^{(n)}(w)\|^2_{\Delta,J} = (2n+1)p_{n}(\Delta,J)\|w\|^2_n\,,
$$
where
\be\label{eq:normN}
p_{n}(\Delta,J) = \frac{2^{n-J}(n-J)!(n+J)!(\Delta+J)_{n-J}(2-\Delta+J)_{n-J}}{(2n+1)!}\,.
\ee
\end{lemma}
\begin{proof}
Set $w = z_{+}^{n}$ or $w=z_{-}^{n}$, use Lemma~\ref{lem:fPlusMinus}, and compute the norm directly from Definitions~\ref{def:principal} and~\ref{def:complementary}. The final formula applies to both principal and complementary series, which justifies our relative normalization of their norms.
\end{proof}

\subsection{Spectrum of $L^2(\Gamma\backslash G)$}\label{ssec:spectrum}
Let $\Gamma$ be a cocompact lattice in $G$. The coset space $\Gamma\backslash G$  is a bundle over $\Gamma\backslash\hh$ where each fibre is $\mathrm{SO}(3)$. More precisely, it is the oriented orthonormal frame bundle of $\Gamma\backslash\hh$.
We are interested in the  Hilbert space $L^2(\Gamma\backslash G)$. Elements of this Hilbert space are functions $F: G\rightarrow\mathbb{C}$ satisfying $F(\gamma g) = F(g)$ for all $\gamma\in\Gamma$, $g\in G$, and having a finite $L^2$-norm.\footnote{More precisely, the elements are equivalence classes of functions under the relation of being equal almost everywhere, as usual in the definition of the $L^2$-space associated to a measure space.}

Let us normalize the Haar measure on $G$ so that $K$ has measure~1. We have
$$
\mathrm{vol}(\Gamma\backslash\hh) = \int\limits_{\Gamma\backslash G}1\,dg < \infty\,.
$$
Let us normalize the inner product on $L^2(\Gamma\backslash G)$ as follows:
$$
(F_1,F_2)_{\Gamma\backslash G} = \frac{1}{\mathrm{vol}(\Gamma\backslash\hh)}\int\limits_{\Gamma\backslash G} F_1(g) \overline{F_2(g)} dg\,.
$$
$G$ acts on elements of $L^2(\Gamma\backslash G)$ by right translation, making the latter into a unitary representation of $G$. Since $\Gamma\backslash G$ is compact, $L^2(\Gamma\backslash G)$ decomposes as a discrete direct sum of unitary irreducible representations of $G$
\be
L^2(\Gamma\backslash G) \cong \mathbb{C}\oplus \bigoplus\limits_{i=1}^{\infty}R_{\Delta_i,J_i}\,.
\label{eq:decomposition2}
\ee
See \cite{GelfandGraev} for the proof of discreteness. Since the action is transitive, the trivial representation appears precisely once, corresponding to the constant functions. In the remainder of this subsection, we will interpret the appearance of $R_{\Delta_i,J_i}$ in the spectrum in terms of the eigensections on $\Gamma\backslash\hh$ discussed in the Introduction.

For each $R_{\Delta_i,J_i}$ in $L^2(\Gamma\backslash G)$, there is a corresponding norm-preserving $G$-linear map
$$
\phi_i: R_{\Delta_i,J_i} \rightarrow L^2(\Gamma\backslash G)\,.
$$
Note that there is a natural real structure ${\rho}$ on $L^2(\Gamma\backslash G)$, provided by complex conjugation ${\rho}(F)(g)=\overline{F(g)}$. Recalling from Proposition~\ref{prop:realStructure} the real structure ${\rho}$ on $R_{\Delta,J}$, we will require that all $\phi_i$ are compatible with the real structure, i.e., ${\rho}\circ\phi_i = \phi_i\circ{\rho}$. If the spectrum exhibits no multiplicity, this reality condition determines $\phi_i$ up to an overall minus sign. More generally, when there are multiplicities, i.e., when $L^2(\Gamma\backslash G)$ contains $\mathbb{C}^{N}\otimes R_{\Delta,J}$ with $N>1$, we need to choose an orthonormal basis in $\mathbb{R}^{N}$ to specify all $\phi_i$ uniquely. We will assume that such a choice has been made.

Let us recall from Definition~\ref{def:kappaMaps} the $K$-linear embeddings $\kappa^{(n)}_{\Delta,J}:V_n\rightarrow R_{\Delta,J}$ and define
$$
\phi^{(n)}_{i} = \phi_i\circ \kappa^{(n)}_{\Delta_i,J_i}:\, V_n\rightarrow C^{\infty}(\Gamma\backslash G)\,.
$$
The image of $\phi^{(n)}_{i}$ is in smooth functions because $K$-finite vectors are smooth.

At this point, we will invoke the correspondence between $K$-linear maps $V_n\rightarrow C^{\infty}(\Gamma\backslash G)$ and smooth, traceless sections of $\mathrm{Sym}^{n}(T^{*}\Gamma\backslash\mathbb{H}^3)$, where $T^{*}\Gamma\backslash\mathbb{H}^3$ is the cotangent bundle of the orbifold. This correspondence is reviewed in detail in Appendix~\ref{app:difgeo}. The main result we need is the following proposition:
\begin{proposition}\label{prop:tensorFields}
Let $\theta : V_{n}\rightarrow C^{\infty}(\Gamma\backslash G)$ be a $K$-linear map. Then there exists a smooth traceless section $\varphi_{\theta}$ of $\mathrm{Sym}^{n}(T^{*}\Gamma\backslash\mathbb{H}^3)$ such that for all $w\in V_n$, $x\in\hh$ and $k\in K$
$$
\theta(w)(g(x,k)) = \varphi_{\theta}(x)(kw)\,,
$$
where we are parametrizing points in $G=NAK$ as $g(x,k)$ with $x\in NA = \hh$ and $k\in K$. Furthermore, the correspondence between $\theta$ and $\varphi_{\theta}$ is one-to-one.
\end{proposition}
Proposition~\ref{prop:tensorFields} is restated for hyperbolic manifolds of any dimension as Proposition~\ref{prop:tensorFieldsApp} in Appendix~\ref{app:difgeo} and proved therein.

We will denote by $\varphi_{i,n}$ the traceless section of $\mathrm{Sym}^{n}(T^{*}\Gamma\backslash\mathbb{H}^3)$ which corresponds to the map $\phi^{(n)}_{i}$, as provided by Proposition~\ref{prop:tensorFields}. This means
\be\label{eq:varphiDef}
\phi^{(n)}_{i}(w)(g(x,k)) = \varphi_{i,n;a_1\ldots a_n}(x)k^{a_1}_{\phantom{a_1}b_1}\ldots k^{a_n}_{\phantom{a_1}b_n}w^{b_1\ldots b_n}\,.
\ee
 $\varphi_{i,n;a_1\ldots a_n}(x)$ are the components of $\varphi_{i,n}$ in a basis of orthonormal 1-forms on $\hh$, and we are using the Einstein summation convention.
 
What is the relationship between $\varphi_{i,n}$ for different values of $n$? Proposition~\ref{prop:nabla} demonstrates that the action of the Lie algebra generators $B_a$ on $C^{\infty}(\Gamma\backslash G)$ becomes the covariant derivative. Therefore, $\varphi_{i,n}$ is obtained by acting on $\varphi_{i,J_i}$ with $n-J_i$ covariant derivatives, symmetrizing and removing traces. We can state it precisely as the following corollary:
\begin{corollary}\label{cor:Bnabla}
$\varphi_{i,n}$ is the unique smooth traceless section of $\mathrm{Sym}^{n}(T^{*}\Gamma\backslash\mathbb{H}^3)$ such that
$$
\varphi_{i,n}(Y) = (\nabla^{n-J_i}\varphi_{i,J_i})(Y)
$$
for any traceless section $Y$ of $\mathrm{Sym}^{n}(T_{*}\Gamma\backslash\mathbb{H}^3)$. Here $\varphi(Y)$ stands for the natural pairing between sections of $(T^{*}\Gamma\backslash\mathbb{H}^3)^{\otimes n}$ and $(T\,\Gamma\backslash\mathbb{H}^3)^{\otimes n}$.
\end{corollary}
\begin{proof}
This follows directly from the second part of Definition~\ref{def:kappaMaps} and Proposition~\ref{prop:nabla}.
\end{proof}

We introduce the following natural norm on the space of sections of $\mathrm{Sym}^{n}(T^{*}\Gamma\backslash\mathbb{H}^3)$:
$$
\|\varphi\|^2 =  \frac{1}{\mathrm{vol}(\Gamma\backslash\hh)}\int\limits_{\Gamma\backslash\hh}\!\!
\varphi(x)\cdot\overline{\varphi(x)}\,d\mu(x)\,.
$$
Here $d\mu(x)$ is the invariant measure on $\Gamma\backslash\hh$ and the dot product $\varphi(x)\cdot\overline{\varphi(x)}$ is defined in~\eqref{eq:dotAndSquare}.

\begin{proposition}
The sections $\varphi_{i,n}$ have the following properties:
\begin{enumerate}
\item reality: $\overline{\varphi_{i,n}(x)} = \varphi_{i,n}(x)$,
\item normalization: $\|\varphi_{i,J_i}\|^2 = 1$. More generally, $\|\varphi_{i,n}\|^2 = (2n+1)p_{n}(\Delta_i,J_i)$, with $p_{n}(\Delta,J)$ given in~\eqref{eq:normN}.
\end{enumerate}
\end{proposition}
\begin{proof}
Since $\phi_i$ preserves the real structure, we have $\overline{\phi^{(n)}_i(w)} = \phi^{(n)}_i(\overline{w})$, and so (1) follows from~\eqref{eq:varphiDef}. To prove part (2), we note that
$$
\|\phi^{(n)}_{i}(w)\|^2 = 
\frac{1}{\mathrm{vol}(\Gamma\backslash\hh)}\int\limits_{\Gamma\backslash\hh}\!\int\limits_{K}
(\varphi_{i,n}(x) \cdot kw)(\varphi_{i,n}(x) \cdot k\overline{w})dk \,d\mu(x)\,.
$$
The integral over $K$ can be done with the help of Lemma~\ref{lem:KHaar2} proved in Appendix~\ref{app:KHaar}. The result is
$$
\|\phi^{(n)}_{i}(w)\|^2_{\Gamma\backslash G} = \|\varphi_{i,n}\|^2 \|w\|_n^2\,.
$$
Part (2) of the proposition now follows from Lemma~\ref{lem:normN} since $\phi_i$ preserves the norm.
\end{proof}

As a result of irreducibility of $R_{\Delta,J}$, the sections $\varphi_{i,n}$ satisfy various differential equations. In order to formulate these equations, we review the definitions of certain standard differential operators acting on smooth sections of $(T^{*}M)^{\otimes n}$ for a Riemannian three-manifold $M$.

\begin{definition}
Let $\varphi$ be a smooth section of $(T^{*}M)^{\otimes n}$. Then $\mathrm{div}(\varphi)$ is the smooth section of $(T^{*}M)^{\otimes(n-1)}$ given by the trace of $\nabla \varphi$ over the first two arguments. We also define the Laplacian of $\varphi$ as
$$
\lap \varphi = -\mathrm{div}(\nabla \varphi)\,,
$$
where $\nabla$ is the covariant derivative (Levi--Civita connection).
\end{definition}

In order to define the curl operator, we first recall the Hodge star operation.
\begin{definition}
Let $\omega$ be the volume form on $M$ and $X$ a vector field on $M$. The Hodge star of $X$ is the section of $\mathrm{\Lambda}^{2}(TM)$ defined by
$$
{\star}X = \omega(\cdot,\cdot,X)^{\sharp}\,,
$$
where $\sharp:\mathrm{\Lambda}^{2}(T^{*}M) \rightarrow \mathrm{\Lambda}^{2}(TM)$ is the musical isomorphism. Let $X_1$, $X_2$ be vector fields on $M$. Their cross product is the vector field defined by
$$
X_1 \times X_2 = \omega(\cdot,X_1,X_2)^{\sharp}\,.
$$
\end{definition}

\begin{definition}\label{def:curl}
Let $\varphi$ be a smooth traceless section of $\mathrm{Sym}^{n}(T^{*}M)$ and let $X_1,\ldots, X_n$ be any collection of vector fields on $M$. We define $\mathrm{curl}(\varphi)$ to be the section of $\mathrm{Sym}^{n}(T^{*}M)$ satisfying
$$
\mathrm{curl}(\varphi)(X_1,\ldots,X_n) = \frac{1}{n}\sum\limits_{i=1}^{n}\nabla \varphi({\star}X_i,X_1,\ldots,\widehat{X_i},\ldots,X_n)\,.
$$
\end{definition}
It is not hard to check that $\mathrm{curl}(\varphi)$ is symmetric and traceless. We are now ready to state the following result:

\begin{proposition}\label{prop:differentialEqs}
Let $\phi:R_{\Delta,J} \rightarrow L^2(\Gamma\backslash G)$ be an injective $G$-linear map and set $\phi^{(n)} = \phi\circ \kappa^{(n)}_{\Delta,J}$. Let $\varphi_{n}$ be the smooth traceless section of $\mathrm{Sym}^{n}(T^{*}\Gamma\backslash\mathbb{H}^3)$ associated to $\phi^{(n)}$ as in~\eqref{eq:varphiDef}. Then
\begin{enumerate}
\item $\lap \varphi_{n} = [\Delta(2-\Delta)+n(n+1)-J^2] \varphi_{n}$.
\item If $J > 0$, we have $\mathrm{div}(\varphi_{J}) = 0$.
\item $\mathrm{curl}(\varphi_{n}) = -i(\Delta-1)\frac{J}{n}\,\varphi_{n}$.
\end{enumerate}
\end{proposition}
To prove part (3) of this proposition, we will need a clean description of the action of the Lie algebra elements $A_a$ on the sections $\varphi$. For $u\in V_1$, let us denote $A(u) = \sum_{a=1}^{3}A_a u^a$. We can now proceed analogously to the discussion preceding Proposition~\ref{prop:nabla}. Given a $K$-linear map $\theta: V_n \rightarrow C^{\infty}(\Gamma\backslash G)$, we can define a new $K$-linear map $\widetilde{\theta}: V_1\otimes V_{n}\rightarrow C^{\infty}(\Gamma\backslash G)$ by acting with the Lie algebra elements $A(u)$:
$$
\forall u\in V_1,\, v\in V_n\,:\qquad \widetilde{\theta}(u\otimes v)(g(x,k)) := \left.\frac{d}{dt}\right|_{t=0}\theta(v)(g(x,k)\mathrm{exp}(t A(u)))\,.
$$

\begin{lemma}\label{lem:Aaction}
Let $\varphi$ be the smooth section of $\mathrm{Sym}^{n}(T^{*}\Gamma\backslash\mathbb{H}^3)$ corresponding to the map $\theta$ and let $\widetilde{\varphi}$ be the smooth section of $(T^{*}\Gamma\backslash\mathbb{H}^3)^{\otimes(n+1)}$ corresponding to the map $\widetilde{\theta}$. Then for any collection of vector fields $X_1,\ldots,X_{n+1}$
$$
\widetilde{\varphi}(X_1,\ldots,X_{n+1}) = \sum\limits_{i=2}^{n+1}\varphi(X_1\times X_i,X_2,\ldots,\widehat{X_{i}},\ldots,X_{n+1})\,.
$$
\end{lemma}
\begin{proof}
Note that $A(u)v = u\times v$ for any $u,v\in V_1$, where $\times$ is the standard cross product on $V_1$. Also note that the cross product of vector fields becomes the cross product on $V_1$ when expressed in the basis $e_a$ of vector fields. The proof then follows by a direct calculation.
\end{proof}

\begin{proof}[Proof of Proposition~\ref{prop:differentialEqs}]
To prove part (1), we act with the Casimir element $C_2$ on $\varphi_{n}$. $A_aA_a$ is the quadratic Casimir of the $\mathrm{SO}(3)$ subgroup and it therefore acts on $\varphi_{n}$ by multiplication by the constant $-n(n+1)$. It follows from Proposition~\ref{prop:nabla} that $-B_aB_a \varphi_{n} = \lap \varphi_{n}$. Part (1) of the Proposition then follows from Lemma~\ref{lem:casimirs}.

To prove part (2), let $\gamma_n: V_{n-1}\rightarrow V_1\otimes V_{n}$ be a nonzero $K$-linear map (unique up to a multiplicative constant), and define $\delta_n: V_1\otimes V_n \rightarrow R_{\Delta,J}$ by $\delta_n(u,v) = B(u)\cdot\kappa_{\Delta,J}^{(n)}(v)$. Since $R_{\Delta,J}$ does not contain $V_{J-1}$, we get $ \delta_{J}\circ \gamma_{J} = 0$. Applying $\phi$ to this relation, and using Proposition~\ref{prop:nabla}, we get $\mathrm{div}(\varphi_{J}) = 0$. 

Finally, to prove part (3), we act with the Casimir element $C'_2$ on $\varphi_{n}$. It follows from Lemma~\ref{lem:Aaction} that
$$
C'_2\cdot \varphi_{n} = \mathrm{div}(\widetilde{\varphi_{n}}) = - n \,\mathrm{curl}(\varphi_{n})\,.
$$
Part (3) now follows directly from Lemma~\ref{lem:casimirs}.
\end{proof}

Proposition~\ref{prop:differentialEqs} implies the main result of this section:
\begin{theorem}\label{thm:spectralRelation}
When $R_{\Delta,0}$ appears inside $L^2(\Gamma\backslash G)$, it gives rise to an eigenfunction $\varphi$ of the Laplacian on $\Gamma\backslash\hh$ with eigenvalue $\Delta(2-\Delta)$
$$
\lap \varphi =\Delta(2-\Delta) \varphi\,.
$$
Similarly, when $R_{\Delta,J}$ with $J> 0$ appears inside $L^2(\Gamma\backslash G)$, it gives rise to a smooth, traceless, divergence-free section $\varphi$ of $\mathrm{Sym}^{J}(T^{*}\Gamma\backslash\mathbb{H}^3)$, which is an eigensection of the curl operator
$$
\mathrm{div}(\varphi) = 0\,,\qquad \mathrm{curl}(\varphi) = -i (\Delta-1)\,\varphi\,.
$$
Furthermore, the correspondence is one-to-one: Given a $\varphi$ satisfying these differential equations, there is a corresponding injective $G$-linear map $\phi : R_{\Delta,J} \rightarrow L^2(\Gamma\backslash G)$. 
\end{theorem}
\begin{proof}
Given a nonzero $\phi:R_{\Delta,J}\rightarrow L^2(\Gamma\backslash G)$, let $\phi^{(J)} = \phi\circ\kappa_{\Delta,J}^{(J)}$ and take $\varphi = \varphi_{J}$ associated to $\phi^{(J)}$ through~\eqref{eq:varphiDef}. Then $\varphi$ satisfies the stated differential equation thanks to~Proposition~\ref{prop:differentialEqs}. Conversely, we need to show that any $\varphi$ satisfying the differential equations produces a $G$-linear map $\phi : R_{\Delta,J} \rightarrow L^2(\Gamma\backslash G)$. It suffices to exhibit the corresponding $(\mathfrak{g},K)$-module in $C^{\infty}(\Gamma\backslash G)$. To do that, we first define $\varphi_{n}$ by acting with $n-J$ covariant derivatives on $\varphi$ as in Corollary~\ref{cor:Bnabla}. This allows us to construct the family of maps $\phi^{(n)}:V_n\rightarrow C^{\infty}(\Gamma\backslash G)$ through~\eqref{eq:varphiDef}. These maps give rise to the correct $(\mathfrak{g},K)$-module. In particular, the commutator $[B_a,B_b]$ is correctly reproduced by the commutator of covariant derivatives since we are on a space of constant sectional curvature $-1$.
\end{proof}

Theorem~\ref{thm:spectralRelation} and its proof explain the connection, first stated in the Introduction in~\eqref{eq:spectrumIntroOld} and~\eqref{eq:spectrumIntro}, between the decomposition of $L^2(\Gamma\backslash G)$ into irreducibles, and the spectral problems~\eqref{eq:laplaceScalar},~\eqref{eq:curl}. Indeed, given the decomposition~\eqref{eq:decomposition2}, the eigensections of the spin-$J$ spectral problem are the sections $\varphi_{k,J_k}$ with $J_k=J$, and $t_k = -i(\Delta_k-1)$ are the corresponding eigenvalues.

From now on, we will write the section $\varphi_{k,J_k}$ more simply as $\varphi_k$. We will also define the spectrum $\Sigma_{\Gamma}$ of the cocompact lattice $\Gamma\subset \mathrm{PSL}_2(\mathbb{C})$ as the sequence $\Sigma_{\Gamma} = ((\Delta_k,J_k))_{k\in\mathbb{Z}_{> 0}}$. Recall Remark~\ref{rmk:spectrumNotation} for the relation between our notation for the full spectrum $((\Delta_k,J_k))_{k\in\mathbb{Z}_{> 0}} = ((1+i t_k,J_k))_{k\in\mathbb{Z}_{> 0}}$ and the individual spin-$J$ spectra $(t_\ell^{(J)})_{\ell\in\mathbb{Z}_{> 0}}$.

\begin{remark}\label{rmk:constants}
For many practical purposes, we can treat the trivial representation uniformly with the rest of the spectrum. Thus, we will set $(\Delta_0,J_0) = (0,0)$ and define $R_{0,0}:=\mathbb{C}$ to be the trivial representation of $G$. The spectral decomposition then reads $L^2(\Gamma\backslash G) = \bigoplus\limits_{k=0}^{\infty} R_{\Delta_{k},J_k}$.
\end{remark}

\section{Correlations and triple products}
\label{sec:correlators}
\subsection{Correlations of smooth vectors}
Our goal will be to derive universal constraints on the spectrum~\eqref{eq:decomposition2} of irreducible representations $R_{\Delta_i,J_i}$ appearing in $L^2(\Gamma\backslash G)$. These will follow from consistency of correlations, defined as follows:

\begin{definition}
Let $F_1,\ldots,F_N$ be smooth functions on $\Gamma\backslash G$. Their correlation is defined as the Haar average of the pointwise product
$$
\langle F_1\ldots F_N\rangle = \frac{1}{\mathrm{vol}(\Gamma\backslash G)}\int\limits_{\Gamma\backslash G}\!F_1(g)\ldots F_N(g)dg\,.
$$
\end{definition}
Smoothness of $F_i$ and compactness of $\Gamma\backslash G$ guarantee that the correlations are finite. The normalization is chosen so that $\langle 1\rangle = 1$. It follows immediately from the definition that correlations are $G$-invariant, i.e.,
$$
\langle \cdot \ldots \cdot\rangle:\; (C^{\infty}(\Gamma\backslash G))^N \rightarrow \mathbb{C}
$$
is an invariant functional.

We will focus on the situation where $F_j$ with $j\in\{1,\ldots,N\}$ are smooth vectors in irreducible representations, i.e., for each $j$, we have $F_j = \phi_{i}(f_j)$ for some $i\in\mathbb{Z}_{>0}$ and some $f_j\in R_{\Delta_i,J_i}^{\infty}$. Note that the space of smooth vectors $R_{\Delta,J}^{\infty}$ coincides with smooth functions on the Riemann sphere belonging to $R_{\Delta,J}$. Furthermore, for any $f\in R_{\Delta_i,J_i}^{\infty}$, $\phi_i(f)$ is a smooth function on $\Gamma\backslash G$. Thus, we will analyze the collection of $G$-invariant functionals
\be\label{eq:correlationFunctions}
\langle\phi_{i_1}(\cdot)\ldots \phi_{i_N}(\cdot)\rangle:\;R^{\infty}_{\Delta_{i_1},J_{i_1}}\times\ldots\times R^{\infty}_{\Delta_{i_N},J_{i_N}}\rightarrow \mathbb{C}\,.
\ee

For $N=1$, the correlation is the orthogonal projection onto the one-dimensional $G$-invariant subspace of $L^2(\Gamma\backslash G)$. This means that $\langle\phi_i(f)\rangle = 0$ for all $i>0$ and all $f\in R^{\infty}_{\Delta_i,J_i}$.

The next simplest case is $N=2$, i.e., the two-point correlation. It can be analyzed through its relation to the inner product on $L^2(\Gamma\backslash G)$ and is covered by the following proposition:
\begin{proposition}\label{prop:2ptFunction}
Let $R_{\Delta_i,J_i}$ and $R_{\Delta_j,J_j}$ appear in $L^2(\Gamma\backslash G)$ and let $\phi_{i}: R_{\Delta_i,J_i}\rightarrow L^2(\Gamma\backslash G)$, $\phi_{j}: R_{\Delta_j,J_j}\rightarrow L^2(\Gamma\backslash G)$ be the corresponding $G$-linear, norm-preserving, real maps, as described in Section~\ref{ssec:spectrum}. The two-point correlations take the form
$$
\langle\phi_i(f_1)\phi_j(f_2)\rangle = \delta_{ij}\,\pi\int\limits_{\mathbb{C}}\! f_1(x)(\mathcal{S}_{\Delta_j,J_j}f_2)(x) dx\,,
$$
where $\mathcal{S}_{\Delta,J}$ is the intertwining operator defined by~\eqref{eq:shadow}.
\end{proposition}
\begin{proof}
Note that $\langle\phi_i(f_1)\phi_j(f_2)\rangle = (\phi_i(f_1),\overline{\phi_j(f_2)})_{\Gamma\backslash G}$, where $ (\cdot,\cdot)_{\Gamma\backslash G}$ is the inner product on $L^2(\Gamma\backslash G)$. Since $\phi_j$ is a real, norm-preserving map, we get
$$
\langle\phi_i(f_1)\phi_j(f_2)\rangle = (\phi_i(f_1),\phi_j({\rho}(f_2)))_{\Gamma\backslash G} = \delta_{ij}(f_1,{\rho}(f_2))_{R_{\Delta_i,J_i}}\,,
$$
where ${\rho}: R_{\Delta_j,J_j}\rightarrow R_{\Delta_j,J_j}$ is the antilinear involution described in Proposition~\ref{prop:realStructure} and $(\cdot,\cdot)_{\Delta_i,J_i}$ is the inner product on $R_{\Delta_i,J_i}$, described in Definitions~\ref{def:principal} and~\ref{def:complementary}. The proposition follows from the explicit forms of ${\rho}$ and $(\cdot,\cdot)_{\Delta_i,J_i}$ quoted therein.
\end{proof}
\begin{remark}
Proposition~\ref{prop:2ptFunction} reflects the observation that the space of invariant maps \linebreak${R^{\infty}_{\Delta_{1},J_{1}}\times R^{\infty}_{\Delta_{2},J_{2}}\rightarrow\mathbb{C}}$ is one-dimensional if $(\Delta_1,J_1)=(\Delta_2,J_2)$ and zero-dimensional otherwise. It is not hard to check explicitly that the distribution $(x_1-x_2)^{-\Delta_1+J_1}\,(\xb_1-\xb_2)^{-\Delta_1-J_1}$ is indeed invariant. Uniqueness follows because the action of $G$ on ordered pairs of points on the Riemann sphere is transitive. In the context of conformal field theory, this statement is known as the uniqueness of two-point functions.
\end{remark}

We will also need a version of Proposition~\ref{prop:2ptFunction} for two-point correlations of $K$-finite vectors. Recall the $K$-linear maps $\phi^{(n)}_i:V_n\rightarrow C^{\infty}(\Gamma\backslash G)$, introduced in Section~\ref{ssec:spectrum}. Their two-point correlations are captured by the following lemma:

\begin{lemma}\label{lem:2ptK}
For all $i,j>0$, $n_1\geq J_i$, $n_2\geq J_j$ and $w_1\in V_{n_1}$, $w_2\in V_{n_2}$, we have
$$
\langle \phi_i^{(n_1)}(w_1) \phi_j^{(n_2)}(w_2) \rangle = \delta_{ij}\,\delta_{n_1 n_2}
\,p_{n_1}(\Delta_i,J_i) \,w_1\cdot w_2\,,
$$
where $p_n(\Delta,J)$ is given in~\eqref{eq:normN}.
\end{lemma}
\begin{proof}
We have
$$
\langle \phi_i^{(n_1)}(w_1) \phi_j^{(n_2)}(w_2) \rangle = (\phi_i^{(n_1)}(w_1),\phi_j^{(n_2)}(\overline{w_2}))_{\Gamma\backslash G} =
\delta_{ij}(\kappa_{\Delta_i,J_i}^{(n_1)}(w_1),\kappa_{\Delta_i,J_i}^{(n_2)}(\overline{w_1}))_{\Delta_i,J_i}\,,
$$
so the claim follows from Lemma~\ref{lem:normN}.
\end{proof}

\subsection{Three-point correlations}\label{ssec:threePtCorrelations}
The next natural step, which will be essential in the rest of this paper, is to study the three-point correlations $\langle\phi_i(f_1)\phi_j(f_2)\phi_k(f_3)\rangle$. It turns out that $G$-invariance fixes the dependence of this object on $f_1,\,f_2,\,f_3$ up to an overall constant. This is an immediate consequence of the following proposition, proved in \cite{Oksak1973,Loke2001}:

\begin{proposition}\label{prop:trilinear}
Let $R_{\Delta_1,J_1},\,R_{\Delta_2,J_2},\,R_{\Delta_3,J_3}$ each be a nontrivial unitary irreducible representation of $G=\mathrm{PSL}_2(\mathbb{C})$. Then the space of invariant trilinear functionals
$$
R^{\infty}_{\Delta_1,J_1}\times R^{\infty}_{\Delta_2,J_2}\times R^{\infty}_{\Delta_3,J_3}\rightarrow\mathbb{C}
$$
is one-dimensional. Each such functional is proportional to the normalized functional
$$
\mathcal{T}_{\Delta_1,J_1;\Delta_2,J_2;\Delta_3,J_3}(f_1,f_2,f_3)= \int\limits_{\mathbb{C}^3}\! f_1(x_1)f_2(x_2)f_3(x_3)
\tau_{\Delta_1,J_1;\Delta_2,J_2;\Delta_3,J_3}(x_1,x_2,x_3)
dx_1dx_2dx_3\,.
$$
Here
\ba
\tau_{\Delta_1,J_1;\Delta_2,J_2;\Delta_3,J_3}(x_1,x_2,x_3) &=
x_{12}^{-h_1-h_2+h_3}x_{13}^{-h_1-h_3+h_2}x_{23}^{-h_2-h_3+h_1}\times\\
&\quad\times\xb_{12}^{-\hb_1-\hb_2+\hb_3}\xb_{13}^{-\hb_1-\hb_3+\hb_2}\xb_{23}^{-\hb_2-\hb_3+\hb_1}\,,
\label{eq:threePointStructure}
\ea
where $x_{ij} = x_i-x_j$ and we introduced the notation
$$
h_i = \frac{\Delta_i-J_i}{2}\,,\quad\hb_i = \frac{\Delta_i+J_i}{2}\,.
$$
\end{proposition}
\begin{remark}
Expression~\eqref{eq:threePointStructure} is well-defined for $x_1,\,x_2,\,x_3\in\mathbb{R}$ and $x_1>x_2>x_3$. For general $x_1,x_2,x_3\in\mathbb{C}$, it is defined by analytic continuation from that region. The result does not depend on the path of analytic continuation because $J_1,\,J_2,\,J_3\in\mathbb{Z}$.
\end{remark}
To prove Proposition~\ref{prop:trilinear}, it will be convenient to first state the following invariance property of $\tau_{\Delta_1,J_1;\Delta_2,J_2;\Delta_3,J_3}(x_1,x_2,x_3)$:
\begin{lemma}\label{lem:3PtInvariance}
Let $\tau(x_1,x_2,x_3)$ be the distribution given in~\eqref{eq:threePointStructure}. Then for any $\begin{psmallmatrix}a&b\\c&d\end{psmallmatrix}\in\mathrm{SL}_2(\mathbb{C})$
$$
\tau\!\left(\tfrac{a x_1+b}{c x_1+d},\tfrac{a x_2+b}{c x_2+d},\tfrac{a x_3+b}{c x_3+d}\right) = \tau(x_1,x_2,x_3)\prod\limits_{i=1}^{3}\left[(c x_i+d)^{2h_i}(\bar{c} \,\bar{x}_i+\bar{d})^{2\widetilde{h}_i}\right]\,.
$$
\end{lemma}
\begin{proof}
The lemma follows immediately from the identity
$
\tfrac{a x_i+b}{c x_i+d} - \tfrac{a x_j+b}{c x_j+d} = \tfrac{x_{ij}}{(cx_i+d)(cx_j+d)}\,.
$
\end{proof}
\begin{proof}[Proof of Proposition~\ref{prop:trilinear}]
We start by proving that $\mathcal{T}$ is invariant. Let $g\cdot f_{i}$ for $i=1,2,3$ be the result of acting on $f_i$ with $g\in G$, see~\eqref{eq:principalAction}. Hence
$$
\begin{aligned}
\mathcal{T}(g\cdot f_1,g\cdot f_2,g\cdot f_3) &= 
\int\limits_{\mathbb{C}^3}\!
\tau(x_1,x_2,x_3)\prod\limits_{i=1}^{3}\left[ (-c\,x_i+a)^{2h_i-2}(-\bar{c}\,\bar{x}_i+\bar{a})^{2\widetilde{h}_i-2}
f_i\!\left(\tfrac{dx_i-b}{-cx_i+a}\right)dx_i\right]\\
&=\int\limits_{\mathbb{C}^3}\!
\tau\!\left(\tfrac{a y_1+b}{c y_1+d},\tfrac{a y_2+b}{c y_2+d},\tfrac{a y_3+b}{c y_3+d}\right)\prod\limits_{i=1}^{3}\left[ (c\,y_i+d)^{-2h_i}(\bar{c}\,\bar{y}_i+\bar{d})^{-2\widetilde{h}_i}
f_i(y_i)dy_i\right]\,.
\end{aligned}
$$
To go to the second line, we changed variables $x_i = (ay_i+b)/(cy_i+d)$. It now follows from Lemma~\ref{lem:3PtInvariance} that $\mathcal{T}(g\cdot f_1,g\cdot f_2,g\cdot f_3) = \mathcal{T}(f_1,f_2,f_3)$. To prove uniqueness, note that the action of $\mathrm{PSL}_2(\mathbb{C})$ on ordered triplets of points on the Riemann sphere is transitive. Hence the value of the distribution $\tau(x_1,x_2,x_3)$ for any triplet $(x_1,x_2,x_3)$ is fixed by its value on any particular triplet and so the space of such distributions is one-dimensional.
\end{proof}

We are now ready to state the general form of three-point correlations.
\begin{corollary}\label{cor:threePtFunctions}
Let $\phi_i,\,\phi_j,\,\phi_k:R_{\Delta_i,J_i},\,R_{\Delta_j,J_j},\,R_{\Delta_k,J_k}\rightarrow L^2(\Gamma\backslash G)$ be defined as in Section~\ref{ssec:spectrum}. Then for any smooth vectors $f_1\in R^{\infty}_{\Delta_i,J_i}$, $f_2\in R^{\infty}_{\Delta_j,J_j}$, $f_3\in R^{\infty}_{\Delta_k,J_k}$, we have
$$
\langle\phi_i(f_1)\phi_j(f_2)\phi_k(f_3)\rangle = c_{ijk}\,\mathcal{T}_{\Delta_i,J_i;\Delta_j,J_j;\Delta_k,J_k}(f_1,f_2,f_3),
$$
for some constant $c_{ijk}\in\mathbb{C}$.
\end{corollary}
We will soon see that $c_{ijk}$ is proportional to the integral of an invariant product of the eigensections $\varphi_i$, $\varphi_j$, $\varphi_k$ provided by Theorem~\ref{thm:spectralRelation}.

\begin{remark}
The discussion of this subsection is familiar from 2D conformal field theory. In that context, the distribution $\tau_{\Delta_1,J_1;\Delta_2,J_2;\Delta_3,J_3}(x_1,x_2,x_3)$ agrees with the three-point correlation function of local primary operators of left- and right-moving weights $(h_1,\hb_1)$, $(h_2,\hb_2)$, $(h_3,\hb_3)$.
\end{remark}
\begin{remark}\label{rmk:cSymmetry}
The three-point correlations $\langle\phi_i(f_1)\phi_j(f_2)\phi_k(f_3)\rangle$ are symmetric under simultaneous permutation of the indices $i,\,j,\,k$ and the functions $f_1,\,f_2,\,f_3$. At the same time, the functional $\mathcal{T}$ transforms under permutations with a possible sign,
$$
\mathcal{T}_{\Delta_2,J_2;\Delta_1,J_1;\Delta_3,J_3}(f_2,f_1,f_3) = (-1)^{J_1+J_2+J_3}\mathcal{T}_{\Delta_1,J_1;\Delta_2,J_2;\Delta_3,J_3}(f_1,f_2,f_3)\,.
$$
It follows that if $J_i+J_j+J_k\in 2\mathbb{Z}$, then $c_{ijk}$ is completely symmetric in its labels, and otherwise it is completely antisymmetric. In particular, $c_{iik} = 0$ if $J_k$ is odd.
\end{remark}

The next step is to study the three-point correlations of $K$-finite vectors. They are uniquely determined in terms of the three-point coefficients $c_{ijk}$ and the spectral data.
\begin{proposition}\label{prop:threePtKFinite}
Let $i,j,k>0$ and let $n_1\geq J_i$, $n_2\geq J_j$, $n_3\geq J_k$ be integers satisfying the triangle inequalities. Then
$$
\langle \phi_i^{(n_1)}(w_1) \phi_j^{(n_2)}(w_2)  \phi_k^{(n_3)}(w_3)\rangle =
c_{ijk}\,\alpha_{n_1,n_2,n_3}(\Delta_i,J_i;\Delta_j,J_j;\Delta_k,J_k)\,T_{n_1,n_2,n_3}(w_1,w_2,w_3)\,.
$$
Here $c_{ijk}$ is the three-point coefficient defined by Corollary~\ref{cor:threePtFunctions}. $T_{n_1,n_2,n_3}$ is the $K$-invariant trilinear form introduced in Definition~\ref{def:Trilinear}. Finally, $\alpha_{n_1,n_2,n_3}(\Delta_1,J_1;\Delta_2,J_2;\Delta_3,J_3)$ is a universal function, independent of $\Gamma$.
\end{proposition}
\begin{proof}
$\langle \phi_i^{(n_1)}(\cdot) \phi_j^{(n_2)}(\cdot)  \phi_k^{(n_3)}(\cdot)\rangle$ is a $K$-invariant trilinear form $V_{n_1}\times V_{n_2}\times V_{n_3} \rightarrow \mathbb{C}$, and so it must be proportional to $T_{n_1,n_2,n_3}$. The proposition then follows immediately from Corollary~\ref{cor:threePtFunctions}, with $\alpha_{n_1,n_2,n_3}$ given by
\be\label{eq:alphaDef}
\!\!\!\alpha_{n_1,n_2,n_3}(\Delta_1,J_1;\Delta_2,J_2;\Delta_3,J_3) =
\frac{\mathcal{T}_{\Delta_1,J_1;\Delta_2,J_2;\Delta_3,J_3}(\kappa_{\Delta_1,J_1}^{(n_1)}(w_1),\kappa_{\Delta_2,J_2}^{(n_2)}(w_2),\kappa_{\Delta_3,J_3}^{(n_3)}(w_3))}{T_{n_1,n_2,n_3}(w_1,w_2,w_3)}
\ee
for any choice of $w_1,w_2,w_3$ such that $T_{n_1,n_2,n_3}(w_1,w_2,w_3)\neq 0$.
\end{proof}
To derive the bootstrap spectral identities, it will be essential to efficiently compute $\alpha_{n_1,n_2,n_3}$. To that end, we will formulate a recursion relation which relates $\alpha_{n_1,n_2,n_3}$ for different values of $(n_1,n_2,n_3)$.
\begin{lemma}\label{lem:recursion}
For any $n_1\geq J_1$, $n_2\geq J_2$, $n_3\geq J_3$, satisfying the triangle inequalities, $\alpha_{n_1,n_2,n_3}(\Delta_1,J_1;\Delta_2,J_2;\Delta_3,J_3)$ satisfies the recursion relations
$$
\begin{aligned}
&-\tfrac{(n_1+n_2-n_3+1)(n_1-n_2+n_3+1)}{(2n_1+1)(2n_1+2)}\alpha_{n_1+1,n_2,n_3}+\tfrac{n_1+n_2-n_3+1}{2n_2+2}\alpha_{n_1,n_2+1,n_3} +
\tfrac{n_1-n_2+n_3+1}{2n_3+2}\alpha_{n_1,n_2,n_3+1}\\
&+\frac{i}{\sqrt{2}}\left[\tfrac{(n_3-n_2)(\Delta_1 -1)J_1}{n_1(n_1+1)}+\tfrac{(\Delta_2 -1)J_2}{n_2+1}-\tfrac{(\Delta_3 -1)J_3}{n_3+1}
\right]\alpha_{n_1,n_2,n_3}\\
&+\tfrac{(n_1-J_1)(n_1+J_1) (\Delta_1+n_1-1) (\Delta_1 -n_1-1)}{n_1 (2 n_1+1)}\alpha_{n_1-1,n_2,n_3} = 0\,,
\end{aligned}
$$\vspace{10pt}
$$
\begin{aligned}
&-\tfrac{(n_2+n_3-n_1+1)(n_2-n_3+n_1+1)}{(2n_2+1)(2n_2+2)}\alpha_{n_1,n_2+1,n_3}+\tfrac{n_2+n_3-n_1+1}{2n_3+2}\alpha_{n_1,n_2,n_3+1} +
\tfrac{n_2-n_3+n_1+1}{2n_1+2}\alpha_{n_1+1,n_2,n_3}\\
&+\frac{i}{\sqrt{2}}\left[\tfrac{(n_1-n_3)(\Delta_2 -1)J_2}{n_2(n_2+1)}+\tfrac{(\Delta_3 -1)J_3}{n_3+1}-\tfrac{(\Delta_1 -1)J_1}{n_1+1}
\right]\alpha_{n_1,n_2,n_3}\\
&+\tfrac{(n_2-J_2)(n_2+J_2) (\Delta_2+n_2-1) (\Delta_2 -n_2-1)}{n_2 (2 n_2+1)}\alpha_{n_1,n_2-1,n_3} = 0\,,
\end{aligned}
$$\vspace{10pt}
$$
\begin{aligned}
&-\tfrac{(n_3+n_1-n_2+1)(n_3-n_1+n_2+1)}{(2n_3+1)(2n_3+2)}\alpha_{n_1,n_2,n_3+1}+\tfrac{n_3+n_1-n_2+1}{2n_1+2}\alpha_{n_1+1,n_2,n_3} +
\tfrac{n_3-n_1+n_2+1}{2n_2+2}\alpha_{n_1,n_2+1,n_3}\\
&+\frac{i}{\sqrt{2}}\left[\tfrac{(n_2-n_1)(\Delta_3 -1)J_3}{n_3(n_3+1)}+\tfrac{(\Delta_1 -1)J_1}{n_1+1}-\tfrac{(\Delta_2 -1)J_2}{n_2+1}
\right]\alpha_{n_1,n_2,n_3}\\
&+\tfrac{(n_3-J_3)(n_3+J_3) (\Delta_3+n_3-1) (\Delta_3 -n_3-1)}{n_3 (2 n_3+1)}\alpha_{n_1,n_2,n_3-1} = 0\,.
\end{aligned}
$$
Note that the three equations are related by cyclic permutations of all labels.
\end{lemma}
The proof of this lemma is somewhat technical, and we supply it in Appendix~\ref{app:recursionProof}.

\subsection{Triple products of eigensections}\label{ssec:triple}
We would now like to relate the three-point correlation coefficients $c_{ijk}$, defined in Corollary~\ref{cor:threePtFunctions}, to triple products of eigensections. Besides being interesting in its own right, this relation can be used to find the complex phase of $c_{ijk}$.

In order to define the triple products, let us recall the normalized sections $\varphi_i$ provided by Theorem~\ref{thm:spectralRelation}. 
\begin{definition}\label{def:cHats}
Let $i,j,k>0$. If $J_i,\,J_j,\,J_k$ satisfy the triangle inequalities, we define the triple products by
$$
\widehat{c}_{ijk} =
\frac{1}{\mathrm{vol}(\Gamma\backslash \mathbb{H}^3)}\int\limits_{\Gamma\backslash\mathbb{H}^3}\!T_{J_i,J_j,J_k}(\varphi_i(x),\varphi_j(x),\varphi_k(x))\,d\mu(x)\,.
$$
On the other hand, if $J_k>J_i+J_j$, we define
$$
\widehat{c}_{ijk} =
\frac{1}{\mathrm{vol}(\Gamma\backslash \mathbb{H}^3)}\int\limits_{\Gamma\backslash\mathbb{H}^3}\!T_{J_k-J_j,J_j,J_k}(\nabla^{J_k-J_i-J_j}\varphi_i(x),\varphi_j(x),\varphi_k(x))\,d\mu(x)\,.
$$
The remaining cases $J_i>J_j+J_k$ and $J_j>J_i+J_k$ are covered by defining $\widehat{c}_{ijk}$ to be completely symmetric or antisymmetric if $J_i+J_j+J_k$ is even or odd.
\end{definition}

\begin{remark}
Note that $\widehat{c}_{ijk}\in\mathbb{R}$ by definition. Note also that $\widehat{c}_{ijk}$ transforms in the same way as $c_{ijk}$ under permutations of $i,j,k$, see Remark~\ref{rmk:cSymmetry}. In the first case, this is because $T_{n_1,n_2,n_3}$ transforms correctly. In the second case, we need to integrate by parts $(J_k-J_i-J_j)$ times to relate $\widehat{c}_{ijk}$ and $\widehat{c}_{jik}$.
\end{remark}

The precise proportionality between three-point correlation coefficients $c_{ijk}$ and triple products $\widehat{c}_{ijk}$ can now be stated as the following theorem:

\begin{theorem}\label{thm:acRelation}
Let $i,j,k>0$. The three-point correlation coefficients $c_{ijk}$ and triple products $\widehat{c}_{ijk}$ are proportional, $c_{ijk} = \mathfrak{q}(\Delta_i,J_i;\Delta_j,J_j;\Delta_k,J_k)\,\widehat{c}_{ijk}$, with a universal proportionality constant $\mathfrak{q}$. Indeed, if $J_i,\,J_j,\,J_k$ satisfy the triangle inequalities, then
$$
\mathfrak{q}(\Delta_i,J_i;\Delta_j,J_j;\Delta_k,J_k) = \frac{(1-2h_i)(1-2h_j)(1-2h_k)}{q(-h_i,-h_j,-h_k)}\,.
$$
On the other hand, if $J_k>J_i+J_j$, we have
$$
\mathfrak{q}(\Delta_i,J_i;\Delta_j,J_j;\Delta_k,J_k) =\frac{(1-2h_i)(1-2h_j)(1-2h_k)}{(\sqrt{2}i)^{J_k-J_i-J_j}(h_k-h_i-h_j+1)_{J_k-J_i-J_j}\,q(-h_i,-h_j,-h_k)}\,.
$$
The remaining cases are covered by declaring $\mathfrak{q}$ to be invariant under all permutations of $i,j,k$. Here $q(n_1,n_2,n_3)$ is the ratio of products of gamma functions given in~\eqref{eq:qDefinition}.
\end{theorem}
Since $\widehat{c}_{ijk}\in\mathbb{R}$, the theorem in particular fixes the complex phase of $c_{ijk}$ modulo $\pi$ in terms of the spectral data. To prove the theorem, we start with the following lemma:
\begin{lemma}
For all $i,j,k>0$ and all $n_1\geq J_1$, $n_2\geq J_j$, $n_3\geq J_k$
$$
\begin{aligned}
&\frac{1}{\mathrm{vol}(\Gamma\backslash \mathbb{H}^3)}\int\limits_{\Gamma\backslash\mathbb{H}^3}\!T_{n_1,n_2,n_3}(\nabla^{n_1-J_i}\varphi_i(x),\nabla^{n_2-J_j}\varphi_j(x),\nabla^{n_3-J_k}\varphi_k(x))\,d\mu(x)=\\
&\qquad=c_{ijk}\,q(n_1,n_2,n_3)\,\alpha_{n_1,n_2,n_3}(\Delta_i,J_i;\Delta_j,J_j;\Delta_k,J_k) \,.
\end{aligned}
$$
\end{lemma}
\begin{proof}
By definition
\ba\label{eq:tripleProductIntegralG}
&\langle \phi_i^{(n_1)}(w_1) \phi_j^{(n_2)}(w_2)  \phi_k^{(n_3)}(w_3)\rangle \\
&\qquad= \frac{1}{\mathrm{vol}(\Gamma\backslash\hh)}\int\limits_{\Gamma\backslash\hh}\!\int\limits_{K}
(\varphi_{i,n_1}(x) \cdot k'w_1)(\varphi_{j,n_2}(x) \cdot k'w_2)(\varphi_{k,n_3}(x) \cdot k'w_3)dk' \,d\mu(x)\\
&\qquad=
\frac{1}{\mathrm{vol}(\Gamma\backslash \mathbb{H}^3)}\int\limits_{\Gamma\backslash\mathbb{H}^3}\!T_{n_1,n_2,n_3}(\nabla^{n_1-J_i}\varphi_i(x),\nabla^{n_2-J_j}\varphi_j(x),\nabla^{n_3-J_k}\varphi_k(x))\,d\mu(x)\\
&\qquad\qquad\qquad\times \frac{T_{n_1,n_2,n_3}(w_1,w_2,w_3)}{q(n_1,n_2,n_3)}\,,
\ea
where to go from second to third line, we integrated over $K$ using Lemma~\ref{lem:KHaar3}. The lemma then follows from Proposition~\ref{prop:threePtKFinite}.
\end{proof}
\begin{proof}[Proof of Theorem~\ref{thm:acRelation}]
To prove the theorem, it remains to evaluate $\alpha_{n_1,n_2,n_3}$ for certain special values of $n_1,n_2,n_3$. In particular, for the first case of the theorem, we need $\alpha_{J_1,J_2,J_3}$, while for the second case, we need $\alpha_{J_3-J_2,J_2,J_3}$. These special cases can be obtained by evaluating the triple integral which defines the trilinear functional $\mathcal{T}_{\Delta_1,J_1;\Delta_2,J_2;\Delta_3,J_3}$ on special test vectors. The calculations are quite technical and we defer them to Appendix~\ref{app:trilinear}. In particular, the results we need are recorded in Lemma~\ref{lem:integral1} and Lemma~\ref{lem:integral3}.
\end{proof}

\begin{remark}
Note that, in the spirit of Remark~\ref{rmk:constants}, Definition~\ref{def:cHats} makes sense also for $i=0$, i.e., when any of the eigenfunctions are the constant function. In that case, $\widehat{c}_{0jk} = \delta_{jk}$ for all $j,k\geq 0$. Consistency with Theorem~\ref{thm:acRelation} then requires that we set $c_{0jk} = (1-2h_j)\delta_{jk}$.
\end{remark}

We can now rewrite Proposition~\ref{prop:threePtKFinite} in a way that manifests reality of both sides:
\ba\label{eq:tripleProductReal}
\langle \phi_i^{(n_1)}(w_1) \phi_j^{(n_2)}(w_2)  \phi_k^{(n_3)}(w_3)\rangle =
&\widehat{c}_{ijk}\,\widehat{\alpha}_{n_1,n_2,n_3}(\Delta_i,J_i;\Delta_j,J_j;\Delta_k,J_k)\\
&\quad\times T_{n_1,n_2,n_3}(w_1,w_2,w_3)\,.
\ea

Here $\widehat{\alpha}_{n_1,n_2,n_3}$ is a rescaled version of the coefficients $\alpha_{n_1,n_2,n_3}$:
\ba\label{eq:alphaHat}
\widehat{\alpha}_{n_1,n_2,n_3}(\Delta_i,J_i;\Delta_j,J_j;\Delta_k,J_k) \coloneqq
 &\mathfrak{q}(\Delta_i,J_i;\Delta_j,J_j;\Delta_k,J_k)\\
&\quad\times\alpha_{n_1,n_2,n_3}(\Delta_i,J_i;\Delta_j,J_j;\Delta_k,J_k)\,,
\ea
with $\mathfrak{q}(\Delta_i,J_i;\Delta_j,J_j;\Delta_k,J_k)$ given in Theorem~\ref{thm:acRelation}.

\begin{remark}\label{rmk:alphaHats}
Since the left-hand side of~\eqref{eq:tripleProductReal} is real whenever $w_1,\,w_2,\,w_3$ are real tensors, it must be that $\widehat{\alpha}_{n_1,n_2,n_3}(\Delta_i,J_i;\Delta_j,J_j;\Delta_k,J_k)$ is real whenever all of its arguments are on the principal or complementary series.

For any $J_i,\, J_j,\, J_k$ and $n_1,\,n_2,\,n_3$, $\widehat{\alpha}_{n_1,n_2,n_3}(\Delta_i,J_i;\Delta_j,J_j;\Delta_k,J_k)$ is a rational function of $\Delta_i$, $\Delta_j$, $\Delta_k$. To see this, note that $\widehat{\alpha}_{n_1,n_2,n_3}$ satisfies the same recursion relations of Lemma~\ref{lem:recursion} as $\alpha_{n_1,n_2,n_3}$. For any $J_i,\, J_j,\, J_k$, these recursion relations can be used to determine all $\widehat{\alpha}_{n_1,n_2,n_3}$ up to an overall $(n_1,n_2,n_3)$-independent function of $\Delta_i$, $\Delta_j$, $\Delta_k$. This function is fixed by comparing~\eqref{eq:tripleProductIntegralG} with the special cases featuring in Definition~\ref{def:cHats}. The result is that whenever $J_i,\,J_j,\,J_k$ satisfy the triangle inequalities, we have
$$
\widehat{\alpha}_{J_i,J_j,J_k}(\Delta_i,J_i;\Delta_j,J_j;\Delta_k,J_k) = \frac{1}{q(J_i,J_j,J_k)}\,.
$$
On the other hand, when $J_i,\,J_j,\,J_k$ do not satisfy the triangle inequalities and $J_i+J_j < J_k$, we have
$$
\widehat{\alpha}_{n,J_k-n,J_k}(\Delta_i,J_i;\Delta_j,J_j;\Delta_k,J_k) = \frac{(-1)^{n+J_j+J_k}}{2J_k+1}
$$
for any $n\in\{J_i,\ldots, J_k-J_j\}$. The remaining cases are covered by using the fact that $\widehat{\alpha}_{n_i,n_j,n_k}(\Delta_i,J_i;\Delta_j,J_j;\Delta_k,J_k)$ transforms under an arbitrary transposition of the labels $i,\,j,\, k$ with the sign $(-1)^{n_i+n_j+n_k+J_i+J_j+J_k}$, e.g.,
$$
\widehat{\alpha}_{n_1,n_2,n_3}(\Delta_i,J_i;\Delta_j,J_j;\Delta_k,J_k) =
(-1)^{n_1+n_2+n_3+J_i+J_j+J_k} \widehat{\alpha}_{n_1,n_3,n_2}(\Delta_i,J_i;\Delta_k,J_k;\Delta_j,J_j)\,.
$$
\end{remark}

\section{Product expansion and associativity}
\label{sec:OPE-and-associativity}
\subsection{Product expansion}
One of the conclusions of the previous section is that the information contained in correlations~\eqref{eq:correlationFunctions} for $N\leqslant 3$ is equivalent to the knowledge of the spectrum $(\Delta_i,J_i)$ and the three-point correlation coefficients $c_{ijk}$ for all $i,j,k>0$. We will now explain how the same data can be used to reconstruct \emph{all} correlations for $N>3$.

The idea is to apply the spectral decomposition~\eqref{eq:decomposition2} to the pointwise product $\phi_i(f_1)\phi_j(f_2)$. Due to $G$-invariance, the spectral decomposition is uniquely determined by the aforementioned spectral data.

\begin{proposition}[Product expansion]\label{prop:ope}
Let $i,j>0$ and let $f_1\in R^{\infty}_{\Delta_i,J_i}$ and $f_2\in R^{\infty}_{\Delta_j,J_j}$. Then $\phi_i(f_1)\phi_j(f_2)\in L^2(\Gamma\backslash G)$ and its spectral decomposition takes the form
$$
\phi_i(f_1)\phi_j(f_2) =  \langle\phi_i(f_1)\phi_j(f_2)\rangle+ \sum\limits_{m=1}^{\infty}c_{ijm}\,\phi_m(f_1\star f_2)\,.
$$
Here $\langle\phi_i(f_1)\phi_j(f_2)\rangle$ is the two-point correlation, $c_{ijm}\in\mathbb{C}$ are the three-point correlation coefficients, and
\be\label{eq:starProduct}
(f_1\star f_2)(x_3) =
\int\limits_{\mathbb{C}^2} f_1(x_1)f_2(x_2)
\frac{\tau_{\Delta_i,J_i;\Delta_j,J_j;2-\Delta_m,-J_m}(x_1,x_2,x_3)}{(1-2h_m)S(\Delta_i,J_i;\Delta_j,J_j|\Delta_m,J_m)}dx_1dx_2\,,
\ee
where

\be\label{eq:SFactor}
S(\Delta_i,J_i;\Delta_j,J_j|\Delta_m,J_m) =
\tfrac{\pi(-1)^{J_i+J_j+J_m}\Gamma(\hb_i-\hb_j+\hb_m)\Gamma(-\hb_i+\hb_j+\hb_m)\Gamma(1-2\hb_m)}{\Gamma(1+h_i-h_j-h_m)\Gamma(1-h_i+h_j-h_m)\Gamma(2h_m)}\,.
\ee
\end{proposition}

To prove the proposition, we first state the following lemma, which provides an interpretation of the function $S(\Delta_i,J_i;\Delta_j,J_j|\Delta_m,J_m)$:
\begin{lemma}\label{lem:shadowTransformTau}
Let $f_1\in R^{\infty}_{\Delta_i,J_i}$, $f_2\in R^{\infty}_{\Delta_j,J_j}$, and $f_3\in R^{\infty}_{\Delta_m,J_m}$. The functional $\mathcal{T}$ satisfies
$$
\mathcal{T}_{\Delta_i,J_i;\Delta_j,J_j;2-\Delta_m,-J_m}(f_1,f_2,\mathcal{S}_{\Delta_m,J_m}(f_3))=
\tfrac{1-\Delta_m+|J_m|}{\pi}S(\Delta_i,J_i;\Delta_j,J_j|\Delta_m,J_m)
\mathcal{T}_{\Delta_i,J_i;\Delta_j,J_j;\Delta_m,J_m}(f_1,f_2,f_3),
$$
where $\mathcal{S}_{\Delta, J}$ is the intertwining operator defined in~\eqref{eq:shadow}.
\end{lemma}
\begin{proof}
The left-hand side is an invariant functional $R^{\infty}_{\Delta_i,J_i}\times R^{\infty}_{\Delta_j,J_j}\times R^{\infty}_{\Delta_m,J_m}\rightarrow\mathbb{C}$ and therefore is proportional to $\mathcal{T}_{\Delta_i,J_i;\Delta_j,J_j;\Delta_m,J_m}(\cdot,\cdot,\cdot)$ by Proposition~\ref{prop:trilinear}. The constant of proportionality is given by
$$
S(\Delta_i,J_i;\Delta_j,J_j|\Delta_m,J_m) =
\frac{\int\limits_{\mathbb{C}}\tau_{\Delta_i,J_i;\Delta_j,J_j;2-\Delta_m,-J_m}(x_1,x_2,x_4)x_{34}^{-2h_m}\overline{x}_{34}^{-2\hb_m}dx_4}{\tau_{\Delta_i,J_i;\Delta_j,J_j;\Delta_m,J_m}(x_1,x_2,x_3)}\,.
$$
After taking the limit $x_1\rightarrow\infty$, the right-hand side reduces to a convolution integral which can be evaluated by inserting the Fourier representation of $x_{34}^{-2h_4}\overline{x}_{34}^{-2\hb_3}$, with the result given by~\eqref{eq:SFactor}.
\end{proof}

\begin{proof}[Proof of Proposition~\ref{prop:ope}]
Since $f_1\in R^{\infty}_{\Delta_i,J_i}$ and $f_2\in R^{\infty}_{\Delta_j,J_j}$, we have $\phi_i(f_1)\phi_j(f_2)\in C^{\infty}(\Gamma\backslash G)\subset L^2(\Gamma\backslash G)$. The spectral decomposition then takes the form
\be\label{eq:ope}
\phi_i(f_1)\phi_j(f_2) = \sum\limits_{m=0}^{\infty}P_{m}(\phi_i(f_1)\phi_j(f_2))\,,
\ee
where $P_0$ is the orthogonal projection on the trivial representation and $P_{m}$ with $m>0$ is the orthogonal projection onto the image of $\phi_m$. The sum over $m$ converges in the $L^2$ sense. The $P_0$ term is precisely the two-point correlation $\langle\phi_i(f_1)\phi_j(f_2)\rangle$, whose general form is given in Proposition~\ref{prop:2ptFunction}.

Similarly, compatibility with the $G$-action fixes $P_{m}(\phi_i(f_1)\phi_j(f_2))$ up to an overall multiplicative constant. This is because all $G$-linear maps
$$
R_{\Delta_i,J_i}^{\infty}\times R_{\Delta_j,J_j}^{\infty} \rightarrow R_{\Delta_m,J_m}
$$
are constant multiples of
$$
(f_1,f_2)\mapsto
\int\limits_{\mathbb{C}^2} f_1(x_1)f_2(x_2)
\tau_{\Delta_i,J_i;\Delta_j,J_j;2-\Delta_m,-J_m}(x_1,x_2,x_3)dx_1 dx_2\,.
$$
The proof of this statement follows from the same argument as in the proof of Proposition~\ref{prop:trilinear}. To find the constant of proportionality between this expression and $P_{m}(\phi_i(f_1)\phi_j(f_2))$, we multiply both sides of~\eqref{eq:ope} by $\phi_m(f_3)$ and compute the expectation value using in turn Proposition~\ref{prop:2ptFunction} and Corollary~\ref{cor:threePtFunctions}. The proof then follows from Lemma~\ref{lem:shadowTransformTau}.
\end{proof}

\begin{remark}
One can apply Proposition~\ref{prop:ope} to write any $N$-point correlation as an infinite, convergent sum of $(N-1)$-point correlations. The sum over $m$ can be exchanged with the operation of integrating over $\Gamma\backslash G$ thanks to $L^2$ convergence and the Cauchy--Schwartz inequality. Proceeding inductively down to $N=3$, we see that all $N$-point correlations are determined by the three-point correlation coefficients $c_{ijk}$ and the spectrum $(\Delta_i,J_i)$. We will refer to the three-point correlation coefficients and the spectrum collectively as the spectral data.
\end{remark}

Our next order of business is to apply the product expansion of Proposition~\ref{prop:ope} to $K$-finite vectors. To describe the result, let us recall the map $T^{\vee}_{n_1,n_2,n_3}: V_{n_1}\otimes V_{n_2}\rightarrow V_{n_3}$ from Definition~\ref{def:Tcheck}.

\begin{corollary}\label{cor:opeKFinite}
Let $i,j>0$, $n_1\geq J_i$, $n_2\geq J_j$ and $w_1\in V_{n_1}$, $w_2\in V_{n_2}$. Then
$$
\begin{aligned}
&\phi^{(n_1)}_i(w_1)\phi_j^{(n_2)}(w_2) = \langle\phi^{(n_1)}_i(w_1)\phi_j^{(n_2)}(w_2)\rangle +\\
&\qquad+\sum\limits_{\substack{k=1}}^{\infty}\;\sum\limits_{n_3=\max(J_k,|n_1-n_2|)}^{n_1+n_2}\widehat{c}_{ijk}
\tfrac{\widehat{\alpha}_{n_1,n_2,n_3}(\Delta_i,J_i;\Delta_j,J_j;\Delta_k,J_k)}{p_{n_3}(\Delta_k,J_k)}
\phi^{(n_3)}_k(T^{\vee}_{n_1,n_2,n_3}(w_1,w_2))\,.
\end{aligned}
$$
\end{corollary}
\begin{proof}
The corollary follows immediately from the spectral decomposition of $\phi^{(n_1)}_i(w_1)\phi_j^{(n_2)}(w_2)$ and the general form of two- and three-point correlations of $K$-finite vectors, stated in Lemma~\ref{lem:2ptK} and equation~\eqref{eq:tripleProductReal}.
\end{proof}

\subsection{The four-point correlations}\label{ssec:fourPoints}
We are finally ready to discuss the four-point correlations. Let us first make a preliminary definition.

\begin{definition}[Conformal partial waves]\label{def:CPW}
The conformal partial waves are given by
$$
\begin{aligned}
&\psi^{ijk\ell}_{m}(x_1,x_2,x_3,x_4) =\\
&\tfrac{1}{(1-2h_m)S(\Delta_i,J_i;\Delta_j,J_j|\Delta_m,J_m)}\int\limits_{\mathbb{C}}
\tau_{\Delta_i,J_i;\Delta_j,J_j;2-\Delta_m,-J_m}(x_1,x_2,x_5)
\tau_{\Delta_m,J_m;\Delta_k,J_k;\Delta_\ell,J_\ell}(x_5,x_3,x_4)
dx_5 \ .
\end{aligned}
$$
$\psi^{ijk\ell}_{m}$ is a function of four complex variables $x_1,\ldots,x_4$, defined whenever all of these are distinct. It depends on $i,j,k,\ell,m$ only through the spectral parameters $(\Delta_i,J_i)$, $(\Delta_j,J_j)$, $(\Delta_k,J_k)$, $(\Delta_\ell,J_\ell)$, $(\Delta_m,J_m)$.
\end{definition}

\begin{remark}
It is possible to find a closed formula for $\psi^{ijk\ell}_{m}(x_1,x_2,x_3,x_4)$ in terms of the Gauss hypergeometric function by evaluating the integral over $x_5$. The formula, which first appeared in literature on conformal field theory \cite{Ferrara:1973vz,Ferrara:1974ny,DO1,DO2}, is recorded in Appendix~\ref{app:partialwave}. However, we will not make use of this formula in the rest of this work.
\end{remark}

Note that $\psi^{ijk\ell}_{m}$ gives rise to a family of $G$-invariant functionals
\be\label{eq:Psi}
\begin{aligned}
&\Psi^{ijk\ell}_{m}:R_{\Delta_i,J_i}^{\infty}\times R_{\Delta_j,J_j}^{\infty}\times R^{\infty}_{\Delta_k,J_k}\times R^{\infty}_{\Delta_\ell,J_\ell}\rightarrow \mathbb{C}\\
&(f_1, f_2,f_3,f_4) \mapsto \Psi^{ijk\ell}_{m}(f_1,f_2,f_3,f_4) \coloneqq \int\limits_{\mathbb{C}^4}
\psi^{ijk\ell}_{m}(x_1,x_2,x_3,x_4) \prod\limits_{n=1}^{4}f_n(x_n)\prod\limits_{n=1}^{4}dx_n.
\end{aligned}
\ee
These functionals are universal in the sense that they depend only on $(\Delta_m,J_m)$ and make no reference to $\Gamma$.

\begin{remark}
For $i=j$, $k=\ell$ and $J_m = 0$, the functional $\Psi^{iikk}_{m}$ is continuous in $\Delta_m$ at $\Delta_m = 0$, and it is natural to define $\Psi^{iikk}_{0}$ to be its limit there. Using the formula in Appendix~\ref{app:partialwave}, we get
$$
(1-2h_i)(1-2h_k)\Psi^{iikk}_{0}(f_1,f_2,f_3,f_4) = \langle\phi_{i}(f_1)\phi_{i}(f_2)\rangle\langle\phi_{k}(f_3)\phi_{k}(f_4)\rangle\,.
$$
\end{remark}

The reason for introducing the functionals $\Psi^{ijk\ell}_{m}$ is that they arise from the product expansion of four-point correlations. The following theorem shows how the four-point correlations are determined in terms of the spectral data:
\begin{theorem}\label{thm:cpwExpansion}
The four-point correlations admit a representation as a convergent sum
$$
\langle\phi_{i}(f_1)\phi_{j}(f_2)\phi_{k}(f_3)\phi_{\ell}(f_4)\rangle =
 \sum\limits_{m=0}^{\infty}c_{ijm}c_{k\ell m} \Psi^{ijk\ell}_{m}(f_1,f_2,f_3,f_4)\,,
$$
where $c_{ijk}$ are the three-point correlation coefficients and $\Psi^{ijk\ell}_{m}$ are the functionals defined by~\eqref{eq:Psi}.
\end{theorem}
\begin{proof}
We apply the product expansion from Proposition~\ref{prop:ope} to $\phi_{i}(f_1)\phi_{j}(f_2)$. We can swap the summation with integration over $\Gamma\backslash G$, which leads to
$$
\langle\phi_{i}(f_1)\phi_{j}(f_2)\phi_{k}(f_3)\phi_{\ell}(f_4)\rangle =
 \langle\phi_i(f_1)\phi_j(f_2)\rangle  \langle\phi_k(f_3)\phi_{\ell}(f_4)\rangle +
 \sum\limits_{m=1}^{\infty}c_{ijm} \langle\phi_m(f_1\star f_2)\phi_k(f_3)\phi_{\ell}(f_4)\rangle\,.
$$
Recalling that $c_{ij0} = (1-2h_i)\delta_{ij}$ and the preceding remark, we see that the first term correctly reproduces the $m=0$ term of the sum. Finally, applying Corollary~\ref{cor:threePtFunctions} to the three-point correlation inside the sum produces
$$
\langle\phi_m(f_1\star f_2)\phi_k(f_3)\phi_{\ell}(f_4)\rangle = c_{mk\ell} \Psi^{ijk\ell}_{m}(f_1,f_2,f_3,f_4)\,.
$$
\end{proof}

In practice, we will be working with four-point correlations of $K$-finite vectors of the general form
$$
\langle\phi_i^{(n_1)}(w_1)\phi_j^{(n_2)}(w_2)\phi_k^{(n_3)}(w_3)\phi_\ell^{(n_4)}(w_4)\rangle
$$
for some $w_{1}\in V_{n_{1}}$, $w_{2}\in V_{n_{2}}$, $w_{3}\in V_{n_{3}}$, $w_{4}\in V_{n_{4}}$. As explained in Section~\ref{ssec:invariant4K}, without loss of generality, we can set $w_1 = v_1^{n_1}$, $w_2 = v_2^{n_2}$, $w_3 = v_3^{n_3}$, $w_4 = v_4^{n_4}$ for null vectors $v_{1,2,3,4}\in V_1$. Since $\langle\phi_i^{(n_1)}(w_1)\phi_j^{(n_2)}(w_2)\phi_k^{(n_3)}(w_3)\phi_\ell^{(n_4)}(w_4)\rangle$ is an invariant form on $V_{n_1}\otimes V_{n_2}\otimes V_{n_3}\otimes V_{n_4}$, it follows from Lemma~\ref{lem:quadrilinear} that, when $n_1+n_2+n_3+n_4$ is even, we can write
\ba\label{eq:4ptCorEven}
&\langle\phi_i^{(n_1)}(v_1^{n_1})\phi_j^{(n_2)}(v_2^{n_2})\phi_k^{(n_3)}(v_3^{n_3})\phi_\ell^{(n_4)}(v_4^{n_4})\rangle =
(v_1\cdot v_2)^{\frac{n_1+n_2-n_3-n_4}{2}} (v_1\cdot v_4)^{\frac{n_1-n_2-n_3+n_4}{2}}\\
 &\qquad\qquad\qquad\qquad\qquad\qquad
 \times(v_1\cdot v_3)^{n_3}(v_2\cdot v_4)^{\frac{-n_1+n_2+n_3+n_4}{2}}
\mathcal{G}_{ijk\ell}(n_1,n_2,n_3,n_4; r)\,.
\ea
Similarly, when $n_1+n_2+n_3+n_4$ is odd, we can write
\ba\label{eq:4ptCorOdd}
&\langle\phi_i^{(n_1)}(v_1^{n_1})\phi_j^{(n_2)}(v_2^{n_2})\phi_k^{(n_3)}(v_3^{n_3})\phi_\ell^{(n_4)}(v_4^{n_4})\rangle =
\frac{[v_1,v_2,v_4]}{\sqrt{2}}(v_2\cdot v_4)^{\frac{-n_1+n_2+n_3+n_4-1}{2}}
 \\
 &\qquad\times(v_1\cdot v_3)^{n_3}(v_1\cdot v_2)^{\frac{n_1+n_2-n_3-n_4-1}{2}} (v_1\cdot v_4)^{\frac{n_1-n_2-n_3+n_4-1}{2}}
\mathcal{G}_{ijk\ell}(n_1,n_2,n_3,n_4; r)\,.
\ea
Here $r = r(v_1,v_2,v_3,v_4)$ is the $K$-invariant ratio defined in~\eqref{eq:rDefinition1} and $\mathcal{G}_{ijk\ell}(n_1,n_2,n_3,n_4; r)$ is a polynomial in $r$ of degree at most $\min(2n_3,n_3+n_4+n_2-n_1)$, thanks to Remark~\ref{rmk:Hdegree}.

We are finally in a position to derive a version of the product expansion of Theorem~\ref{thm:cpwExpansion}, but this time for the polynomial $\mathcal{G}_{ijk\ell}(n_1,n_2,n_3,n_4; r)$:
\begin{proposition}\label{prop:productExpG}
The product expansion of $\mathcal{G}_{ijk\ell}(n_1,n_2,n_3,n_4; r)$ reads: 
$$
\mathcal{G}_{ijk\ell}(n_1,n_2,n_3,n_4; r) = 
\sum\limits_{m=0}^{\infty}\widehat{c}_{ijm}\widehat{c}_{k\ell m}\, \widehat{\Psi}^{ijk\ell}_{m}(n_1,n_2,n_3,n_4;r)\,,
$$
where
\ba\label{eq:PsiHat}
&\!\!\widehat{\Psi}^{ijk\ell}_{m}(n_1,n_2,n_3,n_4;r) \\
&\!\!\!\!\!\coloneqq\sum\limits_{n_5}
\frac{\widehat{\alpha}_{n_1,n_2,n_5}(\Delta_i,J_i;\Delta_j,J_j;\Delta_m,J_m)\widehat{\alpha}_{n_3,n_4,n_5}(\Delta_k,J_k;\Delta_\ell,J_\ell;\Delta_m,J_m)}{p_{n_5}(\Delta_m,J_m)}
\mathcal{H}^{n_1,n_2,n_3,n_4}_{n_5}(r)\,,
\ea
with the sum ranging over $\mathrm{max}(|n_1-n_2|,|n_3-n_4|,J_m)\leq n_5\leq \mathrm{min}(n_1+n_2,n_3+n_4)$. Here $\mathcal{H}^{n_1,n_2,n_3,n_4}_{n_5}(r)$, $p_{n_5}(\Delta_m,J_m)$ and $\widehat{\alpha}_{n_1,n_2,n_5}(\Delta_i,J_i;\Delta_j,J_j;\Delta_m,J_m)$ are given respectively in~\eqref{eq:HcalDef},~\eqref{eq:normN} and~\eqref{eq:alphaHat}.
\end{proposition}
\begin{proof}
We apply the product expansion of Corollary~\ref{cor:opeKFinite} to the first two entries inside the correlation $\langle\phi_i^{(n_1)}(v_1^{n_1})\phi_j^{(n_2)}(v_2^{n_2})\phi_k^{(n_3)}(v_3^{n_3})\phi_\ell^{(n_4)}(v_4^{n_4})\rangle$, and then we use formula~\eqref{eq:tripleProductReal} to evaluate the resulting three-point correlations. The proposition then follows by evaluating $T^{\vee}_{n_1,n_2,n_5}(v_1^{n_1},v_2^{n_2})\cdot T^{\vee}_{n_3,n_4,n_5}(v_3^{n_3},v_4^{n_4})$ using Lemma~\ref{lem:quadrilinear}.
\end{proof}

\begin{remark}
For given $n_1,\,n_2,\,n_3,\,n_4$ and given $J_i,\, J_j,\, J_k,\, J_{\ell},\, J_{m}$, $\widehat{\Psi}^{ijk\ell}_{m}(n_1,n_2,n_3,n_4;r)$ is a polynomial in $r$ of degree at most $\min(2n_3,n_3+n_4+n_2-n_1)$, whose coefficients are rational functions of $\Delta_i$, $\Delta_j$, $\Delta_k$, $\Delta_\ell$, $\Delta_m$. This follows directly from the definition~\eqref{eq:PsiHat} and the fact that $\widehat{\alpha}_{n_1,n_2,n_3}(\Delta_i,J_i;\Delta_j,J_j;\Delta_k,J_k)$ are rational functions of $\Delta_i$, $\Delta_j$, $\Delta_k$, as explained in Remark~\ref{rmk:alphaHats}. Furthermore, the coefficients of the polynomial are real whenever all five representations $R_{\Delta_i,J_i}$, $R_{\Delta_j,J_j}$, $R_{\Delta_k,J_k}$, $R_{\Delta_\ell,J_\ell}$, $R_{\Delta_m,J_m}$ are unitary. 
\end{remark}

\subsection{Spectral identities}
The main idea we will use is that the spectral data $\Delta_i$, $J_i$ and $c_{ijk}$ are constrained by associativity of multiplication in the smooth function ring $C^{\infty}(\Gamma\backslash G)$. Associativity amounts to the equality
$$
(\phi_i(f_1)\,\phi_j(f_2))\,\phi_k(f_3) = \phi_i(f_1)\,(\phi_j(f_2)\,\phi_k(f_3))
$$
for all $i,j,k>0$ and all $(f_1,f_2,f_3)\in R_{\Delta_i,J_i}^{\infty}\times R_{\Delta_j,J_j}^{\infty}\times R_{\Delta_k,J_k}^{\infty}$. On both sides of the equation, we use Proposition~\ref{prop:ope} twice but in different orders. It is convenient to project the equality onto the image of $\phi_{\ell}$ by multiplying by $\phi_\ell(f_4)$ and computing the expectation value. The following corollary then follows immediately from Theorem~\ref{thm:cpwExpansion}:

\begin{corollary}[Spectral identities]\label{cor:crossing}
The spectral data $\Delta_i,\,J_i$, $c_{ijk}$ of any cocompact lattice $\Gamma\subset \mathrm{PSL}_2(\mathbb{C})$ must satisfy the identities
\be\label{eq:crossing}
\sum\limits_{m=0}^{\infty}c_{ijm}c_{k\ell m} \Psi^{ijk\ell}_{m}(f_1,f_2,f_3,f_4) =
\sum\limits_{m=0}^{\infty}c_{i\ell m}c_{kjm} \Psi^{i\ell k j}_{m}(f_1,f_4,f_3,f_2)
\ee
for all $i,j,k,\ell >0$ and all $(f_1,f_2,f_3,f_4)\in R_{\Delta_i,J_i}^{\infty}\times R_{\Delta_j,J_j}^{\infty}\times R_{\Delta_k,J_k}^{\infty}\times R_{\Delta_\ell,J_\ell}^{\infty}$.
\end{corollary}

In practice, we will use these spectral identities in the case when all of $f_1$, $f_2$, $f_3$, $f_4$ are $K$-finite vectors, ranging over the images of maps $\kappa^{(n)}_{\Delta,J}$. In that case, we can state the identities as follows:
\begin{theorem}[Spectral identities for $K$-finite vectors]\label{thm:crossingKFinite}
The spectral data $\Delta_i,\,J_i$, $\widehat{c}_{ijk}$ of any cocompact lattice $\Gamma\subset \mathrm{PSL}_2(\mathbb{C})$ must satisfy the identities
\ba\label{eq:crossingKFinite}
&\sum\limits_{m=0}^{\infty}\widehat{c}_{ijm}\widehat{c}_{k\ell m}\, \widehat{\Psi}^{ijk\ell}_{m}(n_1,n_2,n_3,n_4;r) \\
&\qquad\qquad=(-1)^{n_1+n_2+n_3+n_4}\sum\limits_{m=0}^{\infty}\widehat{c}_{i\ell m}\widehat{c}_{k j m}\, \widehat{\Psi}^{i\ell k j}_{m}(n_1,n_4,n_3,n_2;1-r)
\ea
for all $i,j,k,\ell >0$, and all $n_1\geq J_i$, $n_2\geq J_j$, $n_3\geq J_k$, $n_4\geq J_\ell$. Here $\widehat{\Psi}^{ijk\ell}_{m}(n_1,n_2,n_3,n_4;r)$ are the polynomials of $r$ given in~\eqref{eq:PsiHat}. The identity~\eqref{eq:crossingKFinite} holds as an equality of polynomials in the variable $r$.
\end{theorem}
\begin{proof}
The theorem follows from the trivial equality
\be\label{eq:crossing0}
\langle\phi_i^{(n_1)}(v_1^{n_1})\phi_j^{(n_2)}(v_2^{n_2})\phi_k^{(n_3)}(v_3^{n_3})\phi_\ell^{(n_4)}(v_4^{n_4})\rangle = 
\langle\phi_i^{(n_1)}(v_1^{n_1})\phi_\ell^{(n_4)}(v_4^{n_4})\phi_k^{(n_3)}(v_3^{n_3})\phi_j^{(n_2)}(v_2^{n_2})\rangle\,.
\ee
Note that under the simultaneous transposition $j\leftrightarrow\ell$, $n_2\leftrightarrow n_4$, $v_2\leftrightarrow v_4$, the prefactor multiplying $\mathcal{G}_{ijk\ell}$ on the right-hand side of~\eqref{eq:4ptCorEven} and~\eqref{eq:4ptCorOdd} stays invariant, up to an overall sign $(-1)^{n_1+n_2+n_3+n_4}$. Furthermore, it is easy to see from~\eqref{eq:rDefinition2} that the ratio $r$ transforms as follows:
$$
r(v_1,v_4,v_3,v_2) = 1- r(v_1,v_2,v_3,v_4)\,.
$$
Applying the product expansion of Proposition~\ref{prop:productExpG} on both sides of~\eqref{eq:crossing0} yields the theorem.
\end{proof}

For a given spectrum $\Sigma=((\Delta_i,\,J_i))_{i\in\mathbb{Z}_{\geq0}}$, the spectral identities~\eqref{eq:crossing} or~\eqref{eq:crossingKFinite} constrain the values of the triple products $\widehat{c}_{ijk}$. Moreover, we will see that it may happen that for a given $\Sigma$ there exists no choice of $\widehat{c}_{ijk}$ consistent simultaneously with the spectral identities and the constraint $\widehat{c}_{ijk}\in\mathbb{R}$. In that case, $\Sigma$ can not arise as the spectrum for any $\Gamma$. In the following subsection, we will see how to implement this idea to derive universal bounds on the low-energy spectrum.

\subsection{Spectral bounds from linear programming}
In this work, we will limit ourselves to exploring the consequences of a tiny subset of the spectral identities of Theorem~\ref{thm:crossingKFinite}: those coming from a single four-point correlation of identical representations, i.e., $i=j=k=\ell$. In that case, we can rewrite~\eqref{eq:crossingKFinite} as
\be\label{eq:crossingIdentical}
\sum\limits_{m=0}^{\infty}(\widehat{c}_{\ell\ell m})^2\,
\left[\widehat{\Psi}^{\ell\ell\ell\ell}_{m}(n_1,n_2,n_3,n_4;r)-
(-1)^{n_1+n_2+n_3+n_4}\widehat{\Psi}^{\ell\ell\ell\ell}_{m}(n_1,n_4,n_3,n_2;1-r)
\right]
= 0\,.
\ee
Note that $(\widehat{c}_{\ell\ell m})^2 \geq 0$ since $\widehat{c}_{\ell\ell m}\in\mathbb{R}$. The identities~\eqref{eq:crossingIdentical} can be turned into bounds on the spectrum by using the simple fact that a sum of non-negative real numbers, at least one of which is positive, cannot vanish. It will be more convenient to switch to parametrizing the spectrum by $t$, where $\Delta_{\ell} = 1 + it_{\ell}$, so that $t$ is real on the principal series and pure imaginary on the complementary series with $t\in (0,1)i$. In order to describe a general spectral identity following from~\eqref{eq:crossingIdentical}, we introduce the following notation:

\begin{definition}
For any $n_1,\,n_2,\,n_3,\,n_4\geq J_{\ell}$ and any $k\geq 0$, let $\omega_{n_1,n_2,n_3,n_4;k}(t_{\ell},J_{\ell};t_{m},J_m)$ be the coefficient of $r^k$ in the contribution of $R_{1+it_m,J_m}$ to the spectral identity~\eqref{eq:crossingIdentical}
$$
\begin{aligned}
\omega_{n_1,n_2,n_3,n_4;k}(t_{\ell},J_{\ell};t_{m},J_m)\coloneqq \left.\frac{1}{k!}\frac{d^{k}}{dr^{k}}\right|_{r=0}
&\left[\widehat{\Psi}^{\ell\ell\ell\ell}_{m}(n_1,n_2,n_3,n_4;r)\right.\\
&\quad\left.-(-1)^{n_1+n_2+n_3+n_4}\widehat{\Psi}^{\ell\ell\ell\ell}_{m}(n_1,n_4,n_3,n_2;1-r)
\right]\,.
\end{aligned}
$$
$\omega_{n_1,n_2,n_3,n_4;k}(t_{\ell},J_{\ell};t_{m},J_m)$ is a rational function of $t_\ell$ and $t_{m}$, which can be computed explicitly following Remark~\ref{rmk:alphaHats}.
\end{definition}

Finally, by taking linear combinations of the identities~\eqref{eq:crossingIdentical}, we arrive at:
\begin{corollary}\label{cor:crossingOmega}
Let us take $\Omega_{y}(t_\ell,J_\ell;t_m,J_m)$ to be a finite linear combination of functions $\omega_{n_1,n_2,n_3,n_4;k}(t_{\ell},J_{\ell};t_{m},J_m)$ with real coefficients $y_{n_1,n_2,n_3,n_4;k}$:
\be\label{eq:omegaDef}
\Omega_y(t_\ell,J_\ell;t_m,J_m) = 
\sum\limits_{n_1,n_2,n_3,n_4;k}
y_{n_1,n_2,n_3,n_4;k}\;
\omega_{n_1,n_2,n_3,n_4;k}(t_{\ell},J_{\ell};t_{m},J_m)\,,
\ee
where $y$ stands for the vector with multi-index components $y_{n_1,n_2,n_3,n_4;k}$. The spectral data $t_i,\,J_i$, $\widehat{c}_{ijk}$ of any cocompact lattice $\Gamma\subset \mathrm{PSL}_2(\mathbb{C})$ must satisfy, for any $\ell>0$,
\be\label{eq:crossingOmega}
\sum\limits_{m=0}^{\infty}(\widehat{c}_{\ell\ell m})^2\,
\Omega_{y}(t_\ell,J_\ell;t_m,J_m) = 0\,.
\ee
\end{corollary}
\begin{remark}
From the point of view of the spectral identities in the form~\eqref{eq:crossing}, picking the vector $y$ means picking $f\in (R_{\Delta_\ell,J_\ell}^{\infty})^{\otimes 4}$.
\end{remark}

Corollary~\ref{cor:crossingOmega} enables us to set up linear programs which imply bounds on the spectrum. To state the result precisely, let us denote by $U \subset \mathbb{C}\times \mathbb{Z}$ the set of ordered pairs $(t,J)$ whose elements parametrize all the inequivalent nontrivial unitary irreducible representations $R_{1+i t,J}$ of $\mathrm{PSL}_2(\mathbb{C})$, i.e.,
$$
U=
\{(t ,0)|t\in i(0,1)\}\cup
\{(t ,0)|t\in \mathbb{R}_{\geq 0}\} \cup
\{(t ,J)|t\in \mathbb{R},J\in \mathbb{Z}_{>0}\}\,.
$$
Let $U_{\mathrm{even}} = \{(t,J)\in U| J\in 2\mathbb{Z}\}$. Finally, let us write $\sigma_{\Gamma} = \{(t_m,J_m)| m\in\mathbb{Z}_{> 0}\}\subset U$ for the $(t,J)$ spectrum of a lattice $\Gamma$. Note that our convention is to not include the trivial representation $(t_0,J_0)=(i,0)$ in $\sigma_{\Gamma}$. Our strategy for deriving spectral bounds is spelled out as follows.

\begin{proposition}[Spectral bounds]\label{prop:spectralBounds}
Let $(t_{\ell},J_{\ell})\in U$, let $U'\subset U_{\mathrm{even}}$ be a subset and suppose there exists a finite collection of coefficients $y_{n_1,n_2,n_3,n_4;k}\in\mathbb{R}$ such that $\Omega_{y}(t_\ell,J_\ell;\,t,J)$, given by~\eqref{eq:omegaDef}, satisfies
\begin{enumerate}
\item $\Omega_{y}(t_\ell,J_\ell;\,i,0) = 1$,\vspace{5pt}
\item $\Omega_{y}(t_\ell,J_\ell;\,t,J) \geq 0$ for all $(t,J)\in  U' $.
\end{enumerate}
Then for any cocompact lattice $\Gamma\subset\mathrm{PSL}_2(\mathbb{C})$ with $(t_\ell,J_\ell)\in\sigma_{\Gamma}$, the set ${\sigma_{\Gamma}\cap (U_{\mathrm{even}}\setminus U')}$ must be nonempty.
\end{proposition}
\begin{proof}
Let $\Gamma$ be such that $(t_\ell,J_\ell)\in\sigma_{\Gamma}$ and let $y_{n_1,n_2,n_3,n_4;k}\in\mathbb{R}$ be such that $\Omega_{y}$ satisfies properties (1) and (2). The spectral data of $\Gamma$ then satisfies the identity~\eqref{eq:crossingOmega}. Recall that $\widehat{c}_{\ell\ell 0} = 1$ so that the $m=0$ term on the left-hand side of~\eqref{eq:crossingOmega} equals 1 thanks to property (1). Recall also from Remark~\ref{rmk:cSymmetry} that $\widehat{c}_{\ell \ell m} = 0$ if $J_m$ is odd. It follows that there must be a term on the left-hand side of~\eqref{eq:crossingOmega} with $m>0$, $J_m$ even and $\widehat{c}_{\ell \ell m}\neq 0$ such that $\Omega_{y}(t_\ell,J_\ell;\,t,J) < 0$. Therefore, ${\sigma_{\Gamma}\cap (U_{\mathrm{even}}\setminus U')}$ is nonempty thanks to property (2).
\end{proof}

Since $\omega_{n_1,n_2,n_3,n_4;k}(t_{\ell},J_{\ell};t_{m},J_m)$ can be computed explicitly, conditions (1) and (2) in the statement of Proposition~\ref{prop:spectralBounds} take the form of a concrete linear program in the space of parameters $y_{n_1,n_2,n_3,n_4;k}$. By making various choices of $U'$ and asking whether the resulting linear program admits a feasible solution, Proposition~\ref{prop:spectralBounds} implies a plethora of nontrivial information about the spectra of lattices in $\mathrm{PSL}_2(\mathbb{C})$.

To implement these linear programs in practice, we must contend with the fact that both the space of parameters $y_{n_1,n_2,n_3,n_4;k}$ and the space of constraints (2) are infinite-dimensional, for generic $U'$. To deal with the first infinity, we simply restrict to the finite-dimensional subspace with $y_{n_1,n_2,n_3,n_4;k}$ nonvanishing only if $J_{\ell}\leq n_1,\,n_2,\,n_3,\,n_4\leq N$ for some $N\in\mathbb{Z}_{>0}$. Since $\widehat{\Psi}^{\ell\ell\ell\ell}_{m}(n_1,n_2,n_3,n_4;r)$ is a polynomial in $r$ of degree at most $\min(2n_3,n_2+n_3+n_4-n_1)$, we can then take $k\in\{0,\ldots,2N\}$ without loss of generality.

Let $\mathcal{I}(N) = \{J_{\ell},\ldots,N\}^4\times\{0,\ldots,2N\}$ and fix $(t_{\ell},J_\ell)$. $\omega_{n_1,n_2,n_3,n_4;k}(t_{\ell},J_{\ell};t,J)$ with $(n_1,n_2,n_3,n_4;k)\in \mathcal{I}(N)$ are typically highly linearly dependent in the space of functions of $(t,J)$ on $U_{\mathrm{even}}$. One reason for this is that $\Psi_m^{\ell\ell\ell\ell}: (R_{1+i t_{\ell},J_{\ell}}^{\infty})^{\otimes 4}\rightarrow\mathbb{C}$ is a $G$-invariant functional, and thus acting with elements of the Lie algebra $B_a\in\mathfrak{g}$ leads to linear dependences among $\omega_{n_1,n_2,n_3,n_4;k}(t_{\ell},J_{\ell};t,J)$. From now on, we will assume that for each $N>0$ we have chosen a set $\mathcal{I}'(N) \subseteq \mathcal{I}(N)$ of tuples $(n_1,n_2,n_3,n_4;k)$ such that $\{\omega_{I}(t_{\ell},J_{\ell};t,J)| I\in \mathcal{I}'(N)\}$ is a basis for the space of functions of $(t,J)$ on $U_{\mathrm{even}}$ generated by $\{\omega_{I}(t_{\ell},J_{\ell};t,J)| I\in \mathcal{I}(N)\}$.

In this work, we will choose $U'$ such that Proposition~\ref{prop:spectralBounds} produces estimates on the low-energy spectrum. In particular, we will consider the following two linear programs:

\begin{linpro}\label{linpro:1}
Fix $(t_\ell,J_\ell)\in U$ and suppose a cocompact lattice $\Gamma\subset \mathrm{PSL}_2(\mathbb{C})$ is such that $t_{\ell} = t_{1}^{(J_\ell)} \in \sigma^{(J_\ell)}_{\Gamma}$, where $\sigma^{(J)}_{\Gamma}$ denotes the set of spin-$J$ eigenvalues $t_i^{(J)}$ for $\Gamma \backslash \hh$. Let $t_{*}>0$ and $J_*\neq J_{\ell}$ with $J_{*}\in 2\mathbb{Z}_{\geq 0}$. Let $N>0$. Suppose there exists $y\in \mathbb{R}^{|\mathcal{I}'(N)|}$ such that
\begin{enumerate}
\item $\sum\limits_{I\in\mathcal{I}'(N)}y_{I}\,\omega_{I}(t_\ell,J_\ell;\,i,0) = 1$,\vspace{5pt}
\item $\sum\limits_{I\in\mathcal{I}'(N)}y_{I}\,\omega_{I}(t_\ell,J_\ell;\,t,J) \geq 0$ for all $(t,J)\in  U'_1(t_{*},J_{*})$,
\end{enumerate}
where
$$
U'_1(t_{*},J_{*}) =
(\{(t,J_*)| t \in \mathbb{R},\, t^2\geq t_{*}^2\} \cup
\{(t,J_\ell)|  t \in \mathbb{C},\,t^2\geq t_{\ell}^2\} \cup
\{(t,J)| t \in \mathbb{C}, J\neq J_{\ell}, J_{*}\})\cap U_{\mathrm{even}}\,.
$$
Then we must have $t\in\sigma^{(J_*)}_{\Gamma}$ for some $t$ with $t^2 < t_{*}^2$.
\end{linpro} 
In other words, Linear Program~\ref{linpro:1} provides a universal upper bound on the spectral gap $|t_1^{(J_*)}|$ for a given $t_1^{(J_\ell)}$.

\begin{linpro}\label{linpro:2}
Fix $(t_\ell,J_\ell)\in U$ and suppose a cocompact lattice $\Gamma\subset \mathrm{PSL}_2(\mathbb{C})$ is such that $t_{\ell} = t_{1}^{(J_\ell)}\in \sigma^{(J_\ell)}_{\Gamma}$. Let $t_{*}>0$ such that $t_{*}^2>t^2_{\ell}$. Let $N>0$. Suppose there exists $y\in \mathbb{R}^{|\mathcal{I}'(N)|}$ such that
\begin{enumerate}
\item $\sum\limits_{I\in\mathcal{I}'(N)}y_{I}\,\omega_{I}(t_\ell,J_\ell;\,i,0) = 1$,\vspace{5pt}
\item $\sum\limits_{I\in\mathcal{I}'(N)}y_{I}\,\omega_{I}(t_\ell,J_\ell;\,t,J) \geq 0$ for all $(t,J)\in  U'_2(t_{*})$,
\end{enumerate}
where
$$
U'_2(t_{*}) =
(\{(t_{\ell},J_{\ell})\}\cup
\{(t,J_\ell)| t \in \mathbb{R},\,t^2\geq t_{*}^2\} \cup
\{(t,J)| t \in \mathbb{C}, J\neq J_{\ell}\})\cap U_{\mathrm{even}}\,.
$$
Then $t_{2}^{(J_\ell)}\in \sigma^{(J_\ell)}_{\Gamma}$ must satisfy $(t_{2}^{(J_\ell)})^2 < t_{*}^2$. 
\end{linpro}
In other words, Linear Program~\ref{linpro:2} provides a universal upper bound on the second smallest Laplacian eigenvalue $(t_{2}^{(J)})^2$ for a given value of $t_1^{(J)}$.

We will present bounds on the low-lying eigenvalues which follow from Linear Programs~\ref{linpro:1}~and~\ref{linpro:2} in Section~\ref{sec:bounds}. To obtain optimal bounds, we proceed as follows. For any $N>0$ and either of the linear programs, let
$$
t_{\mathrm{optimal}}(N) = \inf\{t_{*} | \text{ the linear program is feasible}\}\,.
$$
By increasing $N$ we are enlarging the vector space of spectral identities, so $t_{\mathrm{optimal}}(N)$ is a nonincreasing function of $N$. It therefore has a limit as $N\rightarrow \infty$, which is the optimal upper bound following from the spectral identities~\eqref{eq:crossingIdentical}.

Any finite-dimensional linear program with finitely many constraints admits an algorithmic solution. On the other hand, while Linear Programs~\ref{linpro:1}~and~\ref{linpro:2} do have finitely many variables, they still feature an \emph{infinite} set of constraints (2), parametrized by continuous $t$ and discrete $J$. Fortunately, it turns out that certain special features of our linear programs allow them to be translated into \emph{finite-dimensional} semidefinite programs, for which efficient algorithmic approaches exist.

The first simplification that occurs for Linear Programs~\ref{linpro:1}~and~\ref{linpro:2} is that for any $I\in \mathcal{I}'(N)$, $\omega_{I}(t_\ell,J_\ell;\,t,J)$ vanishes identically for all $J>2N$, since in that case the summation range of $n_5$ in~\eqref{eq:PsiHat} is empty. Therefore, we only need to check the constraints (2) in the linear program for the finitely many values $J\in\{0,2,\ldots,2N\}$. To deal with the continuum of possible values of $t$ in the constraints (2), we recall that for any $I$, $t_\ell$, $J_{\ell}$ and $J$, $\omega_{I}(t_\ell,J_\ell;\,t,J)$ is a rational function of $t$. Since at fixed $N$, $I$ only ranges over the finitely many values $\mathcal{I}'(N)$, we can bring the rational functions to a common denominator and write
$$
\omega_{I}(t_\ell,J_\ell;\,t,J) = \frac{P_{I}(t_\ell,J_\ell,J;t)}{\widetilde{P}(t_\ell,J_\ell,J;t)}\,,
$$
where $P_{I}(t_\ell,J_\ell,J;t)$ and $\widetilde{P}(t_\ell,J_\ell,J;t)$ are polynomials in $t$, with coefficients depending on $t_{\ell}$, $J_{\ell}$ and $J$. Furthermore, it turns out that we can choose $\widetilde{P}(t_\ell,J_\ell,J;t)>0$ for all $(t,J)\in U_{\mathrm{even}}$. It follows that the continuously infinite families of positivity conditions from Linear Programs~\ref{linpro:1}~and~\ref{linpro:2} are equivalent to
$$
\sum\limits_{I\in\mathcal{I}'(N)} y_{I} \,P_{I}(t_\ell,J_\ell,J;t) \geq 0 \quad \text{for all}\quad
t^2 \geq T(J)\quad\text{and all}\quad J\in\{0,\ldots,2N\}
$$
for particular values $T(J)$ depending on the details of the linear programs. This type of linear program with polynomial constraints is known to be exactly equivalent to a finite-dimensional semi-definite program~\cite{Simmons-Duffin:2015qma}. It has an efficient implementation provided by \texttt{SDPB}~\cite{Simmons-Duffin:2015qma,  Landry:2019qug}.

\begin{remark}
Using terminology from physics, we will sometimes refer to the eigentensor corresponding to $(t_\ell,J_\ell)$ as the external state and the other eigentensors appearing in the spectral identities as exchanged states.
\end{remark}

\begin{remark} An alternative approach for deriving the spectral identities for $K$-finite vectors is to directly consider products of derivatives of eigentensors integrated over $\Gamma \backslash \hh$, as in~\cite{Bonifacio:2020xoc, Bonifacio:2021msa, Bonifacio:2021aqf}. The product of any two such tensors can be expanded in terms of curl eigentensors using a spectral decomposition, with coefficients given by the triple products. Expanding quadruple products in two different ways then results in the spectral identities. 
\end{remark}

\section{Computing the spectrum using the trace formula}
\label{sec:STF} 

In this section, we discuss how to use the Selberg trace formula to derive bounds on the spectra of specific orbifolds, following the approach of Booker--Str\"ombergsson and Lin--Lipnowski~\cite{Booker-Strombergsson2007, Lin-Lipnowski2018, Lin-Lipnowski2020, Lin-Lipnowski2021}.

Numerical estimates for the eigenvalues $\lambda^{(0)}_i$ of some hyperbolic 3-manifolds obtained using a different approach can be found in~\cite{Inoue:1998nz, Cornish1999, Inoue2001}.

\subsection{Elements of ${\rm PSL}_2( \mathbb{C})$}
First we must review the classification of elements of ${\rm PSL}_2( \mathbb{C})$. See, e.g.,~\cite{Elstrodt1997}. Elements $\gamma \neq \pm I$ of ${\rm SL}_2( \mathbb{C})$ can be classified as follows:
\begin{itemize}
\item hyperbolic if ${\rm tr}(\gamma) \in \mathbb{C}\setminus [-2,2]$ (sometimes such elements are called loxodromic, with the term hyperbolic reserved for elements with zero holonomy)
\item elliptic if ${\rm tr}(\gamma) \in (-2,2)$
\item parabolic if ${\rm tr}(\gamma)=\pm2$
\end{itemize}
Elements of $ G={\rm PSL}_2( \mathbb{C})$ are similarly classified according to the classification of their preimages in ${\rm SL}_2( \mathbb{C})$.
If $\gamma \in G$ is a hyperbolic element then it fixes two points on $\partial \mathbb{H}^3$. There is a unique geodesic on $\mathbb{H}^3$ connecting these points called the axis of $\gamma$. The conjugacy class $[\gamma]$ of a hyperbolic element $\gamma$ in $\Gamma \subset G$ corresponds to the closed geodesic on $\Gamma \backslash \mathbb{H}^3$ given by projecting the axis.  If $\gamma^{-1} \notin [\gamma]$, then the closed geodesic is topologically a circle, otherwise it is topologically a mirrored interval. The latter possibility can only occur on orbifolds (such a geodesic is labelled as a \textit{mirrored arc} by \texttt{SnapPy}~\cite{SnapPy}). The length of the closed geodesic corresponding to $[\gamma]$ is denoted by $\ell(\gamma)$ and the holonomy by $\phi(\gamma)$. Together these define the complex length $\mathbb{C} \ell (\gamma) \coloneqq \ell(\gamma)+i \phi(\gamma)$, which is related to the trace of the preimages $\tilde{\gamma} \in {\rm SL}_2( \mathbb{C})$ of $\gamma$ by
$$
{\rm Tr}( \tilde{\gamma}) = \pm 2 \cosh \left( \frac{\mathbb{C}\ell(\gamma)}{2}\right).
$$
An elliptic element $\gamma$ also fixes two point on $\partial \mathbb{H}^3$ and corresponds to a pure rotation about its axis. The order of an elliptic element $\gamma$ is the smallest positive integer $n$ such that $\gamma^n =1$. The subgroup $\Gamma$ contains elliptic elements only if $\Gamma \backslash \mathbb{H}^3$ is an orbifold. The projection of the axes of the elliptic elements gives the singular set of the orbifold. There are no parabolic elements in subgroups $\Gamma$ that are cocompact.

Denote the centralizer of $\gamma$ in $G$ by $G_{\gamma}$ and the centralizer of $\gamma$ in $\Gamma$ by $\Gamma_{\gamma}$. An important quantity appearing on the geometric side of the trace formula is the covolume of these centralizers, denoted by ${\rm vol}(\Gamma_{\gamma} \backslash G_{\gamma} ) $. For $\gamma$ a hyperbolic element that does not share an axis with an elliptic element in $\Gamma$, we have
\be \label{eq:covol-hyperbolic}
{\rm vol}\left(\Gamma_{\gamma} \backslash G_{\gamma} \right) = \ell(\gamma_0),
\ee
where $\gamma_0$ is a primitive element such that $\gamma=\gamma_0^n$. 

Hyperbolic elements that share an axis with an elliptic element are called bad hyperbolic elements in~\cite{Lin-Lipnowski2020}. Consider a primitive elliptic element $\gamma_e$ of order $n$ and let $\gamma_h$ be a hyperbolic element of minimal length that shares an axis with $\gamma_e$. Up to conjugation, bad hyperbolic elements sharing an axis with $\gamma_e$ take the form $\gamma_e^p \gamma_h^q$ for $p\in \{0,1, \dots , n-1\}$ and $q \in \mathbb{Z}\setminus \{0\}$. The covolumes of the centralizers for these elements are
\be \label{eq:bad-covolume}
{\rm vol}\left(\Gamma_{\gamma_e^p \gamma_h^q} \backslash G_{\gamma_e^p \gamma_h^q} \right) = \frac{\ell(\gamma_h)}{n}.
\ee 
Equation~\eqref{eq:bad-covolume} also holds for the elliptic case $q=0$, unless $n=2p$ and there does not exist another order-2 element in $\Gamma$ that commutes with $\gamma_e$, in which case there is an extra factor of 2~\cite{Lin-Lipnowski2020},
$$
{\rm vol}\left(\Gamma_{\gamma_e^{\frac{n}{2}}} \backslash G_{\gamma_e^{\frac{n}{2}}} \right) = \frac{2\ell(\gamma_h)}{n}.
$$ 
Whether or not a bad hyperbolic element $\gamma_e^p \gamma_h^q$ is primitive, and hence whether or not it appears in the spectrum computed by \texttt{SnapPy}, depends on whether we can write $\gamma_e^p \gamma_h^q = \left(\gamma_e^{p'} \gamma_h^{q'}\right)^m $ for $m>1$. Unlike for regular hyperbolic elements, summing over the conjugacy classes of bad hyperbolic elements in $\Gamma$ is not equivalent to summing over iterates of the primitive conjugacy classes.

\subsection{Selberg trace formula}
\label{subsec:STF}
We follow the presentation of the trace formulas given in~\cite{Lin-Lipnowski2018, Lin-Lipnowski2020, Lin-Lipnowski2021}. We start with the even trace formulas.
\begin{theorem}[Even trace formula for 3-orbifolds]
\label{thm:trace-even}
Let $\Gamma \backslash \mathbb{H}^3$ be a closed oriented hyperbolic 3-orbifold of volume $V$ and let $H: \mathbb{R} \rightarrow \mathbb{R}$ be an even test function that is smooth and compactly supported. For $\gamma \in \Gamma$, define the function 
$$
\Theta(\gamma)= \begin{cases} 
\frac{1}{2}, \quad \text{$\gamma$ is an elliptic element of order 2}, \\
1, \quad \text{otherwise.}
\end{cases}
$$
The Selberg trace formula for functions gives
\be\label{eq:STF}
\sum_{j=1}^{\infty} \hat{H}\left(t^{(0)}_j\right)+\hat{H}\left(i\right)= -\frac{V}{2\pi}H''(0)+ \sum_{[\gamma] \neq 1}  \Theta(\gamma) {\rm vol}(\Gamma_{\gamma} \backslash G_{\gamma} ) \frac{H(\ell(\gamma))}{2(\cosh \ell(\gamma)-\cos \phi(\gamma))},
\ee
where $H'' = d^2 H/dx^2$, $\hat{H}(t) = \int_{-\infty}^{\infty} H(x) e^{-i x t} \, dx$ is the Fourier transform of $H$, and the sum on the left-hand side runs over the Laplace--Beltrami eigenvalues with $\lambda^{(0)}_j = (t^{(0)}_j)^2 +1$.

The even Selberg trace formula for divergence-free, symmetric, traceless, rank-$J$ eigentensors of the curl operator with $J>0$ gives
\be \label{eq:STF-J}
\sum_{j=1}^{\infty} \hat{H}\left(t^{(J)}_j\right) - \delta_{J, 1} \hat{H}(0)= \frac{V}{\pi}\left(J^2 H(0)-H''(0)\right)+ \sum_{[\gamma] \neq 1}  \Theta(\gamma) {\rm vol}(\Gamma_{\gamma} \backslash G_{\gamma} ) \frac{\cos( J \phi(\gamma))H(\ell(\gamma))}{\cosh \ell(\gamma)-\cos \phi(\gamma)},
\ee
where the sum on the left-hand side runs over the spin-$J$ curl eigenvalues $t_j^{(J)}$ of $\Gamma \backslash \mathbb{H}^3$.
\end{theorem}
\begin{proof} These are special cases of the general trace formula. In particular, we can select the irreducible representations $R_{1+i t, J}$ for fixed $J$  by specializing the trace formula in Corollary B.7 of~\cite{Lin-Lipnowski2018} with the function $F(u, \theta) = H(u) \cos (J \theta)$, with the adaptation for order-2 elliptic elements  described in~\cite{Lin-Lipnowski2020}. The correspondence between these representations and divergence-free, symmetric, traceless, rank-$J$ eigentensors of the curl operator is given in Theorem~\ref{thm:spectralRelation}.  
\end{proof}
The right-hand sides of~\eqref{eq:STF}  and~\eqref{eq:STF-J} are the geometric sides of the trace formulas, involving sums over the nontrivial conjugacy classes $[\gamma]$ of $\Gamma$, while
the left-hand sides are the spectral sides.
The case $J=1$ of the even trace formula, corresponding to coexact 1-forms, was considered for manifolds in~\cite{Lin-Lipnowski2018} and for orbifolds in~\cite{Lin-Lipnowski2020}.  
Note that with our conventions the inverse Fourier transform is
$$
H(x) = \frac{1}{2 \pi}\int_{-\infty}^{\infty} \hat{H}(t) e^{i x t} \, dt.
$$ 

There is also an odd trace formula for $J> 0$,  as considered for  $J=1$ in~\cite{Lin-Lipnowski2021}.
\begin{theorem}[Odd trace formula for 3-orbifolds]
\label{thm:trace-odd}
Let $\Gamma \backslash \mathbb{H}^3$ be a closed oriented hyperbolic 3-orbifold and let $K: \mathbb{R} \rightarrow \mathbb{R}$ be a smooth compactly supported odd test function. Then for $J>0$ we have the odd Selberg trace formula
\be \label{eq:odd-selberg}
\sum_{j=1}^{\infty} \hat{K}\left(t^{(J)}_j\right)=  i \sum_{[\gamma] \neq 1} {\rm vol}(\Gamma_{\gamma} \backslash G_{\gamma} ) \frac{\sin( J \phi(\gamma))K(\ell(\gamma))}{\cosh \ell(\gamma)-\cos \phi(\gamma)}.
\ee
\end{theorem}
\begin{proof}
Take $F(u, \theta) = K(u) \sin (J \theta)$ in Corollary B.7 of~\cite{Lin-Lipnowski2018}.
\end{proof}
Note that elliptic conjugacy classes do not contribute to the sum on the right-hand side of the odd trace formula~\eqref{eq:odd-selberg} since $K(0)=0$. 

\begin{remark}
The trace formulas are valid for more general classes of test functions. For example, Theorem 2.2 of~\cite{Lin-Lipnowski2018} shows that compactly supported even functions $H: \mathbb{R} \rightarrow \mathbb{R}$ such that
$$
\int \left( \left|\hat{H}(t) \right|^2 + \left|\hat{H}'(t) \right|^2\right)(\sqrt{1+t^2})^{2\delta} < \infty
$$
for some $\delta > 5/2$ are admissible to use in the even trace formulas.
 The even trace formulas were also shown to hold for Gaussian test functions in~\cite{Lin-Lipnowski2021}.
\end{remark}

\subsection{Ruling out eigenvalues}
\label{subsec:selberg-intervals}
Given an orbifold $M= \Gamma \backslash \mathbb{H}^3$ and its length spectrum computed up to some cutoff $\ell_{\rm max}$, we can use the Selberg trace formula to find constraints on the possible eigenvalues of $M$, following the approach of Booker--Str\"ombergsson~\cite{Booker-Strombergsson2007} and Lin--Lipnowski~\cite{Lin-Lipnowski2018, Lin-Lipnowski2020, Lin-Lipnowski2021}. Given a candidate eigenvalue $t^*$, the idea is to try to prove that it cannot be in the spectrum of $M$ by finding a function $H$ for which the Selberg trace formula would otherwise give a contradiction. In more detail, we consider even functions $H$ with support on an interval $[-\ell_{\rm max}, \ell_{\rm max}]$ such that $\hat{H} \geq 0$ and $\hat{H}\left(t^* \right) = 1$. Using the explicit length spectrum, we can evaluate the geometric side of the trace formula for a particular $H$ and thus obtain the weighted sum over eigenvalues on the spectral side,
$$
I^{(J)}_{\hat{H}} \coloneqq \sum_{j=1}^\infty \hat{H}\left(t_j^{(J)}\right).
$$
If $I^{(J)}_{\hat{H}}<1$, then we know that $t^* $ cannot appear in the spin-$J$ spectrum since by assumption $\hat{H}\left(t^*\right) = 1$. By optimizing the choice of test function $H$, we can hope to constrain the location of possible low-lying eigenvalues to small intervals.

A nice set of test functions to optimize over are of the type used by~\cite{Booker-Strombergsson2007, Lin-Lipnowski2018, Lin-Lipnowski2020, Lin-Lipnowski2021}, given in terms of convolutions of indicator functions. Fix positive integers $m$ and $n$, a length cutoff $\ell_{\rm max}$, and define $\delta \coloneqq  \ell_{\rm max}/(2n+2m)$. The number $m$ determines the convolutional power of the indicator function that we consider and $n$ determines the size of the space of functions to optimize over. Define $h$ as the $m$\textsuperscript{th} convolutional power of the normalized indicator function on the interval $[-\delta, \delta]$,
\be \label{eq:h-basic}
h \coloneqq \left( \frac{1}{2 \delta}\mathds{1}_{[-\delta, \delta]} \right)^{* m}.
\ee
For $k \in \{0 , 1, \dots, n\}$ define 
$$
h_k (x) \coloneqq \frac{1}{2} \left[h(x+k \delta ) + h(x- k \delta) \right],
$$
which has support on the interval $[-(k+m)\delta, (k+m)\delta]$. The Fourier transform of $h_k$ is
$$
\hat{h}_k(t) = \cos (k \delta t) \left( \frac{\sin(\delta t)}{\delta t} \right)^m.
$$
For $\vec{x} \in \mathbb{R}^{n+1}$ with components $x_k$, we define the linear combination
$$
h_{\vec{x}} \coloneqq \sum_{k=0}^n x_k h_k.
$$
Then we take the test function to be the convolution of this linear combination with itself,
$$
H_{\vec{x} } \coloneqq h_{\vec{x} }  * h_{\vec{x} },
$$
which has the non-negative Fourier transform
$$
\hat{H}_{\vec{x} } = \hat{h}_{\vec{x} }^2.
$$
For a candidate eigenvalue $t^*$, consider the set of functions of this form that are normalized at $t^*$,
\be \label{eq:function-set-S}
S(t^*) \coloneqq \left\{ H_{\vec{x} }: \,  \vec{x} \in \mathbb{R}^{n+1}, \quad \sum_{k=0}^n x_k \hat{h}_k(t^*) =1 \right\} \,.
\ee
To try to rule out $t^*$ from the spin-$J$ spectrum, we minimize $I^{(J)}_{\hat{H}}$ over these functions,
\be \label{eq:inf}
\mathcal{J}^{(J)}(t^*)  \coloneqq \inf_{H \in S(t^*) } I^{(J)}_{\hat{H}} .
\ee
If $\mathcal{J}^{(J)}(t^*) <1$, then $t^*$ cannot be in the spectrum.

Finding the infimum is equivalent to a quadratic optimization problem with a linear constraint and can be solved explicitly for all $t^*$. Define the $(n+1)\times (n+1)$ matrix $A^{(J)}$ with components
$$
A^{(J)}_{ kl} \coloneqq \sum_{j=1}^{\infty} \widehat{h_k * h_l}\left(t^{(J)}_j\right).
$$
This matrix can be computed in terms of the length spectrum and other geometric quantities via the even trace formula. This assumes that the even trace formula is valid for this class of functions, which is the case for $m \geq 2$~\cite{Lin-Lipnowski2018, Lin-Lipnowski2021}. Then the method of Lagrange multipliers gives~\cite{Lin-Lipnowski2018}
$$
\mathcal{J}^{(J)}(t^*)  = \left[ \sum_{k, l =0}^n \left(A^{(J)}\right)^{-1}_{kl} \hat{h}_k(t^*) \hat{h}_l(t^*)  \right]^{-1}.
$$
As $n$ and $\ell_{\rm max}$ are increased, $\mathcal{J}^{(J)}(t)$ better approximates the indicator function $m_{\Gamma}(R_{1+it, J} )+m_{\Gamma}(R_{1-it, J}) $, where $m_{\Gamma}( R_{\Delta, J})$ is the dimension of the eigenspace corresponding to the representation $R_{\Delta, J}$.

Recall that $\pm t$ cannot be an eigenvalue if $\mathcal{J}^{(J)}(t)<1$, so using the above procedure we can definitively exclude candidate eigenvalues.  For $J>0$ we therefore obtain a union of intervals containing the absolute values of the curl eigenvalues of $\Gamma \backslash \hh$,
\be \label{eq:interval-union}
\left\{ \big| t_i^{(J)} \big| \right\}_{i=1}^{\infty} \subset [u_1, v_1] \cup [u_2, v_2] \cup  \dots  \cup [u_n, v_n] \cup [u_{n+1}, \infty),
\ee
where $0 \leqslant u_1 < v_1 < u_2 <v_2 \dots$; these are the intervals where $\mathcal{J}^{(J)}(t) \geq 1$. If $|v_k-u_k|$ is small and $\mathcal{J}^{(J)}(t)$ peaks just above a positive integer $N$ in the interval $[u_k, v_k]$, then this suggests that there are $N$ eigenvalues $\pm t^*$ with $|t^*| \subseteq [u_k, v_k]$, but this approach cannot distinguish between this situation and the case of distinct nearby eigenvalues in $[u_k, v_k]$. Similarly, for $J=0$ we obtain a union of imaginary and real intervals that contain the eigenvalues $t_i^{(0)}$.

As an instructive example, we show in Figure~\ref{fig:weeks-bounds}  plots of $\mathcal{J}^{(J)}(t)$ with $J=0,1,2$ and $t \in [0, 7]$ for the Weeks manifold, taking $\ell_{\rm max}=7$, $n=30$, and $m=3$. From the $J=0$ curve we can resolve three clear peaks whose maxima exceed an integer by a small amount: these suggest the locations of the first three distinct nonzero eigenvalues, where the second has multiplicity two. There is also a curve for $t \in i(0,1)$ that shows that there are no small eigenvalues. These values and multiplicities are in agreement with those obtained in~\cite{Cornish-Spergel}. In contrast, the first peak of the $J=1$ curve has height $2.75$, which suggests that there are nearby eigenvalues that have not yet been resolved (this can be confirmed by looking at quotient orbifolds~\cite{Mednykh1998}). Lastly, the $J=2$ curve has two narrow peaks, which suggest the locations of the first two distinct eigenvalues, each with multiplicity two. We similarly show in Figure~\ref{fig:T2-Z2-bounds} plots of $\mathcal{J}^{(J)}(t)$ with $J=0,1,2$ for the smallest closed orientable hyperbolic 3-orbifold~\cite{Gehring-Martin-I, Marshall-Martin-II}, which is described in Appendix~\ref{app:tetrahedra}, taking $\ell_{\rm max}=5$, $n=30$, and $m=3$. This is a good illustration of an example for which we obtain many narrow peaks, allowing for a precise determination of the spectrum.

\begin{figure}[h]
\begin{center}
\epsfig{file=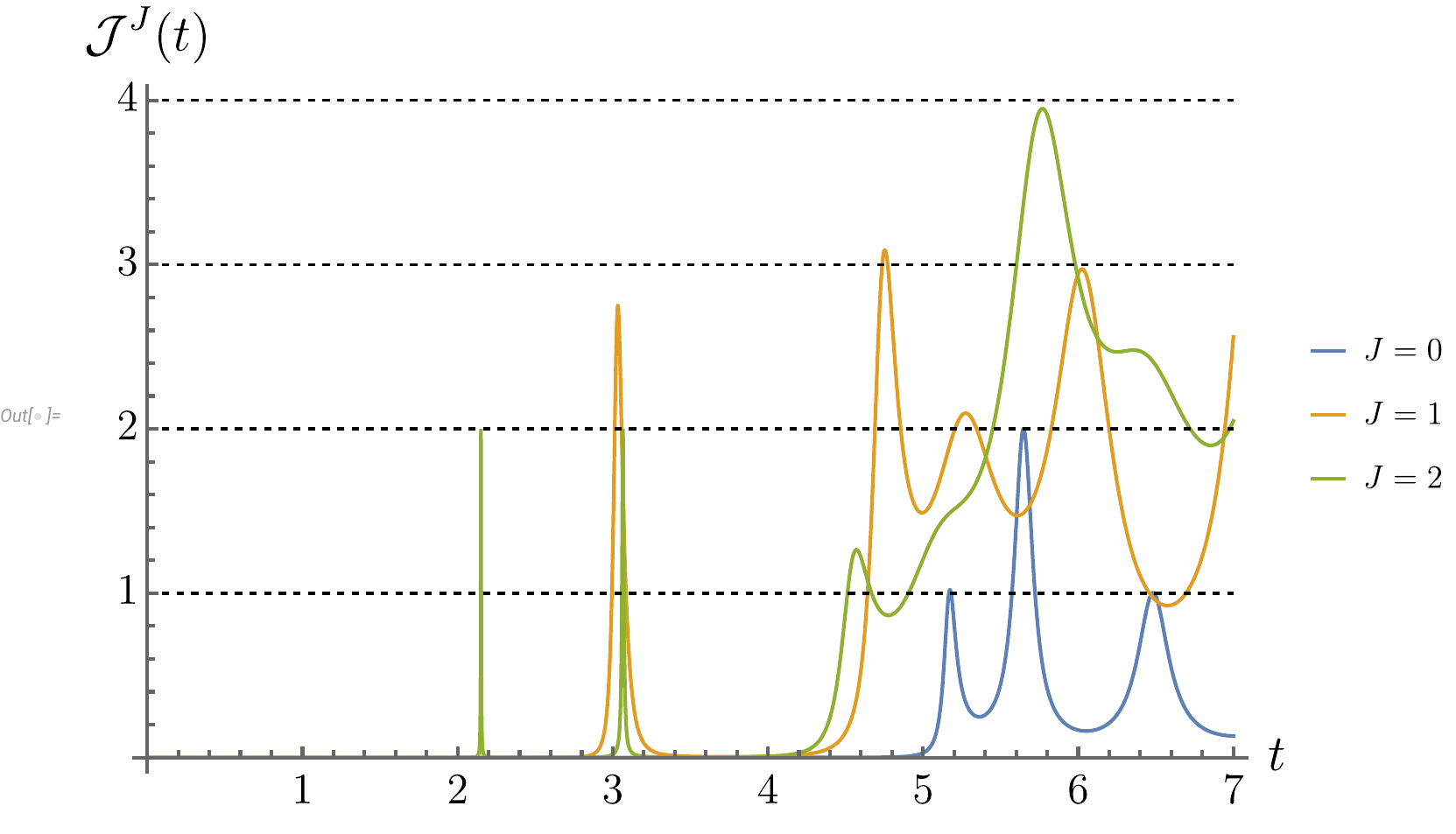,scale=.4}
\end{center}
\caption{Plots of $\mathcal{J}^{(J)}(t)$ with $J=0,1,2$ for the smallest closed orientable hyperbolic 3-manifold, the Weeks manifold m003$(-3,1)$, taking $\ell_{\rm max}=7$, $n=30$, and $m=3$.}
\label{fig:weeks-bounds}
\end{figure}
\begin{figure}[h]
\begin{center}
\epsfig{file=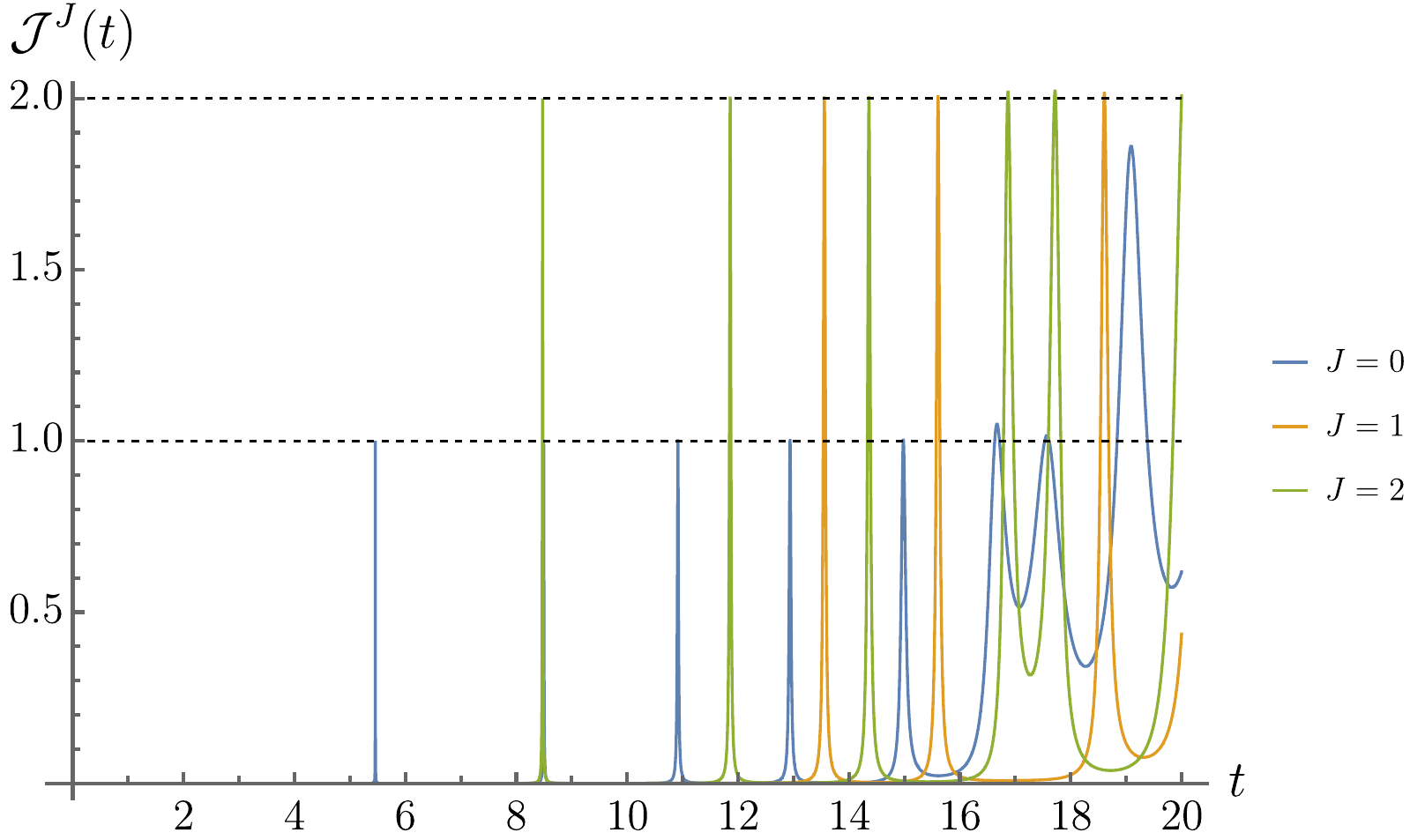,scale=.4}
\end{center}
\caption{Plots of $\mathcal{J}^{(J)}(t)$ with $J=0,1,2$ for the smallest closed orientable hyperbolic 3-orbifold, taking $\ell_{\rm max}=5$, $n=30$, and $m=3$.}
\label{fig:T2-Z2-bounds}
\end{figure}

\subsection{Ruling in eigenvalues}
\label{subsec:ruling-in}
The procedure from Section~\ref{subsec:selberg-intervals} gives a union of intervals containing the absolute values of the low-lying eigenvalues for a particular orbifold. In principle, the finite intervals may or may not contain any eigenvalues. However, as suggested in~\cite{Lin-Lipnowski2018}, we can attempt to use the Selberg trace formula to also verify that a given interval does actually contain an eigenvalue.

Suppose that we want to verify that the interval $[0, v]$ contains the absolute value of a spin-$J$ curl eigenvalue for an orbifold $M = \Gamma \backslash \mathbb{H}^3$, where $J>0$.
Suppose also that we have computed the length spectrum of $M$ up to $\ell_{\rm max}$. For a fixed integer $m\geq 3$ and a fixed $\delta$ satisfying $0<\delta \leqslant \ell_{\rm max}/2m$, we consider the even test function $H_{v}$ with Fourier transform
\be \label{eq:H-nu}
\hat{H}_{v}(t)=\left(1-\frac{t^2}{v^2}\right)\left(\frac{\sin \delta t }{\delta t}\right)^{2m}.
\ee 
The function $\hat{H}_{v}$ is positive for $t \in[0,v)$ and negative for $t>v$. The condition $\delta \leqslant \ell_{\rm max}/2m$ ensures that $H_{v}$ vanishes outside $[-\ell_{\rm max}, \ell_{\rm max}]$, while the condition $m\geq 3$ ensures the validity of the trace formula~\cite{Lin-Lipnowski2018}. From the even trace formula we can explicitly determine the spectral sum $I_{H_{v}}^{(J)}$.
If $I_{H_{v}}^{(J)}>0$, then we can conclude that there must be an eigenvalue $t^{(J)}_j \in \sigma^{(J)}_{\Gamma}$ with $|t^{(J)}_j | \in [0, v)$. For example, by taking $v<u_2$ from~\eqref{eq:interval-union}, we can verify that the smallest non-zero spin-$J$ eigenvalue has absolute value in the interval $[u_1,v_1]$. For $J=0$, the same approach allows us to show that there is an eigenvalue $\lambda^{(0)}_j$ in the interval $ (0, v^2+1)$.

To show that there is an eigenvalue with absolute value in the interval $[u, v]$ with $u>0$, we instead consider even test functions $H_{u, v}$ with Fourier transform
$$
\hat{H}_{u,v}(t)=-\left(1-\frac{t^2}{u^2}\right)\left(1-\frac{t^2}{v^2}\right)\left(\frac{\sin \delta t }{\delta t}\right)^{2m},
$$ 
where $m\geq 4$ and $0<\delta \leqslant \ell_{\rm max}/2m$. This function is non-negative only for $t \in[u,v]$.

\subsection{Determining chiralities}
Let $J>0$ and assume that we have $N(t^*)$ spin-$J$ eigenvalues $\pm t^*$ with $t^* \neq 0$ for the orbifold $\Gamma \backslash \mathbb{H}^3$. Let $m_{\Gamma}( R_{1\pm it^*,J})$ be the multiplicity of $\pm t^*$, so that we have
$$
N(t^*) = m_{\Gamma}( R_{1+it^*, J}) +m_{\Gamma}( R_{1-it^*, J}).
$$
The even trace formula is only sensitive to the total multiplicity $N(t^*)$, not the separate multiplicities of each chirality. However, using the odd trace formula we can find the difference of these multiplicities, which then lets us deduce $m_{\Gamma}( R_{1\pm it^*,J})$.

 Given $t^* \in \sigma^{(J)}_{\Gamma}$ and given fixed $m\geq3$, $n$, and $\ell_{\rm max}$, let $H_{\vec{x}(t^*)} \in S(t^*)$ be the function from the set~\eqref{eq:function-set-S} that minimizes the spectral sum $I_{\hat{H}}^{(J)}$. 
 Explicitly, we have
$$
x_i(t^*) = \mathcal{J}^{(J)}(t^*)  \sum_{j=0}^n\left(A^{(J)} \right)^{-1}_{ ij}\hat{h}_j(t^*).
$$ 
The Fourier transform $\hat{H}_{\vec{x}(t^*)}$ equals one at $t=\pm t^*$ and, for large enough $n$ and $\ell_{\rm max}$, takes small values at other $t$ in $\pm \sigma^{(J)}_{\Gamma}$. We write the spectral sum as
\be \label{eq:epsilon-def}
I_{\hat{H}_{\vec{x}(t^*)}}^{(J)} = N(t^*) + \epsilon^2(t^*), \qquad \epsilon^2(t^*) \coloneqq \sum_{\substack{j=1, \\ \big|t_j^{(J)} \big| \neq |t^*|}}^\infty \hat{H}_{\vec{x}(t^*)}\left(t_j^{(J)}\right).
\ee
Now define $K =  H'_{\vec{x}(t^*)}$, so that $\hat{K}(t) = i t \hat{H}_{\vec{x}(t^*)}(t)$. This function can be used in the odd trace formula~\eqref{eq:odd-selberg} since $m \geq 3$~\cite{Lin-Lipnowski2021}. The left-hand side of~\eqref{eq:odd-selberg} is then
\be \label{eq:lhs-odd}
i \sum_{j=1}^{\infty} t_j^{(J)} \hat{H}_{\vec{x}(t^*)}\left( t_j^{(J)} \right) = i t^* \left(m_{\Gamma}( R_{1+it^*, J}) -m_{\Gamma}( R_{1-it^*, J}) \right) + i \tilde{\epsilon}(t^*),
\ee
where 
$$
\tilde{\epsilon}(t^*) \coloneqq \sum_{\substack{j=1, \\ \big|t_j^{(J)} \big| \neq |t^*|}}^{\infty} t_j^{(J)} \hat{H}_{\vec{x}(t^*)}\left( t_j^{(J)} \right) .
$$
We can explicitly find~\eqref{eq:lhs-odd} from the length spectrum using the odd trace formula~\eqref{eq:odd-selberg}. If $| \tilde{\epsilon}(t^*)|$ is sufficiently small, we can then use this to determine the integer multiplicities $m_{\Gamma}( R_{1\pm it^*, J})$.

To uniquely determine the multiplicity of each chirality, we need a sufficiently strong upper bound on $|\tilde{\epsilon}(t^*)|$. Suppose that we have an upper bound on the number of spin-$J$ eigenvalues of $\Gamma \backslash \hh$ whose absolute values lie in a unit interval centered at $\nu$,
$$
 \Bigg|  \left\{ \big|t_i^{(J)} \big| \right\}_{i=1}^{\infty} \cap \left[\nu - \frac{1}{2}, \nu+\frac{1}{2} \right]  \Bigg|   \leqslant   \alpha_J \nu^2 + \beta_J, \quad \nu \geq \nu_{\rm min}\geq0,
$$
where $\alpha_J$ and $\beta_J$ are constants. Such a bound is called a local Weyl law and can be obtained for a particular orbifold by following the approach from Section 5 of~\cite{Lin-Lipnowski2021}. 
 Then we can obtain the following bound on $|\tilde{\epsilon}(t^*)|$:
$$
\begin{aligned}
\left| \tilde{\epsilon}(t^*) \right| & \leqslant \sum_{\substack{j=1, \\ \big|t_j^{(J)} \big| \neq |t^*|}}^{\infty} \big| t_j^{(J)} \big| \hat{H}_{\vec{x}(t^*)}\left( t_j^{(J)} \right) \leqslant \sum_{\nu=0}^{\infty} \sum_{\substack{\big| t_j^{(J)}  \big| \in [\nu-\frac{1}{2},\nu+\frac{1}{2}] , \\ \big|t_j^{(J)} \big| \neq |t^*|}} \big| t_j^{(J)} \big| \hat{H}_{\vec{x}(t^*)}\left( t_j^{(J)} \right) \nn \\
& \leqslant \sum_{\nu=0}^{k} \sum_{\substack{\big| t_j^{(J)}  \big| \in [\nu-\frac{1}{2},\nu+\frac{1}{2}] , \\ \big|t_j^{(J)} \big| \neq |t^*|}} \big| t_j^{(J)} \big| \hat{H}_{\vec{x}(t^*)}\left( t_j^{(J)} \right) +\sum_{\nu=k+1}^{\infty} \sum_{\big| t_j^{(J)}  \big| \in [\nu-\frac{1}{2},\nu+\frac{1}{2}]}\big| t_j^{(J)} \big| \hat{H}_{\vec{x}(t^*)}\left( t_j^{(J)} \right) \nn \\
& \leqslant \epsilon^2(t^*)  \left(k+ \frac{1}{2}\right)  +|\vec{x}(t^*)|_1^2 \sum_{\nu=k+1}^{\infty} \frac{ \left(\nu+ \frac{1}{2}\right) (\alpha_J \nu^2 + \beta_J)  }{(\delta  \left(\nu-\frac{1}{2}\right))^{2m} } \eqqcolon F(k),
\end{aligned}
$$
where $\delta =  \ell_{\rm max}/(2n+2m)$ and $ k \geq \nu_{\rm min}-1$ is a non-negative integer. The last step used
$$
\hat{H}_{\vec{x}(t^*)}\left( t \right)  =  \left[\sum_{l=0}^n x_l(t^*) \cos (l \delta t) \right]^2 \left( \frac{\sin(\delta t)}{\delta t} \right)^{2m}\leqslant \frac{|\vec{x}(t^*)|_1^2}{(\delta t)^{2m}}
$$
together with the local Weyl law to bound each term in the infinite sum, while the definition of $\epsilon(t^*)$ from~\eqref{eq:epsilon-def} was used to bound the finite sum.
To optimize the bound we choose $k$ to minimize $F(k)$. If the resulting interval $[(-i I^{(J)}_{\hat{K}}- F(k))/t^*,(-i I^{(J)}_{\hat{K}}+F(k))/t^*]$ contains exactly one integer, then we can uniquely determine the multiplicities $m_{\Gamma}( R_{1\pm it^*,J})$.  

\subsection{Examples}
We applied the above procedures to bound the eigenvalues of various examples, all with volume less than $35$. We considered examples from the following (partially overlapping) classes of orbifolds: 
\begin{itemize}
\item \texttt{SnapPy}'s census of closed orientable manifolds~\cite{SnapPy}.
\item Manifolds and link orbifolds (orbifolds whose singular locus is a link) obtained by Dehn filling manifolds in \texttt{SnapPy}'s census of orientable cusped manifolds, using small Dehn filling coefficients.
\item Tetrahedral orbifolds and their quotients, as described in Appendix~\ref{app:tetrahedra}. 
\item Closed orientable manifolds from \texttt{SnapPy}'s censuses of Platonic manifolds~\cite{platonic}. 
\item The small orbifolds listed in Heard's thesis~\cite{Heard-thesis}. For many of these examples we used the program \texttt{Orb} to find matrix generators of $\Gamma$~\cite{orb}.\footnote{The version of \texttt{Orb} available at~\cite{orb} does not easily run on recent versions of macOS. We ported the source code from \texttt{Qt3} to \texttt{Qt4} to build a version that runs on macOS 13.0.1.} 
\end{itemize}
The exact list of examples can be found in the ancillary file ``3D-spectral-data.txt," together with the bounds we obtained on their spectra. We used \texttt{SnapPy} to compute the length spectra for each example, using fixed precision floating-point numbers. For the orbifolds that are not link orbifolds, we computed their length spectra by first defining their Dirichlet domains in \texttt{SnapPy} using the matrix generators included in the ancillary files.

We show in Table~\ref{tab:large-eigenvalues} the examples we found from our list with the largest spin-0 spectral gaps and the largest first spin-$J$ curl eigenvalues for $J=1,\dots, 4$. 
\begin{table}
\centering
  \resizebox{\textwidth}{!}{
  \begin{tabular}{  | c |c | c| c|c| c|}
  \thickhline
  & 1 & 2 & 3 & 4 & 5   \\ \thickhline
 $\big|t_1^{(0)}\big| $  &\makecell{orb$(0.0527)$ \\ $7.81369(3)^1$} & \makecell{orb$(0.0660)$ \\ $7.40279(3)^1$}  & \makecell{orb$(0.1178)$ \\  $6.75173(6)^1$} & \makecell{orb$(0.0409)$ \\  $6.734106(7)^1$}  & \makecell{m$345(2,0)$ \\ $6.02863(9)^1$} \\ \hline
  $\big|t^{(1)}_1\big|$  &\makecell{orb$(0.0391)$ \\ $13.5574(8)^{\pm 1}$ } & \makecell{orb$(0.0718)$ \\ $11.3739(4)^{\pm 1}$ } & \makecell{orb$(0.0858)$ \\ $10.8147(5)^{\pm 1}$ }  & \makecell{orb$(0.0781)$ \\ $10.2517(3)^{\pm 1}$ } & \makecell{orb$(0.0662)$ \\ $10.1045(7)^{-1}$ }   \\ \hline
  $\big|t^{(2)}_1\big|$  &\makecell{orb$(0.0391)$ \\ $8.47501(4)^{\pm 1}$ } & \makecell{orb$(0.0718)$ \\ $7.19644(3)^{\pm 1}$ } & \makecell{orb$(0.0858)$ \\ $6.12140(2)^{\pm 1}$ }  & \makecell{orb$(0.0662)$ \\ $5.5743(2)^{- 1}$ } & \makecell{orb$(0.0933)$ \\ $5.044753(5)^{\pm 1}$ }   \\ \hline  
    $\big|t^{(3)}_1\big|$  &\makecell{orb$(0.0391)$ \\ $11.7600(3)^{\pm 1}$ } & \makecell{orb$(0.0409)$ \\ $10.123(3)^{+1}$ } & \makecell{orb$(0.0660)$ \\ $9.616(5)^{-1}$ }  & \makecell{orb$(0.1269)$ \\ $9.1749(7)^{\pm 1}$ } & \makecell{orb$(0.0718)$ \\ $9.11114(8)^{\pm 1}$ }   \\ \hline  
        $\big|t^{(4)}_1\big|$  &\makecell{orb$(0.0391)$ \\ $3.8516179(4)^{\pm 1}$ } & \makecell{orb$(0.0409)$ \\ $3.4557791(5)^{+1}$ } & \makecell{orb$(0.1028)$ \\ $2.9217(2)^{-1}$ }  & \makecell{orb$(0.0662)$ \\ $2.32(6)$ }  & \makecell{orb$(0.0660)$ \\ $2.02979(8)^{-1}$ }  \\ \hline  
 \thickhline
\end{tabular}}
\caption{Example orbifolds with large eigenvalues $\big| t_1^{(J)} \big|$ for $J=0, \ldots, 4$. orb$(a)$ denotes the orbifold with volume $V \approx a$ from Table 4.1 or Table 4.3 of Heard's thesis~\cite{Heard-thesis}. The parentheses denote the uncertainty in the last digit of the eigenvalue such that the corresponding interval contains the eigenvalue, e.g., for orb$(0.0527)$ we find $\big|t_1^{(0)}\big|  \in [7.81366, 7.81372]$. The superscripts denote the eigenvalue multiplicities if we were able to determine this, with $+$ or $-$ denoting the chirality. For cases in which an orbifold and one of its covers have the same smallest eigenvalue, we have only listed the smaller orbifold. }
\label{tab:large-eigenvalues}
\end{table}
For each eigenvalue in the table, we used the approach described in Section~\ref{subsec:ruling-in} to verify that there is an eigenvalue in the quoted interval.\footnote{For the eigenvalues with multiplicity two, we quote the interval in which $\mathcal{J}^{(J)}(t) \geq 2$ rather than $\mathcal{J}^{(J)}(t) \geq 1$. We are guaranteed to have an eigenvalue in this smaller interval as the corresponding orbifolds are all achiral, which means that their spin-$J$ eigenvalues with $J>0$ come in pairs with opposite chiralities.} orb$(0.0391)$ is the smallest closed orientable hyperbolic 3-orbifold~\cite{Gehring-Martin-I, Marshall-Martin-II} and it has the largest first spin-$J$ eigenvalues amongst our examples for $J=1, \dots, 4$; its double cover $\Gamma^+(T_2) \backslash \mathbb{H}^3$ has the same $t_1^{(2)}$ and $t_1^{(4)}$. We show in Table~\ref{tab:T2-Z2-eigenvalues} some eigenvalues for orb$(0.0391)$ computed at higher precision and for larger values of $J$. Let us add a few remarks about some of the other orbifolds appearing in Table~\ref{tab:large-eigenvalues}: orb$(0.0409)$ is the second-smallest closed orientable hyperbolic 3-orbifold~\cite{Marshall-Martin-II},  orb$(0.0527)$ is the third-smallest orbifold we know (an arithmetic description for an orbifold of this volume is given in~\cite{smallest-arithmetic-manifold}), orb$(0.0718)$ is the tetrahedral orbifold $\Gamma^+(T_3) \backslash \mathbb{H}^3$, orb$(0.0781)$ is the tetrahedral orbifold $\Gamma^+(T_2) \backslash \mathbb{H}^3$, orb$(0.0858)$ is the quotient of the tetrahedral orbifold $\Gamma^+(T_5) \backslash \mathbb{H}^3$ described in Appendix~\ref{app:tetrahedra}, orb$(0.0933)$ is the quotient of the tetrahedral orbifold  $\Gamma^+(T_4) \backslash \mathbb{H}^3$ described in Appendix~\ref{app:tetrahedra}, orb$(0.1178)$ is the smallest known orbifold with all elliptic elements of order two~\cite{Heard-thesis}, and m$345(2,0)$ is a link orbifold with volume $V \approx 0.3159$. See~\cite{Atkinson2012, Atkinson2017} for more discussion of small volume link orbifolds.

\begin{table}
\begin{center}
  \resizebox{\textwidth}{!}{
  \begin{tabular}{  | c |c | c| c|c| c| c | }
  \thickhline
 $J=0 $  &5.4574691592(5) & 8.494100975(7) & 10.91535883(4) & 12.9382783(2)& 14.981647(6)\\   
 $J=1$  &13.5572082(2) & 15.6075150(3) & 18.607848(4)& 20.2825(2) & 21.474(2)   \\   
 $J=2 $  & 8.475000851(6) & 11.85814782(5) & 14.3611760(2)&16.863777(9) & 17.722926(8)   \\   
 $J=3$ & 11.75997183(5)& 16.0027595(7) &17.60783(2) & 19.0854(5) & 20.0365(4) \\
 $J=4$ & 3.85161774528(6) & 9.49996163(2)& 13.3528534(2) &15.257644(4)&16.67488(8) \\
 $J=5$ & 8.854663434(8) &11.96430372(6) & 13.9851772(2)&17.262576(7) &  \\
 $J=6$ & 0 & 4.6555236272(3) & 7.387048105(4)& 10.44231496(4) &12.1675583(2)  \\
 $J=7$ & 13.18143467(9) & 15.9125565(4) & 18.19365(2)  & & \\
 $J=8$ & 0 & 6.797730017(2) & 9.37023688(2) & 12.5363956(2) & 14.1291619(4)   \\
 $J=9$ & 4.3942968302(2)& 8.610423081(6) & 11.65042016(5) &13.2029898(2) & 15.1960708(9)  \\
 $J=10$ & 0 & 2.19569290764(5) &5.1920621318(5) & 7.616579020(5) & 9.85973901(4)  \\
  \thickhline
\end{tabular}}
\caption{Low-lying curl eigenvalues $t_i^{(J)}$ for the smallest closed orientable hyperbolic 3-orbifold for $J\leq 10$.  The parentheses denote the uncertainty in the last digit of the eigenvalue, as in Table~\ref{tab:large-eigenvalues}.  For the eigenvalues shown, those with $J=0$ have multiplicity one, the positive ones with $J>0$ each come in a positive and negative chirality pair, and the zero eigenvalues have multiplicity one.  These were computed using the procedure from Section~\ref{subsec:selberg-intervals} with $m=3$ and $n=300$, using the length spectrum at 300 digits of precision with cutoff $\ell_{\rm max} =5$.
}
\label{tab:T2-Z2-eigenvalues}
\end{center}
\end{table}

Amongst manifolds, the largest spin-0 spectral gap we find is for the Meyerhoff manifold m$003(-2,3)$ with $\big|t_1^{(0)}\big| \in(5.3114,5.3689)$, which is consistent with~\cite{Cornish-Spergel}. To get a more precise estimate we can look at the link orbifold m$115(2,0)$, which has the Meyerhoff manifold as a double cover and the same spectral gap. The spectral gap of m$115(2,0)$ is $\big|t_1^{(0)}\big|  \in  (5.3116, 5.3125)$.

\section{Bootstrap bounds}\label{sec:bounds}
	
In this section, we present our bootstrap bounds for the Laplacian eigenvalues of closed hyperbolic 3-orbifolds. 
To compare our bounds to explicit examples, we also display lower bounds on the eigenvalues of some  orbifolds, as determined from the Selberg trace formula using the approach described in Section~\ref{sec:STF}. If we have an upper bound and a lower bound for the eigenvalue on the horizontal axis, then we display the example as a red dot with horizontal error bars, whereas we use a blue dot if we only have a lower bound. 

\subsection{External spin-0 bounds}
We start with bootstrap bounds from the spectral identities with identical external $J=0$ eigenfunctions. We show an upper bound on the gap $\big| t_2^{(0)} \big|-\big| t_1^{(0)} \big|$ for different values of $\big| t_1^{(0)} \big|$ in Figure~\ref{fig:bound_0_0}, together with examples. The bound applies more generally to the difference of the first two distinct non-vanishing eigenvalues, so we show this quantity for the examples with degenerate first eigenvalues. For this plot we used 18 spectral identities with $N=4$, i.e., including exchanged states up to $J=8$, and we apply Linear Program~\ref{linpro:2}. We have only shown the bound for the principal series scalars, $t_1^{(0)} \in \mathbb{R}$, but there is also a bound for the complementary series. We did not find examples that were close to saturating this bound, unlike for the analogous bound for hyperbolic surfaces~\cite{Bonifacio:2021msa}. In Figure~\ref{fig:bound_0_2} we show an upper bound on the smallest spin-2 eigenvalue $|t^{(2)}_1|$  for different values of $\big| t_1^{(0)} \big|$. For this plot we used 35 spectral identities with $N=6$, using Linear Program~\ref{linpro:1}.
\begin{figure}[h]
\begin{center}
\epsfig{file=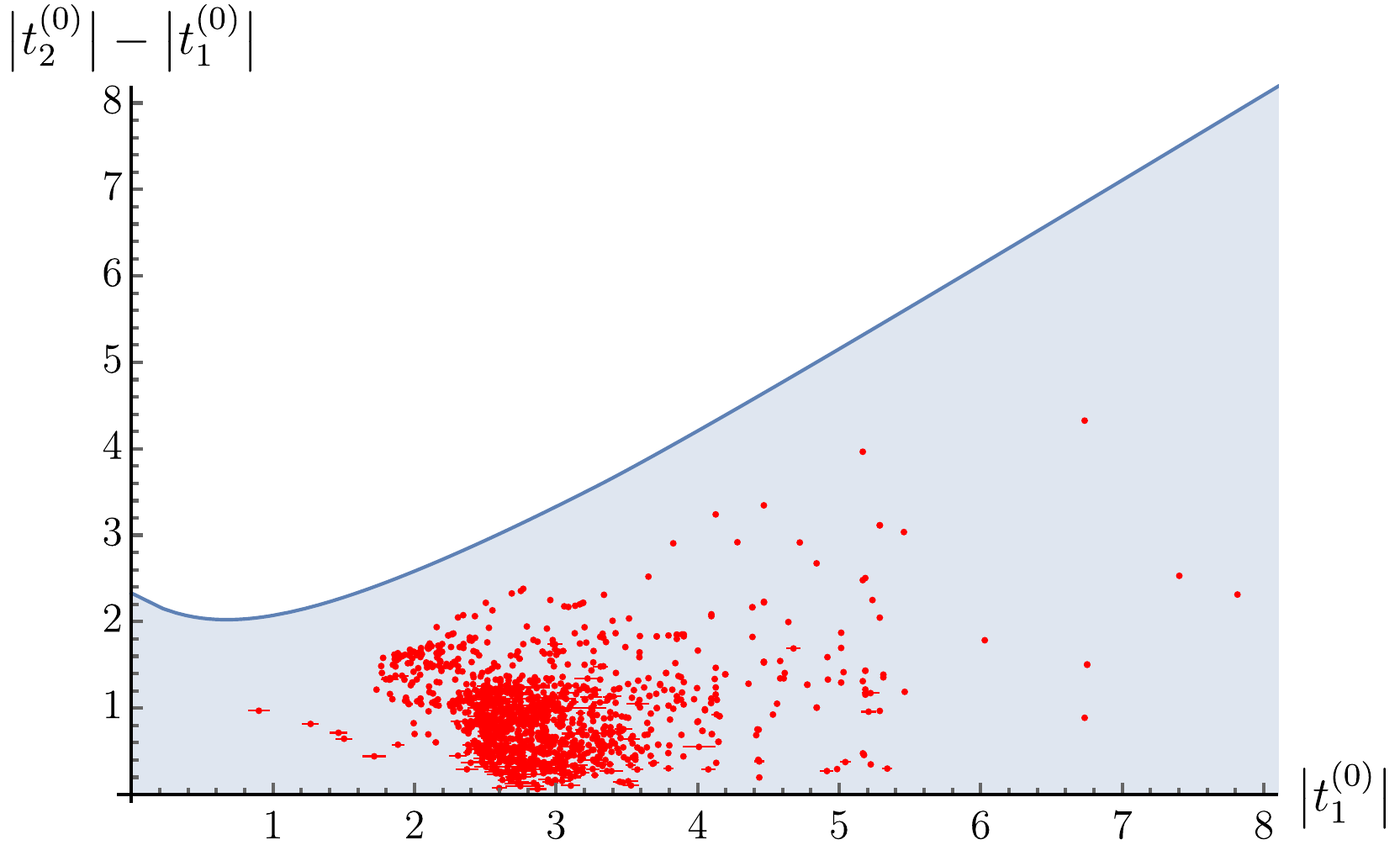,scale=.4}
\end{center}
\caption{An upper bound on the gap $\big| t_2^{(0)} \big|-\big| t_1^{(0)} \big|$ between the first two non-vanishing eigenvalues of the scalar Laplacian for different values of $\big| t_1^{(0)} \big|$ with $t_1^{(0)} \in \mathbb{R}$. The  blue region is allowed by the bound. The red dots represent various example  orbifolds.}
\label{fig:bound_0_0}
\end{figure}

\begin{figure}[h]
\begin{center}
\epsfig{file=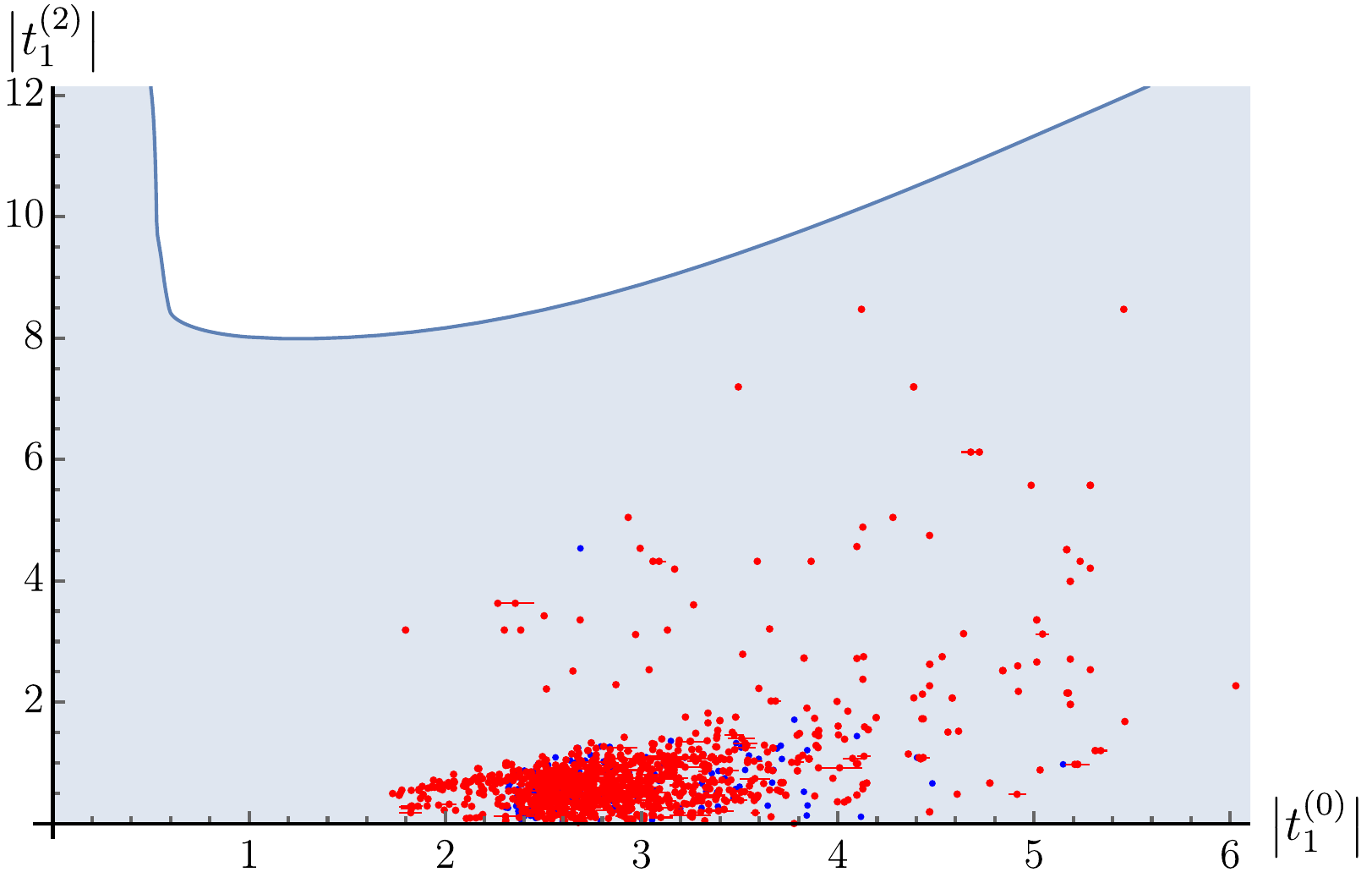,scale=.4}
\end{center}
\caption{An upper bound on the smallest spin-2 eigenvalue for different values of the spin-0 spectral gap with $t_1^{(0)} \in \mathbb{R}$. The red dots are example  orbifolds for which we have upper and lower bounds on $\big|t_1^{(0)} \big|$. The blue dots are example  orbifolds for which we only have lower bounds on $\big|t_1^{(0)} \big|$.}
\label{fig:bound_0_2}
\end{figure}

\subsection{External spin-1 bounds}
Now we consider identical external $J=1$ modes using the framework of Linear Program~\ref{linpro:1}. An upper bound on the spin-0 spectral gap $\big| t_1^{(0)} \big|$ for different values of the smallest spin-1 eigenvalue $\big| t^{(1)}_1 \big|$  is shown in Figure~\ref{fig:bound_1_0}, where an upper bound $\big|t^{(0)}_1 \big|< c$ is to be interpreted as $\lambda^{(0)}_1 < c^2+1$. For this plot we used 18 spectral identities with $N=3$. The closest example to saturating this bound is the orbifold m203$(0,3)(2,0)$ of volume $V \approx 0.4866$, which has $t^{(1)}_1=0$ and $\big| t_1^{(0)} \big|=5.4629(6)$. If we use 101 spectral identities with $N=8$, then the bootstrap bound at $t^{(1)}_1=0$ gives $\big| t^{(0)}_1 \big|<5.5289$. We checked this result rigorously using rational arithmetic by rationalizing the output of \texttt{SDPB} and checking positivity of the resulting rational polynomials using the algorithm from Appendix B of \cite{Kravchuk:2021akc}. 

In Figure~\ref{fig:bound_1_2}, which is presented in the introduction, we show an upper bound on the smallest spin-2 eigenvalue $\big| t^{(2)}_1 \big|$ for different values of $\big| t^{(1)}_1 \big|$. For this plot we used 29 consistency conditions with $N =4$. Table~\ref{tab:bound_1_2-examples} lists some examples that are close to this bound, together with upper bounds on $\big| t^{(2)}_1 \big|$.
We have rigorously checked the upper bound on $\big| t^{(2)}_1 \big|$ for $ t^{(1)}_1 = 0$ by rationalizing the output of \texttt{SDPB}. This result, together with the upper bound on $\big| t^{(0)}_1 \big|$ mentioned above, can be stated as the following theorem, which was given in the Introduction:

 \begin{theoremLambda1-Betti}
Let $M$ be a closed oriented hyperbolic 3-orbifold with positive first Betti number. The first positive eigenvalue of the Laplace--Beltrami operator on $M$ satisfies
$$
\lambda^{(0)}_1 < \frac{252549}{8000} \approx 31.569.
$$
Similarly, the first eigenvalue $\lambda_1^{(2)}$ of the $J=2$ spectral problem \eqref{eq:spinJequation} satisfies
$$
\lambda_1^{(2)}<\frac{46685}{8000}\approx 5.8356\,.
$$
\end{theoremLambda1-Betti} 

Note that the bounds in Figures~\ref{fig:bound_1_2} and \ref{fig:bound_1_0} are also valid with $\big| t^{(1)}_1 \big|$  replaced by $\big| t^{(1)}_i \big|$ for any $i \geq 1$, since the odd spins do not appear as exchanged states in the consistency conditions we consider.

\begin{figure}[h]
\begin{center}
\epsfig{file=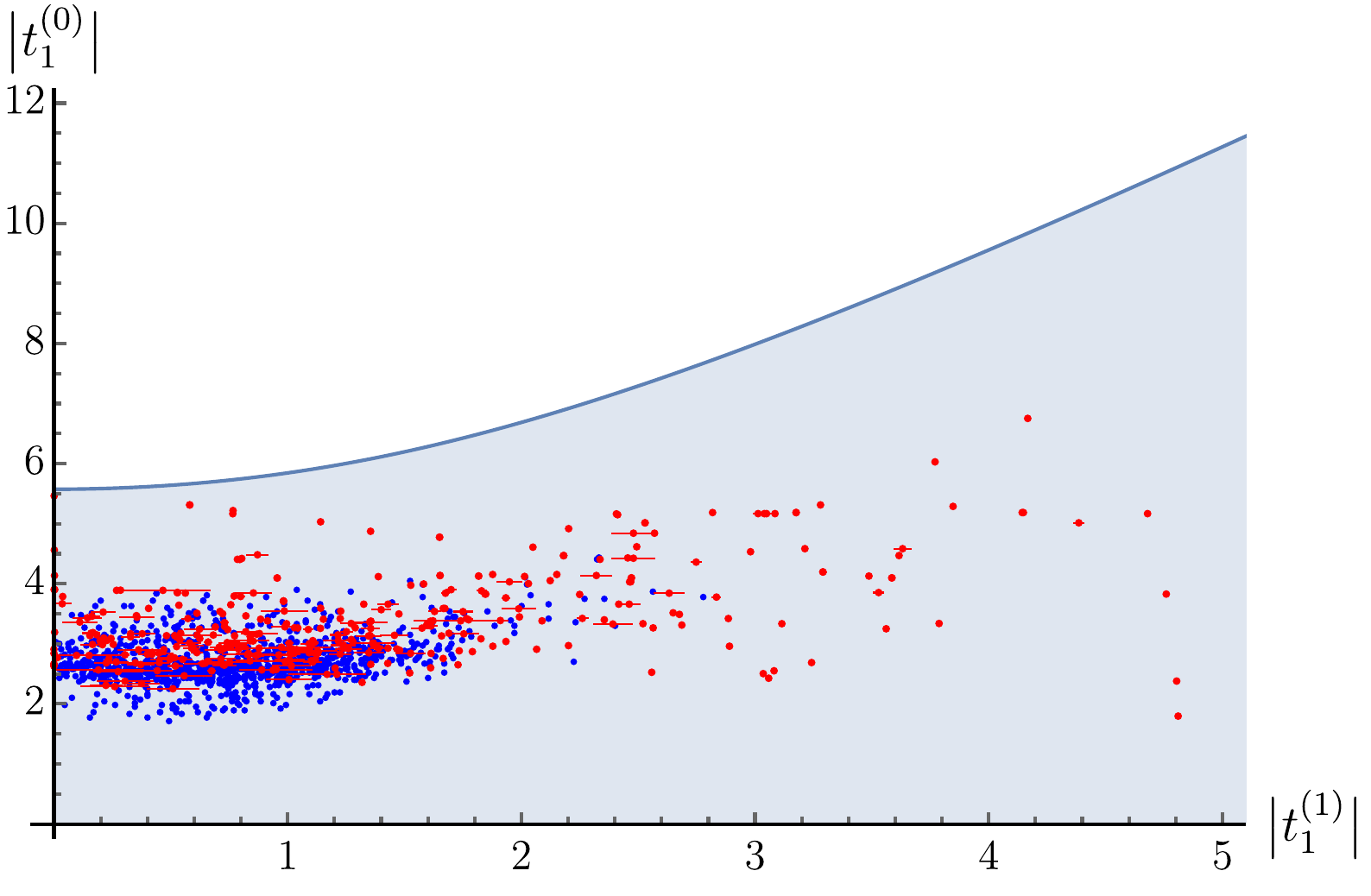,scale=.4}
\end{center}
\caption{An upper bound on the spin-0 spectral gap for different values of the smallest spin-1 eigenvalue. Here an upper bound $\big|t^{(0)}_1 \big|< c$ means that $\lambda^{(0)}_1 < c^2+1$.}
\label{fig:bound_1_0}
\end{figure}

\subsection{External spin-2 bounds}
Lastly, we consider identical external $J=2$ modes. An upper bound on the spin-0 spectral gap $\big|t_1^{(0)}\big|$ for different values of the smallest spin-2 eigenvalue $\big| t^{(2)}_1 \big|$ is shown in Figure~\ref{fig:bound_2_0}. An upper bound on the spin-2 gap $\big| t^{(2)}_2 \big|-\big| t^{(2)}_1 \big|$  for different values of $\big| t^{(2)}_1 \big|$ is shown in Figure~\ref{fig:bound_2_2}. The bound applies more generally to the difference of the absolute values of the first two distinct eigenvalues, so we show this quantity for the examples with degenerate first eigenvalues. We also show an upper bound on the smallest spin-4 eigenvalue $\big| t^{(4)}_1 \big|$  for different values of $\big| t^{(2)}_1 \big|$ in Figure~\ref{fig:bound_2_4}. For each of these plots we used 39 spectral identities with $N=5$. The first and third bounds use Linear Program~\ref{linpro:1}, while the second bound follows Linear Program~\ref{linpro:2}.

\begin{figure}[h]
\begin{center}
\epsfig{file=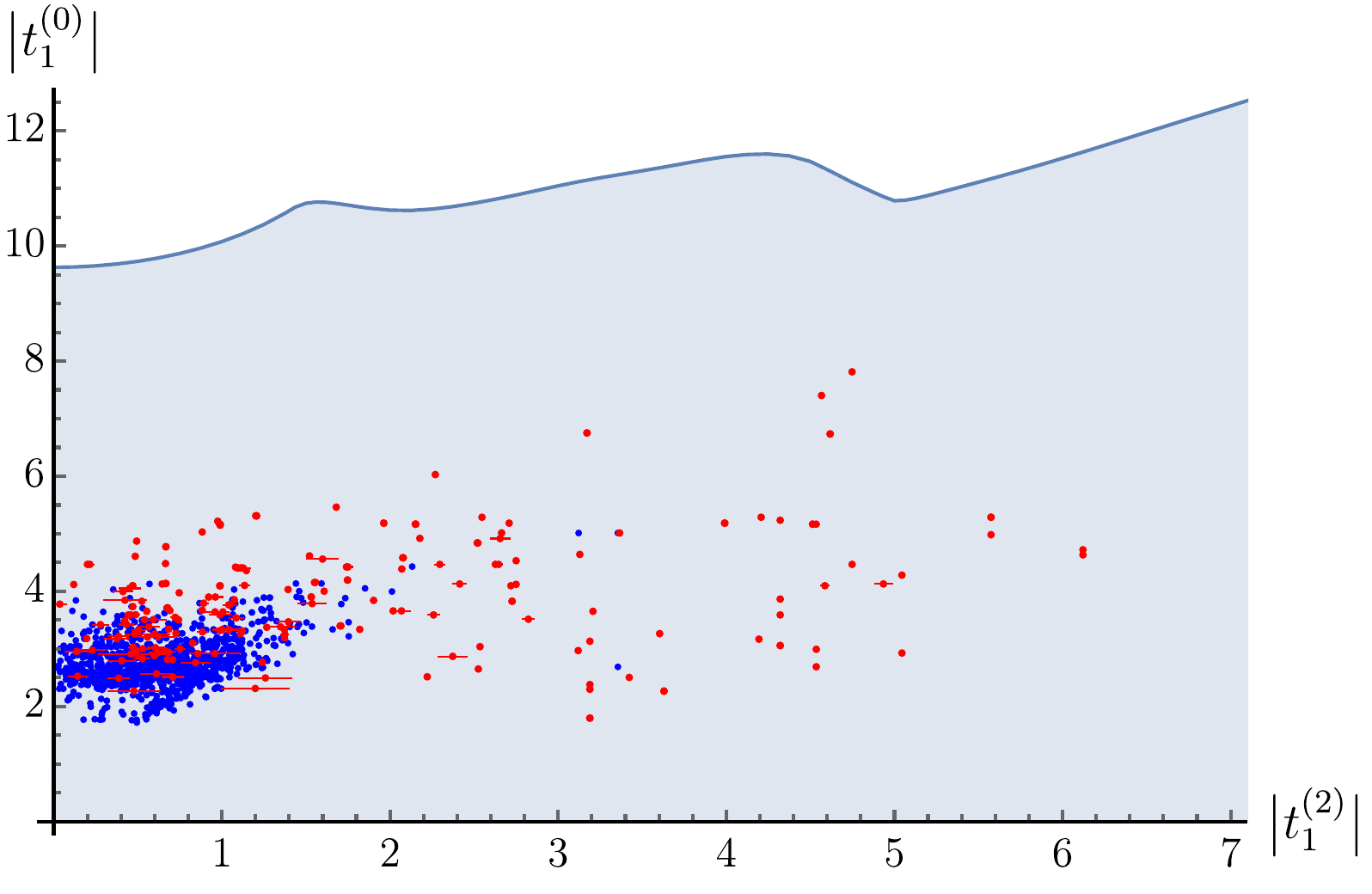,scale=.4}
\end{center}
\caption{An upper bound on the spin-0 spectral gap for different values of the smallest spin-2 eigenvalue. Here an upper bound $\big|t^{(0)}_1 \big|< c$ means that $\lambda^{(0)}_1 < c^2+1$.}
\label{fig:bound_2_0}
\end{figure}

\begin{figure}[h]
\begin{center}
\epsfig{file=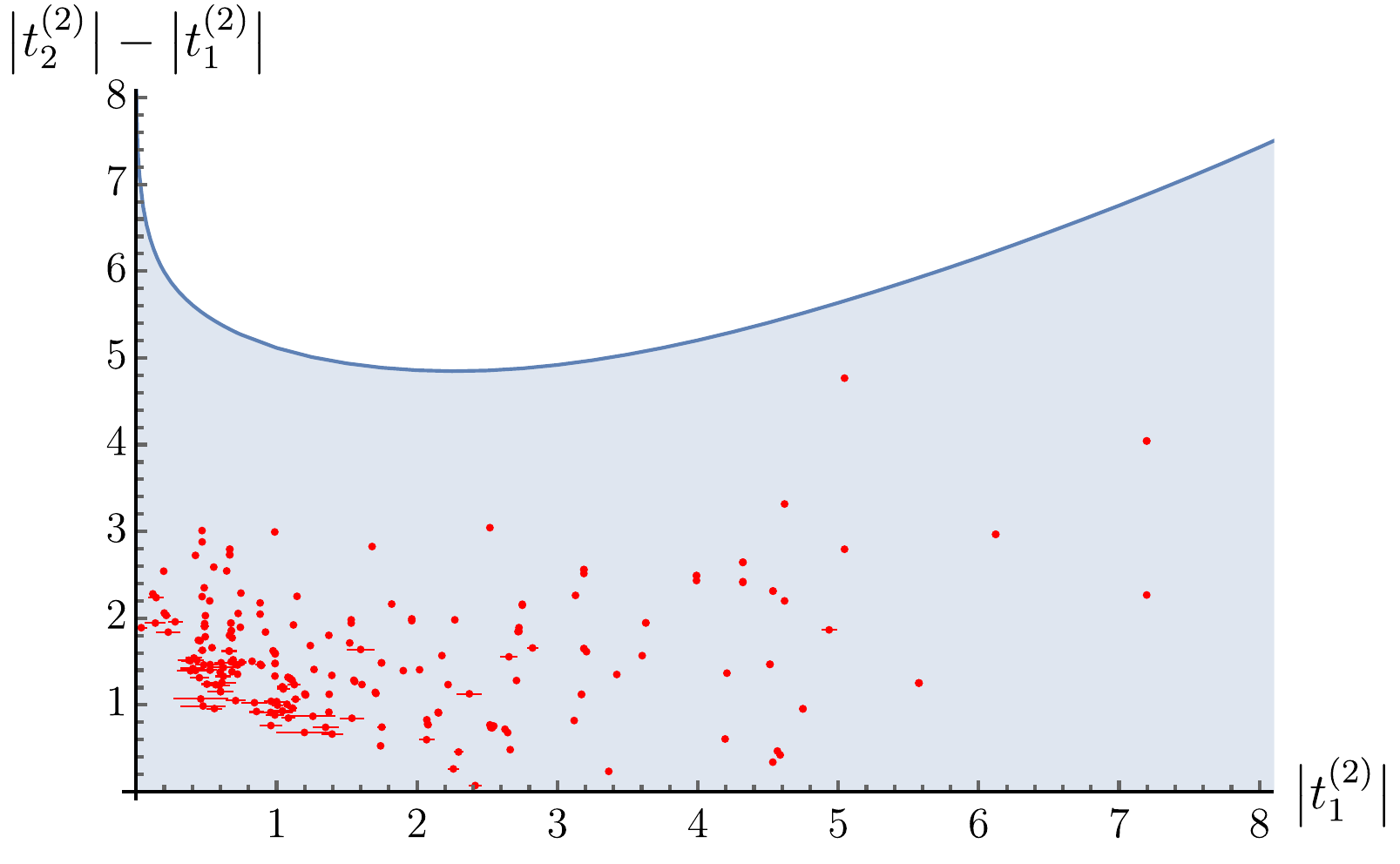,scale=.4}
\end{center}
\caption{An upper bound on the spin-2 gap $\big| t^{(2)}_2 \big|-\big| t^{(2)}_1 \big|$ for different values of the smallest spin-2 eigenvalue.}
\label{fig:bound_2_2}
\end{figure}

\begin{figure}[h]
\begin{center}
\epsfig{file=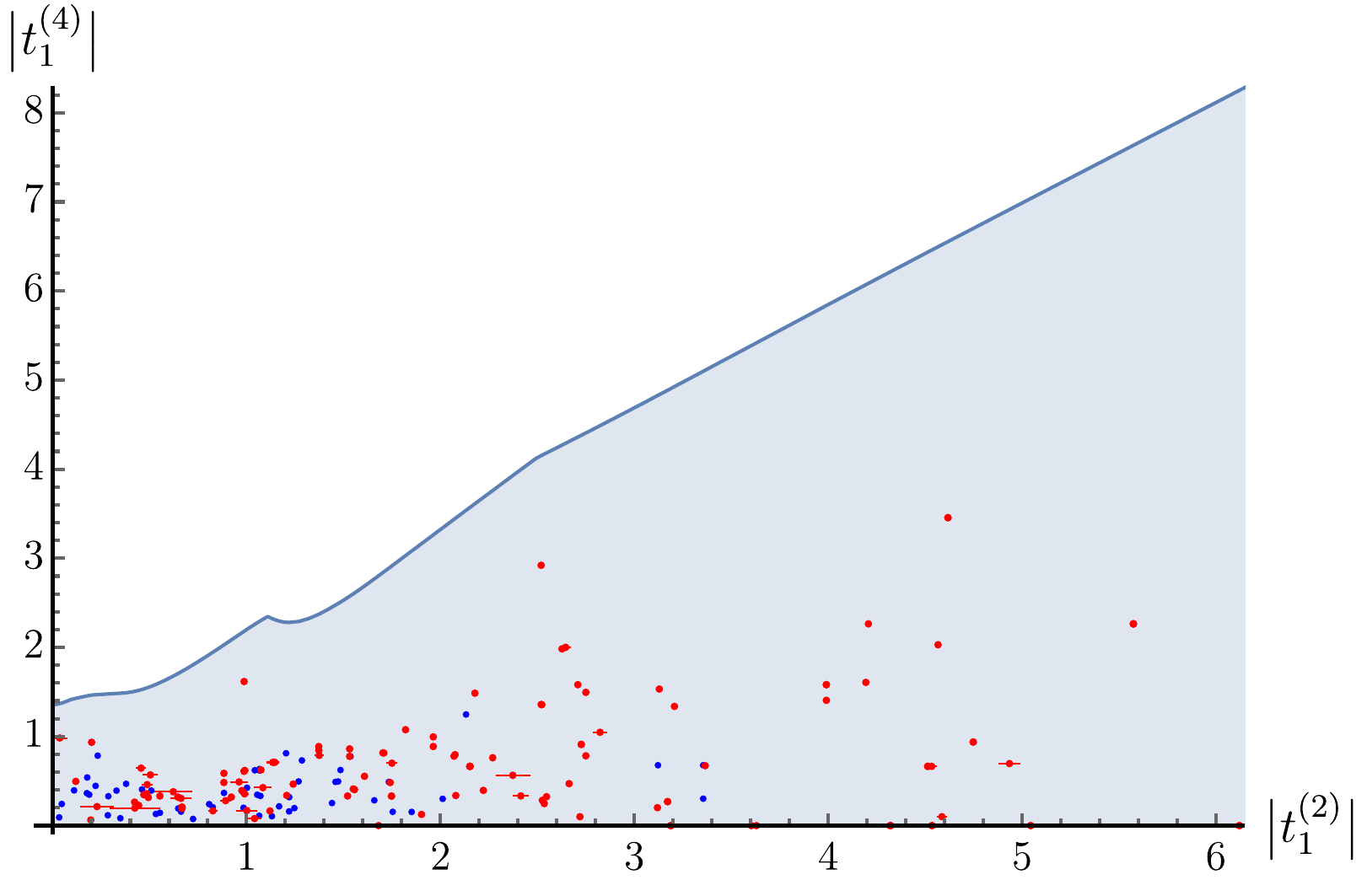,scale=.4}
\end{center}
\caption{An upper bound on the smallest spin-4 eigenvalue for different values of the smallest spin-2 eigenvalue.}
\label{fig:bound_2_4}
\end{figure}

\newpage 
	
\section{Bounds from the trace formula}
\label{sec:selberg-general-bounds}
In this section, we use the Selberg trace formula to derive bounds for manifolds and orbifolds of a particular volume. In particular, we find upper bounds on the spectral gap $\lambda^{(0)}_1$ and the systole length in terms of volume.

\subsection{Bounds on the spectral gap of the Laplace--Beltrami operator}
We start by looking for upper bounds on the spectral gap $\lambda^{(0)}_1$ for orbifolds.
	\begin{proposition}\label{STFpragmatic}
	 Let us define 
	 $$
	\begin{aligned}
	N_{j, \omega}(V)&:= \frac{1}{j!} \frac{d^j}{d\beta^j} \left(1-\frac{V}{2\pi }\frac{e^{-\beta}}{4 \sqrt{\pi\beta^3 }}\right)\Bigg|_{\beta=\omega}\,,\\
	S_{j}(\lambda)&:=\frac{1}{j!} e^{\omega\lambda}\frac{d^j}{d\beta^j}  \left(e^{-\beta\lambda}\right)\Bigg|_{\beta=\omega}\,,\quad T_{j, \omega }(p):=-\frac{1}{j!}e^{\frac{p}{4 \omega}}\frac{d^j}{d\beta^j}  \left(\frac{e^{-\frac{p}{4 \beta}-\beta}}{2 \sqrt{\pi \beta}}\right)\Bigg|_{\beta=\omega}\,,
	\end{aligned}
	$$
	and 
	$$
	c^2_\gamma:= \frac{\Theta(\gamma) {\rm vol}(\Gamma_{\gamma} \backslash G_{\gamma} )}{2(\cosh \ell(\gamma)-\cos \phi(\gamma))}\geqslant 0\,.
	$$
	In terms of these variables and denoting $p_\gamma=\ell^2_\gamma$,  $\lambda=t^2+1$, the Selberg trace formula~\eqref{eq:STF} implies the following conditions:
	\begin{equation}\label{eq:STFpragmatic}
	N_{j, \omega} (V)+ \sum_{i=1}^{\infty}e^{-\omega\lambda^{(0)}_i} S_{j}(\lambda^{(0)}_i) + \sum_{[\gamma]\neq 1} c_\gamma^2 e^{-\frac{p_\gamma}{4 \omega}} T_{j, \omega }(p_\gamma)=0\,,  \quad  \forall\ j\in \mathbb{Z}_{\geq 0}\,.
	\end{equation}
	\end{proposition}
	
	\begin{proof}
	We consider the $J=0$ Selberg trace formula with Gaussian test functions,
	$$
	\hat{H}(t):=e^{-\beta \left(t^2+1\right)}\,,\quad H(\ell)=\frac{e^{-\frac{\ell^2}{4 \beta}-\beta}}{2 \sqrt{\pi \beta}}\,,
	$$
	which were shown to be admissible test functions in~\cite{Lin-Lipnowski2021}.
	Using these functions in ~\eqref{eq:STF} leads to 
	$$
	 \left(1-\frac{V}{8\pi }\frac{e^{-\beta}}{\sqrt{\pi\beta^3 }}\right) + \sum_{i=1}^{\infty}e^{-\beta \left((t_i^{(0)})^2+1\right)} = \sum_{[\gamma] \neq 1} c^2_\gamma  \frac{e^{-\frac{\ell_\gamma^2}{4 \beta}-\beta}}{2 \sqrt{\pi \beta}}.
	$$ 
The result follows by taking derivatives of this equation with respect to $\beta$ at $\beta=\omega$.
	\end{proof}
	
	The form of the Selberg trace formula in Proposition~\ref{STFpragmatic} is amenable to semi-definite programming owing to the following proposition:

\begin{proposition}\label{prop:spectralBoundsSTF}
Let $\lambda^{(0)}_1$ be the first nonzero eigenvalue of the Laplace--Beltrami operator on an oriented closed hyperbolic 3-manifold $M$ with volume V.  For a given $\Lambda\in \mathbb{N}$ and $\lambda_*>0$,  if there exists a vector $\vec{\alpha}$ with components $\{\alpha_j\}_{j=0,\cdots, \Lambda}$ such that
\begin{enumerate}
	\item\label{eq:functional}
	 $\sum_{j=0}^{\Lambda}\alpha_j N_{j, \omega}(V)=1$\,,
	\item $\sum_{j=0}^{\Lambda}\alpha_j S_j(\lambda) \geqslant 0$ \, for all \ $\lambda\geqslant \lambda_*,$
\item $ \sum_{j=0}^{\Lambda} \alpha_j T_{j, \omega}(p)\geqslant 0$\, for all \ $p\geqslant 0$,
	\end{enumerate}
	then $\lambda^{(0)}_1 <\lambda_*$.
\end{proposition}

\begin{proof}
Let us denote 
\be \label{eq:vec-notation}
\vec{\alpha} \cdot \vec{N}_{\omega, \Lambda}(V) \coloneqq \sum_{j=0}^{\Lambda}\alpha_j N_{j, \omega}(V), \quad   \vec{\alpha} \cdot \vec{S}_{\Lambda}(\lambda ) \coloneqq \sum_{j=0}^{\Lambda}\alpha_j S_j(\lambda), \quad  \vec{\alpha} \cdot \vec{T}_{\omega, \Lambda}(p ) \coloneqq  \sum_{j=0}^{\Lambda} \alpha_j T_{j, \omega}(p).
\ee
We observe from~\eqref{eq:STFpragmatic} that the Selberg trace formula gives 
\begin{equation} \label{eq:STF-lambda-eq}
\sum_i e^{-\omega\lambda^{(0)}_i}\vec{\alpha} \cdot \vec{S}_{\Lambda}(\lambda^{(0)}_i)=-\vec{\alpha} \cdot \vec{N}_{\omega, \Lambda}(V) -\sum_{[\gamma]\neq 1} c_\gamma^2\vec{\alpha} \cdot \vec{T}_{\omega, \Lambda}(p_\gamma)\,.
\end{equation}
The right-hand side of~\eqref{eq:STF-lambda-eq} is negative, by our first and third assumptions and the fact that $c_{\gamma}^2\geq0$, and therefore the spectral sum on the left-hand side is also negative. By our second assumption, we have $\vec{\alpha} \cdot \vec{S}_{\Lambda}(\lambda )\geq 0$ for all $\lambda\geqslant \lambda_*$, so there must be a contribution to the spectral sum with $\lambda <\lambda_*$, which proves the proposition.
\end{proof}

Using Proposition~\ref{prop:spectralBoundsSTF}, we can find upper bounds on $\lambda^{(0)}_1$ in terms of volume. In Figure~\ref{fig:selberg-upper-bound}, which is presented in the introduction, we show the resulting numerical upper bound on $\big|t^{(0)}_1\big|$ for $\Lambda=25$ and $\omega=1/2$, together with lower bounds on the spectral gaps of some example orbifolds. For $V = 0.94$, the upper bound we find is less than $47.32$. We can extend this to a rigorous bound for $V\geq 0.94$ to obtain the following result:
\begin{theoremLambda1}
Let $\lambda^{(0)}_1$ be the first nonzero eigenvalue of the Laplace--Beltrami operator on an oriented closed hyperbolic 3-manifold. This satisfies the following upper bound:
\be
\lambda^{(0)}_1 < \frac{193785}{4096} \approx 47.3108.
\ee
\end{theoremLambda1}
\begin{proof}
The smallest closed hyperbolic 3-manifold is the Weeks manifold~\cite{Gabai2009}, which has volume $V= 0.94270736 \ldots >47/50$. It is therefore sufficient to verify the upper bound on the spectral gap for manifolds with $V>47/50$. Consider the following vector, which is also included in the ancillary file ``alpha-manifold.txt":
$$
\begin{aligned}
\vec{\alpha}(V) & = \frac{-1}{10^8}\Big(\frac{71121394927352129365612271758976459 \sqrt{2} V-15944911639937280000000000000 \sqrt{e} \pi ^{3/2}}{39862279099843200000
   \left(4 \sqrt{e} \pi ^{3/2}-\sqrt{2}
   V\right)}, \nn \\
& -19245520229997,9006560617729,-1814726225141,6547197712933, \nn \\
&10721444671903,  26479728197731, 52987341181353,97645713842881,15955921283
   4925,\nn \\
   &232045826988633,298547797283859,  338853972868892,337870740962673, 294614195206713, \nn \\ 
   & 223363554736767,146183319207000,818246341579
   11,  38702935837026,15224422720686, \nn \\
   & 4873578082873,1231368841228,234699426881,31398670005,2589176277,  97567520\Big). \nn
\end{aligned}
$$
We will show that this vector satisfies the three conditions in Proposition~\ref{prop:spectralBoundsSTF} with $\omega=1/2$, $\Lambda=25$, and $\lambda^*=193785/4096$ for  $47/50  \leqslant V \leqslant 25$. This then implies that all manifolds with $V\leqslant 25$ satisfy the stated bound on the spectral gap by Proposition~\ref{prop:spectralBoundsSTF}. 

The condition $ \vec{\alpha}(V) \cdot \vec{N}_{1/2, 25}(V)=1$ follows from the normalization of $\vec{\alpha}(V)$. To verify the other two conditions in Proposition~\ref{prop:spectralBoundsSTF}, let us define the following two order-25 polynomials with rational coefficients:
$$
\begin{aligned}
P_1(x) &  \coloneqq \vec{\alpha}(V) \cdot \vec{S}_{25}(193785/4096 +x ) -\vec{\alpha}(V) \cdot \vec{S}_{25}(193785/4096  ), \\
P_2(x) &  \coloneqq \sqrt{2 \pi e} \left[ \vec{\alpha}(47/50)  \cdot \vec{T}_{1/2, 25}(x )-\vec{\alpha}(47/50)  \cdot \vec{T}_{1/2, 25}(0 )\right]+2438516 .
\end{aligned}
$$
It can be checked rigorously that $P_1(x) \geq 0$ and $P_2(x) \geq 0$ for $x \geq 0$. We verified this using the algorithm from Appendix B of~\cite{Kravchuk:2021akc}.
We also note that we can write
$$
\begin{aligned}
\vec{\alpha}(V) \cdot \vec{S}_{25}(193785/4096  ) & = q_1 -\frac{q_2 V}{2 \sqrt{2 e \pi^3}-V}, \label{eq:V-condition-1} \\
\sqrt{2 \pi e} \left[ \vec{\alpha}(V)-\vec{\alpha}(47/50) \right] \cdot \vec{T}_{1/2, 25}(p )  & = \frac{ q_3  \sqrt{e \pi^3}(50V-47 )}{\sqrt{2} (47 \sqrt{2}-200 \sqrt{e \pi^3})(\sqrt{2} V-4 \sqrt{e \pi^3})}, \\
\sqrt{2 \pi e}  \, \vec{\alpha}(47/50)  \cdot \vec{T}_{1/2, 25}(0 )-2438516 & = \frac{\sqrt{e \pi ^3} q_4 - \sqrt{2} q_5}{47 \sqrt{2} -200 \sqrt{e \pi^3}},
\end{aligned}
$$
where $q_1, \dots q_5$ are particular positive rational numbers. It can be checked straightforwardly that for $47/50  \leqslant V \leqslant 25$ we have
$$
\begin{aligned}
 \vec{\alpha}(V) \cdot \vec{S}(193785/4096 ) & \geq 0\quad  \text{and} \quad \left[ \vec{\alpha}(V)-\vec{\alpha}(47/50) \right] \cdot \vec{T}_{1/2, 25}(p ) \geq 0. \nn
\end{aligned}
$$
For example, the right-hand side of~\eqref{eq:V-condition-1} is positive at $V=0$ and vanishes at $V=2 q_1 \sqrt{2 e \pi^3}/(q_1 + q_2)<2 \sqrt{2 e \pi^3}$, and it can be verified that this zero is greater than 25 using rational approximations to $e$ and $\pi$.
Similarly, we can also show that
$$
\sqrt{2 \pi e}  \,\vec{\alpha}(47/50)  \cdot \vec{T}_{1/2, 25}(0 )-2438516 >0. \nn
$$
By combining these inequalities with the non-negativity of the polynomials $P_1(x)$ and $P_2(x)$ for $x \geq 0$, we deduce that for $47/50  \leqslant V \leqslant 25$ and $ x \geq 0$ we have
$$
 \vec{\alpha}(V) \cdot \vec{S}_{25}(193785/4096 +x ) \geq 0 \quad  \text{and}  \quad\vec{\alpha}(V) \cdot \vec{T}_{1/2, 25}(x ) \geq 0 . \nn
$$

To complete the proof, we need to show that manifolds with $V>25$ also satisfy the stated bound on $\lambda^{(0)}_1$. For this we use the following vector:
$$
\begin{aligned}
\vec{\alpha}'(V)  = \frac{1}{10^8} & \Big( \frac{119054041669332482567 \sqrt{3} V+1600000000 e^{1/3} \pi ^{3/2}}{2 \left(8 e^{1/3} \pi
   ^{3/2}-3 \sqrt{3} V\right)},-3514199558672594815, \nn \\
   &\quad  23240623445190391,-74712908174666 \Big). \nn
\end{aligned}
$$
We show that this vector satisfies the three conditions in Proposition~\ref{prop:spectralBoundsSTF} with $\omega=1/3$, $\Lambda=3$, and $\lambda^*=193785/4096 $ for  $ V \geq 25$. This then implies that all manifolds with $V\geq  25$ satisfy the bound on the spectral gap. As before, the condition $ \vec{\alpha}'(V) \cdot \vec{N}_{1/3, 3}(V)=1$ follows from the normalization of $\vec{\alpha}'(V)$. Let us now define the following two cubic polynomials with rational coefficients:
$$
\begin{aligned}
P_3(x) &  \coloneqq \vec{\alpha}'(V) \cdot \vec{S}_3(193785/4096 +x ) -\vec{\alpha}'(V) \cdot \vec{S}_3(193785/4096  ), \nn \\
P_4(x) &  \coloneqq \sqrt{3 \pi} e^{1/3} \vec{\alpha}'(V)  \cdot\left[  \vec{T}_{1/3, 3}(x )-  \vec{T}_{1/3, 3}(0 )\right]. \nn
\end{aligned}
$$
These polynomials are easily verified to be non-negative for $x \geq 0$. We can also verify that for $V \geq 25$ we have
$$
\vec{\alpha}'(V) \cdot \vec{S}_3(193785/4096  ) \geq 0 \quad  \text{and}  \quad  \vec{\alpha}'(V)  \cdot \vec{T}_{1/3, 3}(0 ) \geq 0,
$$
since these are simple rational functions of $V$.
Combining these inequalities with the non-negativity of the polynomials $P_3(x)$ and $P_4(x)$ for $x \geq 0$, we deduce that for $V \geq 25$ and $x \geq 0$ we have
$$
\vec{\alpha}'(V) \cdot \vec{S}_3(193785/4096 +x ) \geq 0 \quad  \text{and}  \quad  \vec{\alpha}'(V)  \cdot  \vec{T}_{1/3, 3}(x ) \geq 0.
$$
This then implies the stated bound on $\lambda^{(0)}_1$.
\end{proof}
	
	Theorem~\ref{thm:gap-bound} is the strongest available upper bound on the spectral gap for closed hyperbolic 3-manifolds, as far as we know.\footnote{Another way to obtain an explicit bound of this form is by combining the upper bound on Cheeger's constant $h\leqslant 57.7131$ from Callahan's thesis~\cite{Callahan-thesis} with Buser's inequality $\lambda^{(0)}_1\leqslant 4 h+10h^2$~\cite{Buser-inequality}, which gives $\lambda^{(0)}_1< 33549$. } This bound can be strengthened by increasing the derivative order $\Lambda$. Numerically it seems to converge above $47$. This is still somewhat larger than the largest manifold spectral gap we know, namely $\lambda^{(0)}_1 = 29.218(5)$ for the Meyerhoff manifold. 
	
	To get a bound that is valid for all closed oriented hyperbolic 3-orbifolds, we must increase $\Lambda$. Taking $\Lambda=97$ and $\omega=1/2$ gives the following bound:	\begin{theorem} \label{thm:gap-bound-orbifold}
Let $\lambda^{(0)}_1$ be the first nonzero eigenvalue of the Laplace--Beltrami operator on an oriented closed hyperbolic 3-orbifold. This satisfies the following bound:
$$
\lambda^{(0)}_1 < 400.
$$
\end{theorem}
\begin{proof}
The proof is similar to the proof of Theorem~\ref{thm:gap-bound}. The smallest oriented closed hyperbolic 3-orbifold is the order-2 quotient of the tetrahedral orbifold $\Gamma^+(T_2) \backslash \hh$~\cite{Gehring-Martin-I, Marshall-Martin-II}, as described in Appendix~\ref{app:tetrahedra}, which has volume $V= 0.039050\ldots >18/461$. It is therefore sufficient to verify the upper bound on the spectral gap for orbifolds with $V>18/461$. To this end, consider the  vector $\vec{\alpha}(V)$ defined in the ancillary file ``alpha-orbifold.txt," which is too large to display here.
This vector satisfies the three conditions in Proposition~\ref{prop:spectralBoundsSTF} with $\omega=1/2$, $\Lambda=97$, and $\lambda^*=400$ for  $18/461  \leqslant V \leqslant 25$. This then implies that all orbifolds with $V\leqslant 25$ satisfy $\lambda^{(0)}_1 < 400$ by Proposition~\ref{prop:spectralBoundsSTF}. The condition $ \vec{\alpha}(V) \cdot \vec{N}_{1/2, 97}(V)=1$ again follows from the normalization of $\vec{\alpha}(V)$. To verify the other two conditions, we define the following two order-97 polynomials with rational coefficients:
$$
\begin{aligned}
P'_1(x) &  \coloneqq \vec{\alpha}(V) \cdot \vec{S}_{97}(400 +x ) -\vec{\alpha}(V) \cdot \vec{S}_{97}(400  ), \\
P'_2(x) &  \coloneqq \sqrt{2 \pi e} \left[ \vec{\alpha}(18/461)  \cdot \vec{T}_{1/2, 97}(x )-\vec{\alpha}(18/461)  \cdot \vec{T}_{1/2, 97}(0 )\right]+92641675559.
\end{aligned}
$$
It can be checked rigorously that $P'_1(x) \geq 0$ and $P'_2(x) \geq 0$ for $x \geq 0$ using the algorithm from Appendix B of~\cite{Kravchuk:2021akc}.
We can also check that the following inequalities hold for $18/461  \leqslant V \leqslant 25$:
$$
\begin{aligned}
&\vec{\alpha}(V) \cdot \vec{S}_{97}(400   )   \geq 0 , \\
&\sqrt{2 \pi e} \left[ \vec{\alpha}(V)-\vec{\alpha}(18/461) \right] \cdot \vec{T}_{1/2, 97}(p )   \geq 0, \\
&\sqrt{2 \pi e}  \, \vec{\alpha}(18/461)  \cdot \vec{T}_{1/2, 97}(0 )-92641675559  \geq 0,
\end{aligned}
$$
since these are simple rational functions of $V$.
By combining these inequalities with the non-negativity of the polynomials $P'_1(x)$ and $P'_2(x)$ for $x \geq 0$, we deduce that for $18/461 \leqslant V \leqslant 25$ and $ x \geq 0$ we have
$$
 \vec{\alpha}(V) \cdot \vec{S}_{97}(400 +x ) \geq 0 \quad  \text{and}  \quad\vec{\alpha}(V) \cdot \vec{T}_{1/2, 97}(x ) \geq 0 . \nn
$$
To complete the proof, we need to show that orbifolds with volume  $V>25$ also satisfy $\lambda^{(0)}_1 < 400$, but this follows from the proof of Theorem~\ref{thm:gap-bound}.
\end{proof} 

\subsection{Analytical bounds on $\lambda^{(0)}_1$ at large volume}
	In this subsection, we use the Selberg trace formula to derive an upper bound on $\lambda^{(0)}_1$ for large volume, following the approach used by~\cite{bourque2023linear} for hyperbolic surfaces. This leads to Theorem~\ref{thm:lambda1As}, which we restate here:
\begin{theoremAs}
For any $\varepsilon>0,$  there exists $V_*>0$ such that for all oriented closed hyperbolic 3-orbifolds with volume $V>V_{*}$, we have
\begin{equation*}\label{eq:3}
\lambda^{(0)}_1< 1+ \frac{54\left(1+\varepsilon\right)}{\left[ \log V\right]^2}\,.
\end{equation*}
\end{theoremAs}

\begin{remark}
One can refine the above theorem to give a precise value of $V_*$ at the expense of introducing a parameter $\alpha$.  The refined version reads: for any $\alpha>1$ and $\varepsilon>0,$  we have
\begin{equation*}\label{eq:4}
\lambda^{(0)}_1\leqslant 1+ \frac{54\left(1+\varepsilon\alpha\right)}{\left[ \log V\right]^2}\quad \text{for all}\quad V>\max \left\{\exp\left(\frac{\pi}{\varepsilon}\right),\exp\left(\sqrt{\frac{54(1+\varepsilon\alpha)}{(\alpha-1)}}\right)\right\}\,.
\end{equation*}
\end{remark}

The rest of this subsection is devoted to proving Theorem~\ref{thm:lambda1As}.  For that purpose, the following key lemma is instrumental:
\begin{lemma}\label{lemma:1}
Given an oriented closed hyperbolic 3-orbifold with volume $V$,  if we can find an admissible even function $f: \mathbb{R} \rightarrow \mathbb{R}$ such that
\begin{equation}
\text{(a)}\ f(x)\geqslant 0 \quad \forall x\in\mathbb{R}\,,\quad \quad \text{(b)}\ \widehat{f}(t) \leqslant 0\quad \forall t \geq v\,,  \quad \quad \text{(c)}\ \widehat{f}(Ri) <-\frac{V}{2\pi R^3} f''(0)\,,
\end{equation}
then we have
$\lambda^{(0)}_1 < 1+ \frac{v^2}{R^2}\,,$
where $\lambda^{(0)}_1$ is the first non-zero eigenvalue of the Laplace--Beltrami operator.
\end{lemma}

\begin{proof}
Given the admissible even function $f: \mathbb{R} \rightarrow \mathbb{R}$,  we define $H_R(x):=f(x/R)/R$. $H_R$ is admissible to use in the even trace formula since $f$ is admissible. Using the behaviour of the Fourier transform under scaling, we have
\begin{enumerate}
\item $H_R(x)\geqslant 0$ \, for all \ $x\in\mathbb{R}$,
\item $\widehat{H}_R(t) \leqslant 0$ \, for all \ $t\geq v/R$,
\item $\widehat{H}_R(i) <-\frac{V}{2\pi} H_R''(0)$.
\end{enumerate}
We then consider the Selberg trace formula~\eqref{eq:STF} with $H=H_R$. It follows that
$$
\sum_{n=1}^{\infty} \widehat{H}_R\left(t^{(0)}_n\right)= \left[-\frac{V}{2\pi}H_R''(0)-\widehat{H}_R\left(i\right)\right]+ \sum_{[\gamma] \neq 1}  \Theta(\gamma) {\rm vol}(\Gamma_{\gamma} \backslash G_{\gamma} ) \frac{H_R(\ell(\gamma))}{2(\cosh \ell(\gamma)-\cos \phi(\gamma))}\,.
$$
The right-hand side of this equation is positive by the first and third conditions stated above, so the spectral sum on the left-hand side is also positive.
Since $\widehat{H}_R(t) \leqslant 0$ for all $t \geq v/R$, there must be a contribution to the spectral sum with $\lambda^{(0)}_1 < 1+ \frac{v^2}{R^2}$ and the lemma follows. 
\end{proof}

Now we come back to the proof of the Theorem~\ref{thm:lambda1As}.  The idea is to explicitly construct a function $f$ satisfying the conditions of Lemma~\ref{lemma:1}.
\begin{proof}[Proof of Theorem~\ref{thm:lambda1As}]
Given $\varepsilon>0$,  we make the following choices of $v$ and $R$:
\begin{equation}\label{choice}
v^2=\frac{27}{2\pi^2}\left(1+\varepsilon\right)\,,\quad  R=\frac{\log (V)}{2\pi}\,.
\end{equation}
For these choices, we  consider the admissible even function $f= h + v^{-2} h''$, where $h$ is the function defined in~\eqref{eq:h-basic} with $m=6$ and $\delta = \pi/3$, so that 
\begin{equation}\label{func}
\widehat{f}(t):=\left(1-\frac{t^2}{v^2}\right)\left(\frac{\sin\frac{\pi t}{3}}{\frac{\pi t}{3}}\right)^{6}.
\end{equation}
The function $f$ can be written as a piecewise quintic polynomial on the interval $[-2 \pi, 2 \pi]$ and it vanishes elsewhere. Using the algorithm from Appendix~B of~\cite{Kravchuk:2021akc}, we verify that each quintic polynomial is non-negative on its domain if $v^2\geq \frac{45}{11\pi^2}$. Our choice of $v^2$ satisfies $v^2>\frac{45}{11\pi^2}$, so we have $ f(x)\geqslant 0$ for all $x\in\mathbb{R}$.
From~\eqref{func}, we also see that $ \widehat{f}(t) \leqslant 0$ for all $t\geq v$. Furthermore, we have
$$
f''(0) = -\frac{27 \varepsilon}{8 \pi^3(1+\varepsilon)},
$$
so we obtain
$$\frac{\widehat{f}(Ri)}{-\frac{V}{2\pi R^3} f''(0)} \sim \frac{\pi}{\varepsilon \log(V)}$$
as $V\rightarrow\infty$. There therefore exists some $V_*>0$ such that for $V>V_*$ we have
$$
\widehat{f}(Ri) <-\frac{V}{2\pi R^3} f''(0)\,.
$$
We therefore see that all of the conditions of Lemma~\ref{lemma:1} are satisfied for $V>V_*$ and the choices of $v$ and $R$ given by~\eqref{choice}, which proves the theorem.
\end{proof}

\subsection{Bounds on the systole length}
A systole of a closed hyperbolic manifold $M$ is one of its shortest closed geodesics. The systole length of $M$, denoted $\rm{sys}(M)$, is the length of a systole. We can use the Selberg trace formula to find an upper bound on the systole length of manifolds of a given volume. Some previous systole bounds for hyperbolic 3-manifolds are discussed in~\cite{Gromov83, adams_reid_2000, White2002, Katz-Schaps-Vishne2007, Gendulphe2011, lakeland2014, Murillo-thesis}.

\begin{proposition}\label{prop:systole}
For a given $\Lambda\in \mathbb{N}$ and $p_*>0$,  if there exists a vector $\vec{\alpha}$ with components $\{\alpha_j\}_{j=0,\cdots, \Lambda}$ such that 
\begin{enumerate}
\item $\sum_{j=0}^{\Lambda}\alpha_j N_{j, \omega}(V)=1\,,$ 
\item $\sum_{j=0}^{\Lambda}\alpha_j S_j(\lambda) \geqslant 0$ \,  for all\ $\lambda\geqslant 0,$ 
\item $ \sum_{j=0}^{\Lambda} \alpha_j T_{j, \omega}(p)\geqslant 0 $ \,for all\ $p\geqslant p_*$,
\end{enumerate}
then the length of a systole of any closed hyperbolic manifold $M$ with volume $V$ satisfies $\rm{sys}(M) < \sqrt{p_*}$. 
\end{proposition}
\begin{proof}
Using the notation from~\eqref{eq:vec-notation}, observe that the Selberg trace formula gives 
\begin{equation} \label{eq:systole-prop}
\sum_{[\gamma]\neq 1} c_\gamma^2 \vec{\alpha} \cdot \vec{T}_{\omega, \Lambda}(p_\gamma)=-\vec{\alpha} \cdot \vec{N}_{\omega, \Lambda}(V) -\sum_i e^{-\omega\lambda^{(0)}_i}\vec{\alpha} \cdot \vec{S}_{\Lambda}(\lambda^{(0)}_i)\,.
\end{equation}
The right-hand side of~\eqref{eq:systole-prop} is negative, by our first two assumptions, and therefore the sum over nontrivial conjugacy classes on the left-hand side is also negative.
Since $\vec{\alpha} \cdot \vec{T}_{\omega, \Lambda}(p)\geq 0$ for all $p\geqslant p_*$, by our third assumption, and $c_{\gamma}^2 \geq 0$, there must be a contribution to the sum with $p < p_*$, which proves the proposition.
\end{proof}

Taking $\Lambda=31$ gives the bound shown in Figure~\ref{fig:selberg-systole-bound}. The manifold closest to the bound is $M={\rm m}007(3,1)$, which has volume $V \approx 1.0149$ and sys$(M) \approx 0.8314$, while the $\Lambda=31$ bound for this volume is ${\rm sys} < 0.8951$. For closed hyperbolic manifolds, the systole can also be bounded from above using the fact that we can embed a ball of radius $\rm{sys}(M)/2$ (see, e.g.,~\cite{Murillo-thesis}). A ball of radius $r$ in $\hh$ has volume $\pi(\sinh(2r)-2r)$, so we get 
\be \label{eq:exact-simple-bound}
V(M) \geq \pi (\sinh({\rm sys}(M))-{\rm sys}(M)).
\ee
Numerically inverting~\eqref{eq:exact-simple-bound} gives the upper bound shown as the dashed line in Figure~\ref{fig:selberg-systole-bound}.

\begin{figure}[htbp]
\begin{center}
\epsfig{file=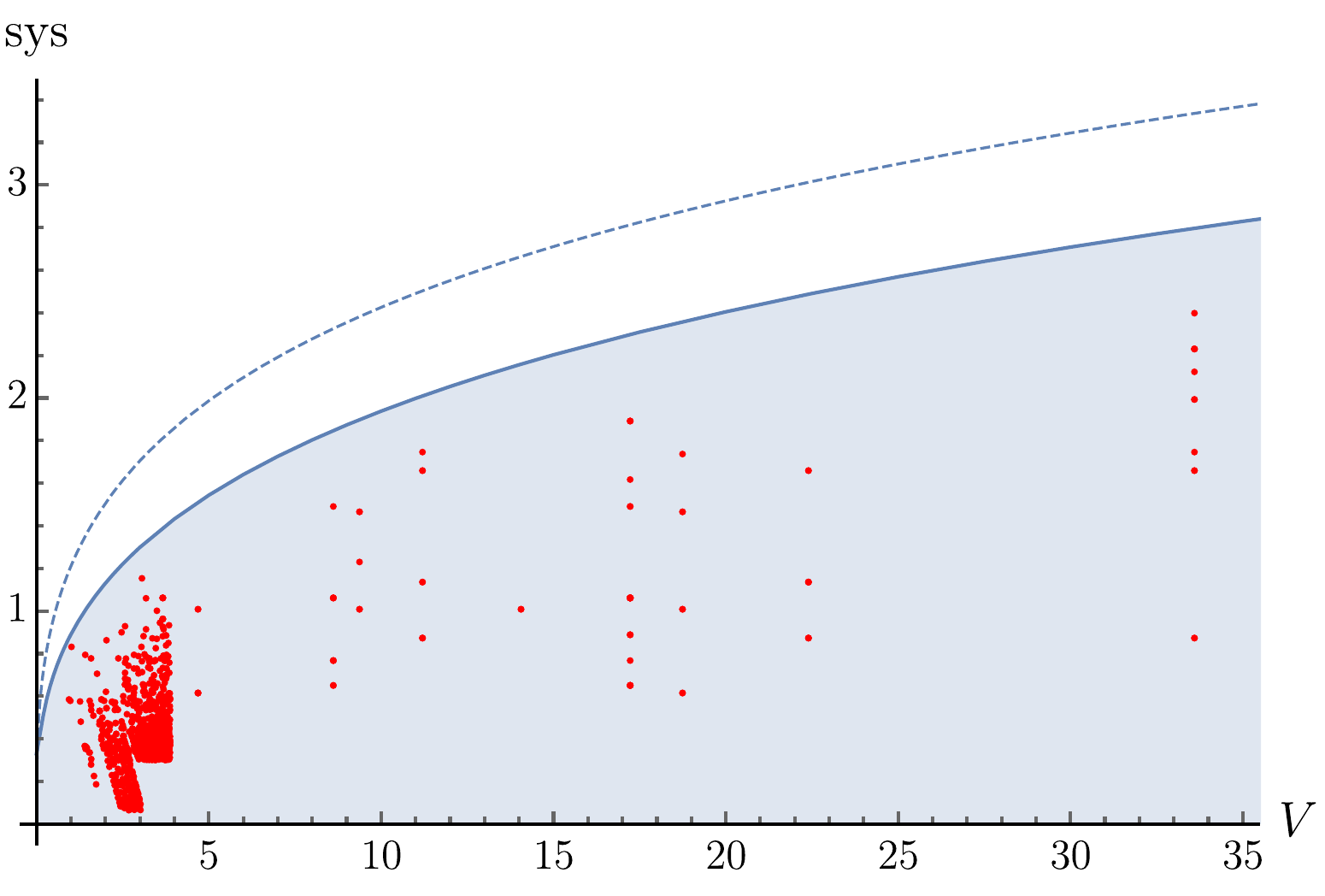,scale=.4}
\end{center}
\caption{An upper bound on the systole length of a closed hyperbolic 3-manifold from Selberg's trace formula. The blue region is allowed and the red dots correspond to example manifolds. The dashed line is the upper bound deduced from embedding a ball of radius ${\rm sys}(M)/2$, namely $f^{-1}(V)$ where $f(x) = \pi(\sinh x -x)$. }
\label{fig:selberg-systole-bound}
\end{figure}
	
 \subsection{Analytical bounds on the systole for large volume}
	In this subsection, we find an asymptotic upper bound on the systole length for closed hyperbolic 3-manifolds, following the approach used by~\cite{bourque2023linear} for hyperbolic surfaces. The best known asymptotic upper bound for non-compact finite volume hyperbolic 3-manifolds is 
$$
\mathrm{sys}(M) \leq \frac{4}{3} \log V -c,
$$ 
 for some constant $c$~\cite{lakeland2014}. 
The fact that we can embed a ball of radius $\rm{sys}(M)/2$ in a closed manifold gives the following bound for large volume:
 \begin{proposition}\label{simple-systole}
For any $\epsilon>0$, there exists $V_*>0$ such that for all oriented closed hyperbolic 3-manifolds $M$ with volume $V>V_{*}$, we have
$$
\mathrm{sys}(M)\leq\log(V)- \log \left( \frac{\pi}{2} \right)+ \epsilon \,,
$$
where $\mathrm{sys}(M)$ is the systole length of $M$.
\end{proposition}
\begin{proof}
A ball of radius $r$ in $\hh$ has volume 
$$
V(B_r)=\pi(\sinh(2r)-2r).
$$
Fix $\epsilon>0$ and define $\epsilon' \coloneqq 1-e^{-\epsilon}>0$. There is some $r_*$ such that  $V(B_r) > \frac{\pi}{2}e^{2r}(1-\epsilon')$ for $r>r_*$. In a closed manifold $M$ of volume $V$, we can embed a ball of radius ${\rm sys}(M)/2$, so we have
$$
V \geq  \frac{\pi}{2}e^{\mathrm{sys}(M)}(1-\epsilon') \implies \mathrm{sys}(M) \leq \log(V)-\log( \pi/2)+\epsilon,
$$
for ${\rm sys}(M)>2 r_*$. Taking $V_*=\frac{\pi}{2} e^{2 r_*}$ then gives the result. 
\end{proof}
There is no lower bound on the systole for closed hyperbolic 3-manifolds, but there are examples with systole growing like $\frac{2}{3} \log V-c'$ for constant $c'$~\cite{Katz-Schaps-Vishne2007}, so this limits the optimal upper bound. We were not able to improve upon the coefficient of the logarithm in Proposition~\ref{simple-systole}  using the Selberg trace formula, but we can improve upon the constant $\log \frac{\pi}{2} \approx 0.4516$. This gives Theorem~\ref{thm:systole}, which we restate here:
\begin{theoremSys}
There exists $V_*>0$ such that for all closed oriented hyperbolic 3-manifolds $M$ of volume $V>V_*$, we have
$$
\mathrm{sys}(M)< \log(V)- 0.8482\,.
$$
\end{theoremSys}

The rest of this subsection is devoted to proving the above theorem.  For that purpose, the following key lemma is instrumental:

\begin{lemma}\label{lemma:3}
Let $M$ be an oriented closed hyperbolic 3-manifold with systole length $\mathrm{sys}(M)$ and volume $V$. If we can find an admissible even function $f: \mathbb{R} \rightarrow \mathbb{R}$ such that
$$
\text{(a)}\ f(x)\leqslant 0  \quad \forall x\geqslant R_*\,,\quad \quad \text{(b)}\ \widehat{f}(t) \geqslant 0 \quad \forall t\in\mathbb{R} \cup [-i,i]\,,  \quad \quad \text{(c)}\ \widehat{f}(i) >-\frac{V}{2\pi} f''(0)\,,
$$
then we have
$\mathrm{sys}(M) < R_*\,.$
\end{lemma}

\begin{proof}
We apply the Selberg trace formula~\eqref{eq:STF} to the test function $f$, specializing to manifolds using~\eqref{eq:covol-hyperbolic}. This gives 
\be\label{eq:STFdd}
\sum_{[\gamma] \neq 1}   \ell(\gamma_0) \frac{f(\ell(\gamma))}{2(\cosh \ell(\gamma)-\cos \phi(\gamma))}=\sum_{n=1}^{\infty} \widehat{f}\left(t^{(0)}_n\right)+\frac{V}{2\pi}f''(0)+\widehat{f}\left(i\right).
\ee
The right-hand side of~\eqref{eq:STFdd} is positive by the second and third conditions on $f$, so the sum over conjugacy classes on the left-hand side must also be positive. Since $f(x)\leqslant 0$ for all $x\geqslant R_*$,  there must be a group element $\gamma$ with $\ell(\gamma)<R_*$. From the classification of elements of ${\rm PSL}_2( \mathbb{C})$, this group element corresponds to a closed geodesic of length $\ell(\gamma)$, which proves the lemma.
\end{proof}

To prove Theorem~\ref{thm:systole}, we apply Lemma~\ref{lemma:3} with a particular test function.

\begin{proof}[Proof of Theorem~\ref{thm:systole}]
Following~\cite{bourque2023linear},  we consider even test functions of the form
\begin{equation}
\widehat{f}(t)=\left(\frac{\eta_\alpha(Rt/2)}{4^\alpha \Gamma(\alpha+1)}\right)^2e^{-b^2t^2/2}(c-1+b^2t^2)\,,
\end{equation}
where
$$
\eta_\alpha(t):=  \Gamma(\alpha+1)\sum_{n=0}^{\infty} \frac{(-1)^n}{\Gamma(n+1)\Gamma(n+\alpha+1)}\left(\frac{t}{2}\right)^{2n}\,
$$
is an entire even function. When $\alpha$ is a non-negative integer,  we have  $\eta_\alpha(t)=2^\alpha\Gamma(\alpha+1)\frac{J_\alpha(t)}{t^{\alpha}}$, where $J_{\alpha}$ is a Bessel function of the first kind. We will use the same choice of constants $\alpha=0.559$,  $c=2.3726$ and $\kappa_0=0.1814$ as in~\cite{bourque2023linear}, but we make a different choice for $b$, namely $b=0.733$. For $\alpha>1/2$, we can write $f$ as the convolution of a continuous function supported on the interval $[-R, R]$ and the function $e^{-x^2/2b^2}(c-x^2/b^2)$~\cite{bourque2023linear}, so $f$ is admissible by Proposition D.1 of~\cite{Lin-Lipnowski2021}. 

For $\alpha=0.559$,  $b>0$, $c=2.3726$ and $\kappa_0=0.1814$,  Corollary $6.10$ of~\cite{bourque2023linear} implies that there exists $R_0>0$ such that $f(R+b\kappa)\leq 0$ for $R\geq R_0$ and $\kappa \geq \kappa_0$.  Moreover,  we also have $\widehat{f}(t) \geqslant 0$ for all $t\in\mathbb{R} \cup [-i,i]$ if $b^2<c-1$.  We therefore only need to check condition $c)$ of Lemma~\ref{lemma:3}. 

To find $f''(0)=-\frac{1}{2 \pi} \int_{-\infty}^{\infty}\mathrm{d}t\ t^2 \widehat{f}(t)$, we note that
$$
\int_{-\infty}^{\infty}\mathrm{d}t\ t^2 \widehat{f}(t)=2\int_{0}^{\infty}\mathrm{d}t\ t^2 \widehat{f}(t) = \frac{2^{-2 \alpha+\frac{1}{2}} \Gamma \left(\alpha +\frac{1}{2}\right)}{b^3} \left((c-1) I_{3}\left(-\frac{R^2}{2 b^2}\right)+3 I_{5}\left(-\frac{R^2}{2 b^2}\right)\right),
$$
where 
$$
\begin{aligned}
I_{n}(t):=\, _2\tilde{F}_2\left(\frac{n}{2},\alpha +\frac{1}{2};\alpha +1,2 \alpha +1;t\right).
\end{aligned}
$$
We further observe that 
$$
\lim_{R\to\infty} R^{2\alpha+1}\int_{-\infty}^{\infty}\mathrm{d}t\ t^2 \widehat{f}(t) =   \frac{2 \Gamma(1-\alpha)}{\pi}\left(\frac{b^2}{2}\right)^{\alpha-1}\left(c+1-2\alpha\right)\,,
$$
which leads to 
$$
\frac{\widehat{f}(i)}{\frac{V}{4\pi^2}\int_{-\infty}^{\infty}\mathrm{d}x\ t^2 \widehat{f}(t) }\sim \frac{(c-1-b^2)e^{b^2/2}e^{R-\log V}}{\frac{ \Gamma(1-\alpha)}{2\pi^2}\left(\frac{b^2}{2}\right)^{\alpha-1}\left(c+1-2\alpha\right)} 
$$
as $R\to\infty$.
For any $\rho>1$, we can choose 
$$
R= \log V + \log\rho -b^2/2 + \log\left(\frac{ \Gamma(1-\alpha)}{2\pi^2}\left(\frac{b^2}{2}\right)^{\alpha-1}\frac{c+1-2\alpha}{c-1-b^2}\right),
$$
and there exists some $V_*$ such that for $V>V_*$ we have $R \geq R_0$ and $\widehat{f}(i) >-\frac{V}{2\pi} f''(0)$.

Hence,  for $V>V_*$  the conditions of Lemma~\ref{lemma:3} are satisfied with $R_*=R+ b\kappa_0$. Taking $b=0.733$, which satisfies  $b^2<c-1$, and $\rho=1+10^{-5}$, we get
$$
R_* <\log V- 0.8482.
$$
Lemma~\ref{lemma:3} implies that $\mathrm{sys}(M) < R_*\,$ for $V>V_*$, which proves the theorem.
\end{proof}

\section{Final remarks}
\label{sec:conclusion}
In this paper, we studied the spectra of hyperbolic 3-orbifolds $M = \Gamma\backslash\hh$ using linear programming applied to spectral identities arising from a) the associativity of the product of functions on $\Gamma\backslash\mathrm{PSL}_2(\mathbb{C})$ and b) the Selberg trace formula. The first approach was directly inspired by a modern incarnation of the conformal bootstrap program~\cite{Belavin:1984vu, Rattazzi:2008pe, Rychkov:2009ij}. In fact, also the approach based on the Selberg trace formula is formally almost identical to a version of the conformal bootstrap, this time the so-called annulus modular bootstrap, used to constrain boundary conditions of two-dimensional conformal field theories~\cite{Collier:2021ngi}.

There are several natural extensions of our analysis. In the context of the approach using associativity, the following come to mind:
\begin{enumerate}
\item We only exploited spectral identities arising from four identical irreducible representations. In other words, we specialized to the case $i=j=k=\ell$ of Theorem~\ref{thm:crossingKFinite}. It can be expected that stronger bounds will follow from incorporating additional spectral identities, as has been done originally for conformal field theories in~\cite{Kos:2014bka} and more recently for hyperbolic surfaces in~\cite{Kravchuk:2021akc,Bonifacio:2021aqf}.

\item The spectral identities from associativity constrain not only the spectrum of the Laplacians, but also the integrals of triple products of the eigenfunctions. It would be very interesting to use linear programming to prove new upper bounds on the triple products. For arithmetic manifolds, such bounds would translate to bounds on values of $L$-functions, which have been extensively studied in the number theory literature~\cite{Sarnak94, Bernstein2006, Philippe-Venkatesh2010, Blomer2023}. In the context of the conformal bootstrap, the analogous task is bounding the three-point functions, for which there exist well-developed methods.

\item By replacing $\rm{PSL}_2(\mathbb{C})$ with its double cover $\rm{SL}_2(\mathbb{C})$, we can study the spectrum of the Dirac operator on hyperbolic 3-orbifolds.
\end{enumerate}

Regarding the approach based on the Selberg trace formula, there are also interesting future avenues to explore:
\begin{enumerate}
\item One should be able to bound eigenvalue multiplicities and the number of small eigenvalues ($\lambda^{(0)}_i \leq 1$), as was done for hyperbolic surfaces in~\cite{bourque2023linear}.
\item It would be desirable to prove an optimal version of the asymptotic bounds on $\lambda^{(0)}_1$ and the systole length at large volume following from the trace formula.
\end{enumerate}

An important virtue of the approach based on associativity is that all of the spectral identities follow directly from representation theory of $G=\mathrm{PSL}_2(\mathbb{C})$. As a result, this method can be generalized to various other contexts simply by replacing $\mathrm{PSL}_2(\mathbb{C})$ by a different locally compact group. For example, the case $\mathrm{SO}(n,1)$ yields hyperbolic $n$-manifolds, $\mathrm{SU}(n,1)$ the complex hyperbolic manifolds, and $\mathrm{GL}_2(\mathbb{Q}_p)$ regular graphs.

To conclude, we would like to raise the question: Just how powerful is associativity, really? It may seem surprising that nearly sharp bounds follow from a modest subset of the associativity constraints. Could it be the case that if we were to impose all of the spectral identities of Theorem~\ref{thm:crossingKFinite} for all $i,j,k,\ell$, then the bounds would have to be sharp? We can formulate this question as follows:
\begin{problem}
Describe the set of discrete spectra compatible with the spectral identities~\eqref{eq:crossing} and the constraints $\widehat{c}_{ijk}\in\mathbb{R}$. Every cocompact lattice $\Gamma\subset\mathrm{PSL}_2(\mathbb{C})$ provides such a spectrum. Does this exhaust all solutions?
\end{problem}

Why might one hope that cocompact lattices provide the only solutions with discrete spectra? The role of the spectral identities is to implement the consistency of an abelian algebra which is a representation of $G$. It is known that every abelian von Neumann algebra is the algebra of bounded functions on a measure space $X$.\footnote{We thank Petr Kravchuk for making us aware of this result.} In this case, we would like to prove the existence of such a measure space with a $G$-action. Ergodicity of this action, together with suitable assumptions on the spectrum, could guarantee that we must have $X = \Gamma\backslash G$ for a cocompact lattice $\Gamma$.
It would be interesting to see if this idea can be made rigorous.

\newpage
\appendix

\section{Some differential geometry}\label{app:difgeo}
The $D$-dimensional hyperbolic space can be realized as the locus
$$
\mathbb{H}^D = \{X\in\mathbb{R}^{1,D}:\,X_0^2-\sum_{i=1}^{D}X_{i}^2 = 1,\,X_0>0\}\,,
$$
where $\mathbb{R}^{1,D}$ is equipped with the flat metric of indefinite signature $ds^2 = -dX_0^2+\sum_{i=1}^{D}dX_i^2$. The group of orientation-preserving isometries of $\mathbb{H}^D$ is $G=\mathrm{SO}_0(1,D)$, acting on $\mathbb{R}^{1,D}$ in the obvious way.

We will make use of the Iwasawa decomposition $G=NAK$. Here $K = \mathrm{SO}(D)$ is the maximal compact subgroup
$$
K = \left\{
\left(
\begin{array}{cc}
1 & 0 \\
 0 & k
\end{array}
\right),\,k\in\mathrm{SO}(D)\right\}\,,
$$
and $N\simeq(\mathbb{R}^{D-1},+)$, $A\simeq(\mathbb{R}_{>0},\times)$ are the subgroups
$$
N = \left\{n(y),\,y\in\mathbb{R}^{D-1}\right\}\,,\quad
A = \left\{a(z),\,z\in\mathbb{R}_{>0}\right\}\,,
$$
where
$$
n(y)=
    \left(
\begin{array}{ccc}
1+\frac{y^2}{2} & y^t & -\frac{y^2}{2} \\
 y & I_{D-1} &  -y\\
 \frac{y^2}{2} & y^t &1-\frac{y^2}{2}
\end{array}
\right)\,,\quad
a(z)=\left(
    \begin{array}{ccc}
      \frac{z^2+1}{2z} & 0 & \frac{z^2-1}{2z}\\
      0 & I_{D-1} &0\\
      \frac{z^2-1}{2z}& 0& \frac{z^2+1}{2z}
    \end{array}
    \right)\,.
$$
Here $y\in\mathbb{R}^{D-1}$ is a column vector, $y^2$ is its squared norm, and $y^t$ is its transpose. $I_{D-1}$ is the identity matrix of size $D-1$. Let us assemble $y\in\mathbb{R}^{D-1}$ and $z\in\mathbb{R}_{>0}$ into a single vector $x=(y,z)\in\mathbb{R}^{D-1}\times\mathbb{R}_{>0}$ and write $\widetilde{n}(x) = n(y)a(z)$. We will parametrize a general element of $G$ as $g(x,k) = \widetilde{n}(x) k$.

Since $K$ is the stabilizer of $\widehat{X}=(1,0,\ldots,0)\in \mathbb{R}^{1,D}$, the map $G\rightarrow\mathbb{H}^{D}$ given by $g\mapsto g\cdot \widehat{X}$ yields a diffeomorphism $NA\simeq\mathbb{H}^{D}$. In terms of the coordinates $x\in\mathbb{R}^{D-1}\times\mathbb{R}_{>0}$ on $NA$ introduced above, the metric on $\mathbb{H}^D$ takes the form
$$
ds^2 = (x^{D})^{-2}\sum\limits_{a=1}^{D}(dx^a)^2\,.
$$
It will be convenient to introduce the vector fields $e_a = x^{D}\frac{\partial}{\partial x^a}$ and their dual one-forms $e^a = (x^{D})^{-1}dx^a$, so that $ds^2 = \sum_{a=1}^{D}e^ae^a$.

\begin{lemma}\label{lem:leftIwasawa}
Let $g\in G$ and $\widetilde{n}(x)\in N A$. Then $g\,\widetilde{n}(x) = \widetilde{n}(g\cdot x) r(g,x)$, where $r(g,x)\in K$ is given by
$$
r(g,x) = 
\left(
\begin{array}{cc}
1 & 0 \\
 0 & d g_x
\end{array}
\right)\,.
$$
Here $d g_x : T_x\mathbb{H}^D\rightarrow T_{g\cdot x}\mathbb{H}^D$ is the differential of the isometry $g:\mathbb{H}^{D}\rightarrow\mathbb{H}^D$, written in the basis of the vector fields $e_a$, i.e.,
$$
(dg_x)^{a}_{\phantom{a}b} = \frac{x^D}{(g\cdot x)^D}\frac{\partial (g\cdot x)^a}{\partial x^b}\,.
$$
\end{lemma}
\begin{proof}
Let $g_{1},g_2\in G$. Since $(dg_1g_2)_x = (dg_1)_{g_2\cdot x}\circ (dg_2)_x$, it follows that
$$
r(g_1g_2,x) = r(g_1,g_2\cdot x) r(g_2,x)\,.
$$
Therefore, it suffices to prove the Lemma for all $g$ in a set which generates $G$. $G$ is generated by $n(y)$ and $\tau n(y)\tau$, where $\tau = \mathrm{diag}(1,\ldots,1,-1)$. That the Lemma holds for $n(y)$ and for $\tau$ can be checked by a direct computation.
\end{proof}

Let $\Gamma$ be a discrete subgroup of $G$ so that $\Gamma\backslash \mathbb{H}^D$ is a hyperbolic $D$-orbifold. We will now describe the correspondence between subspaces of $C^{\infty}(\Gamma \backslash G)$ transforming irreducibly under right $K$-multiplication and smooth sections of powers of the cotangent bundle $T^*(\Gamma\backslash\mathbb{H}^D)$. For simplicity, we will focus on the symmetric traceless tensor representations, but the same construction works also for general irreducible representations.

Let $(V_{s},\rho_{s})$ be the irreducible representation of $K=SO(D)$ consisting of symmetric traceless tensors of rank $s$, with $V_{s}$ the underlying vector space and $\rho_{s}: K\rightarrow \mathrm{GL}(V_{s})$ a group homomorphism. Let $V_{s}^{*}$ be the dual of $V_s$. There is an obvious identification between maps $\varphi:\mathbb{H}^D\rightarrow V_s^{*}$ and traceless sections of the symmetric tensor power of the cotangent bundle $\mathrm{Sym}^s(T^{*}\mathbb{H}^D)$, obtained by contracting $\varphi$ with the one forms $e^a$, so that $\varphi$ corresponds to the tensor field $\varphi_{a_1\ldots a_s}(x)e^{a_1}(x)\ldots e^{a_s}(x)$, with repeated indices summed over. The main result we record in this appendix is

\begin{proposition}\label{prop:tensorFieldsApp}
Let $\theta : V_{s}\rightarrow C^{\infty}(\Gamma\backslash G)$ be a $K$-linear map. Then there exists a smooth traceless section $\varphi_{\theta}$ of $\mathrm{Sym}^s(T^{*}\Gamma\backslash\mathbb{H}^D)$ such that
\be
\theta(v)(g(x,k)) = \varphi_{\theta}(x)(\rho_{s}(k)v)\,,
\label{eq:thetaCorrespondence}
\ee
Furthermore, the correspondence between $\theta$ and $\varphi_{\theta}$ is one-to-one.
\end{proposition}
\begin{proof}
$K$-linearity of $\theta$ means that $\forall \widetilde{k}\in K$, we have $\theta(\rho_{s}(\widetilde{k})v)(g(x,k)) = \theta(v)(g(x,k\widetilde{k}))$. It follows from the Peter--Weyl theorem for $K$ at fixed $x$ that $\theta(v)(g(x,k)) = \varphi_{\theta}(x)(\rho_{s}(k)v)$ for some smooth symmetric traceless tensor field $\varphi_{\theta}:\mathbb{H}^D\rightarrow V_s^{*}$. It remains to implement left $\Gamma$-invariance. From Lemma~\ref{lem:leftIwasawa}, we have $\forall\gamma\in\Gamma$
$$
\theta(v)(g(x,k))=\theta(v)(\gamma g(x,k)) = \theta(v)(g(\gamma\cdot x,r(\gamma,x)k))
$$
and hence
$$
\varphi_{\theta}(x)(v) = \varphi_{\theta}(\gamma\cdot x)(\rho_{s}(r(\gamma,x))v)\,.
$$
By Lemma~\ref{lem:leftIwasawa}, the right-hand side is the pullback of the tensor field $\varphi_{\theta}$ by the isometry $\gamma\in\Gamma$, and hence $\varphi_{\theta}$ is a section of $\mathrm{Sym}^s(T^{*}\Gamma\backslash\mathbb{H}^D)$. Conversely, any $\varphi$ gives rise to $\theta$ through~\eqref{eq:thetaCorrespondence}, and so the correspondence is one-to-one.
\end{proof}
Note that the proof extends to reducible representations by writing them as a direct sum of irreducibles and applying the Peter--Weyl theorem to each irreducible separately. In particular, we get a correspondence between $K$-linear maps $V_1^{\otimes s}\rightarrow C^{\infty}(\Gamma\backslash G)$ and smooth sections of $(T^{*}\Gamma\backslash\mathbb{H}^D)^{\otimes s}$.

Having discussed the action of $K$ on $C^{\infty}(\Gamma\backslash G)$, it remains to understand the action of the noncompact generators of $G$. Let $\mathfrak{g}$ and $\mathfrak{k}$ be the Lie algebras of $G$ and $K$. We have $\mathfrak{g} = \mathfrak{k}\oplus \mathfrak{b}$, where
$$
\mathfrak{b} = \{B(v): v\in\mathbb{R}^D\}\,,\qquad
B(v) =
\left(
\begin{array}{cc}
0 & v^{t} \\
v & 0_{D}
\end{array}
\right)\,.
$$
Here $v$ is a column vector in $\mathbb{R}^{D}$ and $0_{D}$ is the zero $D$ by $D$ matrix. $\mathfrak{g}$ acts on elements of $C^{\infty}(\Gamma\backslash G)$ by differential operators. In particular, given a $K$-linear map $\theta: V_1^{\otimes s}\rightarrow C^{\infty}(\Gamma\backslash G)$, we can define a new $K$-linear map $d\theta: V_1^{\otimes (s+1)}\rightarrow C^{\infty}(\Gamma\backslash G)$ by acting with the Lie algebra elements $B(u)$
$$
\forall u\in V_1,\, v\in V_1^{\otimes s}\,:\qquad d\theta(u\otimes v)(g(x,k)) := \left.\frac{d}{dt}\right|_{t=0}\theta(v)(g(x,k)\mathrm{exp}(t B(u)))\,.
$$

\begin{proposition}\label{prop:nabla}
Let $\varphi_{\theta}$ be the smooth section of $(T^{*}\Gamma\backslash\mathbb{H}^D)^{\otimes s}$ corresponding to the map $\theta$ as in Proposition~\ref{prop:tensorFields}. Similarly, let $\varphi_{d\theta}$ be the smooth section of $(T^{*}\Gamma\backslash\mathbb{H}^D)^{\otimes (s+1)}$ corresponding to the map $d\theta$. Then
$$
\varphi_{d\theta} = \nabla \varphi_{\theta}\,,
$$
where $\nabla: \Gamma^{\infty}((T^{*}\Gamma\backslash\mathbb{H}^D)^{\otimes s})\rightarrow\Gamma^{\infty}((T^{*}\Gamma\backslash\mathbb{H}^D)^{\otimes (s+1)})$ is the covariant derivative. Here $\Gamma^{\infty}(Y)$ stands for the space of smooth sections of bundle $Y$.
\end{proposition}
\begin{proof}
$\varphi_{d\theta}$ is defined by
$$
\varphi_{d\theta}(x)(ku\otimes kv) = d\theta(u\otimes v)(g(x,k)) =  \left.\frac{d}{dt}\right|_{t=0}\theta(v)(g(x,k)\mathrm{exp}(t B(u)))\,.
$$
In order to rewrite the right-hand side in terms of $\varphi_{\theta}$, we can use
$$
g(x,k)\mathrm{exp}(t B(u)) = g(x',k')\,,
$$
where
\ba
x'^a &= x^a + x^{D}(ku)^{a} t+ O(t^2), \\
k' &= \mathrm{exp}(t Q(ku))k+O(t^2)\,.
\label{eq:deformation}
\ea
Here $Q(u)\in\mathfrak{k}$ is given by
$$
Q(u) =
\left(
\begin{array}{cc}
0_{D-1} & \hat{u} \\
-\hat{u}^t &0 \\
\end{array}
\right)\,,
$$
and $\hat{u}\in\mathbb{R}^{D-1}$ is the truncation of $u$ to the first $D-1$ components. It follows that
$$
\varphi_{d\theta}(x)(ku\otimes kv) =
\left.\frac{d}{dt}\right|_{t=0}\varphi_{\theta}(x')(k' v)\,.
$$
That the right-hand side equals the covariant derivative of $\varphi_{\theta}$ applied to $ku\otimes kv$ can now be checked directly using~\eqref{eq:deformation} by comparing it with the standard definition of the covariant derivative.
\end{proof}

\section{$\mathrm{SO}(3)$ Haar integrals}\label{app:KHaar}
As in the main text, we will normalize the Haar measure on $K=\mathrm{SO}(3)$ as $\int_K dk = 1$. Elements of $K$ are $3\times 3$ matrices with components $k^{a}_{\phantom{a} b}$. Recall that for $w\in V_n$, we have $(k w)^{a_1\ldots a_n} = k^{a_1}_{\phantom{a_1}b_1}\ldots k^{a_n}_{\phantom{a_n}b_n}w^{b_1\ldots b_n}$, and that for $w_1,\,w_2\in V_n$, we defined $w_1\cdot w_2 := w_1^{a_1\ldots a_n}w_2^{a_1\ldots a_n}$. The objective of this appendix is to record two lemmas about Haar integrals of products of matrix elements.
\begin{lemma}\label{lem:KHaar2}
Let $w_{1},\,w_{2},\,\widetilde{w}_1,\widetilde{w}_2\in V_n$. Then
$$
\int\limits_{K}
(w_1\cdot k \widetilde{w}_1)\, (w_2\cdot k\widetilde{w}_2)dk
= \frac{(w_1\cdot w_2) (\widetilde{w}_1\cdot \widetilde{w}_2)}{2n+1}\,.
$$
\end{lemma}
\begin{proof}
Consider the map $\alpha: V_n\rightarrow V_n$ defined by
$$
t_1 \mapsto \int\limits_{K}
(w_1\cdot kt_1)\, k^{-1}w_2\,dk\,.
$$
This map is $K$-invariant, and so by Schur's lemma $\alpha = c\,\mathrm{id}$. The constant of proportionality follows by taking the trace $\mathrm{dim}(V_n)\,c = \mathrm{tr}(\alpha) = w_1\cdot w_2$, and the claim follows.
\end{proof}

\begin{lemma}\label{lem:KHaar3}
Let $w_{1},\,\widetilde{w}_1 \in V_{n_1}$,  $w_{2},\,\widetilde{w}_2 \in V_{n_2}$ and  $w_{3},\,\widetilde{w}_3 \in V_{n_3}$. Then
$$
\int\limits_{K}
(w_1\cdot k\widetilde{w}_1)\, (w_2\cdot k\widetilde{w}_2)\, (w_3\cdot k\widetilde{w}_3)dk
= \frac{T_{n_1,n_2,n_3}(w_1,w_2,w_3)\,T_{n_1,n_2,n_3}(\widetilde{w}_1,\widetilde{w}_2,\widetilde{w}_3)}{q(n_1,n_2,n_3)}\,,
$$
where $q(n_1,n_2,n_3)$ is given in~\eqref{eq:qDefinition} and $T_{n_1, n_2, n_3}$ is the $K$-invariant trilinear form defined in Definition~\ref{def:Trilinear}.
\end{lemma}
\begin{proof}
The trilinear maps
$$
(w_1,w_2,w_3)\mapsto \int\limits_{K}
(w_1\cdot k\widetilde{w}_1)\, (w_2\cdot k\widetilde{w}_2)\, (w_3\cdot k\widetilde{w}_3) \, dk
$$
and
$$
(\widetilde{w}_1,\widetilde{w}_2,\widetilde{w}_3)\mapsto \int\limits_{K}
(w_1\cdot k\widetilde{w}_1)\, (w_2\cdot k\widetilde{w}_2)\, (w_3\cdot k\widetilde{w}_3) \, dk
$$
are $K$-invariant and therefore the integral must be proportional to
$$
T_{n_1,n_2,n_3}(w_1,w_2,w_3)\,T_{n_1,n_2,n_3}(\widetilde{w}_1,\widetilde{w}_2,\widetilde{w}_3)\,.
$$
It remains to determine the constant of proportionality. To do that, we evalute the integral over $K$ in the statement of the lemma for a special choice of $w_{1,2,3}$ and $\widetilde{w}_{1,2,3}$.

It is more convenient to replace $\mathrm{SO}(3)$ with its double cover $\mathrm{SU}(2)$. The latter has a two-dimensional fundamental representation $U_1$. The complete list of irreducible representations consists of $U_{m}=\mathrm{Sym}^m(U_1)$ of dimension $m+1$ with $m\in\mathbb{Z}_{\geq 0}$. The elements of $U_m$ are completely symmetric tensors $t^{\alpha_{1}\ldots \alpha_{m}}$. We have $V_n\simeq U_{2n}$. $U_m$ admits an invariant bilinear form $\langle t_1,t_2\rangle_m := \epsilon_{\alpha_1\beta_1}\ldots \epsilon_{\alpha_m\beta_m} t_1^{\alpha_1\ldots\alpha_m}t_2^{\beta_1\ldots\beta_m}$, where $\epsilon_{11}=\epsilon_{22}=0$ and $\epsilon_{12}=-\epsilon_{21}=1$. The space of invariant trilinear forms $U_{m_1}\otimes U_{m_2}\otimes U_{m_3}\rightarrow\mathbb{C}$ is one-dimensional if $m_1$, $m_2$, $m_3$ satisfy the triangle inequalities and $m_1+m_2+m_3$ is even, and zero-dimensional otherwise. We will normalize the invariant form $\widehat{T}_{m_1,m_2,m_3}: U_{m_1}\otimes U_{m_2}\otimes U_{m_3}\rightarrow\mathbb{C}$ as follows:
$$
\widehat{T}_{m_1,m_2,m_3}(u_1^{m_1},u_2^{m_2},u_3^{m_3}) = 
\langle u_1,u_2\rangle_{1}^{\frac{m_1+m_2-m_3}{2}}
\langle u_1,u_3\rangle_{1}^{\frac{m_1+m_3-m_2}{2}}
\langle u_2,u_3\rangle_{1}^{\frac{m_2+m_3-m_1}{2}}\,.
$$
Here $u_1,\,u_2,\,u_3\in U_1$ and $u^m\in U_m$ stands for the factorized symmetric tensor with components $u^{\alpha_1}\ldots u^{\alpha_m}$.

With this preparation, the statement of Lemma~\ref{lem:KHaar3} is equivalent to
$$
\int\limits_{\mathrm{SU}(2)}
\langle t_1, k\widetilde{t}_1\rangle_{m_1}\,
\langle t_2, k\widetilde{t}_2\rangle_{m_2}\,
\langle t_3, k\widetilde{t}_3\rangle_{m_3}
dk
= \frac{\widehat{T}_{m_1,m_2,m_3}(t_1,t_2,t_3)\,\widehat{T}_{m_1,m_2,m_3}(\widetilde{t}_1,\widetilde{t}_2,\widetilde{t}_3)}{q(m_1/2,m_2/2,m_3/3)}
$$
for all $t_1,\,\widetilde{t}_1\in U_{m_1}$, $t_2,\,\widetilde{t}_2\in U_{m_2}$, $t_3,\,\widetilde{t}_3\in U_{m_3}$, where the Haar measure on $\mathrm{SU}(2)$ is normalized as $\int_{\mathrm{SU}(2)}1dk = 1$. To find the overall constant, we parametrize $k\in\mathrm{SU}(2)$ as
$$
k(\alpha,\beta,\theta) =
\begin{pmatrix}
\cos\theta e^{i\alpha} & \sin\theta e^{i\beta} \\
-\sin\theta e^{-i\beta}& \cos\theta e^{-i\alpha}
\end{pmatrix}\,,
$$
with $\alpha,\beta\in[0,2\pi)$ and $\theta\in[0,\pi/2]$. The measure then takes the form
$$
dk = \frac{\sin(2\theta)d\alpha d\beta d\theta}{(2\pi)^2}\,.
$$
Finally, we set $t_1 = u_1^{m_1}$, $\widetilde{t}_1 = \widetilde{u}_1^{m_1}$, $t_2 = u_2^{m_2}$, $\widetilde{t}_2 = \widetilde{u}_2^{m_2}$, $t_3 = u_3^{m_3}$, $\widetilde{t}_3 = \widetilde{u}_3^{m_3}$, where
$$
u_1 = \begin{pmatrix}0\\-1\end{pmatrix},\quad
u_2 = \begin{pmatrix}1\\0\end{pmatrix},\quad
u_3 = \begin{pmatrix}1\\-1\end{pmatrix},\quad
\widetilde{u}_1 = \begin{pmatrix}1\\0\end{pmatrix},\quad
\widetilde{u}_2 = \begin{pmatrix}0\\1\end{pmatrix},\quad
\widetilde{u}_3 = \begin{pmatrix}1\\1\end{pmatrix}\,.
$$
With this choice, the lemma becomes equivalent to the following integral identity:
$$
\begin{aligned}
\int\limits_{0}^{\frac{\pi}{2}}\int\limits_{0}^{2\pi}\int\limits_{0}^{2\pi}
e^{i(m_1-m_2)\alpha}(\cos\theta)^{m_1+m_2}
&\left[2 \cos \alpha  \cos \theta +2 i \sin \beta  \sin \theta \right]^{m_3}
\frac{\sin(2\theta)d\alpha d\beta d\theta}{(2\pi)^2} \\
&= \frac{1}{q(m_1/2,m_2/2,m_3/3)}\,.
\end{aligned}
$$
Since the integrand is a polynomial in $e^{\pm i\alpha}$, $e^{\pm i\beta}$, $e^{\pm i\theta}$, performing the integral is elementary after using the binomial formula on the last factor, which completes the proof.
\end{proof}
\begin{proof}[Proof of Lemma~\ref{lem:qDefinition}]
Consider the following integral, which produces an element of $V_{n_1}\otimes V_{n_2}\otimes V_{n_3}$:
$$
\begin{aligned}
\int\limits_{K}
&(kw_1)^{a_1\ldots a_{n_1}}\, (kw_2)^{b_1\ldots b_{n_2}}\, (kw_3)^{c_1\ldots c_{n_3}} dk\\
&=  \frac{T_{n_1,n_2,n_3}(w_1,w_2,w_3)}{q(n_1,n_2,n_3)}
((b_{n_1}\otimes b_{n_2}\otimes b_{n_3})^{-1} \circ T^*_{n_1,n_2,n_3}(1))
^{a_1\ldots a_{n_1},b_1\ldots b_{n_2},c_1\ldots c_{n_3}}\,,
\end{aligned}
$$
where the right-hand side follows from Lemma~\ref{lem:KHaar3}. Let us apply $T_{n_1,n_2,n_3}$ to both sides. Since $T_{n_1,n_2,n_3}$ is $K$-linear, the left-hand side equals $T_{n_1,n_2,n_3}(w_1,w_2,w_3)$. On the other hand, the right hand side equals
$$
\frac{T_{n_1,n_2,n_3}(w_1,w_2,w_3)}{q(n_1,n_2,n_3)} \,
T_{n_1,n_2,n_3}\circ (b_{n_1}\otimes b_{n_2}\otimes b_{n_3})^{-1} \circ T^*_{n_1,n_2,n_3}(1)\,,
$$
which completes the proof.
\end{proof}

\section{Computations of trilinear functionals}\label{app:trilinear}
\begin{lemma}\label{lem:integral1}
Suppose $J_1+J_2\leqslant J_3$ and let $n_1\geq J_1$, $n_2\geq J_2$ be such that $n_1+n_2=J_3$. Then
$$
\begin{aligned}
&\alpha_{n_1,n_2,J_3}(\Delta_1,J_1;\Delta_2,J_2;\Delta_3,J_3) = \\
&=
(-1)^{n_1+J_2+J_3}\frac{(\sqrt{2}i)^{J_3-J_1-J_2}(h_3-h_1-h_2+1)_{J_3-J_1-J_2}q(-h_1,-h_2,-h_3)}{(1-2h_1)(1-2h_2)(1-2h_3)(2J_3+1)}\,.
\end{aligned}
$$
\end{lemma}
\begin{proof}
Let
$$
f_1 = \kappa^{(n_1)}_{\Delta_1,J_1}(z_{+}^{n_1}),\quad
f_2 = \kappa^{(n_2)}_{\Delta_2,J_2}(z_{+}^{n_2}),\quad
f_3 = \kappa^{(J_3)}_{\Delta_3,J_3}(z_{-}^{J_3})\,,
$$
so that $\alpha_{n_1,n_2,J_3}(\Delta_1,J_1;\Delta_2,J_2;\Delta_3,J_3)=\mathcal{T}_{\Delta_1,J_1;\Delta_2,J_2;\Delta_3,J_3}(f_1,f_2,f_3)$. In order to evaluate the right-hand side, we proceed as follows. Note that
\be\label{eq:mapTau}
(f_1,f_2)\mapsto
\int\limits_{\mathbb{C}^2}\!\, f_1(x_1)f_2(x_2)
\tau_{\Delta_1,J_1;\Delta_2,J_2;\Delta_3,J_3}(x_1,x_2,x_3)dx_1dx_2
\ee
is a $G$-linear map $R_{\Delta_1,J_1}^{\infty}\times R_{\Delta_2,J_2}^{\infty}\rightarrow R_{2-\Delta_3,-J_3}^{\infty}$. Therefore, for the special choice of $f_1$, $f_2$ above,~\eqref{eq:mapTau} must be proportional to the vector
$$
\kappa_{2-\Delta_3,-J_3}^{(J_3)}(z_{+}^{J_3})(x_3) \propto (|x_3|^2+1)^{-\Delta_3-J_3}\,.
$$
The constant of proportionality can be found by setting $x_3 = 0$ inside the $x_1$, $x_2$ integrals in~\eqref{eq:mapTau}. The resulting integral over $x_1$, $x_2$ can be evaluated by changing variables as $x_{1}=1/y_{1}$, $x_{2}=1/y_{2}$ to bring it to the form
$$
\int\limits_{\mathbb{C}^2}\!\, \widetilde{f}_1(y_1)\widetilde{f}_2(y_2)
(y_2-y_1)^{-h_1-h_2+h_3}(\bar{y}_2-\bar{y}_1)^{-\hb_1-\hb_2+\hb_3}dy_1dy_2\,,
$$
at which point we insert the Fourier representation of $\widetilde{f}_1$, $\widetilde{f}_2$. Finally, the integral over $x_3$ is elementary.
\end{proof}

\begin{lemma}\label{lem:integral2}
Suppose $J_1+J_2\geq J_3$. Then
$$
\begin{aligned}
&\alpha_{J_1,J_2,J_1+J_2}(\Delta_1,J_1;\Delta_2,J_2;\Delta_3,J_3) = 
\frac{q(-h_1,-h_2,-h_3)}{(1-2h_1)(1-2h_2)(1-2h_3)}\frac{(J_1+J_2+J_3)!}{(2 J_1+2 J_2+1)!}
\times\\
&\qquad\times(-i \sqrt{2})^{J_1+J_2-J_3}(1-h_1+h_2-h_3)_{J_1+J_2-J_3} (-\hb_1+\hb_2+\hb_3)_{J_1+J_2-J_3}
\times
\\
&\qquad\times\pFq{3}{2}{J_3-J_1-J_2,h_1-h_2-h_3+1,\hb_1-\hb_2+\hb_3}{h_1-h_2-h_3-2 J_2+1,\hb_1-\hb_2+\hb_3-2J_1}{1}\,.
\end{aligned}
$$
\end{lemma}
\begin{proof}
We proceed analogously to the proof of Lemma~\ref{lem:integral1}, this time setting
$$
f_1 = \kappa^{(J_1)}_{\Delta_1,J_1}(z_{+}^{J_1}),\quad
f_2 = \kappa^{(J_2)}_{\Delta_2,J_2}(z_{+}^{J_2}),\quad
f_3 = \kappa^{(J_1+J_2)}_{\Delta_3,J_3}(z_{-}^{J_1+J_2})\,,
$$
so that $\alpha_{J_1,J_2,J_1+J_2}(\Delta_1,J_1;\Delta_2,J_2;\Delta_3,J_3)=\mathcal{T}_{\Delta_1,J_1;\Delta_2,J_2;\Delta_3,J_3}(f_1,f_2,f_3)$. For these choices of $f_1$, $f_2$,~\eqref{eq:mapTau} must be proportional to the vector
$$
\kappa_{2-\Delta_3,-J_3}^{(J_1+J_2)}(z_{+}^{J_1+J_2})(x_3) \propto \overline{x}_3^{J_1+J_2-J_3}(|x_3|^2+1)^{-J_1-J_2-\Delta_3}\,.
$$
The constant of proportionality can be found by expanding~\eqref{eq:mapTau} around $x_3=0$ to extract the coefficient of $\xb_3^{J_1+J_2-J_3}$. The resulting double integral over $x_1$, $x_2$ can again be evaluated by going to Fourier space and the final integral over $x_3$ is elementary.
\end{proof}

\begin{lemma}\label{lem:integral3}
Suppose $J_1$, $J_2$, $J_3$ satisfy the triangle inequalities. Then
$$
\alpha_{J_1,J_2,J_3}(\Delta_1,J_1;\Delta_2,J_2;\Delta_3,J_3) = 
\frac{q(-h_1,-h_2,-h_3)}{(1-2h_1)(1-2h_2)(1-2h_3)q(J_1,J_2,J_3)}
$$
\end{lemma}
\begin{proof}
We will relate $\alpha_{J_1,J_2,J_3}$ and $\alpha_{J_1,J_2,J_1+J_2}$ using the recursion relation of Lemma~\ref{lem:recursion} and apply Lemma~\ref{lem:integral2}. Set $\widetilde{\alpha}_{n_1,n_2,n_3} = q(n_1,n_2,n_3)\,\alpha_{n_1,n_2,n_3}$. In the special case $n_1=J_1$, $n_2=J_2$, the recursion relation reduces to
$$
\begin{aligned}
\widetilde{\alpha}_{J_1,J_2,n_3+1} &= 
\tfrac{i [n_3 (n_3+1) (\Delta_1-\Delta_2)-(\Delta_3-1) J_3 (J_1-J_2)]}{\sqrt{2} n_3 (n_3+1)}\widetilde{\alpha}_{J_1,J_2,n_3}+\\
&\quad+\tfrac{\left(J_1-J_2-n_3\right) \left(J_1-J_2+n_3\right) \left(n_3-J_3\right) \left(J_3+n_3\right) \left(-\Delta _3+n_3+1\right) \left(\Delta _3+n_3-1\right)}{2 n_3^2 \left(2 n_3-1\right) \left(2 n_3+1\right)}\widetilde{\alpha}_{J_1,J_2,n_3-1}\,.
\end{aligned}
$$
We can use this recursion to solve for all $\widetilde{\alpha}_{J_1,J_2,n_3}$ with $J_3\leqslant n_3\leqslant J_1+J_2$ in terms of $\widetilde{\alpha}_{J_1,J_2,J_3}$. By comparing the solution with Lemma~\ref{lem:integral2}, we arrive at the simple result for $\alpha_{J_1,J_2,J_3}$.
\end{proof}

\section{The conformal partial waves}\label{app:partialwave}
In this appendix, we give the explicit form of the conformal partial waves from Definition~\ref{def:CPW}. They can be written as
$$
\begin{aligned}
&\psi^{\Delta_1,J_1;\Delta_2,J_2;\Delta_3,J_3;\Delta_4,J_4}_{\Delta_5,J_5}(x_1,x_2,x_3,x_4) =\\
&\frac{1}{x_{12}^{h_1+h_2}x_{34}^{h_3+h_4}}\left(\frac{x_{24}}{x_{14}}\right)^{h_1-h_2}\left(\frac{x_{14}}{x_{13}}\right)^{h_3-h_4}
\times\frac{1}{\xb_{12}^{\hb_1+\hb_2}\xb_{34}^{\hb_3+\hb_4}}\left(\frac{\xb_{24}}{\xb_{14}}\right)^{\hb_1-\hb_2}\left(\frac{\xb_{14}}{\xb_{13}}\right)^{\hb_3-\hb_4}\frac{\psi(z,\bar{z})}{1-2h_5},
\end{aligned}
$$
where
$$
z = \frac{x_{12}x_{34}}{x_{13}x_{24}}\,,\quad
\bar{z} = \frac{\xb_{12}\xb_{34}}{\xb_{13}\xb_{24}},
$$
and
$$
\psi(z,\bar{z}) =
k^{h_1,h_2,h_3,h_4}_{h_5}(z)k^{\hb_1,\hb_2,\hb_3,\hb_4}_{\hb_5}(\bar{z})
+\tfrac{S(\Delta_3,J_3;\Delta_4,J_4;2-\Delta_5,-J_5)}{S(\Delta_1,J_1;\Delta_2,J_2|\Delta_5,J_5)}
k^{h_1,h_2,h_3,h_4}_{1-h_5}(z)k^{\hb_1,\hb_2,\hb_3,\hb_4}_{1-\hb_5}(\bar{z})\,.
$$
Here
$$
k^{h_1,h_2,h_3,h_4}_{h_5}(z) = z^{h_5}{}_2F_1(h_5-h_1+h_2,h_5+h_3-h_4;2h_5;z)\,.
$$

\section{Recursion relation for the triple products}\label{app:recursionProof}
The main purpose of this appendix is to supply a proof of the recursion relation for the triple products $\alpha_{n_1,n_2,n_3}$, recorded in Lemma~\ref{lem:recursion}. As an intermediate step, we will describe the action of the $\mathfrak{so}(1,3)$ Lie algebra on the complete set of $K$-finite vectors $\kappa^{(n)}_{\Delta,J}(w)$ in $R_{\Delta,J}$, which is the topic of the following subsection.

\subsection{Action of the Lie algebra on $K$-finite vectors}
Let us start by recalling the $K$-linear maps $T_{n_1,n_2,n_3}^{\vee}: V_{n_1}\otimes V_{n_2}\rightarrow V_{n_3}$, introduced just before Corollary~\ref{cor:opeKFinite}. We will need an explicit form of these maps for $n_1=1$. The following lemma easily follows from Definition~\ref{def:Trilinear}.
\begin{lemma}\label{lem:vectorOps}
Let $v\in V_1$, $w\in V_n$ for any $n\geq 1$. Then
$$
\begin{aligned}
&T^{\vee}_{1,n,n+1}(v,w)^{a_1\ldots a_{n+1}} = \tfrac{1}{n+1}\left[\sum\limits_{j=1}^{n+1}v^{a_j} w^{a_1\ldots\widehat{a}_j\ldots a_{n+1}}-
\tfrac{2}{2n+1}\sum\limits_{1\leqslant j<k\leqslant n+1}\delta^{a_j a_k}v^{b}w^{b a_1\ldots \widehat{a}_j\ldots \widehat{a}_k\ldots a_{n+1}}\right]\,,\\
&T^{\vee}_{1,n,n}(v,w)^{a_1\ldots a_{n}} = \tfrac{1}{\sqrt{2}\,n}\sum\limits_{j=1}^{n}\epsilon^{a_j b c}v^b w^{c a_1\ldots\widehat{a}_j\ldots a_{n}}\,,\\
&T^{\vee}_{1,n,n-1}(v,w)^{a_1\ldots a_{n-1}} = v^b w^{b a_1\ldots a_{n-1}}\,,
\end{aligned}
$$
where identical $\mathrm{SO}(3)$ indices are summed over, following the Einstein summation convention.
\end{lemma}

To describe the action of the Lie algebra elements $A_a$ on $\kappa^{(n)}_{\Delta,J}(w)$, recall that for $\widetilde{v}\in V_1$, we have $(A_a \widetilde{v})^c = \epsilon_{abc}\widetilde{v}^b$. The following lemma then follows directly from the second line of Lemma~\ref{lem:vectorOps}:
\begin{lemma}\label{lem:AAction}
Let $v\in V_1$ and let $A(v) = \sum_{a=1}^{3} v^a A_a$. For any $w\in V_n$, we have
$$
A(v)\kappa^{(n)}_{\Delta,J}(w) = \sqrt{2}n\, \kappa^{(n)}_{\Delta,J}(T^{\vee}_{1,n,n}(v,w))\,.
$$
\end{lemma}

While $A_a$ maps the image of $\kappa^{(n)}_{\Delta,J}$ into itself, the Lie algebra generators $B_a$ can change the value of $n$ by up to one unit. The main result of this subsection is the following explicit form of their action:
\begin{lemma}\label{lem:BActionKfinite}
Let $v\in V_1$, $w\in V_n$ and let $B(v) = \sum_{a=1}^{3}v^aB_a$. Then
\be\label{eq:BActionGeneral}
B(v) \kappa^{(n)}_{\Delta,J}(w) = \sum\limits_{m=n-1}^{n+1}
\beta_{n,m}(\Delta,J)\,\kappa^{(m)}_{\Delta,J}(T^{\vee}_{1,n,m}(v,w))\,,
\ee
where
$$
\beta_{n,n+1}(\Delta,J) = 1\,,\quad
\beta_{n,n}(\Delta,J) = -\tfrac{i \sqrt{2}(\Delta -1) J}{n+1}\,,\quad
\beta_{n,n-1}(\Delta,J) = \tfrac{(n^2-J^2)[(\Delta-1)^2-n^2]}{n (2 n+1)}\,.
$$
\end{lemma}
\begin{proof}
Consider the $K$-linear map $\mu_n:V_1\otimes V_n\rightarrow R_{\Delta,J}$ given by $\mu_n(v\otimes w) = B(v)\kappa^{(n)}_{\Delta,J}(w)$. The general form~\eqref{eq:BActionGeneral} follows from applying $\mu_n$ to the decomposition $V_1\otimes V_n = V_{n-1}\oplus V_{n}\oplus V_{n+1}$ and Schur's lemma.

It remains to determine the constants $\beta_{n,m}(\Delta,J)$ for $m\in\{n-1,n,n+1\}$. By setting $v = z_{+}$ and $w=z_{+}^{n}$, we immediately get $\beta_{n,n+1}(\Delta,J)=1$.

To find $\beta_{n,n}(\Delta,J)$, we will act on $\kappa^{(n)}_{\Delta,J}(w)$ with the Casimir $C_2' = \sum_{a=1}^{3}A_aB_a$. By combining~\eqref{eq:BActionGeneral} with Lemma~\ref{lem:AAction}, we find
$$
A(v_1)B(v_2)\kappa^{(n)}_{\Delta,J}(w) = 
\sqrt{2}\sum\limits_{m=n-1}^{n+1}m\,
\beta_{n,m}(\Delta,J)\,\kappa^{(m)}_{\Delta,J}(T^{\vee}_{1,m,m}(v_1,T^{\vee}_{1,n,m}(v_2,w)))\,.
$$
By taking this equation and replacing the two index tensor $v_1^av_2^b$ by the Kronecker delta $\delta^{ab}$, we get
$$
C_2'\,\kappa^{(n)}_{\Delta,J}(w) = -\frac{n+1}{\sqrt{2}}\beta_{n,n}(\Delta,J)\, \kappa^{(n)}_{\Delta,J}(w)\,.
$$
Since $C_2'$ acts on $R_{\Delta,J}$ by multiplication by the constant $i(\Delta-1)J$, we obtain
$$
\beta_{n,n}(\Delta,J) = -\frac{\sqrt{2}i(\Delta-1)J}{n+1}\,.
$$
Similarly, to find $\beta_{n,n-1}(\Delta,J)$, we will act on $\kappa^{(n)}_{\Delta,J}(w)$ with the Casimir $C_2 = \sum_{a=1}^{3}(A_aA_a-B_aB_a)$, which acts on $R_{\Delta,J}$ by multiplication by $-\Delta(\Delta-2)-J^2$. Note that $\sum_{a=1}^{3}A_aA_a$ is the quadratic Casimir of the $\mathfrak{so}(3)$ algebra, so that
$$
\sum_{a=1}^{3}A_aA_a \kappa^{(n)}_{\Delta,J}(w) = -n(n+1) \kappa^{(n)}_{\Delta,J}(w)\,.
$$
Thus it must be that
\be\label{eq:BBCasimir}
\sum_{a=1}^{3}B_aB_a\, \kappa^{(n)}_{\Delta,J}(w) = [\Delta(\Delta-2)+J^2-n(n+1)] \kappa^{(n)}_{\Delta,J}(w)\,.
\ee
Let us compare the right-hand side with the result of applying~\eqref{eq:BActionGeneral} twice in succession:
$$
B(v_1)B(v_2) \kappa^{(n)}_{\Delta,J}(w) = 
 \sum\limits_{m=n-1}^{n+1}\sum\limits_{m'=m-1}^{m+1}
\beta_{n,m}(\Delta,J)\beta_{m,m'}(\Delta,J)\,\kappa^{(m')}_{\Delta,J}(T^{\vee}_{1,m,m'}(v_1,T^{\vee}_{1,n,m}(v_2,w)))\,.
$$
Replacing $v_1^av_2^b\rightarrow \delta^{ab}$ and comparing with~\eqref{eq:BBCasimir} leads to
$$
\frac{2n+3}{2n+1}\beta_{n+1,n}(\Delta,J)-\frac{n+1}{2n}(\beta_{n,n}(\Delta,J))^2+\beta_{n,n-1}(\Delta,J) = \Delta(\Delta-2)+J^2-n(n+1)\,.
$$
Since $\beta_{n,n}(\Delta,J)$ is known, this equation yields a recurrence relation for $\beta_{n,n-1}(\Delta,J)$. Furthermore, since $V_{J-1}$ does not appear in $R_{\Delta,J}$, we have the initial condition $\beta_{J,J-1}(\Delta,J) = 0$. The unique solution of the recurrence with this initial condition takes the form
$$
\beta_{n,n-1}(\Delta,J) = \frac{(n^2-J^2)[(\Delta-1)^2-n^2]}{n (2 n+1)}\,,
$$
which completes the proof.
\end{proof}

\subsection{Recursion relation for triple products}
We are now ready to show that the triple product coefficients $\alpha_{n_1,n_2,n_3}$, defined by Proposition~\ref{prop:threePtKFinite}, satisfy the recursion relation stated in Lemma~\ref{lem:recursion}.
\begin{proof}[Proof of Lemma~\ref{lem:recursion}]
The proof will follow from an infinitesimal form of $G$-invariance of the three-point correlations. Let $\mathcal{T}:R^{\infty}_{\Delta_1,J_1}\times R^{\infty}_{\Delta_2,J_2}\times R^{\infty}_{\Delta_3,J_3}\rightarrow\mathbb{C}$ be an invariant functional. For any $n_1\geq J_1$, $n_2\geq J_2$, $n_3\geq J_3$ and $w_1\in V_{n_1}$, $w_2\in V_{n_2}$, $w_3\in V_{n_3}$, we have from $K$-invariance
\be\label{eq:tripleGeneralApp}
\mathcal{T}(\kappa^{(n_1)}_{\Delta_1,J_1}(w_1),\kappa^{(n_2)}_{\Delta_2,J_2}(w_2),\kappa^{(n_3)}_{\Delta_3,J_3}(w_3)) = \alpha_{n_1,n_2,n_3}\, T_{n_1,n_2,n_3}(w_1,w_2,w_3)\,,
\ee
where $T_{n_1,n_2,n_3}:V_{n_1}\otimes V_{n_2}\otimes V_{n_3}\rightarrow\mathbb{C}$ is the $K$-invariant functional from Definition~\ref{def:Trilinear}. $G$-invariance of $\mathcal{T}$ reads
$$
\mathcal{T}(g\cdot\kappa^{(n_1)}_{\Delta_1,J_1}(w_1),g\cdot\kappa^{(n_2)}_{\Delta_2,J_2}(w_2),g\cdot\kappa^{(n_3)}_{\Delta_3,J_3}(w_3)) = 
\mathcal{T}(\kappa^{(n_1)}_{\Delta_1,J_1}(w_1),\kappa^{(n_2)}_{\Delta_2,J_2}(w_2),\kappa^{(n_3)}_{\Delta_3,J_3}(w_3))\,,
$$
where $g\in G$ acts in representations $R_{\Delta_i,J_i}$ with $i=1,2,3$. Setting $g = \exp(t B(v))$, and expanding the above to the linear order in $t$ gives
$$
\begin{aligned}
&\mathcal{T}(B(v)\kappa^{(n_1)}_{\Delta_1,J_1}(w_1),\kappa^{(n_2)}_{\Delta_2,J_2}(w_2),\kappa^{(n_3)}_{\Delta_3,J_3}(w_3))  \\
&\qquad+\mathcal{T}(\kappa^{(n_1)}_{\Delta_1,J_1}(w_1),B(v)\kappa^{(n_2)}_{\Delta_2,J_2}(w_2),\kappa^{(n_3)}_{\Delta_3,J_3}(w_3))\\
&\qquad\qquad+\mathcal{T}(\kappa^{(n_1)}_{\Delta_1,J_1}(w_1),\kappa^{(n_2)}_{\Delta_2,J_2}(w_2),B(v)\kappa^{(n_3)}_{\Delta_3,J_3}(w_3)) = 0\,.
\end{aligned}
$$
Using Lemma~\ref{lem:BActionKfinite}, and expression~\eqref{eq:tripleGeneralApp}, this becomes
\be\label{eq:recurrenceDer}
\begin{aligned}
&\sum\limits_{m=n_1-1}^{n_1+1}\alpha_{m,n_2,n_3}\,
\beta_{n_1,m}(\Delta_1,J_1)\,
T^{\vee}_{1,n_1,m}(v,w_1)\cdot 
T^{\vee}_{n_2,n_3,m}(w_2,w_3)\\
&\qquad+
\sum\limits_{m=n_2-1}^{n_2+1}\alpha_{n_1,m,n_3}\,
\beta_{n_2,m}(\Delta_2,J_2)\,
T^{\vee}_{1,n_2,m}(v,w_2)\cdot T^{\vee}_{n_3,n_1,m}(w_3,w_1)\\
&\qquad\qquad+
\sum\limits_{m=n_3-1}^{n_3+1}\alpha_{n_1,n_2,m}\,
\beta_{n_3,m}(\Delta_3,J_3)\,
T^{\vee}_{1,n_3,m}(v,w_3)\cdot T^{\vee}_{n_1,n_2,m}(w_1,w_2)=0\,.
\end{aligned}
\ee
This equation must hold for all $v\in V_1$ and all $w_1\in V_{n_1}$, $w_2\in V_{n_2}$, $w_3\in V_{n_3}$. Each of the nine terms on the right-hand side is an invariant $K$-linear map $V_1\otimes V_{n_1}\otimes V_{n_2}\otimes V_{n_3}\rightarrow \mathbb{C}$. The space of such maps is at most three-dimensional. The linear dependences between the nine maps of the form $T^{\vee}_{1,\bullet,m}(v,\bullet)\cdot T^{\vee}_{\bullet,\bullet,m}(\bullet,\bullet)$ can be worked out using Lemma~\ref{lem:quadrilinear}. In practice, we set $w_1 = z_1^{n_1}$, $w_2 = z_2^{n_2}$, $w_3 = z_2^{n_3}$ and $v = z_4$ with null vectors $z_i$, and substitute into~\eqref{eq:recurrenceDer} from Lemma~\ref{lem:quadrilinear}. This substitution turns the left-hand side of~\eqref{eq:recurrenceDer} into a quadratic polynomial in the ratio $r(z_2,z_3,z_4,z_1)$ (up to an overall $\alpha$-independent prefactor). The three coefficients of this polynomial give three linearly independent equations for $\alpha_{n_1,n_2,n_3}$ and its shifts. These linear equations are precisely the recurrence relations of Lemma~\ref{lem:recursion}.
\end{proof}

\section{Compact tetrahedral orbifolds}
\label{app:tetrahedra}

In this appendix we discuss compact orientable hyperbolic tetrahedral orbifolds and their quotients. Useful references for this subject include~\cite{Thurston-notes, Brunner1985, MacLachlan-Reid, Elstrodt1997}. 

Consider the compact hyperbolic tetrahedron $T= T(\lambda_1, \lambda_2, \lambda_3; \mu_1, \mu_2, \mu_3)$ with dihedral angles $\pi/\lambda_i$ and $\pi/\mu_i$, as represented schematically in Fig.~\ref{fig:tetrahedron}. There are nine such tetrahedra. Let $\Gamma(T)$ be the group generated by reflections in the faces of the tetrahedron $T$ and let $\Gamma^+(T)$ be the index-2 subgroup consisting of the orientation-preserving isometries. $\Gamma^+(T)$ is called a tetrahedral group and has the following presentation:
$$
\Gamma^+\left(T\right) = \langle a, b, c \, | \, a^{\lambda_1}=b^{\lambda_2} = c^{\lambda_3} = (b c)^{\mu_1} = (ca)^{\mu_2} = (ab)^{\mu_3} =1 \rangle ,
$$
where the generators correspond to rotations around the edges as indicated in Fig.~\ref{fig:tetrahedron}. 
\begin{figure}
\centering
{\resizebox{7cm}{!}{%
\begin{tikzpicture}[line join = round, line cap = round]
\pgfmathsetmacro{\factor}{2};
\coordinate [label=above:A] (A) at (0.25*\factor,2*\factor,1*\factor);
\coordinate [label=left:B] (B) at (-2*\factor,-1*\factor,-1*\factor);
\coordinate [label=below:C] (C) at (0,-2*\factor,1*\factor);
\coordinate [label=right:D] (D) at (2*\factor,-1*\factor,-1*\factor);
\coordinate [label=below left:${a, \, l_1, \, \lambda_1}$] (BC) at (-1*\factor,-3/2*\factor,0);
\coordinate [label=above left:${b, \, l_2, \, \lambda_2}$] (AC) at (.125*\factor, 0,1*\factor);
\coordinate [label=left:${c, \, l_3, \, \lambda_3}$] (BC) at (-.875*\factor,1/2*\factor,0);
\coordinate [label=above right:${bc, \, l_4, \, \mu_1}$] (DA) at (1.125*\factor,1/2*\factor,0);
\coordinate [label=above:${ca, \, l_5,\, \mu_2}$] (DB) at (0,-1*\factor,-1*\factor);
\coordinate [label=below right:${a b, \, l_6,\, \mu_3}$] (DC) at (1*\factor,-3/2*\factor,0);

\foreach \i in {A,B,C,D}
\draw[thick, opacity=1] (A)--(C)--(B)--cycle;
\draw[dashed, opacity=1] (B)--(D);
\draw[thick, opacity=1] (A)--(D)--(C);
\draw[thick, opacity=1] (B)--(C)--(D);
\end{tikzpicture}
}}
\caption{Schematic representation of a compact hyperbolic tetrahedron with edge lengths $l_j$, $j=1, \dots, 6$, and dihedral angles $\pi/\lambda_i$ and $\pi/\mu_i$, $i=1,2,3$. The rotation generators for each edge are also shown.}
\label{fig:tetrahedron}
\end{figure}
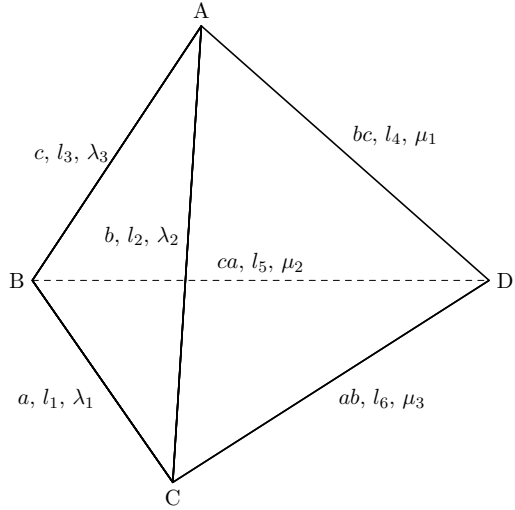
The fundamental domain of the orbifold $\Gamma^+(T) \backslash \mathbb{H}^3$ consists of two copies of the tetrahedron. The  nine compact tetrahedra $T_i$, $i=1, \dots 9$, and the approximate volumes of the corresponding orientable orbifolds are given in Table~\ref{tab:tetrahedral-data}.  The tetrahedral group $\Gamma^+(T_i)$ is arithmetic for $i \neq 8$. Schematic representations of each tetrahedron including approximate edge lengths and the isotropy groups of the vertices are shown in Fig.~\ref{fig:all-tetrahedra}. In general, the isotropy group of a vertex at which edges of orders $(p,q,r)$ meet is a finite spherical symmetry group.
\begin{table}[H]
\centering
  \begin{tabular}{  |c |c |}
  \thickhline
   $(p,q,r)$ & Isotropy group   \\ \thickhline
   $(2,2,n)$ & $D_{2n}, |D_{2n}| = 2n$ (dihedral symmetry)  \\ \hline
  $(2,3,3)$ & $ A_4 $ (tetrahedral symmetry)  \\ \hline
    $(2,3,4)$ & $ S_4 $ (octahedral symmetry)  \\ \hline
      $(2,3,5)$ & $ A_5 $ (icosahedral symmetry)  \\ 
 \thickhline
\end{tabular}
\end{table}

\begin{table}
\centering
  \begin{tabular}{  | c |c | c| }
  \thickhline
   Tetrahedron & $(\lambda_1, \lambda_2, \lambda_3; \mu_1, \mu_2, \mu_3)$ & $V\left(\Gamma^+(T_i) \backslash \mathbb{H}^3\right)$   \\ \thickhline
   $T_1$ & $(2, 2, 3; 3, 5, 2)$ & $0.1435$  \\ \hline
   $T_2$ & $(2,2,3;2,5,3)$ & $0.0781$  \\ \hline
   $T_3$ & $(2,2,4;2,3,5)$ & $0.0718$  \\ \hline
   $T_4$ & $(2,2,5;2,3,5)$ & $0.1867$  \\ \hline
   $T_5$ & $(2,3,3;2,3,4)$ & $0.1715$  \\ \hline
   $T_6$ & $(2,3,4; 2,3,4)$ & $0.4445$  \\ \hline
   $T_7$ & $(2,3,3;2,3,5)$ & $0.4106$  \\ \hline
   $T_8$ & $(2,3,4;2,3,5)$ & $0.7173$  \\ \hline
   $T_9$ & $(2,3,5;2,3,5)$ & $1.0043$  \\ \hline
 \thickhline
\end{tabular}
\caption{Compact orientable hyperbolic tetrahedral orbifolds.}
\label{tab:tetrahedral-data}
\end{table}

To evaluate the geometric side of the trace formula, we need the length spectrum of each tetrahedral orbifold, as well as information about the elliptic and bad hyperbolic conjugacy classes.  We can use \texttt{SnapPy} to find the primitive length spectrum up to some cutoff by inputting explicit matrix generators for the group to construct a Dirichlet domain~\cite{SnapPy, Hodgson-Weeks1994}.\footnote{For arithmetic orbifolds, one can also use the approach described in~\cite{Lin-Lipnowski2020}.} Explicit matrix generators for all tetrahedral groups are given in an unpublished note by Lakeland~\cite{Lakeland-note}. We also need data on the elliptic conjugacy classes and the bad hyperbolic elements. We summarize this data for the tetrahedral orbifolds, as well as some quotients, in tables below; the cases $T_2$ and $T_4$ can be found in~\cite{Lin-Lipnowski2020} and $T_8$ can be found in~\cite{Aurich-Marklof}.  To obtain and check this data we used a combination of geometry and group theory, following~\cite{Lin-Lipnowski2020}. The covolumes of the centralizers  and lengths are written in terms of the edge lengths $l_j$, which can be computed explicitly from the dihedral angles (see, e.g., Eq.~(3.3) of~\cite{Murakami2005}). 

\begin{figure}
\begin{tabular}{c  c c}
\subfloat[$T_1$]{ {\resizebox{5cm}{!}{%
\begin{tikzpicture}[line join = round, line cap = round]
\pgfmathsetmacro{\factor}{1.5};
\coordinate [label=right:$A_5$] (A) at (2*\factor,-1*\factor,-1*\factor);
\coordinate [label=left:$A_5$] (B) at (-2*\factor,-1*\factor,-1*\factor);
\coordinate [label=above:$A_4$] (C) at (0.25*\factor,2*\factor,1*\factor);
\coordinate [label=below:$D_4$] (D) at (0,-2*\factor,1*\factor);
\coordinate [label=below right:${5, \, l_5 \approx1.617}$] (AB) at (0,-1*\factor,-1*\factor);
\coordinate [label=below right:${2, \, l_6 \approx 1.061}$] (AD) at (1*\factor,-3/2*\factor,0);
\coordinate [label=below left:${2, \, l_1 \approx 1.061}$] (BD) at (-1*\factor,-3/2*\factor,0);
\coordinate [label=above right:${3, \, l_4 \approx 1.226}$] (AC) at (1.125*\factor,1/2*\factor,0);
\coordinate [label=left:${3, \, l_3 \approx 1.226}$] (BC) at (-.875*\factor,1/2*\factor,0);
\coordinate [label=above right:${2, \, l_2 \approx 0.5306}$] (CD) at (.125*\factor, 0,1*\factor);

\foreach \i in {A,B,C,D}
\draw[thick, opacity=1] (A)--(D)--(B);
\draw[dashed, opacity=1] (A)--(B);
\draw[thick, opacity=1] (A) --(D)--(C)--cycle;
\draw[thick, opacity=1] (B)--(D)--(C)--cycle;
\end{tikzpicture}
}}
} &
\subfloat[$T_2$]{  {\resizebox{5cm}{!}{%
\begin{tikzpicture}[line join = round, line cap = round]
\pgfmathsetmacro{\factor}{1.5};
\coordinate [label=right:$A_5$] (A) at (2*\factor,-1*\factor,-1*\factor);
\coordinate [label=left:$A_5$] (B) at (-2*\factor,-1*\factor,-1*\factor);
\coordinate [label=above:$D_6$] (C) at (0.25*\factor,2*\factor,1*\factor);
\coordinate [label=below:$D_6$] (D) at (0,-2*\factor,1*\factor);
\coordinate [label=below right:${5, \, l_5 \approx1.383}$] (AB) at (0,-1*\factor,-1*\factor);
\coordinate [label=below right:${3, \, l_6 \approx .8683}$] (AD) at (1*\factor,-3/2*\factor,0);
\coordinate [label=below left:${2, \, l_1 \approx .9727}$] (BD) at (-1*\factor,-3/2*\factor,0);
\coordinate [label=above right:${2, \, l_4 \approx .9727}$] (AC) at (1.125*\factor,1/2*\factor,0);
\coordinate [label=left:${3, \, l_3 \approx .8683}$] (BC) at (-.8683*\factor,1/2*\factor,0);
\coordinate [label=above right:${2, \, l_2 \approx .3942}$] (CD) at (.125*\factor, 0,1*\factor);

\foreach \i in {A,B,C,D}
\draw[thick, opacity=1] (A)--(D)--(B);
\draw[dashed, opacity=1] (A)--(B);
\draw[thick, opacity=1] (A) --(D)--(C)--cycle;
\draw[thick, opacity=1] (B)--(D)--(C)--cycle;
\end{tikzpicture}
}}
 } &
\subfloat[$T_3$]{  {\resizebox{5cm}{!}{%
\begin{tikzpicture}[line join = round, line cap = round]
\pgfmathsetmacro{\factor}{1.5};
\coordinate [label=below:$D_{10}$] (A) at (0,-2*\factor,1*\factor);
\coordinate [label=left:$S_4$] (B) at (-2*\factor,-1*\factor,-1*\factor);
\coordinate [label=above:$D_8$] (C) at (0.25*\factor,2*\factor,1*\factor);
\coordinate [label=right:$A_5$] (D) at (2*\factor,-1*\factor,-1*\factor);
\coordinate [label=below left:${2, \, l_1 \approx .8425}$] (BA) at (-1*\factor,-3/2*\factor,0);
\coordinate [label=above right:${2, \, l_2 \approx .6269}$] (CA) at (.125*\factor, 0,1*\factor);
\coordinate [label=left:${4, \, l_3 \approx .5306}$] (BC) at (-.875*\factor,1/2*\factor,0);
\coordinate [label=above right:${2, \, l_4 \approx 1.061}$] (DC) at (1.125*\factor,1/2*\factor,0);
\coordinate [label=below right:${3, \, l_5 \approx 1.226}$] (DB) at (0,-1*\factor,-1*\factor);
\coordinate [label=below right:${5, \, l_6 \approx .8085}$] (AD) at (1*\factor,-3/2*\factor,0);

\foreach \i in {A,B,C,D}
\draw[thick, opacity=1] (D)--(A)--(B);
\draw[dashed, opacity=1] (D)--(B);
\draw[thick, opacity=1] (D) --(A)--(C)--cycle;
\draw[thick, opacity=1] (B)--(A)--(C)--cycle;
\end{tikzpicture}
}}
}  \\
\subfloat[$T_4$]{ {\resizebox{5cm}{!}{%
\begin{tikzpicture}[line join = round, line cap = round]
\pgfmathsetmacro{\factor}{1.5};
\coordinate [label=below:$D_{10}$] (A) at (0,-2*\factor,1*\factor);
\coordinate [label=left:$A_5$] (B) at (-2*\factor,-1*\factor,-1*\factor);
\coordinate [label=above:$D_{10}$] (C) at (0.25*\factor,2*\factor,1*\factor);
\coordinate [label=right:$A_5$] (D) at (2*\factor,-1*\factor,-1*\factor);
\coordinate [label=below left:${2, \, l_1 \approx 1.439}$] (BA) at (-1*\factor,-3/2*\factor,0);
\coordinate [label=above right:${2, \, l_2 \approx .9136}$] (CA) at (.125*\factor, 0,1*\factor);
\coordinate [label=left:${5, \, l_3 \approx .9964}$] (BC) at (-.875*\factor,1/2*\factor,0);
\coordinate [label=above right:${2, \, l_4 \approx 1.439}$] (DC) at (1.125*\factor,1/2*\factor,0);
\coordinate [label=below right:${3, \, l_5 \approx1.903}$] (DB) at (0,-1*\factor,-1*\factor);
\coordinate [label=below right:${5, \, l_6 \approx .9964}$] (AD) at (1*\factor,-3/2*\factor,0);

\foreach \i in {A,B,C,D}
\draw[thick, opacity=1] (D)--(A)--(B);
\draw[dashed, opacity=1] (D)--(B);
\draw[thick, opacity=1] (D) --(A)--(C)--cycle;
\draw[thick, opacity=1] (B)--(A)--(C)--cycle;
\end{tikzpicture}
}}
 } &
\subfloat[$T_5$]{  {\resizebox{5cm}{!}{%
\begin{tikzpicture}[line join = round, line cap = round]
\pgfmathsetmacro{\factor}{1.5};
\coordinate [label=below:$S_{4}$] (A) at (0,-2*\factor,1*\factor);
\coordinate [label=left:$A_4$] (B) at (-2*\factor,-1*\factor,-1*\factor);
\coordinate [label=above:$A_{4}$] (C) at (0.25*\factor,2*\factor,1*\factor);
\coordinate [label=right:$S_4$] (D) at (2*\factor,-1*\factor,-1*\factor);
\coordinate [label=below left:${2, \, l_1 \approx 1.128}$] (BA) at (-1*\factor,-3/2*\factor,0);
\coordinate [label=above right:${3, \, l_2 \approx 1.015}$] (CA) at (.125*\factor, 0,1*\factor);
\coordinate [label=left:${3, \, l_3 \approx 0.7691}$] (BC) at (-.875*\factor,1/2*\factor,0);
\coordinate [label=above right:${2, \, l_4 \approx 1.128}$] (DC) at (1.125*\factor,1/2*\factor,0);
\coordinate [label=below right:${3, \, l_5 \approx 1.015}$] (DB) at (0,-1*\factor,-1*\factor);
\coordinate [label=below right:${4, \, l_6 \approx 1.128}$] (AD) at (1*\factor,-3/2*\factor,0);

\foreach \i in {A,B,C,D}
\draw[thick, opacity=1] (D)--(A)--(B);
\draw[dashed, opacity=1] (D)--(B);
\draw[thick, opacity=1] (D) --(A)--(C)--cycle;
\draw[thick, opacity=1] (B)--(A)--(C)--cycle;
\end{tikzpicture}
}}
} &
\subfloat[$T_6$]{ {\resizebox{5cm}{!}{%
\begin{tikzpicture}[line join = round, line cap = round]
\pgfmathsetmacro{\factor}{1.5};
\coordinate [label=below:$S_{4}$] (A) at (0,-2*\factor,1*\factor);
\coordinate [label=left:$S_4$] (B) at (-2*\factor,-1*\factor,-1*\factor);
\coordinate [label=above:$S_{4}$] (C) at (0.25*\factor,2*\factor,1*\factor);
\coordinate [label=right:$S_4$] (D) at (2*\factor,-1*\factor,-1*\factor);
\coordinate [label=below left:${2, \, l_1 \approx 1.700}$] (BA) at (-1*\factor,-3/2*\factor,0);
\coordinate [label=above right:${3, \, l_2 \approx 1.567}$] (CA) at (.125*\factor, 0,1*\factor);
\coordinate [label=left:${4, \, l_3 \approx 1.384}$] (BC) at (-.875*\factor,1/2*\factor,0);
\coordinate [label=above right:${2, \, l_4 \approx 1.700}$] (DC) at (1.125*\factor,1/2*\factor,0);
\coordinate [label=below right:${3, \, l_5 \approx 1.567}$] (DB) at (0,-1*\factor,-1*\factor);
\coordinate [label=below right:${4, \, l_6 \approx 1.384}$] (AD) at (1*\factor,-3/2*\factor,0);

\foreach \i in {A,B,C,D}
\draw[thick, opacity=1] (D)--(A)--(B);
\draw[dashed, opacity=1] (D)--(B);
\draw[thick, opacity=1] (D) --(A)--(C)--cycle;
\draw[thick, opacity=1] (B)--(A)--(C)--cycle;
\end{tikzpicture}
}}
 } \\
\subfloat[$T_7$]{  {\resizebox{5cm}{!}{%
\begin{tikzpicture}[line join = round, line cap = round]
\pgfmathsetmacro{\factor}{1.5};
\coordinate [label=below:$A_{5}$] (A) at (0,-2*\factor,1*\factor);
\coordinate [label=left:$A_4$] (B) at (-2*\factor,-1*\factor,-1*\factor);
\coordinate [label=above:$A_{4}$] (C) at (0.25*\factor,2*\factor,1*\factor);
\coordinate [label=right:$A_5$] (D) at (2*\factor,-1*\factor,-1*\factor);
\coordinate [label=below left:${2, \, l_1 \approx 1.761}$] (BA) at (-1*\factor,-3/2*\factor,0);
\coordinate [label=above right:${3, \, l_2 \approx 1.627}$] (CA) at (.125*\factor, 0,1*\factor);
\coordinate [label=left:${3, \, l_3 \approx 0.9291}$] (BC) at (-.875*\factor,1/2*\factor,0);
\coordinate [label=above right:${2, \, l_4 \approx 1.761}$] (DC) at (1.125*\factor,1/2*\factor,0);
\coordinate [label=below right:${3, \, l_5 \approx 1.627}$] (DB) at (0,-1*\factor,-1*\factor);
\coordinate [label=below right:${5, \, l_6 \approx 2.044}$] (AD) at (1*\factor,-3/2*\factor,0);

\foreach \i in {A,B,C,D}
\draw[thick, opacity=1] (D)--(A)--(B);
\draw[dashed, opacity=1] (D)--(B);
\draw[thick, opacity=1] (D) --(A)--(C)--cycle;
\draw[thick, opacity=1] (B)--(A)--(C)--cycle;
\end{tikzpicture}
}}
} &
\subfloat[$T_8$]{ {\resizebox{5cm}{!}{%
\begin{tikzpicture}[line join = round, line cap = round]
\pgfmathsetmacro{\factor}{1.5};
\coordinate [label=below:$A_{5}$] (A) at (0,-2*\factor,1*\factor);
\coordinate [label=left:$S_4$] (B) at (-2*\factor,-1*\factor,-1*\factor);
\coordinate [label=above:$S_{4}$] (C) at (0.25*\factor,2*\factor,1*\factor);
\coordinate [label=right:$A_5$] (D) at (2*\factor,-1*\factor,-1*\factor);
\coordinate [label=below left:${2, \, l_1 \approx 2.273}$] (BA) at (-1*\factor,-3/2*\factor,0);
\coordinate [label=above right:${3, \, l_2 \approx 2.133}$] (CA) at (.125*\factor, 0,1*\factor);
\coordinate [label=left:${4, \, l_3 \approx 1.487}$] (BC) at (-.875*\factor,1/2*\factor,0);
\coordinate [label=above right:${2, \, l_4 \approx 2.273}$] (DC) at (1.125*\factor,1/2*\factor,0);
\coordinate [label=below right:${3, \, l_5 \approx 2.133}$] (DB) at (0,-1*\factor,-1*\factor);
\coordinate [label=below right:${5, \, l_6 \approx 2.224}$] (AD) at (1*\factor,-3/2*\factor,0);

\foreach \i in {A,B,C,D}
\draw[thick, opacity=1] (D)--(A)--(B);
\draw[dashed, opacity=1] (D)--(B);
\draw[thick, opacity=1] (D) --(A)--(C)--cycle;
\draw[thick, opacity=1] (B)--(A)--(C)--cycle;
\end{tikzpicture}
}}
 } &
\subfloat[$T_9$]{  {\resizebox{5cm}{!}{%
\begin{tikzpicture}[line join = round, line cap = round]
\pgfmathsetmacro{\factor}{1.5};
\coordinate [label=below:$A_{5}$] (A) at (0,-2*\factor,1*\factor);
\coordinate [label=left:$A_5$] (B) at (-2*\factor,-1*\factor,-1*\factor);
\coordinate [label=above:$A_{5}$] (C) at (0.25*\factor,2*\factor,1*\factor);
\coordinate [label=right:$A_5$] (D) at (2*\factor,-1*\factor,-1*\factor);
\coordinate [label=below left:${2, \, l_1 \approx 2.826}$] (BA) at (-1*\factor,-3/2*\factor,0);
\coordinate [label=above right:${3, \, l_2 \approx 2.684}$] (CA) at (.125*\factor, 0,1*\factor);
\coordinate [label=left:${5, \, l_3 \approx 2.302}$] (BC) at (-.875*\factor,1/2*\factor,0);
\coordinate [label=above right:${2, \, l_4 \approx 2.826}$] (DC) at (1.125*\factor,1/2*\factor,0);
\coordinate [label=below right:${3, \, l_5 \approx 2.684}$] (DB) at (0,-1*\factor,-1*\factor);
\coordinate [label=below right:${5, \, l_6 \approx 2.302}$] (AD) at (1*\factor,-3/2*\factor,0);

\foreach \i in {A,B,C,D}
\draw[thick, opacity=1] (D)--(A)--(B);
\draw[dashed, opacity=1] (D)--(B);
\draw[thick, opacity=1] (D) --(A)--(C)--cycle;
\draw[thick, opacity=1] (B)--(A)--(C)--cycle;
\end{tikzpicture}
}}
} 
\end{tabular}
\caption{Schematic representations of the tetrahedra $T_1, \dots, T_9$ indicating their approximate edge lengths and the isotropy groups of their vertices. The angles and relative lengths are not intended to be accurate.}
\label{fig:all-tetrahedra}
\end{figure}
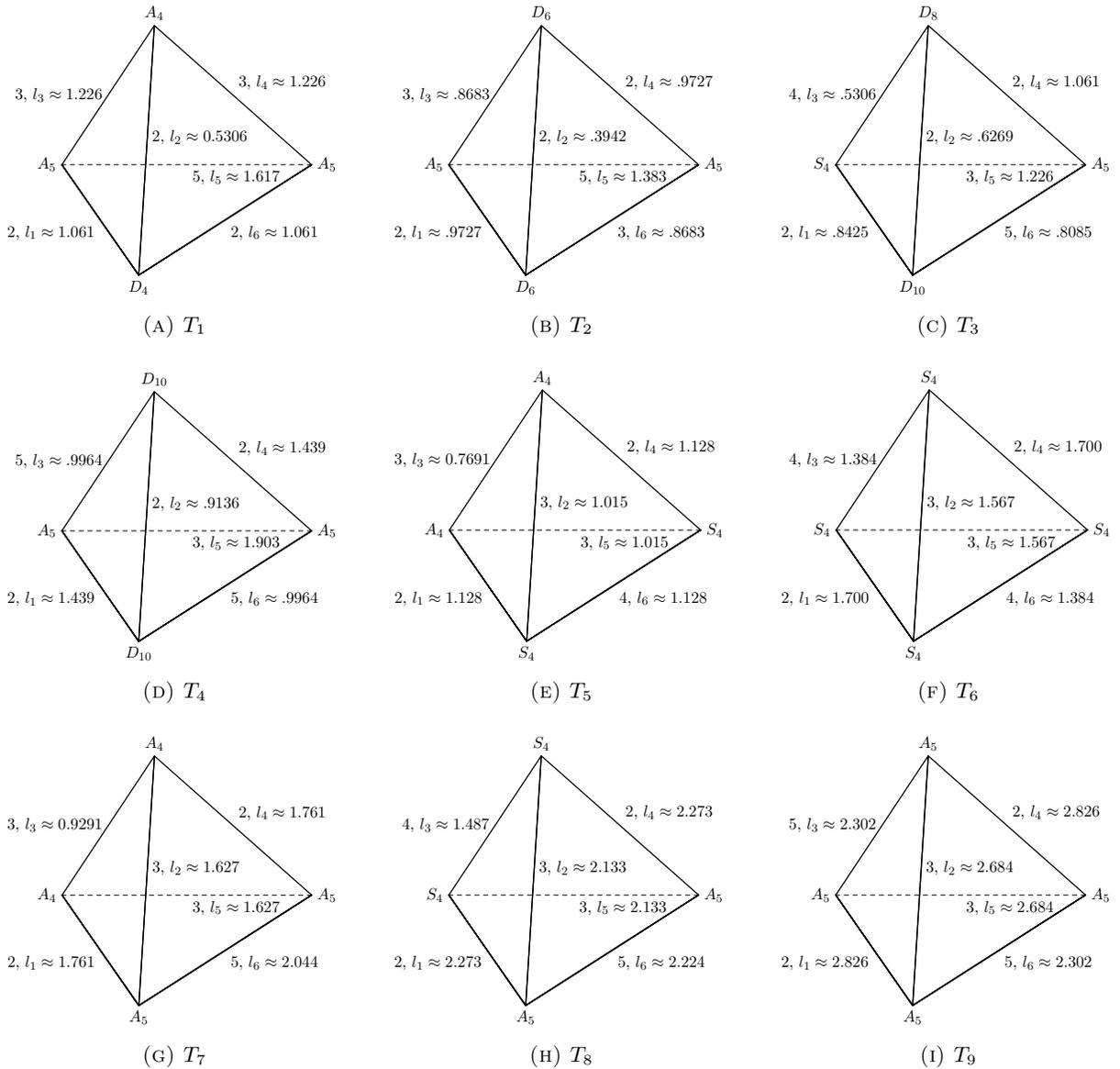

\subsection{$T_1$}
A schematic representation of the tetrahedron $T_1$ is shown in Fig.~\ref{fig:all-tetrahedra}. The exact edge lengths of $T_1$ are
$$
\begin{aligned}
l_1& =l_6=2l_2=\cosh^{-1} \left[ \frac{1}{2}(1+ \sqrt{5})  \right],\\ 
l_3& = l_4= \cosh^{-1} \left[  \frac{1+\sqrt{5}}{2\sqrt{3-\sqrt{5}}}  \right], \\ 
l_5& = \cosh^{-1} \left[  \frac{1}{2} (3+\sqrt{5})\right] .
\end{aligned}
$$

In Table~\ref{tab:T1-data} we list data for the elliptic and bad hyperbolic conjugacy classes of $\Gamma^+(T_1)$.  Each row in the table corresponds to a distinct conjugacy class of elliptic elements (they share an axis if the rows are not separated by a horizontal line) and each elliptic conjugacy class is represented. The first column gives the order of the elliptic element; the second column gives the rotation angle; the third column gives the covolume of the centralizer,   ${\rm vol}(\Gamma_{\gamma_e} \backslash G_{\gamma_e})$; and the fourth column gives, for each primitive elliptic element $\gamma_e$, the complex length of a hyperbolic element $\gamma_h$ of minimal length that commutes with $\gamma_e$. The geodesics in the last column of this table, and in the tables that follow, are topologically mirrored intervals unless otherwise stated.
\begin{table}
\centering
  \begin{tabular}{  | c |c | c| c |}
  \thickhline
   Order & Rotation angle & Covolume  & Complex length  \\ \thickhline
   2 & $\pi$ & $l_1$  & $2 l_1$ \\ \hline
   2 & $\pi$ & $l_2$   &  $2l_2+ \frac{i\pi }{2} $\\ \hline
2 & $\pi$ & $l_6$   & $2 l_6 $\\ \hline
3 & $2 \pi/3$ & $2(l_3+l_4)/3$ & $2 (l_3+l_4) $ \\  \hline
5 & $2 \pi/5$ & $2l_5/5$  & $2l_5 $  \\ 
5 & $4 \pi/5$ & $2l_5/5$   & --- \\ 
 \thickhline
\end{tabular}
\caption{Elliptic conjugacy classes of  $\Gamma^+(T_1)$.}
\label{tab:T1-data}
\end{table}

 Let us explain how to find the elliptic conjugacy classes for this example:
 \begin{enumerate}
 \item There is one edge in $T_1$ with dihedral angle $\pi/5$, so there can be at most four order-5 elliptic elements, corresponding to the nontrivial rotations around this edge. This edge joins vertices with isotropy groups $A_5$. Since every order-5 element in $A_5$ is conjugate to its inverse, there are two order-5 elliptic conjugacy classes in $\Gamma^+(T_1)$, corresponding to rotations by $2 \pi/5$ and $4\pi /5$. 
  \item There are two edges in $T_1$ with dihedral angle $\pi/3$, joining vertices with isotropy groups $A_4$ and $A_5$, as shown in Fig.~\ref{fig:all-tetrahedra}. In $A_4$ there are two conjugacy classes of order-$3$ elements, with each order-3 element in a separate conjugacy class to its inverse, whereas  in $A_5$ every order-3 element is conjugate to its inverse. There is therefore one order-3 elliptic conjugacy class in $\Gamma^+(T_1)$. 
\item The three order-2 elements of $D_4 = \mathbb{Z}_2 \times \mathbb{Z}_2$ are in separate conjugacy classes and commute. They can be realized by the following $\rm{SL}_2(\mathbb{C})$ matrices~\cite{Lakeland-note}:
$$
\begin{pmatrix}
0 & i \\i & 0
\end{pmatrix}, \quad 
\begin{pmatrix}
i & 0 \\ 0 & -i
\end{pmatrix}, \quad
\begin{pmatrix}
0 & 1 \\ -1 & 0
\end{pmatrix},
$$
corresponding to the elements $a$, $b$, and $ab$, respectively. The centralizers of these elements in $\Gamma^+(T_1)$ contain hyperbolic elements of length $2l_1$, $2l_2$, and $2l_6$, respectively, which shows that they are in separate conjugacy classes.
\end{enumerate}

The tetrahedron $T_1$ consists of two copies of $T_3$~\cite{MaclachlanReid1989}. We can extend $\Gamma^+(T_1)$ by adding an order-2 rotation $d$ with the relations $d bd = c^{-1} bc$, $d ac d = (ac)^{-1}$,  and $dcd=(bc)^{-1}$. This extended group is isomorphic to the tetrahedral group $\Gamma^+(T_3)$, which is described below.

\subsection{$T_2$}

Table~\ref{tab:T2-data} contains information about the elliptic conjugacy classes of the tetrahedral group $\Gamma^+(T_2)$. The Selberg trace formula for coexact 1-forms was applied to this example in~\cite{Lin-Lipnowski2020}.
\begin{table}
\centering
  \begin{tabular}{  | c |c | c| c| c| }
  \thickhline
   Order & Rotation angle & Covolume   & Complex length \\ \thickhline
2 & $\pi$ & $l_1+l_2+l_4$ & $2(l_1+l_2+l_4) $ \\ \hline
3 & $2 \pi/3$ & $2l_3/3$  & $2 l_3+i \pi $  \\ \hline
3 & $2 \pi/3$ & $2l_6/3$   &$2 l_6+i \pi $   \\ \hline
5 & $2 \pi/5$ & $2l_5/5$  & $2l_5 $   \\ 
5 & $4 \pi/5$ & $2l_5/5$   & --- \\ 
 \thickhline
\end{tabular}
\caption{Elliptic conjugacy classes of  $\Gamma^+(T_2) $.}
\label{tab:T2-data}
\end{table}

The tetrahedron $T_2$ admits an order-2 rotational symmetry about the geodesic which is the perpendicular bisector of edges 2 and 5. This geodesic has length $2 l'$ where $l' \approx1.06128$.
We can extend $\Gamma^+(T_2)$ by this rotation to get a larger group. Denoting the rotation by $d$, this extended group has the presentation~\cite{MacLachlan-Reid}
\ba \label{eq:T2-extension}
\langle a, b, c, d \, | \, & a^{2}=b^{2} = c^{3} = d^2= (b c)^{2} = (ca)^{5} = (ab)^{3} =(db)^2=1, \, \nonumber \\
& d a d = bc, \, dcd=ba \rangle.
\ea
Table~\ref{tab:T2-Z2-data} contains information about the elliptic conjugacy classes of this group. The corresponding orbifold is the smallest compact orientable hyperbolic 3-orbifold~\cite{Marshall-Martin-II}, with volume $V \approx 0.03905$. 
\begin{table}
\centering
  \begin{tabular}{  | c |c | c| c| c| }
  \thickhline
   Order & Rotation angle & Covolume  & Complex length \\ \thickhline
   2 & $\pi$ & $l'$  & $2l' $  \\ \hline
2 & $\pi$ & $(l_1+l_2+l_4)/2$ & $l_1+l_2+l_4 + \frac{i \pi}{2} $  \\ \hline
3 & $2 \pi/3$ & $2l_3/3$  & $2 l_3+i \pi $  \\  \hline
5 & $2 \pi/5$ & $l_5/5$  &  $l_5 $ \\ 
5 & $4 \pi/5$ & $l_5/5$ & --- \\ 
 \thickhline
\end{tabular}
\caption{Elliptic conjugacy classes of the extension of $\Gamma^+(T_2)$ in~\eqref{eq:T2-extension}.}
\label{tab:T2-Z2-data}
\end{table}

\subsection{$T_3$}

Table~\ref{tab:T3-data} contains information about the elliptic conjugacy classes of $\Gamma^+(T_3)$. The orbifold $\Gamma^+(T_3) \backslash \mathbb{H}^3$ achieves the minimal volume in its commensurability class~\cite{MaclachlanReid1989}, so there are no quotients to consider.
\begin{table}
\centering
  \begin{tabular}{  | c |c | c| c| c| }
  \thickhline
   Order & Rotation angle & Covolume   & Complex lengths \\ \thickhline
2 & $\pi$ & $l_1+l_2$   & $2(l_1+l_2) $  \\ \hline
2 & $\pi$ & $l_4$ & $2l_4 $   \\ \hline
 4 & $\pi/2$ & $l_3/2$ & $2 l_3 $    \\  
 2 & $\pi$ & $l_3/2$ & ---  \\ \hline
3 & $2 \pi/3$ & $2l_5/3$ & $2 l_5 $   \\ \hline
5 & $2 \pi/5$ & $2l_6/5$ & $2l_6+i \pi $   \\ 
5 & $4 \pi/5$ & $2l_6/5$ &  --- \\ 
 \thickhline
\end{tabular}
\caption{Elliptic conjugacy classes of  $\Gamma^+(T_3) $.}
\label{tab:T3-data}
\end{table}

\subsection{$T_4$}

Table~\ref{tab:T4-data} contains information about the elliptic conjugacy classes of $\Gamma^+(T_4)$, as determined in~\cite{Lin-Lipnowski2020}.
\begin{table}
\centering
  \begin{tabular}{  | c |c | c| c| }
  \thickhline
   Order & Rotation angle & Covolume & Complex length   \\ \thickhline
2 & $\pi$ & $l_1+l_2+l_4$  & $2(l_1+l_2+l_4) $   \\ \hline
3 & $2 \pi/3$ & $2l_5/3$  & $2 l_5 $ \\  \hline
5 & $2 \pi/5$ & $2l_3/5$  & $2 l_3+ i \pi $ \\ 
5 & $4 \pi/5$ & $2l_3/5$  & --- \\ \hline
5 & $2 \pi/5$ & $2l_6/5$  & $2 l_6+ i \pi  $  \\ 
5 & $4 \pi/5$ & $2l_6/5$  & ---  \\  
 \thickhline
\end{tabular}
\caption{Elliptic conjugacy classes of  $\Gamma^+(T_4) $.}
\label{tab:T4-data}
\end{table}

We can extend $\Gamma^+(T_4)$ by the order-2 rotation about the perpendicular bisector of edges 2 and 5. This geodesic has length $2 l'$ with $l' \approx1.06128$. Denoting the extra rotation by $d$, the extended group has the presentation
\ba \label{eq:T4-extension}
\langle a, b, c, d \, | \, & a^{2}=b^{2} = c^{5} = d^2= (b c)^{2} = (ca)^{3} = (ab)^{5} =(db)^2=1, \\
& d a d = bc, \, dcd=ba \rangle.
\ea
Table~\ref{tab:T4-Z2-data} contains information about the elliptic conjugacy classes of this group. The corresponding orbifold has volume $V \approx 0.0933$.
\begin{table}
\centering
  \begin{tabular}{  | c |c | c| c| }
  \thickhline
   Order & Rotation angle & Covolume  & Complex length  \\ \thickhline
   2 & $\pi$ & $l'$ & $2l' $  \\ \hline
2 & $\pi$ & $(l_1+l_2+l_4)/2$  & $l_1+l_2+l_4 + \frac{i \pi}{2} $   \\ \hline
3 & $2 \pi/3$ & $l_5/3$   & $ l_5 $ \\  \hline
5 & $2 \pi/5$ & $2l_3/5$  & $2 l_3+ i \pi $  \\ 
5 & $4 \pi/5$ & $2l_3/5$ & --- \\ 
 \thickhline
\end{tabular}
\caption{Elliptic conjugacy classes of the extension of $\Gamma^+(T_4)$  in~\eqref{eq:T4-extension}.}
\label{tab:T4-Z2-data}
\end{table}

\subsection{$T_5$}

Table~\ref{tab:T5-data} contains information about the elliptic conjugacy classes of $\Gamma^+(T_5) $. 
\begin{table}
\centering
  \begin{tabular}{  | c |c | c| c|}
  \thickhline
   Order & Rotation angle & Covolume & Complex length  \\ \thickhline
2 & $\pi$ & $l_1$ & $2l_1+ \frac{i \pi }{2} $   \\ \hline
2 & $\pi$ & $l_4$  & $2l_4+ \frac{i \pi }{2} $ \\ \hline
4 & $\pi/2$ & $l_6/2$ & $2l_6 $   \\ 
2 & $\pi$ & $l_6/2$ & --- \\ \hline
3 & $2 \pi/3$ & $2(l_2+l_3+l_5)/3$ & $2 (l_2+l_3+l_5)  $   \\  
 \thickhline
\end{tabular}
\caption{Elliptic conjugacy classes of $\Gamma^+(T_5) $.}
\label{tab:T5-data}
\end{table}
\noindent 

We can extend $\Gamma^+(T_5)$ by the order-2 rotation about the perpendicular bisector of edges 3 and 6. This geodesic has length $2 l'$ for  $l' \approx 1.52857$. Denoting the extra rotation by $d$, the extended group has the presentation
\ba \label{eq:T5-extension}
\langle a, b, c, d \, | \, & a^{2}=b^{3} = c^{3} = d^2= (b c)^{2} = (ca)^{3} = (ab)^{4} =(dc)^2=1, \\
& d a d = bc, \, dbd=ac\rangle.
\ea
Table~\ref{tab:T5-Z2-data} contains information about the elliptic conjugacy classes of this group. The corresponding orbifold has volume $V \approx 0.08577$.
\begin{table}
\centering
  \begin{tabular}{  | c |c | c| c| c|}
  \thickhline
   Order & Rotation angle & Covolume & Complex length  \\ \thickhline
   2 & $\pi$ & $l'$ & $2l'$ \\ \hline
2 & $\pi$ & $l_1$ & $2l_1+ \frac{i \pi }{2} $  \\ \hline
4 & $\pi/2$ & $l_6/4$  & $l_6 + \frac{i \pi }{4} $   \\ 
2 & $\pi$ & $l_6/4$ & --- \\ \hline
3 & $2 \pi/3$ & $(l_2+l_3+l_5)/3$ & $l_2+l_3+l_5  $   \\  
 \thickhline
\end{tabular}
\caption{Elliptic conjugacy classes of the extension of $\Gamma^+(T_5)$  in~\eqref{eq:T5-extension}.}
\label{tab:T5-Z2-data}
\end{table}
\noindent

\subsection{$T_6$}

Table~\ref{tab:T6-data} contains information about the elliptic conjugacy classes of $\Gamma^+(T_6) $.
\begin{table}
\centering
  \begin{tabular}{  | c |c | c| c|}
  \thickhline
   Order & Rotation angle & Covolume  & Complex length \\ \thickhline
2 & $\pi$ & $l_1$  & $2l_1 $ \\ \hline
2 & $\pi$ & $l_4$ & $2l_4 $  \\ \hline
4 & $\pi/2$ & $l_3/2$ & $2l_3 $    \\ 
2 & $\pi$ & $l_3/2$ & ---\\ \hline
4 & $\pi/2$ & $l_6/2$  & $2l_6 $  \\ 
2 & $\pi$ & $l_6/2$ & --- \\ \hline
3 & $2 \pi/3$ & $2 l_2 /3$  & $2 l_2 $  \\  \hline
3 & $2 \pi/3$ & $2 l_5 /3$  & $2 l_5 $ \\  
 \thickhline
\end{tabular}
\caption{Elliptic conjugacy classes of $\Gamma^+(T_6)$.}
\label{tab:T6-data}
\end{table}

There are three order-2 rotational symmetries of $T_6$ that we can use to extend the group:
\begin{enumerate}
\item We can extend $\Gamma^+(T_6)$ by the order-2 rotation about the perpendicular bisector of edges 1 and 4. This geodesic has length $2 l_1'$ for  $l_1' \approx 0.63297$. Denoting the extra rotation by $d$, the extended group has the presentation
\ba \label{eq:T6-extension-a}
\langle a, b, c, d \, | \, & a^{2}=b^{3} = c^{4} = d^2= (b c)^{2} = (ca)^{3} = (ab)^{4} =(da)^2=1,  \\
& \, dbd=(ca)^{-1}, \, dcd=(ab)^{-1} \rangle.
\ea
Table~\ref{tab:T6-Z2-a-data} contains information about the elliptic conjugacy classes of this extension of $\Gamma^+(T_6) $. 
\item We can extend $\Gamma^+(T_6)$ by the order-2 rotation about the perpendicular bisector of edges 2 and 5. This geodesic has length $2 l_2'$ for  $l_2' \approx 0.88137$ and is topologically a circle. 
Denoting the extra rotation by $d$, the extended group has the presentation
\ba \label{eq:T6-extension-b}
\langle a, b, c, d \, | \, & a^{2}=b^{3} = c^{4} = d^2= (b c)^{2} = (ca)^{3} = (ab)^{4}=(db)^2 =1, \\
&  \, dad=bc,  \, dcd=ba \rangle.
\ea
Table~\ref{tab:T6-Z2-b-data} contains information about the elliptic conjugacy classes of this extension of $\Gamma^+(T_6) $.
\item We can extend $\Gamma^+(T_6)$ by the order-2 rotation about the perpendicular bisector of edges 3 and 6. This geodesic has length $2 l_3'$ for  $l_3' \approx 1.12838$. 
Denoting the extra rotation by $d$, the extended group has the presentation
\ba \label{eq:T6-extension-c}
\langle a, b, c, d \, | \, & a^{2}=b^{3} = c^{4} = d^2= (b c)^{2} = (ca)^{3} = (ab)^{4} =(dc)^2=1,  \\
&  \, dad=bc, \, dbd=ac \rangle.
\ea
Table~\ref{tab:T6-Z2-c-data} contains information about the elliptic conjugacy classes of this extension of $\Gamma^+(T_6) $. 
\end{enumerate}
The corresponding orbifolds all have volume $V \approx 0.2222$. We can also extend $\Gamma^+(T_6)$ by all three of the above rotations to get a larger group for which the corresponding orbifold has volume $V \approx 0.1111$. This orbifold has the smallest volume of any compact orientable arithmetic hyperbolic 3-orbifold associated with a quaternion algebra defined over an imaginary quadratic field~\cite{MaclachlanReid1989}. Table~\ref{tab:T6-Z4-data} contains information about the elliptic conjugacy classes of this larger extension of $\Gamma^+(T_6) $.
\begin{table}
\centering
  \begin{tabular}{  | c |c | c| c|}
  \thickhline
   Order & Rotation angle & Covolume  & Complex length \\ \thickhline
   2 & $\pi$ & $l_1'$  & $2l'_1 +0.417233 i$ \\ \hline
    2 & $\pi$ & $l_1'$  & $2l'_1 -0.417233 i$ \\ \hline
2 & $\pi$ & $l_1/2$  & $l_1+\frac{i \pi}{2} $ \\ \hline
2 & $\pi$ & $l_4/2$ & $l_4 +\frac{i \pi}{2} $  \\ \hline
4 & $\pi/2$ & $l_3/2$ & $2l_3 $    \\ 
2 & $\pi$ & $l_3/2$ & ---\\ \hline
3 & $2 \pi/3$ & $2 l_2 /3$  & $2 l_2 $  \\
 \thickhline
\end{tabular}
\caption{Elliptic conjugacy classes of the extension of $\Gamma^+(T_6)$ in~\eqref{eq:T6-extension-a}.}
\label{tab:T6-Z2-a-data}
\end{table}
\begin{table}
\centering
   \begin{tabular}{  | c |c | c| c|}
  \thickhline
   Order & Rotation angle & Covolume  & Complex length \\ \thickhline
      2 & $\pi$ & $2l_2'$  & $2l'_2 $ \\ \hline
2 & $\pi$ & $l_1$  & $2l_1 $ \\ \hline
4 & $\pi/2$ & $l_3/2$ & $2l_3 $    \\ 
2 & $\pi$ & $l_3/2$ & ---\\ \hline
3 & $2 \pi/3$ & $ l_2 /3$  & $ l_2 $  \\  \hline
3 & $2 \pi/3$ & $ l_5 /3$  & $ l_5 $ \\  
 \thickhline
\end{tabular}
\caption{Elliptic conjugacy classes of the extension of $\Gamma^+(T_6)$ in~\eqref{eq:T6-extension-b}. The hyperbolic geodesic for the first row is topologically a circle.}
\label{tab:T6-Z2-b-data}
\end{table}
\begin{table}
\centering
  \begin{tabular}{  | c |c | c| c|}
  \thickhline
   Order & Rotation angle & Covolume  & Complex length \\ \thickhline
   2 & $\pi$ & $l_3'$  & $2l'_3+0.594503i$ \\ \hline
    2 & $\pi$ & $l_3'$  & $2l'_3-0.594503 i$ \\ \hline
2 & $\pi$ & $l_1$  & $2l_1 $ \\ \hline
4 & $\pi/2$ & $l_3/4$ & $l_3 + \frac{i \pi}{4} $    \\ 
2 & $\pi$ & $l_3/4$ & ---\\ \hline
4 & $\pi/2$ & $l_6/4$  & $l_6+ \frac{i \pi}{4}  $  \\ 
2 & $\pi$ & $l_6/4$ & --- \\ \hline
3 & $2 \pi/3$ & $2 l_2 /3$  & $2 l_2 $  \\ 
 \thickhline
\end{tabular}
\caption{Elliptic conjugacy classes of the extension of $\Gamma^+(T_6)$ in~\eqref{eq:T6-extension-c}.}
\label{tab:T6-Z2-c-data}
\end{table}
\begin{table}
\centering
  \begin{tabular}{  | c |c | c| c|}
  \thickhline
   Order & Rotation angle & Covolume  & Complex length \\ \thickhline
      2 & $\pi$ & $l_1'/2$  & $l'_1+ 1.36218 i $ \\ \hline
    2 & $\pi$ & $l_1'/2$  & $l'_1- 1.36218 i $ \\ \hline
          2 & $\pi$ & $l_2'$  & $2l'_2 $ \\ \hline
   2 & $\pi$ & $l_3'/2$  & $l'_3+ 1.27354 i$ \\ \hline
    2 & $\pi$ & $l_3'/2$  & $l'_3- 1.27354 i $ \\ \hline
2 & $\pi$ & $l_1/2$  & $l_1+ \frac{i \pi}{2} $ \\ \hline
4 & $\pi/2$ & $l_3/4$ & $l_3 + \frac{i \pi}{4} $    \\ 
2 & $\pi$ & $l_3/4$ & ---\\ \hline
3 & $2 \pi/3$ & $ l_2 /3$  & $ l_2 $  \\ 
 \thickhline
\end{tabular}
\caption{Elliptic conjugacy classes of the extension of $\Gamma^+(T_6)$ by all three order-2 rotations around perpendicular bisectors of opposite edges of  $T_6$.}
\label{tab:T6-Z4-data}
\end{table}

\subsection{$T_7$}

Table~\ref{tab:T7-data} contains information about the elliptic conjugacy classes of $\Gamma^+(T_7) $.
\begin{table}
\centering
  \begin{tabular}{  | c |c | c| c| c|}
  \thickhline
   Order & Rotation angle & Covolume & Complex length  \\ \thickhline
2 & $\pi$ & $l_1$ & $2l_1+ \frac{i \pi}{2} $  \\ \hline
2 & $\pi$ & $l_4$  & $2l_4+ \frac{i \pi}{2} $  \\ \hline
3 & $2 \pi/3$ & $2(l_2+l_3+l_5)/3$  & $2 (l_2+l_3+l_5) $\\ \hline
5 & $2 \pi/5$ & $2l_6/5$ & $2l_6  $   \\ 
5 & $4 \pi/5$ & $2l_6/5$ & --- \\ 
 \thickhline
\end{tabular}
\caption{Elliptic conjugacy classes of $\Gamma^+(T_7) $.}
\label{tab:T7-data}
\end{table}

We can extend $\Gamma^+(T_7)$ by the order-2 rotation about the perpendicular bisector of edges 3 and 6. This geodesic has length $2 l'$ with $l' =1.061275$ and is topologically a circle.
Denoting the extra rotation by $d$, the extended group has the presentation
\ba \label{eq:T7-extension}
\langle a, b, c, d \, | \, & a^{2}=b^{3} = c^{3} = d^2= (b c)^{2} = (ca)^{3} = (ab)^{5} =(dc)^2=1, \\
& d a d = bc, \, dbd=ac \rangle.
\ea
Table~\ref{tab:T7-Z2-data} contains information about the elliptic conjugacy classes of this group. The corresponding orbifold has volume $V \approx 0.2053$.
\begin{table}
\centering
  \begin{tabular}{  | c |c | c| c| c|}
  \thickhline
   Order & Rotation angle & Covolume & Complex length  \\ \thickhline
   2 & $\pi$ & $2l'$ & $2l'  $  \\ \hline
2 & $\pi$ & $l_1$ & $2l_1+  \frac{i \pi}{2} $  \\ \hline
3 & $2 \pi/3$ & $(l_2+l_3+l_5)/3$  & $l_2+l_3+l_5$\\ \hline
5 & $2 \pi/5$ & $l_6/5$ & $l_6 $   \\ 
5 & $4 \pi/5$ & $l_6/5$ & --- \\ 
 \thickhline
\end{tabular}
\caption{Elliptic conjugacy classes of the extension of $\Gamma^+(T_7)$ in~\eqref{eq:T7-extension}. The hyperbolic geodesic for the first row is topologically a circle.}
\label{tab:T7-Z2-data}
\end{table}

\subsection{$T_8$}

Table~\ref{tab:T8-data} contains information about the elliptic conjugacy classes of $\Gamma^+(T_8)$. The elliptic conjugacy classes of $\Gamma(T_8)$ are given in~\cite{Aurich-Marklof}.
\begin{table}
\centering
  \begin{tabular}{  | c |c | c| c| c|}
  \thickhline
   Order & Rotation angle & Covolume  & Complex length  \\ \thickhline
2 & $\pi$ & $l_1$ & $2l_1 $    \\ \hline
2 & $\pi$ & $l_4$  & $2l_4 $  \\ \hline
4 & $\pi/2$ & $l_3/2$ & $2l_3 $   \\ 
2 & $\pi$ & $l_3/2$ & ---\\  \hline
3 & $2 \pi/3$ & $2l_2/3$ & $2 l_2 $    \\ \hline
3 & $2 \pi/3$ & $2l_5/3$  & $2 l_5 $  \\ \hline
5 & $2 \pi/5$ & $2l_6/5$ & $2l_6 $ \\ 
5 & $4 \pi/5$ & $2l_6/5$  & --- \\ 
 \thickhline
\end{tabular}
\caption{Elliptic conjugacy classes of $\Gamma^+(T_8)$.}
\label{tab:T8-data}
\end{table}

We can extend $\Gamma^+(T_8)$ by the order-2 rotation about the perpendicular bisector of edges 3 and 6. This geodesic has length $2 l'$ with $l' =2.76088$. Denoting the extra rotation by $d$, the extended group has the presentation
\ba \label{eq:T8-extension}
\langle a, b, c, d \, | \, & a^{2}=b^{3} = c^{4} = d^2= (b c)^{2} = (ca)^{3} = (ab)^{5} =(dc)^2=1,\\
& \, d a d = bc, \, dbd=ac \rangle.
\ea
Table~\ref{tab:T8-Z2-data} contains information about the elliptic conjugacy classes of this group. The corresponding orbifold has volume $V \approx 0.3587$.
\begin{table}
\centering
  \begin{tabular}{  | c |c | c| c|}
  \thickhline
   Order & Rotation angle & Covolume  & Complex length  \\ \thickhline
   2 & $\pi$ & $l'$ & $2l' $\\ \hline
2 & $\pi$ & $l_1$  & $2l_1 $   \\  \hline
4 & $\pi/2$ & $l_3/4$  & $l_3 + \frac{i \pi}{4} $   \\ 
2 & $\pi$ & $l_3/4$ & ---  \\  \hline
3 & $2 \pi/3$ & $2l_2/3$ & $2 l_2 $  \\ \hline
5 & $2 \pi/5$ & $l_6/5$  & $l_6 $  \\ 
5 & $4 \pi/5$ & $l_6/5$  & --- \\ 
 \thickhline
\end{tabular}
\caption{Elliptic conjugacy classes of the extension of $\Gamma^+(T_8)$ in~\eqref{eq:T8-extension}.}
\label{tab:T8-Z2-data}
\end{table}

\subsection{$T_9$}

Table~\ref{tab:T9-data} contains information about the elliptic conjugacy classes of $\Gamma^+(T_9)$.
\begin{table}
\centering
  \begin{tabular}{  | c |c | c| c|}
  \thickhline
   Order & Rotation angle & Covolume  & Complex length \\ \thickhline
2 & $\pi$ & $l_1$ & $2l_1 $   \\ \hline
2 & $\pi$ & $l_4$  & $2l_4$ \\ \hline
3 & $2 \pi/3$ & $2l_2/3$ & $2 l_2 $   \\ \hline
3 & $2 \pi/3$ & $2l_5/3$  & $2 l_5 $\\ \hline
5 & $2 \pi/5$ & $2l_3/5$ & $2l_3  $   \\ 
5 & $4 \pi/5$ & $2l_3/5$  & --- \\ \hline
5 & $2 \pi/5$ & $2l_6/5$ & $2l_6 $   \\ 
5 & $4 \pi/5$ & $2l_6/5$ & --- \\ 
 \thickhline
\end{tabular}
\caption{Elliptic conjugacy classes of $\Gamma^+(T_9)$.}
\label{tab:T9-data}
\end{table}

There are three order-2 rotational symmetries of $T_9$ that we can use to extend the group:
\begin{enumerate}
\item We can extend $\Gamma^+(T_9)$ by the order-2 rotation about the perpendicular bisector of edges 1 and 4. This geodesic has length $2 l_1'$ for  $l_1' \approx 0.767197$. 
Denoting the extra rotation by $d$, the extended group has the presentation
\ba \label{eq:T9-extension-a}
\langle a, b, c, d \, | \, & a^{2}=b^{3} = c^{5} = d^2= (b c)^{2} = (ca)^{3} = (ab)^{5} =(da)^2=1,\\
& \, dbd=(ca)^{-1}, \, dcd=(ab)^{-1} \rangle.
\ea
Table~\ref{tab:T9-Z2-a-data} contains information about the elliptic conjugacy classes of this extension of $\Gamma^+(T_9) $.
\item We can extend $\Gamma^+(T_9)$ by the order-2 rotation about the perpendicular bisector of edges 2 and 5. This geodesic has length $2 l_2'$ for  $l_2' \approx 1.06128$ and is topologically a circle. 
Denoting the extra rotation by $d$, the extended group has the presentation
\ba \label{eq:T9-extension-b}
\langle a, b, c, d \, | \, & a^{2}=b^{3} = c^{5} = d^2= (b c)^{2} = (ca)^{3} = (ab)^{5} =1, \\
&  \, dad=bc, \, dbd=b^{-1}, \, dcd=ba \rangle.
\ea
Table~\ref{tab:T9-Z2-b-data} contains information about the elliptic conjugacy classes of this extension of $\Gamma^+(T_9) $.
\item We can extend $\Gamma^+(T_9)$ by the order-2 rotation about the perpendicular bisector of edges 3 and 6. This geodesic has length $2 l_3'$ for  $l_3' \approx 1.61692$ and is topologically a circle. 
Denoting the extra rotation by $d$, the extended group has the presentation
\ba \label{eq:T9-extension-c}
\langle a, b, c, d \, | \, & a^{2}=b^{3} = c^{5} = d^2= (b c)^{2} = (ca)^{3} = (ab)^{5} =(dc)^2=1, \\
&  \, dad=bc, \, dbd=ac \rangle.
\ea
Table~\ref{tab:T9-Z2-c-data} contains information about the elliptic conjugacy classes of this extension of $\Gamma^+(T_9) $.
\end{enumerate}
The corresponding three orbifolds all have volume $V \approx 0.5021$. We can also extend $\Gamma^+(T_9)$ by all three of the above rotations to get a larger group for which the corresponding orbifold has volume $V \approx 0.2511$. Table~\ref{tab:T9-Z4-data} contains information about the elliptic conjugacy classes of this larger extension of $\Gamma^+(T_9) $.
\begin{table}
\centering
  \begin{tabular}{  | c |c | c| c|}
  \thickhline
   Order & Rotation angle & Covolume  & Complex length \\ \thickhline
   2 & $\pi$ & $l_1'$ & $2l_1'+ \frac{ i \pi}{5} $   \\ \hline
      2 & $\pi$ & $l_1'$ & $2l_1'-   \frac{ i \pi}{5}$   \\ \hline
2 & $\pi$ & $l_1/2$ & $l_1+ \frac{i \pi}{2} $   \\ \hline
2 & $\pi$ & $l_4/2$  & $l_4+ \frac{i \pi}{2}  $ \\ \hline
3 & $2 \pi/3$ & $2l_2/3$ & $2 l_2 $   \\ \hline
5 & $2 \pi/5$ & $2l_3/5$ & $2l_3 $   \\ 
5 & $4 \pi/5$ & $2l_3/5$  & --- \\ 
 \thickhline
\end{tabular}
\caption{Elliptic conjugacy classes of the extension of $\Gamma^+(T_9)$ in~\eqref{eq:T9-extension-a}.}
\label{tab:T9-Z2-a-data}
\end{table}
\begin{table}
\centering
  \begin{tabular}{  | c |c | c| c|}
  \thickhline
   Order & Rotation angle & Covolume  & Complex length \\ \thickhline
   2 & $\pi$ & $2l_2'$ & $2l_2' $   \\ \hline
2 & $\pi$ & $l_1$ & $2l_1 $   \\ \hline
3 & $2 \pi/3$ & $l_2/3$ & $l_2 $   \\ \hline
3 & $2 \pi/3$ & $l_5/3$  & $l_5$\\ \hline
5 & $2 \pi/5$ & $2l_3/5$ & $2l_3  $   \\ 
5 & $4 \pi/5$ & $2l_3/5$  & --- \\ 
 \thickhline
\end{tabular}
\caption{Elliptic conjugacy classes of the extension of $\Gamma^+(T_9)$ in~\eqref{eq:T9-extension-b}. The hyperbolic geodesic for the first row is topologically a circle.}
\label{tab:T9-Z2-b-data}
\end{table}
\begin{table}
\centering
  \begin{tabular}{  | c |c | c| c|}
  \thickhline
   Order & Rotation angle & Covolume  & Complex length \\ \thickhline
   2 & $\pi$ & $2 l_3'$ & $2l_3' $   \\ \hline
2 & $\pi$ & $l_1$ & $2l_1$   \\ \hline
3 & $2 \pi/3$ & $2l_2/3$ & $2 l_2$   \\ \hline
5 & $2 \pi/5$ & $l_3/5$ & $l_3  $   \\ 
5 & $4 \pi/5$ & $l_3/5$  & --- \\ \hline
5 & $2 \pi/5$ & $l_6/5$ & $l_6$   \\ 
5 & $4 \pi/5$ & $l_6/5$ & --- \\ 
 \thickhline
\end{tabular}
\caption{Elliptic conjugacy classes of the extension of $\Gamma^+(T_9)$ in~\eqref{eq:T9-extension-c}. The hyperbolic geodesic for the first row is topologically a circle.}
\label{tab:T9-Z2-c-data}
\end{table}
\begin{table}
\centering
  \begin{tabular}{  | c |c | c| c|}
  \thickhline
   Order & Rotation angle & Covolume  & Complex length \\ \thickhline
   2 & $\pi$ & $l_1'/2$ & $l_1'+ \frac{2 \pi i}{5} $   \\ \hline
      2 & $\pi$ & $l_1'/2$ & $l_1'- \frac{2 \pi i}{5}  $   \\ \hline
        2 & $\pi$ & $l_2'$ & $2l_2'$   \\ \hline
        2 & $\pi$ & $l_3'$ & $2l_3'$   \\ \hline
2 & $\pi$ & $l_1/2$ & $l_1+ \frac{i \pi}{2} $   \\ \hline
3 & $2 \pi/3$ & $l_2/3$ & $l_2 $   \\ \hline
5 & $2 \pi/5$ & $l_3/5$ & $l_3 $   \\ 
5 & $4 \pi/5$ & $l_3/5$  & --- \\ 
 \thickhline
\end{tabular}
\caption{Elliptic conjugacy classes of the extension of $\Gamma^+(T_9)$ by all three order-2 rotations around perpendicular bisectors of opposite edges of  $T_9$.}
\label{tab:T9-Z4-data}
\end{table}

\newpage
\bibliography{refs.bib}
\bibliographystyle{jhep}

\end{document}